\documentclass [leqno, 10pt]{amsart}

\usepackage{amssymb, amsmath, amsfonts, amsthm, amsbsy, amscd}
\usepackage[bbgreekl]{mathbbol} %
\usepackage[mathscr]{euscript}
\usepackage{esint}

\usepackage{url}
\usepackage[linktocpage]{hyperref}
\newtheorem{theorem}{Theorem}[section]
\newtheorem*{theorem*}{Theorem}{\,}
\newtheorem{corollary}{Corollary}[section]
\newtheorem{lemma}{Lemma}[section]
\newtheorem{definition}{Definition}[section]
\newtheorem*{definition*}{Definition}{}

\theoremstyle{definition}

\newtheorem{remark}{Remark}[section]

\numberwithin{equation}{section}

\newcommand{\quat}{\mathbb{H}}
\newcommand{\cayley}{\mathbb{O}}
\newcommand{\met}{\mathsf{h}}
\newcommand{\ctr}{\mathcal{Z}}
\newcommand{\uR}{\mathsf{R}}
\newcommand{\ideal}{\mathscr{I}}
\newcommand{\fie}{\mathbb{k}}
\newcommand{\cubic}{\mathcal{C}}
\newcommand{\lich}{\Delta_{L}}
\newcommand{\culap}{\square}
\newcommand{\ck}{\mathsf{CK}}
\newcommand{\symat}{S_{0}}
\newcommand{\ih}{\mathsf{i}_{h}}
\newcommand{\tliea}{\mathscr{T}^{\ast}}
\newcommand{\tlie}{\mathscr{T}}

\newcommand{\sN}{\mathscr{N}}
\newcommand{\sT}{\mathscr{T}}

\newcommand{\uf}{U}

\renewcommand{\r}{\mathsf{r}}

\renewcommand{\H}{\mathsf{H}}
\newcommand{\so}{\mathfrak{so}}

\newcommand{\hgss}{\hat{\mathbb{g}}}

\newcommand{\gss}{\mathbb{g}}

\newcommand{\sflat}{\curlyvee}
\newcommand{\ssharp}{\curlywedge}

\newcommand{\affhar}{\mathfrak{affhar}}
\newcommand{\projhar}{\mathfrak{projhar}}

\newcommand{\nahm}{\mathfrak{Nahm}}
\newcommand{\mtens}{\mu}
\newcommand{\alg}{\mathbb{A}}
\newcommand{\halg}{\hat{\mathbb{A}}}
\newcommand{\balg}{\mathbb{B}}
\newcommand{\mprod}{\circ}
\newcommand{\pol}{\mathsf{Pol}}
\newcommand{\h}{\mathfrak{h}}
\newcommand{\har}{\mathsf{Har}}

\newcommand{\hol}{\mathsf{hol}}
\newcommand{\dev}{\mathsf{dev}}

\newcommand{\cut}{\mathsf{Cut}}
\newcommand{\cprod}{\circledcirc}
\newcommand{\sprod}{\odot}
\newcommand{\sbl}{\mathsf{\si}}

\newcommand{\tf}{\mathsf{tf}}

\renewcommand{\div}{\mathsf{div}\,}

\newcommand{\symk}{S^{k}(\ctm)}

\newcommand{\symkt}{S^{k}_{0}(\ctm)}
\newcommand{\symlt}{S^{l}_{0}(\ctm)}
\newcommand{\symkp}{S^{k+1}(\ctm)}
\newcommand{\symkpt}{S^{k+1}_{0}(\ctm)}

\newcommand{\symkmt}{S^{k-1}_{0}(\ctm)}

\newcommand{\symktv}{S^{k}_{0}(TM)}

\newcommand{\symkptv}{S^{k+1}_{0}(TM)}

\newcommand{\cmf}{\mathscr{C}}
\newcommand{\curvmod}{\mathscr{C}}

\newcommand{\om}{\omega}

\newcommand{\mlt}{\circ}
\newcommand{\hmlt}{\hat{\circ}}
\newcommand{\tbt}{\mathcal{C}}

\newcommand{\hess}{\mathsf{Hess}\,}

\renewcommand{\S}{\mathcal{S}}
\newcommand{\adj}{\mathsf{adj}}
\newcommand{\monge}{\mathcal{M}}

\newcommand{\vr}{\delta}

\newcommand{\mt}{\mathcal{M}}
\newcommand{\ten}{[\tilde{\nabla}]}
\newcommand{\hrho}{\hat{\rho}}
\newcommand{\vnabla}{\nabla(t)}
\newcommand{\ven}{[\vnabla]}

\newcommand{\hsh}{\star_{h}\,}
\newcommand{\hsth}{\star_{\tilde{h}}\,}

\newcommand{\hodge}{\square}
\newcommand{\dad}{d^{\ast}}
\newcommand{\derham}{\mathsf{dR}}
\newcommand{\vol}{\mathsf{vol}}

\newcommand{\Pis}{\Pi^{\sharp}}
\newcommand{\tD}{\tilde{D}}

\newcommand{\ka}{\kappa}

\newcommand{\lf}{\mathsf{L}}

\newcommand{\lag}{\mathsf{Lag}}
\newcommand{\bach}{\mathcal{O}}
\def\lbt{\ensuremath{\mkern+6mu{\mkern-6mu ^{\lambda}\mkern-1mu}}}
\def\ibt{\ensuremath{\mkern+6mu{\mkern-6mu ^{6-n}\mkern-1mu}}}
\def\twobt{\ensuremath{\mkern+6mu{\mkern-6mu ^{2}\mkern-1mu}}}
\newcommand{\bachl}{\lbt{\bach}}
\newcommand{\bachi}{\ibt{\bach}}
\newcommand{\bachtwo}{\twobt{\bach}}
\newcommand{\lbachl}{\lbt{\mathcal{P}}}
\newcommand{\lbachi}{\ibt{\mathcal{P}}}

\newcommand{\sR}{\mathscr{R}}
\newcommand{\sW}{\mathscr{W}}

\newcommand{\sric}{\mathscr{R}\mathit{ic}}
\newcommand{\qA}{\mathscr{Q}_{A^{\flat}}}
\newcommand{\qT}{\mathscr{Q}_{\mathscr{T}}}
\newcommand{\qR}{\mathscr{Q}_{\mathscr{R}}}
\newcommand{\qW}{\mathscr{Q}_{\mathscr{W}}}
\newcommand{\qH}{\mathscr{Q}_{\mathscr{H}}}
\newcommand{\sH}{\mathscr{H}}
\newcommand{\qY}{\mathscr{Q}_{\mathscr{Y}}}
\newcommand{\q}{\mathscr{Q}}
\newcommand{\sY}{\mathscr{Y}}
\newcommand{\clap}{\mathcal{C}}

\renewcommand{\part}{\vdash}

\newcommand{\rad}{\mathbb{E}}

\newcommand{\sile}{\sigma^{\lambda, \epsilon}}

\newcommand{\Id}{\text{Id}}

\newcommand{\dum}{\,\cdot\,\,}
\newcommand{\Ga}{\Gamma}

\newcommand{\ctn}{T^{\ast}N}
\newcommand{\mob}{\mathcal{M}}

\newcommand{\ric}{\mathsf{Ric}}

\newcommand{\Q}{\mathcal{Q}}
\newcommand{\sh}{\mathsf{S}}

\newcommand{\bD}{\bar{D}}

\newcommand{\nmb}{\mathsf{N}}
\newcommand{\nmp}{\nu}

\newcommand{\nm}{\mathsf{W}}

\newcommand{\gfunc}{\mathcal{G}}

\newcommand{\lap}{\Delta}

\renewcommand{\j}{\mathsf{i}}

\newcommand{\taut}{\mathbb{O}}

\newcommand{\la}{\lambda}
\newcommand{\ep}{\epsilon}

\newcommand{\zmodtwo}{\mathbb{Z}/2\mathbb{Z}}
\newcommand{\reat}{\mathbb{R}^{\times}}

\newcommand{\reap}{\mathbb{R}^{+}}
\newcommand{\ext}{\Omega}
\newcommand{\cinf}{C^{\infty}}

\newcommand{\Det}{\text{Det}\,}

\newcommand{\ben}{[\Bar{\nabla}]}

\newcommand{\htau}{\hat{\tau}}

\newcommand{\eno}{\text{End}}

\newcommand{\si}{\sigma}

\newcommand{\pr}{\partial}

\newcommand{\ctm}{T^{\ast}M}

\newcommand{\sign}{\text{sgn}}
\newcommand{\bnabla}{\bar{\nabla}}

\newcommand{\enb}{\{\nabla\}}
\newcommand{\enbr}{\langle \nabla \rangle}
\def\op{\ensuremath{\mkern+6mu{\mkern-6mu ^{\mathsf{op}}}}}

\def\brt#1{\ensuremath{\mkern+0mu{\mkern+1mu ^{t}\mkern-2mu#1}}}
\def\brtopt{\ensuremath{\mkern+6mu{\mkern-6mu ^{1-t}\mkern-3mu}}}
\def\brtone{\ensuremath{\mkern+6mu{\mkern-6mu ^{1}\mkern-3mu}}}
\def\brtzero{\ensuremath{\mkern+6mu{\mkern-6mu ^{0}\mkern-3mu}}}
\def\brtone{\ensuremath{\mkern+6mu{\mkern-6mu ^{1}\mkern-3mu}}}

\def\brtop{\,{\mkern-6mu ^{1-t}\mkern1mu}}

\def\brtb{{\mathchar'27\mkern-6mu}}
\newcommand{\lcp}{\mkern-6mu\brtb\mkern2mu\nabla}
\def\mr#1{\{\mkern-1mu#1\}}

\def\tbfs#1#2{\textbf{#1}\index{#2!#1}}
\def\tbsf#1#2{\textbf{#1}\index{#1!#2}}
\def\tbf#1{\textbf{#1}\index{#1}}

\def\tbss#1#2{\textbf{#1 #2}\index{#2!#1}}

\newcommand{\en}{[\nabla]}
\newcommand{\pen}{[\pnabla]}

\newcommand{\pnabla}{\brt\nabla}

\def\bkrt{\ensuremath{\mkern+6mu{\mkern-6mu ^{t}\mkern-3mu}}}
\newcommand{\pktnabla}{\bkrt\mkern1mu\tilde{\nabla}}
\newcommand{\pkten}{[\pktnabla]}

\newcommand{\lie}{\mathfrak{L}}
\newcommand{\clie}{\mathscr{L}}
\newcommand{\klie}{\mathscr{K}}
\newcommand{\kliea}{\mathscr{K}^{\ast}}

\newcommand{\re}{\text{Re\,}}
\newcommand{\im}{\text{Im\,}}

\renewcommand{\xi}{\frac{\partial}{\partial x^{i}}}

\newcommand{\F}{\mathcal{F}}

\newcommand{\C}{\mathcal{C}}

\newcommand{\B}{\mathcal{B}}

\newcommand{\lb}{\langle}
\newcommand{\ra}{\rangle}

\newcommand{\ste}{\mathbb{V}}

\newcommand{\sted}{\mathbb{V}^{\ast}}

\newcommand{\stedz}{\mathbb{V}^{\ast}\setminus\{0\}}
\newcommand{\projp}{\mathbb{P}^{+}}

\newcommand{\A}{\mathcal{A}}
\newcommand{\al}{\alpha}
\newcommand{\be}{\beta}
\newcommand{\ga}{\gamma}

\newcommand{\spn}{\text{Span}\,}

\newcommand{\emf}{\mathcal{E}}

\newcommand{\hnabla}{\widehat{\nabla}}
\newcommand{\tnabla}{\tilde{\nabla}}

\newcommand{\sll}{\mathfrak{sl}}

\newcommand{\eul}{\mathbb{X}}

\newcommand{\proj}{\mathbb{P}}

\newcommand{\con}{\mathsf{con}}
\newcommand{\vect}{\mathfrak{vec}}

\DeclareMathOperator{\diff}{Diff}
\DeclareMathOperator{\Aut}{Aut}

\newcommand{\g}{\mathfrak{g}}

\newcommand{\ad}{\text{ad}}

\newcommand{\tensor}{\otimes}

\newcommand{\rea}{\mathbb R}
\newcommand{\com}{\mathbb C}

\newcommand{\tr}{\mathsf{tr} \,}

\DeclareMathOperator{\aut}{\mathfrak{aut}}
\newcommand{\ann}{\text{Ann}\,}
\newcommand{\K}{\mathcal{F}}
\newcommand{\ct}{\mathcal{K}}
\newcommand{\bt}{\mathcal{L}}
\newcommand{\nbt}{\mathcal{L}}
\renewcommand{\L}{\mathscr{L}}

\newcommand{\defeq}{:=}

\renewcommand{\P}{\mathcal{P}}
\newcommand{\inter}{\mathcal{I}}
\newcommand{\dens}{\mathcal{V}}

\def\crn#1#2{{\vcenter{\vbox{
        \hbox{\kern#2pt \vrule width.#2pt height#1pt
           }
          \hrule height.#2pt}}}}

\begin{document}
\title[Einstein AH structures]{Geometric structures modeled on affine hypersurfaces and generalizations of the Einstein Weyl and affine hypersphere equations}

\author{Daniel J.~F. Fox} 
\address{Departamento de Matem\'atica Aplicada\\ EUIT Industrial\\ Universidad Polit\'ecnica de Madrid\\Ronda de Valencia 3\\ 28012 Madrid Espa\~na}
\email{daniel.fox@upm.es}

\begin{abstract}
An affine hypersurface (AH) structure is a pair comprising a conformal structure and a projective structure such that for any torsion-free connection representing the projective structure the completely trace-free part of the covariant derivative of any metric representing the conformal structure is completely symmetric. AH structures simultaneously generalize Weyl structures and abstract the geometric structure determined on a non-degenerate co-oriented hypersurface in flat affine space by its second fundamental form together with either the projective structure induced by the affine normal or that induced by the conormal Gau\ss \, map. There are proposed notions of Einstein equations for AH structures which for Weyl structures specialize to the usual Einstein Weyl equations and such that the AH structure induced on a non-degenerate co-oriented affine hypersurface is Einstein if and only if the hypersurface is an affine hypersphere. It is shown that a convex flat projective structure admits a metric with which it generates an Einstein AH structure, and examples are constructed on mean curvature zero Lagrangian submanifolds of certain para-K\"ahler manifolds. The rough classification of Riemannian Einstein Weyl structures by properties of the scalar curvature is extended to this setting. Known estimates on the growth of the cubic form of an affine hypersphere are partly generalized. The Riemannian Einstein equations are reformulated in terms of a given background metric as an algebraically constrained elliptic system for a cubic tensor. From certain commutative nonassociative algebras there are constructed examples of exact Riemannian signature Einstein AH structures with self-conjugate curvature but which are not Weyl and are neither projectively nor conjugate projectively flat. 
\end{abstract}

\maketitle

\setcounter{tocdepth}{2}  %

{\footnotesize }
\tableofcontents

\section{Introduction and overview of results}
An affine hypersphere is a non-degenerate hypersurface in flat affine space the affine normals of which meet in a point (which may be at infinity), its center. These are the umbilical hypersurfaces in affine geometry. Due principally to examples constructed by E. Calabi in \cite{Calabi-completeaffine} and work of S.Y. Cheng and S.T. Yau, \cite{Cheng-Yau-mongeampere, Cheng-Yau-realmongeampere, Cheng-Yau-affinehyperspheresI}, resolving a precise conjecture made by Calabi in \cite{Calabi-completeaffine}, it is known that, unlike the situation for Euclidean umbilics, there is an abundance of these hypersurfaces which are not hyperquadrics. In particular the interior of the cone over a properly convex domain is foliated by properly embedded hyperbolic affine hyperspheres asymptotic to the boundary of the cone and having center at the vertex of the cone. This means there is a hyperbolic affine hypersphere canonically associated to the universal cover of a convex flat real projective structure, and this in part explains their importance.

The idea motivating this paper is that the equations distinguishing the affine hyperspheres among all hypersurfaces in flat affine space should admit a formulation as the Einstein equations for some geometric structure induced on such hypersurfaces. The resulting intrinsically defined structures are called here AH (affine hypersurface) structures to reflect their origin, as indeed there is induced on every co-oriented non-degenerate immersed hypersuface in flat affine space a pair of such structures (one via the affine normal, and one via the conormal Gau\ss\, map). Moreover, there is a notion of Einstein equations for AH structures which for Weyl structures specializes to the usual Einstein Weyl equations.

Weyl structures are geometric structures generalizing the notion of a positive homothety class of pseudo-Riemannian metrics. The usual definition is that a Weyl structure comprises a torsion-free affine connection $\nabla$ compatible with a conformal class of metrics $[h]$ in the sense that for each representative metric $h \in [h]$ there is a one-form $\ga_{i}$ such that $\nabla_{i}h_{jk} = 2\ga_{i}h_{jk}$. Einstein Weyl structures generalized Einstein metrics and have been studied heavily. AH structures can also be viewed as a direct generalization of Weyl structures obtained by relaxing the compatibility requirement between $\nabla$ and $[h]$. There is an involutive notion of conjugacy of AH structures such that the two AH structures induced on an affine hypersurface are conjugate, and the self-conjugate AH structures are exactly the Weyl structures.

The central theme of this paper is the study of a subclass of AH structures called \textbf{Einstein} which has the following properties:
\begin{itemize}
\item The Einstein AH equations specialize for Weyl structures to the usual Einstein Weyl equations. 
\item The AH structure conjugate to an Einstein AH structure is Einstein.
\item Each member of the pair of conjugate AH structures induced on a hypersurface in flat affine space is Einstein if and only if the hypersurface is an affine hypersphere.
\item A convex flat real projective structure carries a canonically related pair of conjugate Einstein AH structures.
\item There are Einstein AH structures which are not Weyl and which are not locally equivalent to ones induced on affine hyperspheres.
\end{itemize}
Thus Einstein AH structures provide a framework accomodating simultaneously the usual Einstein Weyl structures (and hence the usual Einstein manifolds) and convex flat real projective structures, but including other spaces as well. 

Most of the basic structures and notions associated to a conformal class of metrics (more generally, a Weyl structure), e.g. the conformal Weyl tensor or the Yamabe problem, admit generalizations to AH structures. On the other hand, there appear for AH structures some new objects, which are trivial for Weyl structures. In particular, while the Einstein AH equations generalize the Einstein Weyl equations, their definition is not the most naive generalization of the Einstein Weyl condition, which turns out to be inadequate, and in order to obtain a tractable theory, it is necessary to impose a further condition which for an Einstein Weyl structure on a manifold of dimension at least $3$ is automatic. The extra condition does, however, directly generalize the condition made the definition of the two-dimensional Einstein Weyl equations by D. Calderbank in \cite{Calderbank-mobius}. 
 
The purpose of the present paper is to define AH structures, define the Einstein equations for AH structures, present some examples, and begin the work of relating properties of these structures to curvature conditions. An AH structure is said to have self-conjugate curvature if its curvature tensor is identical with that of the conjugate AH structure. It is for the class of Einstein AH structures with self-conjugate curvature that the strongest results are obtained, and it may come to be seen that this is the class of AH structures on which attention should be focused. In some still obscure sense the projectively flat Einstein AH structures are real analogues of extremal K\"ahler metrics, and the general contours of the theories describing the two bear some resemblance. This paper focuses on the Riemannian signature case, but the Lorentzian signature case is certainly interesting as well. While much more work needs to be done, particularly in regards to the systematic construction of compact examples, formulations of the Einstein equations are obtained which should be amenable to further study, and it is hoped that the article will convince the reader that AH structures merit further attention. For example, it seems that the formalism offers possibilities for bringing elliptic PDE techniques to bear on the study of convex flat projective structures. A large number of problems are suggested, and the subject seems quite rich.

In the remainder of the introduction the contents of the paper are described in more detail and some context is given. The overall organization of the paper should be evident from the table of contents.

\subsubsection{}
In sections \ref{ahsection} and \ref{codazziprojectivesection} there are given the basic definitions of AH structures and their conjugates, and the basic properties of the curvature of AH structures are worked out in detail. 

A choice of a transverse subbundle along a smoothly immersed hypersurface $M$ in the $(n+1)$-dimensional flat real affine space $\ste$ splits the exact sequence defining the normal bundle and determines on the hypersurface a torsion-free affine connection $\nabla$. When the second fundamental form of the hypersurface (viewed as a normal bundle valued covariant symmetric two tensor) is non-degenerate a distinguished affine normal subbundle is determined by imposing a compatibility condition on $\nabla$ and the second fundamental form. Locally identifying the second fundamental form with a metric $H_{ij}$ taking values in the line bundle of $2/n$-densities, the requirement is that $\nabla_{i}H_{jk}$ be trace-free (traces are taken with the bivector $H^{ij}$ dual to $H_{ij}$). If the immersion is co-oriented then the second fundamental form induces a conformal class $[h]$. %
 The conormal Gau\ss\, map sends a point of $M$ to the ray in the oriented projectivization $\projp(\sted)$ of the dual space $\sted$ determined by the annhilator of the tangent space to $M$ at $p$ and the given co-orientation. The pullback via the conormal Gau\ss\, map of the flat projective structure on $\projp(\sted)$ gives a flat projective structure $\ben$ on $M$. The projective structure $\ben$ is in a certain sense dual to the projective structure $\en$ generated by $\nabla$; there is a unique representative $\bnabla \in \ben$ such that $\bnabla_{i}H_{jk}$ is trace-free, and the difference tensor $\bnabla - \nabla$ is simply $H^{kp}\nabla_{i}H_{jp}$. The triple $(\en, [h], \ben)$ is canonically associated to a co-oriented non-degenerate hypersurface immersion in flat affine space. All these structures are invariant under affine motions of the ambient affine space, and all the affinely invariant geometry of the hypersurface is encoded in them. 

Given a pair $(\en, [h])$ comprising a projective structure $\en$ and a conformal structure $[h]$ there is a representative $\nabla \in \en$, called \textbf{aligned}, distinguished by the requirement that $\nabla_{i}H_{jk}$ be trace-free, where the normalized (density valued) representative  $H_{ij} = |\det h|^{-1/n}h_{ij}$ of the conformal structure does not depend on the choice of $h \in [h]$. The pair $(\en, [h])$ is a \textbf{generalized affine hypersurface (AH) structure} if moreover $\nabla_{i}H_{jk}$ is completely symmetric, in which case the weighted tensor $\bt_{ij}\,^{k} = H^{kp}\nabla_{i}H_{jp}$ is called its \textbf{cubic torsion}. There is a built in notion of (involutive) conjugacy of AH structures, generalizing the relation between the AH structures $(\ben, [h])$ and $(\en, [h])$ of the previous paragraph, and such that the self-conjugate AH structures are exactly the usual Weyl structures. Precisely, the aligned representative $\bnabla$ of the AH structure $(\ben, [h])$ conjugate to a given AH structure $(\en, [h])$ is defined by $\bnabla = \nabla + \bt_{ij}\,^{k}$.

An AH structure is \textbf{exact} if the aligned representative $\nabla$ makes parallel some non-vanishing density. In this case there is distinguished a positive homothety class of representatives of $h$ by the requirement that $\nabla |\det h| = 0$. If the AH structure is not exact, then to each $h \in [h]$ there is associated a one-form $\ga_{i}$ called the \textbf{Faraday primitive} and defined by $2n\ga_{i} = h^{pq}\nabla_{i}h_{pq}$. The differential $F_{ij} = -d\ga_{ij}$ does not depend on the choice of $h$ and is called the \textbf{Faraday two-form}. These notions are exacly as for Weyl structures.

The curvature $R_{ijk}\,^{l}$ of the AH structure is the curvature of $\nabla$. The sign and index conventions are consistent with $2\nabla_{[i}\nabla_{j]}X^{k} = R_{ijp}\,^{k}X^{p}$ (see section \ref{preliminariessection} for the notational conventions). The usual Ricci trace $R_{ij} = R_{pij}\,^{p}$ is not the only rank two trace of the curvature tensor. The trace $Q_{ij} = H^{pq}R_{ipqj}$ must also be considered, although there is only one (density-valued) scalar trace $R = H^{ij}R_{ij} = H^{ij}Q_{ij}$. Replacing the AH structure by its conjugate leaves some linear combination of these traces unchanged, and sends some linear combination to its opposite. More generally, the curvature can be decomposed into its self-conjugate and anti-self-conjugate parts, and conditions on the curvature which are invariant under conjugacy are particularly natural geometrically. An AH structure is \textbf{(conjugate) projectively flat} if $\en$ (resp. $\ben$) is projectively flat. For example, the AH structure induced on a non-degenerate co-oriented hypersurface in flat affine space is conjugate projectively flat because the conjugate projective structure $\ben$ is that induced via the conormal Gau\ss\, map by pullback of a flat projective structure. 

As is nicely explained in \cite{Schoen-variationaltheory} one way to try to find Einstein Riemannian metrics is the following. Given a compact manifold $M$ let $\si(M) = \sup_{[g]}\inf_{g \in [g]}\left(\vol_{g}(M)^{(2-n)/n}\int_{M} \sR_{g}\,d\vol_{g}\right)$, where the sup is taken over all conformal structures on $M$, and $\sR_{g}$ is the scalar curvature of the metric $g$. The expectation, which is often true, is that a metric for which the $\vol_{g}(M)^{(2-n)/n}\int_{M} \sR_{g}\,d\vol_{g} = \si(M)$ is Einstein. The tendency is that a metric critical for the normalized total scalar curvature minimizes within its conformal class but maximizes with respect to transverse variations. This suggests the desirability of having a structure which is to AH structures as a conformal structure is to its representative metrics. In this analogy a representative metric should be thought of as corresponding to a Weyl structure. There is such a structure, and it is called here by the not entirely pleasant name \textbf{Codazzi projective structure}. These are described in section \ref{codazziprojectivesection}. 

Given a conformal structure $[h]$ two torsion-free affine connections are \textbf{conformal projectively equivalent} if their difference tensor is pure trace, traces being taken using the conformal structure, and a conformal projective equivalence class of connections each of which generates with $[h]$ an AH structure is called a \textbf{Codazzi projective structure}. One should think of a conformal class of projective structures; if the aligned representatives of two AH structures are conformal projectively equivalent then their difference tensor has the form $2\al_{(i}\delta_{j)}\,^{k} - H_{ij}\al^{k}$ (Theorem \ref{conformalprojectivediff}), which is the conformal action of the one-form $\al_{i}$ on the space of connections. In this case the two AH structures are said to be \textbf{subordinate} to the Codazzi projective structure which their aligned representatives generate; such AH structures have the same cubic torsion. On a non-degenerate hypersurface in flat affine space the connections induced by different choices of transverse subbundle are conformal projectively equivalent, so determine a Codazzi projective structure; the condition defining the affine normal subbundle selects a particular subordinate AH structure. 

While Codazzi projective structures are not themselves studied in great detail, they serve a useful role in organizing the analysis of the curvature of AH structures. In particular, it makes sense to speak of tensors and operators associated to an AH structure which are invariant in the sense that they depend only on the underlying Codazzi projective structure. In general any conformally invariant tensor or operator has a generalization of this sort, although there appear some new tensors and operators which are identically null in the usual conformal setting. In particular there are for an AH structure analogues $W_{ijk}\,^{l}$ and $W_{ijk}$ of the usual conformal Weyl and conformal Cotton tensors; however now these decompose into their self-conjugate parts $A_{ijk}\,^{l}$ and $A_{ijk}$ and anti-self-conjugate parts $E_{ijk}\,^{l}$ and $E_{ijk}$, the latter necessarily vanishing for Weyl structures. If $A_{ijk}\,^{l}$ or $E_{ijk}\,^{l}$ vanishes then $A_{ijk}$ or $E_{ijk}$ is invariant. In particular, the tensor $A_{ijk}\,^{l}$ necessarily vanishes in $3$ dimensions, so in this case $A_{ijk}$ is invariant. It came as a bit of a surprise that in dimension at least $3$ the trace $A_{i} = A_{ipq}H^{pq}$ is always invariant, and need not be zero. As will now be explained, it seems that any reasonable generalization of the Einstein equations has to include the condition $A_{i} = 0$ (this is automatic for affine hypersurfaces). An AH structure satisfying $A_{i} = 0$ will be called \textbf{conservative}.

\subsubsection{}
In section \ref{einsteinahsection} the Einstein equations are defined, their two-dimensional specialization is briefly discussed, and there are described some important examples.

For an affine hypersurface of dimension at least $3$, the AH structure induced via the affine normal is projectively flat if and only if the hypersurface is an affine hypersphere. In two dimensions while it is still true that the AH structure on an affine hypersphere must be projectively flat, the converse is not true. Rather, in two dimensions an affine hypersurface is an affine hypersphere if and only if the induced AH structure is projectively flat and has self-conjugate curvature. This characterization of affine hyperspheres in dimension at least $3$ suggests as a possible notion of Einstein equations the requirement that an AH structure be both projectively flat and conjugate projectively flat. However, this is much too strong as it would amount in essence to defining Einstein AH structures to be those covered by affine hyperspheres; the spirit of this condition is similar to the condition that a metric be conformally flat Einstein. Moreover, it gives the wrong notion in dimension $2$. 

A less restrictive conclusion is that any notion of Einstein such that the AH structures induced on affine hyperspheres are Einstein must be preserved under conjugacy. All rank two traces of the curvature of an AH structure are expressible as linear combinations of the self-conjugate and anti-self-conjugate parts of $R_{ij}$, which in turn are linear combinations of $R_{ij}$ and $Q_{ij}$. Any definition of Einstein AH generalizing the Einstein Weyl equations has to include the condition that the symmetric trace-free part of $R_{ij}$ vanish, and if such a definition is to be preserved by conjugacy, this means, by the preceeding remark, that it must require the vanishing of the symmetric trace-free part of any rank two trace of the curvature. The \textbf{naive Einstein} equations require that the symmetrized tensors $R_{(ij)}$ and $Q_{(ij)}$ be pure trace, or, what is equivalent, that any linear combination of symmetric parts of rank two traces of the curvature be pure trace. However, for a number of reasons these conditions are insufficiently strong. 

Most importantly, they do not in general impose any condition resembling the constancy of the scalar curvature and they do not restrict to the two-dimensional Einstein Weyl equations studied by D. Calderbank in \cite{Calderbank-mobius}. In dimension at least $3$ the traced differential Bianchi identity implies that the Einstein modification $\sR_{ij} - \tfrac{1}{2}\sR h_{ij}$ of the Ricci tensor of a pseudo-Riemannian metric is divergence free, which in the four-dimensional Lorentzian case is of fundamental importance physically. When the connection and the metric are linked only weakly as in the AH condition, the traced differential Bianchi identity and the naive Einstein equations are together insufficient to imply an analogous condition on the weighted scalar curvature, and such an extra condition, at the very least, must be imposed as part of the definitions. In dimension at least $3$ the obstruction is the tensor $A_{i}$. Precisely from \eqref{ebtw} of Lemma \ref{diffbianchilemma} it can be seen that in dimension at least $3$, a naive Einstein AH structure is conservative if and only if 
\begin{align}\label{einsteini}
\nabla_{i}R + n\nabla^{p}F_{ip} = 0. 
\end{align}
Here $F_{ij} = \tfrac{1}{n}R_{ijp}\,^{p}$ is the Faraday two-form. In two dimensions the tensor $A_{i}$ is not defined, but the condition $\nabla_{i}R + 2\nabla^{p}F_{ip} = 0$ does make sense. Moreover, this is exactly the condition taken by Calderbank as the definition of the Einstein Weyl equations in dimension $2$ (in two dimensions the usual Einstein Weyl equations are vacuous, so the definition has to be supplemented by a new condition).

Thus a notion of Einstein AH structure generalizing the Einstein Weyl equations, including their two dimensional formulation, is obtained by considering conservative naive Einstein AH structures, where in $2$ dimensions a naive Einstein AH structure is said to be \textbf{conservative} if it satisfies \eqref{einsteini}.
\begin{definition*}
An AH structure on a manifold of dimension $n \geq 2$ is \textbf{naive Einstein} if there vanishes the trace-free symmetric part of every rank two trace of its curvature, and \textbf{Einstein} if it is naive Einstein and satisfies \eqref{einsteini}.
\end{definition*}\footnote{In $2017$, while preparing a new version of the present article, the author realized that in this definition and the preceding paragraphs (as well as Definition \ref{einsteindefinition} and the paragraph following) all references to $R_{ij}$ and $Q_{ij}$ were intended to refer the \textit{symmetric} parts of $R_{ij}$ and $Q_{ij}$. The definition of naive Einstein, and so also Einstein, used throughout the paper has required (in all versions) the vanishing of the trace-free \textit{symmetric} parts of these ranks two traces of the curvature, and not the stronger vanishing of their trace-free parts (which would force the vanishing of the Faraday curvature, which is not assumed as part of the definitions of the Einstein-like conditions). The expository error, which must have been terribly confusing to readers, is due to a combination of overediting aimed at streamlining the exposition of a previous even more verbose version, combined with the psychological tendency to read in one's own writing what one intends, independently of what one has written.}

An explicit example of an AH structure which is naive Einstein but not Einstein is given in section \ref{naiveexample}. While it seems clear that the gap between the two is quite large, it does not seem much easier to produce examples of naive Einstein AH structures than it is to produce examples of Einstein AH structures. In fact, the only examples of naive Einstein structures available found so far are obtained by perturbing some feature of an explicit example of an Einstein AH structure

\subsubsection{}\label{labourieloftinsection}
In the present paper the focus will be on the case of dimension at least $3$, although occasionally comments will be made about the two-dimensional case. Stronger results are available in two dimensions, and these are briefly discussed now because the picture they yield is helpful in thinking about what should be true in higher dimensions. In two dimensions the following theorems provide convincing evidence that notion of Einstein equations is correct.
\begin{theorem}\label{2dtheorem}
On a compact, orientable surface $M$ of genus at least $2$, an Einstein Riemannian signature AH structure $(\en, [h])$ is exact with negative scalar curvature and is projectively flat and conjugate projectively flat. For a distinguished metric the cubic tensor $L_{ijk} = \bt_{ij}\,^{p}h_{pk}$ is the real part of a cubic holomorphic differential with respect to the complex structure determined by the conformal structure and the given orientation
\end{theorem}
The two-dimensional case will be described in detail, and this theorem proved, in \cite{Fox-2dahs}. It can be deduced from Lemma \ref{2deinsteinlemma}, Theorem \ref{todtheorem}, and the Riemann-Roch Theorem. The story in genus zero or one is a bit harder to summarize briefly, so is omitted. From Theorem \ref{2dtheorem} there can be deduced
\begin{theorem*}
On a compact, orientable surface $M$ of genus at least $2$ there is an oriented mapping class group equivariant bijection between the deformation space of Einstein AH structures and the fiber bundle over the Teichm\"uller space of $M$ the fiber of which over a given conformal structure comprises the cubic holomorphic differentials with respect to the complex structure determined by the conformal structure and the given orientation. To conjugate Einstein AH structures there correspond opposite cubic differentials.
\end{theorem*} 
This theorem is closely related to, though not identical with, a theorem due independently to F. Labourie and J. Loftin (see \cite{Labourie-flatprojective} and \cite{Loftin-affinespheres}) showing that the same bundle over Teichm\"uller space parameterizes the deformation space of convex flat real projective structures on the surface. The construction of a cubic holomorphic differential from an Einstein AH structure is given by Theorem \ref{2dtheorem}, while the construction of an Einstein AH structure given a cubic holomorphic differential can be effected by essentially the argument used by Loftin to prove Theorem $2$ of \cite{Loftin-affinespheres}. This theorem could be deduced from that of Labourie and Loftin if it could be shown directly that the projective structure underlying an AH structure as in Theorem \ref{2dtheorem} is necessarily convex; as things stand, this is true, but follows from Theorem \ref{2dtheorem} coupled with the Labourie-Loftin theorem.

\subsubsection{}
In the higher dimensional case it is not completely clear that the given definition of Einstein AH is the best one. In particular there is available the following somewhat stronger notion. A natural condition which implies the conservation condition is that the curvature of the AH structure be self-conjugate. In the presence of the naive Einstein equations this is equivalent to the vanishing of the anti-self-conjugate Weyl tensor $E_{ijk}\,^{l}$. This condition implies the conservation condition because $\bt^{abc}E_{iabc}$ is a multiple of $A_{i}$, as is shown in Lemma \ref{diffbianchilemma} by involved computations using the differential Bianchi identity, but examples constructed in section \ref{leftinvariantsection} show that it is more restrictive. The choice has been made to define the Einstein equations as above, and to refer to this stronger notion as \textit{Einstein AH with self-conjugate curvature}. 

Because the self-conjugate Weyl tensor vanishes in three dimensions, it can be shown that on a $3$-manifold a closed naive Einstein AH structure with self-conjugate curvature is projectively flat and conjugate projectively flat, and it follows both that it is Einstein and that it is locally immersible as an affine hypersphere in flat affine space. Thus were the self-conjugacy of the curvature included in the definition of Einstein, then in three dimensions the Einstein equations would imply projectively flat for a closed AH structure. This resembles the situation for the Einstein equations of a metric, which in three dimensions force the metric to have constant sectional curvature. Both this observation and that an affine hypersphere has always self-conjugate curvature, suggest including the self-conjugacy of the curvature in the definition of Einstein. Also, as will be discussed below, for Einstein AH structures with self-conjugate curvature than can be proved stronger theoresm about the the size of the cubic torsion than can be proved for general Einstein AH structures. This and the preceeding remarks could be taken as indicating that the self-conjugacy of the curvature should be included in the definition of Einstein. On the other hand, in two-dimensions self-conjugacy of the curvature is equivalent to the naive Einstein equations, and the conservation condition has to be imposed, and so with a definition including self-conjugacy of the curvature the two-dimensional case would, as for Weyl structures, require a special treatment. This suggests that in higher dimensions the two conditions, conservative and self-conjugate curvature, should be regarded as independent notions. It seems preferable to have a definition which applies equally in dimension two and higher dimensions, and it seems that the self-conjugacy of the curvature and the conservation condition are conceptually distinct, although related. For these reasons, and because many basic results about Einstein Weyl structure generalize straightforwardly to Einstein AH structures, it seems that the terminologies chosen here are reasonable.

It should be mentioned that a third notion, called \textbf{strongly Einstein}, is defined in Definition \ref{stronglyeinsteindefinition}, but this notion is not central, and motivating it adequately would require discussion of technicalities distracting here in the introduction.

\subsubsection{}
 In section \ref{affinehyperspheresection} there are considered non-degenerate co-oriented hypersurface immersions in manifolds with affine connections and projective structures. It is shown that a on a non-degenerate co-oriented immersed hypersurface in a manifold with projective structure there is induced a conformal projective structure, which is a Codazzi projective structure if the ambient manifold is projectively flat. It is shown that each torsion free affine connection representing the ambient projective structure determines an affine normal bundle and induces on the hypersurface a projective structure which forms with the conformal structure determined by the second fundamental form an AH structure provided the ambient affine connection is projectively flat. For immersions in flat affine space there is described the conormal Gau\ss\, map and the flat projective structure which it induces, and it is shown that the AH structure which this constitutes with the second fundamental form is conjugate to that induced via the affine normal. Finally it is shown that the induced AH structure is Einstein if and only if the hypersurface is an affine hypersphere. It should be remarked that the AH formalism may be useful for studying other classes of affine hypersurfaces, e.g. affine maximal hypersurfaces. In section \ref{convexprojectivesection} it is shown, using a fundamental theorem of Cheng and Yau on the solvability of a particular Monge-Amp\`ere equation on a properly convex domain, that a convex flat real projective structure $\en$ admits a unique conformal structure $[h]$ with which it forms an exact Einstein AH structure with negative scalar curvature. Since there is a unique affine hypersphere associated to the universal cover of a convex flat projective structure, this also follows from the results of section \ref{affinehyperspheresection}, but a direct intrinsic proof has been given rather than passing through this correspondence, which is quite deep. 

\subsubsection{}
In section \ref{leftinvariantsection} there are constructed various examples of left-invariant Einstein AH structures on semisimple Lie groups which, however, do not have self-conjugate curvature. In particular such examples having Riemannian signature are constructed on $S^{3}$.

\subsubsection{}
Some idea of the content of the definition is given by the following characterization of Riemannian signature Einstein AH structures on a compact $n$-dimensional manifold $M$. By Theorem \ref{classtheorem}, such a structure is equivalent to the data of a Riemannian metric $h_{ij}$, a one-form $\ga_{i}$, and a completely symmetric, completely trace-free tensor $L_{ijk} = L_{(ijk)}$ such that, writing $D$ for the Levi-Civita connection of $h_{ij}$, and $\sR_{ij}$ and $\sR$ for its Ricci and scalar curvatures, respectively, there hold the equations
\begin{align}\label{confeinintro}
\begin{split}
&\sR_{ij} = \tfrac{1}{n}\sR h_{ij} + \tfrac{1}{4}\left(L_{ip}\,^{q} L_{iq}\,^{p} - \tfrac{1}{n}|L|_{h}^{2}h_{ij}\right) + (2-n)\left(\ga_{i}\ga_{j} - \tfrac{1}{n}|\ga|_{h}^{2}h_{ij}\right),\\
&D_{i}(\sR - \tfrac{1}{4}|L|_{h}^{2} - (n+2)|\ga|_{h}^{2}) = 0, \quad D_{p}L_{ij}\,^{p} = 0, \quad \ga_{p}L_{ij}\,^{p} = 0,\quad D_{(i}\ga_{j)} = 0.
\end{split}
\end{align}
The usual Einstein Weyl equations are recovered in the case $L_{ijk} \equiv 0$. These equations say that the vector dual to $\ga_{i}$ is $h$-Killing, and that $L$ is divergence free and annihilated by $\ga_{i}$. When $n = 2$ and $M$ is oriented they imply more, namely that $L$ is the real part of a holomorphic cubic differential (with respect to the complex structure determined by $h$ and the given orientation). The equations \eqref{confeinintro} imply
\begin{align}\label{fieldeq}
\begin{split}
\sR_{ij} - \tfrac{1}{2}\sR h_{ij} + \tfrac{n-2}{2n}\ka h_{ij}& = \tfrac{1}{4}\left( L_{ip}\,^{q} L_{iq}\,^{p} - \tfrac{1}{2}|L|_{h}^{2}h_{ij}\right)+ (2-n)\left(\ga_{i}\ga_{j} + \tfrac{1}{2}|\ga|_{h}^{2}h_{ij}\right), 
\end{split}
\end{align}
in which $\ka = \sR - \tfrac{1}{4}|L|_{h}^{2} - (n+2)|\ga|_{h}^{2}$ is constant. This last equation looks like the gravitational field equations with cosmological constant and the righthand side as stress-energy tensor (that this is divergence free is non-trivial). Of course $h_{ij}$ is Riemannian, and this resemblance is purely formal, but in the Lorentzian case solving the field equations \eqref{fieldeq} will also lead to Einstein AH structures, and this suggests other points of view on the latter. 

That there vanish $E_{ijk}\,^{l}$ imposes the further condition that $L_{ijk}$ be what some authors call a \textbf{Codazzi tensor}:
\begin{align}\label{codazzicondition}
D_{[i}L_{j]kl} = 0.  
\end{align}
The extra strength of this condition is apparent from the Weitzenb\"ock type identities described in section \ref{structuretheoremsection} and also the refined Kato inequality of Lemma \ref{katolemma} to which it leads. This kind of Codazzi condition plays an important role in various problems, for example the analysis of the initial conditions for the (physical) Einstein equations, or in proving Bernstein type theorems for Riemannian hypersurfaces (as in \cite{Schoen-Simon-Yau}) or for affine hypersurfaces. A relevant and motivating survey is \cite{Calabi-bernsteinproblems}. The consequences here of \eqref{codazzicondition} can be seen in Theorems \ref{eigentheorem} and \ref{cubictorsiontheorem} recounted later in the introduction.

\subsubsection{}
Theorem \ref{convextheorem} shows that given a convex flat real projective structure $\en$ there is a unique conformal structure $[h]$ such that $(\en, [h])$ is an exact Einstein AH structure with negative scalar curvature. This gives a large class of examples which are not Einstein Weyl, and for which much is known by combinatorial and geometric topological methods. Since the existence of non-trivial, non-hyperbolic convex flat projective structures will not be familiar to all readers, a few remarks are made about their abundance. 

In two dimensions W. Goldman proved in \cite{Goldman-convex} that the deformation space of convex flat projective structures on a compact orientable surface of genus $g > 1$ is diffeomorphic to an open cell of real dimension $16(g-1)$; this follows also from the theorem of F. Labourie and J. Loftin discussed in section \ref{labourieloftinsection}. %
 In higher dimensions there is of course not available such a complete description. While it suggests that convex flat projective structures are abundant in higher dimensions as well, phenomena such as Mostow rigidity might be taken to suggest the contrary. However, Y. Benoist has shown that in every dimension $n \geq 2$ there exist non-symmetric divisible strictly convex sets in $\projp(\ste)$. In these examples the discrete group dividing the set is isomorphic to a cocompact lattice in $O^{+}(1, n)$. Examples not obtained by deforming hyperbolic ones were later found by Benoist and M. Kapovich. In particular, in \cite{Kapovich}, Kapovich proved:%
\begin{theorem}[M. Kapovich, \cite{Kapovich}]
For every $n \geq 4$ there is a compact $n$-manifold which admits a strictly convex flat real projective structure, and which admits no Riemannian metric of constant negative sectional curvature, although it does admit a metric of negative sectional curvature.
\end{theorem}
The proof is constructive; the projective structures are constructed on examples due to M. Gromov and W. Thurston, \cite{Gromov-Thurston}, who showed the claimed metric properties of these manifolds. The strict convexity is proved using a theorem of Y. Benoist, \cite{Benoist-convexesdivisblesI}, that a divisible open subset of $\projp(\ste)$ is strictly convex if and only if it is divisible by a Gromov hyperbolic discrete group of projective automorphisms.

In short, even for manifolds for which there is no hope of uniformization using hyperbolic structures, there can be convex flat projective structures. In lieu of Theorem \ref{convextheorem} and recast in the language of this paper, this shows that there are manifolds which admit no hyperbolic metric but which admit Einstein AH structures of a particularly nice sort. One motivation for the present paper is the idea that techniques for producing Einstein AH structures (in particular with self-conjugate curvature) could be useful also for producing convex flat projective structures. In particular they will be useful in three dimensions, since in that case a closed Einstein AH structure with self-conjugate curvature is necessarily projectively flat. More generally it is hoped that the analytic techniques afforded by the AH formalism will be useful in the study of convex flat projective structures. 

\subsubsection{}
Because the symplectic form identifies the normal bundle of a Lagrangian submanifold with its cotangent bundle, the second fundamental form of the Lagrangian submanifold with respect to a connection on the ambient symplectic manifold can be viewed as a covariant three tensor on the submanifold, which is completely symmetric if the connection is torsion-free and compatible with the symplectic form. Because of the similarity of the Gau\ss-Codazzi equations for the second fundamental form with \eqref{codazzicondition} this leads to the expectation that there should be natural AH structures on Lagrangian submanifolds of K\"ahler and para-K\"ahler manifolds. In section \ref{lipkcurvaturesection} this expectation is realized for Lagrangian submanifolds of para-K\"ahler manifolds. Theorem \ref{inducedcodazzitheorem} shows that on an immersed non-degenerate Lagrangian submanifold of a para-K\"ahler manifold there is induced in a natural way a pencil of Codazzi structures. Theorem \ref{mczerotheorem} shows that on an immersed mean curvature zero spacelike Lagrangian submanifold of a para-K\"ahler manifold of constant para-holomorphic sectional curvature there is induced an exact Riemannian Einstein AH structure having as a distinguished metric the induced metric, and which is projectively flat and conjugate projectively flat. This suggests a close relationship between such Lagrangian submanifolds and affine hyperspheres or convex flat projective structures, and in fact there is a sort of local correspondence, though this will be explained in detail elsewhere. 

Logically section \ref{lipkcurvaturesection} would come at the end of section \ref{einsteinahsection}, but because sketching the necessary background about para-K\"ahler structures takes some space, and it is convenient to use some notation not introduced until section \ref{structuretheoremsection}, it is located where it is.

\subsubsection{}
In section \ref{curvatureconditionssection} and those following attention is restricted to Riemannian AH structures. This is not because other signatures are not interesting; on the contrary, the example of Lorentzian Einstein-Weyl structures, as for example in N. Hitchin's \cite{Hitchin} or C. LeBrun and L. Mason's \cite{Lebrun-Mason-einsteinweyl}, suggests that the Lorentzian case will be very interesting. Rather this restriction is made because the availability of the usual tools from elliptic PDE facilitates the proof of general structural results, and leads to further problems which should be approachable with more sophisticated analysis. The Lorentzian case requires a different set of techniques, in particular the best results for Lorentzian Einstein Weyl structures are obtained using twistorial methods as in \cite{Hitchin} and \cite{Lebrun-Mason-einsteinweyl}. Another justification for the focus on the Riemannian case is that this is the case relevant for the study of convex flat projective structures.

\subsubsection{}
In section \ref{scalarcurvaturesection} there are analyzed the properties of the scalar curvature of Riemannian signature Einstein AH structures on compact manifolds. Essentially all the results of this sort for Einstein Weyl structures carry over to the more general setting. While the end results, and the ideas behind them, are essentially the same as in the Weyl case, their proofs sometimes require a bit more work. The key technical tool is the Gauduchon metric. In section \ref{bochnervanishingsection} there is proved for AH structures (not necessarily Einstein) with some positivity condition on the Ricci curvature a Bochner vanishing type theorem (Theorem \ref{bochnertheorem}) which generalizes (and slightly strengthens) a theorem of B. Alexandrov and S. Ivanov for the Weyl case. This is used to deduce Theorem \ref{negexacttheorem}, which gives a rough classification of compact Riemannian signature Einstein AH manifolds by the properties of their scalar curvature, which extends the results known for Einstein Weyl structures.

\subsubsection{}
Section \ref{structuretheoremsection} contains the main structural results of the paper. Sections \ref{symmetrictensoralgebrasection} and \ref{differentialoperatorssymmetricsection} analyze the algebra of trace-free symmetric tensors and the basic first order differential operators that a metric determines on such tensors. The latter are usually called generalized gradients or Stein-Weiss operators; see e.g. \cite{Branson-steinweiss} for the general theory. Section \ref{curvatureoperatorsection} describes the action of the curvature operator on trace-free symmetric tensors. In section \ref{weitzenbocksection} there are given Weitzenb\"ock identities, refined Kato inequalities, and the resulting differential inequalities and vanishing theorems for tensors annihilated by the differential operators of section \ref{differentialoperatorssymmetricsection}, e.g. conformal Killing tensors and trace-free Codazzi tensors. The results parallel those for conformal Killing forms obtained by U. Semmelman in \cite{Semmelmann}. It should be mentioned that the utility of these results is somewhat limited by a poor understanding of the action of the curvature operator on completely trace-free completely symmetric tensors of rank greater than two. For example the curvature conditions necessary for vanishing results are themselves opaque. The analysis of the curvature operator acting on trace-free symmetric two tensors was in essence made in \cite{Berger-Ebin}, and is reviewed in section \ref{curvatureoperatorsection}; to make progress with tensors of higher rank it seems that a better understanding is needed of the interaction of the curvature operator with the Cartan multiplication of trace-free symmetric tensors. As in this paper there is needed only the special case of this material for tensors of rank $3$, it is developed in more detail than is needed, but it seems that it has applications in many other contexts (a few of which are briefly recalled), and is itself of interest. Moreover, it is as easy to treat tensors of general rank as it is to treat those of rank $3$, and the general point of view only clarifies the structure of the results. 

In section \ref{eigensection}, using the results of the preceeding sections, the Riemannian Einstein equations are rewritten in terms of a Gauduchon metric and the cubic torsion. The theorem stated next is a special case of the slightly more general Theorem \ref{eigentheorem}; it shows that the Riemannian Einstein AH equations reduce to an algebraically constrained elliptic equation for a completely trace free completely symmetric covariant three tensor.

\begin{theorem*}[special case of Theorem \ref{eigentheorem}]
If $M$ is a manifold of dimension at least three and $(\en, [h])$ an exact Riemannian signature Einstein AH structure with self-conjugate curvature and cubic torsion $\bt_{ij}\,^{k}$, then for a distinguished metric $h \in [h]$ there is a constant $\ka$ such that $L_{ijk} \defeq \bt_{ij}\,^{p}h_{pk}$ solves the system
\begin{align}
\label{lapeinstein}
&\lap_{h}L_{ijk} = -2 \sR_{p(ij}\,^{q}L_{k)q}\,^{p} + \sR_{p(i}L_{jk)}\,^{p},& &D^{p}L_{ijp} = 0,&
& D_{[i}L_{j]kl} = 0,&\\
\label{lapeinstein0}&\sR_{ij} = \tfrac{1}{4}L_{ip}\,^{q}L_{jq}\,^{p} + \ka h_{ij}.
\end{align}
in which $D$ is the Levi-Civita connection of $h$, $\lap_{h} = D^{p}D_{p}$ is the rough Laplacian, the curvature convention is $2D_{[i}D_{j]}X^{k} = \sR_{ijp}\,^{k}X^{p}$, and $\sR_{ij} \defeq \sR_{pij}\,^{p}$. Conversely, if $M$ is compact, then a Riemannian metric $h$ with Levi-Civita connection $D$ and a completely trace-free completely symmetric covariant tensor $L_{ijk}$ solving the first equation of \eqref{lapeinstein} solves all the equations \eqref{lapeinstein}, and if $h$ and $L$ solve \eqref{lapeinstein} and \eqref{lapeinstein0} then the connection $\nabla = D - \tfrac{1}{2}h^{kp}L_{ijp}$ is the aligned representative of an exact Einstein AH structure with self-conjugate curvature for which $h$ is a distinguished metric. 
\end{theorem*}

This reformulation is useful in two distinct ways. First, it can be used for analyzing the properties of Einstein AH structures, and second, it offers a way of constructing examples. With respect to the first, the main result is a partial generalization of estimates on the growth of the cubic form of an affine hypersphere obtained by E. Calabi in \cite{Calabi-completeaffine} and \cite{Calabi-improper} to bounds on the norm of the cubic torsion for Riemannian signature Einstein AH structures with self-conjugate curvature. This yields estimates on the scalar and Ricci curvatures of a Gauduchon or distinguished metric. This is contained in theorem \ref{cubictorsiontheorem}, which is restated here for convenience.
\begin{theorem*}[Theorem \ref{cubictorsiontheorem}]
Let $(\en, [h])$ be a Riemannian signature Einstein AH structure with self-conjugate curvature on a manifold $M$ of dimension $n \geq 3$. Suppose either $M$ is compact or $(\en, [h])$ is exact and a distinguished metric is complete. Suppose that $C^{ijkl}\sW_{ijkl} \geq 0$ (this is automatic if $n = 3$). Then one of the following mutually exclusive possibilities holds:
\begin{enumerate}
\item $R > 0$ and $(\en, [h])$ is an Einstein Weyl structure which is either not closed or is exact. 
\item $R \equiv 0$ and $(\en, [h])$ is closed Einstein Weyl. If $M$ is compact and $(\en, [h])$ is not exact then the universal cover of $M$ equipped with the pullback of a Gauduchon metric is isometric to a product metric on $\rea \times N$ where $N$ is simply-connected with an Einstein metric of positive scalar curvature; if $3 \leq n \leq 4$, then $N$ is diffeomorphic to $S^{n-1}$. There is induced on $N$ an exact Riemannian signature Einstein Weyl structure which has positive scalar curvature, and for which the induced metric $g$ is a distinguished metric.
\item $R$ is negative and $\nabla$-parallel and $(\en, [h])$ is exact. A distinguished metric $h \in [h]$ has non-positive scalar curvature $\sR_{h}$ and Ricci curvature $\sR_{ij}$ satisfying
\begin{align*}
\sR_{ij} \leq \tfrac{(n-2)(n+1)}{n(n+2)}\bt^{abc}\bt_{abc}H_{ij}. 
\end{align*}
Moreover, $\sR_{h} = 0$ if and only if $|L|^{2}_{h}$ is constant, in which case $L$ is $h$-parallel, $C^{ijkl}\sW_{ijkl} = 0$ and $4\sR_{ij} = \mr{L}_{ij}$. 
\item $n = 3$, $R$ is somewhere positive and somewhere non-positive, and $(\en, [h])$ is not closed. The scalar curvature $\sR$ of a Gauduchon metric $h \in [h]$ satisfies $\sR \leq 5|\ga|^{2}_{h}$. 
\end{enumerate}
If $[h]$ is locally conformally flat, then in cases $(1)$ and $(2)$, $(\en, [h])$ is exact and a distinguished metric has constant sectional curvature, positive or identically zero according to whether $R$ is positive or zero. If $(\en, [h])$ is conjugate projectively flat, then case $(4)$ does not occur, and in cases $(1)$, $(\en, [h])$ is exact and a distinguished metric has constant positive sectional curvature.
\end{theorem*}
The undefined terminology is all defined in due course. The tensor $\sW_{ijkl}$ is the usual conformal Weyl tensor of $[h]$, and the condition $C^{ijkl}\sW_{ijkl} \geq 0$ is automatic if $[h]$ is conformally flat or $\en$ is either projectively or conjugate projectively flat. Despite this, the theorem is not wholly satisfactory in that the geometric meaning of the condition $C^{ijkl}\sW_{ijkl} \geq 0$ is obscure. Also unsatisfactory is the assumption that in the exact non-compact case a distinguished metric be complete; it would be preferable to conclude this from some assumption stated only in terms of $\en$ and $[h]$. 

The proof consists in obtaining a differential inequality for the Laplacian of a power of the norm of the cubic torsion (this is Theorem \ref{calabiestimatetheorem}) and applying the maximum principle in a form due to S.Y. Cheng and S.T. Yau and used by them in \cite{Cheng-Yau-maximalspacelike} and \cite{Cheng-Yau-affinehyperspheresI}. %
The self-conjugacy of the curvature is necessary for the results to hold, as is shown by the examples of section \ref{leftinvariantsection}.
If in addition to the hypotheses of the theorem, $(\en, [h])$ is projectively flat, then the Ricci curvature of a distinguished metric is non-positive; this is Theorem \ref{calabitheorem} which is essentially Calabi's original theorem, and is explained in section \ref{calabiriccisection}.

Theorem \ref{eigentheorem} shows that by finding a cubic tensor solving the algebraically constrained elliptic system \eqref{lapeinstein}-\eqref{lapeinstein0} there can be constructed Einstein AH structures with self-conjugate curvature. As is explained in \ref{eigensectionintro} below, this approach is used in section \ref{cubicformsection} to construct examples in the case of flat $h$, exploiting the observation that in this case \eqref{lapeinstein} is vacuous for an $h$-parallel tensor $L_{ijk}$, and the system reduces to the purely algebraic \eqref{lapeinstein0}. In general one expects that with some non-positivity condition on the curvature the equations \eqref{lapeinstein}-\eqref{lapeinstein0} should admit solutions, but no theorem is yet available. Of course the problem contains the problem of finding ordinary Einstein metrics, so one should be restrained in one's hopes.

In three dimensions this is particular interesting. Because the self-conjugate conformal Weyl tensor necessarily vanishes in three dimensions, it follows from Lemma \ref{projflatahlemma}, that in this case the resulting AH structure is necessarily projectively and conjugate projectively flat. This means that flat projective structures can be constructed on $3$-manifolds by solving an elliptic system. %

Methods are needed for solving \eqref{lapeinstein}-\eqref{lapeinstein0}. One approach is to study the heat flow associated to \eqref{lapeinstein}; the constraint \eqref{lapeinstein0} introduces complications not seen in other contexts, e.g. harmonic maps. A second approach, which seems more promising, is to develop an analogue of Ricci flow for AH structures. This may require recasting the usual Ricci flow in the metric affine (so-called Palatini) formalism, that is to say, treating the metric and the connection as \textit{a priori} uncoupled fields. Here it appears ideas coming from study of the renormalization group flow could be relevant. Equations \eqref{fieldeq} and \eqref{lapeinstein0} resemble the steady state of the equation for metric in the equations of the renormalization group flow for a nonlinear sigma model though with the $3$-form $H_{ijk}$ measuring the strength of the $B$-field replaced by the symmetric tensor $L_{ijk}$, while \eqref{lapeinstein} resembles the equations controlling $H_{ijk}$; see e.g. equations $(2.2)-(2.4)$ of \cite{Oliynyk-Suneeta-Woolgar} and the discussion of Ricci flow for connections with torsion in \cite{Streets}. While at the moment this appears to be no more than a vague analogy, it is suggestive of the form of the appropriate analogue of the Ricci flow.

For either approach one would like to have a better understanding of the variational character of the equations \eqref{lapeinstein}. Some brief remarks about this are made in section \ref{variationalsection}.

In the physics literature, a completely symmetric covariant tensor of rank $s$ is said to have \textit{spin $s$}. The problem of consistently coupling the Einstein gravitational action with the Lagrangian for a (massless) field of spin greater than two is an old one, going back to at least the work of M. Fierz and W. Pauli, \cite{Fierz-Pauli}. Some standard references studying gauge theories including such fields are \cite{Fronsdal}, \cite{deWit-Freedman}, \cite{Curtright-massless}, and \cite{Weinberg-Witten}. See \cite{Vasiliev-survey} for a recent survey. Some other relevant recent references are \cite {Dolan-Nappi-Witten} and \cite{Boulware-Deser-Kay}, \cite{Deser-Zumino}, \cite{Deser-Waldron-partiallymassless}, \cite{Deser-Waldron-conformalinvariance}, and \cite{Deser-Waldron-nullpropagation}. Most results are negative, demonstrating the inconsistency or impossibility of putative couplings, and there is little or no experimental evidence for such fields. Nonetheless, the tensor $L_{ijk}$ of Theorem \ref{eigentheorem} is a spin-$3$ field in this sense, and equations like \eqref{lapeinstein} (in Lorentz signature), and Lagrangians leading to them, have been studied in the physics literature, e.g. by Fronsdal in \cite{Fronsdal}, or Vasiliev (see \cite{Vasiliev-survey}). 
The problem of finding a variational description of the Einstein AH equations appears quite similar to the problem of consistently coupling a spin-$3$ field to the gravitational action. In \cite{Boulanger-Kirsch}, N. Boulanger and I. Kirsch, building on Kirsch's earlier \cite{Kirsch}, proposed a model of spontaneous symmetry breaking of the analytic diffeomorphism group in which gravity is modified at high energies by interactions with a massive spin-$3$ field. In \cite{Baekler-Boulanger-Hehl}, P. Baekler, N. Boulanger, and F.~W. Hehl show that in the metric affine formalism the traceless part of the nonmetricity tensor (which is essentially what is here called the cubic torsion) can be regarded as a massless spin-$3$ field and show that this perspective yields interpretations of Fronsdal's and Vasiliev's Lagrangians for such fields. As the relevant physical literature is huge and (as may be evident from this remark) this author's competence to analyze it limited, no more will be said on this point, but it would be interesting to compare the equations considered here with the models proposed in \cite{Boulanger-Kirsch} and \cite{Baekler-Boulanger-Hehl}.

\subsubsection{}
In section \ref{automorphismsection} it is shown that the analytic tools afforded by the AH formalism can be applied to deduce results about convex flat real projective structures. For example:%
\begin{theorem*}[Theorem \ref{twodconvexautotheorem}]
A convex flat real projective structure on a compact, orientable surface of genus at least $2$ admits no non-trivial infinitesimal projective automorphism.
\end{theorem*}
This theorem is certainly not new. It follows from the theorem of Labourie-Loftin that a projective automorphism of a convex flat real projective structure on a compact, orientable surface of genus at least $2$ is an automorphism of the associated conformal structure, and this theorem then follows from the non-existence of conformal Killing fields on such surfaces. Proposition $4.1.2$ of \cite{Loftin-affinespheres} shows more, namely that a generic convex flat real projective structure on a compact, orientable surface of genus at least $2$ admits no projective automorphism whatsoever. In any case, both the more general Theorem \ref{convexoneparametertheorem} from which the above theorem follows, and the use of the Bochner method in the context of flat projective structures, seem to be new. The corresponding theorem in higher dimensions, Theorem \ref{convexoneparametertheorem}, says essentially that if a convex real projective structure on a compact, orientable manifold $M$ admits a continuous group of projective automorphisms then the group is a torus and $M$ is foliated by flat Riemannian submanifolds. 

These theorems follow from Theorem \ref{projkilltheorem}, which is a Bochner style vanishing theorem for a more general class of AH structures. Given an AH structure $(\en,[h])$ there is considered the vector space of \textbf{projective harmonic} vector fields, which comprises those vector fields $X$ for which there vanishes $H^{ij}(\lie_{X}\en)_{ij}\,^{k}$, which is the trace, using $[h]$, of the Lie derivative of $\en$. This evidently contains the infinitesimal projective automorphisms of $\en$. In the presence of a particular upper bound on the Ricci tensor of $\nabla$ there is deduced a vanishing theorem for projective harmonic vector fields. The argument is motivated by a result of R. Couty \cite{Couty} showing that on a compact manifold there are no infinitesimal projective automorphisms of a Riemannian metric with non-positive curvature which is somewhere negative. By the results of section \ref{convexprojectivesection} the vanishing theorem for projective harmonic vector fields applies to convex flat real projective structures and implies for them a vanishing theorem for infinitesimal projective automorphisms. A strong conclusion is reached only because of Theorem \ref{calabitheorem} of Calabi implying the non-positivity of the Ricci curvature of a distinguished metric of the Einstein AH structure determined by a convex flat projective structure. A clean result is obtained in two dimensions because the Gau\ss-Bonnet theorem links the curvature hypothesis in the vanishing theorem to the topology.

\subsubsection{}\label{eigensectionintro}
Theorem \ref{eigentheorem} shows that the problem of constructing an exact Einstein AH structure with self-conjugate curvature for which a distinguished metric is flat reduces to the purely algebraic problem of constructing completely trace-free cubic tensors satisfying certain quadratic equations. While a complete solution is not given here, it is shown in section \ref{cubicformsection} how to produce some explicit solutions. Precisely, there are constructed on $\rea^{n}$ examples of exact Einstein AH structures with negative scalar curvature and self-conjugate curvature for which the flat Euclidean metric is a distinguished metric and which are neither projectively flat nor conjugate projectively flat. This shows that there are Einstein AH structures with self-conjugate curvature which are not Weyl and do not come from convex flat projective structures or affine hyperspheres.

These examples are constructed by constructing certain commutative nonassociative algebras. From Theorem \ref{eigentheorem} is is apparent that the condition on $L_{ijk}$ for a pair $(h, L)$ in which $h$ is flat to comprise the distinguished representative of an exact Einstein AH with cubic torsion $\bt_{ij}\,^{k} = H^{kp}L_{ijp}$ reduces to a purely algebraic condition. A cubic tensor is naturally viewed as the structure tensor of an algebra, and the problem of finding cubic tensors of the desired sort is recast as that of constructing a commutative, non-associative, non-unital algebra the trace form of which is a constant multiple of a non-degenerate symmetric invariant bilinear form. These algebras are named \textbf{Einstein commutative Codazzi algebras} and are closely related to special case of a proposed non-associative non-unital generalization of what are usually called \textit{Frobenius algebras}. In order to situate the algebraic structures which arise in the usual geographies, and because these algebraic notions seem to merit further exploration they are recounted in more detail than is needed for the uses which are made of them. In particular, it might be a tractable problem to give a structure theory for commutative Codazzi algebras. The multiplication of a commutative Codazzi algebra determines a harmonic cubic homogeneous polynomial satisfying certain conditions, and some such polynomials are constructed explicitly. In \cite{Kinyon-Sagle}, M. Kinyon and A. Sagle have defined and studied the \textbf{Nahm algebra} of a Lie algebra. In Theorem \ref{compactnahmtheorem} it is shown that the Nahm algebra of a compact simple Lie algebra is an algebra of the desired sort, so gives rise to an Einstein AH structure. In \cite{Cartan-cubic}, E. Cartan classified the isoparametric hypersurfaces of a round sphere having three distinct principal curvatures. These come in one parameter families given as the level sets of certain harmonic homogeneous cubic polynomials associated to the real definite signature composition algebras, and, as is explained in section \ref{isoparametricsection}, the corresponding multiplications also yield Einstein AH structures.

\subsubsection{}
Although the main theme of the paper is the Einstein equations, a secondary theme is that many constructions for and problems about conformal structures have analogues or generalizations for AH structures. Section \ref{codprojsection} presents some evidence for this point of view.

In section \ref{mobiussection} it is shown that the notions of M\"obius structure and the conformal Laplacian both generalize straightforwardly to the setting of Codazzi projective structures. In section \ref{yamabesection} it is shown that the usual Yamabe problem generalizes to the context of AH structure. A \textbf{restricted Codazzi projective structure} is defined to be an equivalence class of AH structures, the difference tensor of the aligned representatives of two of which is determined by an exact one-form. If one AH structure subordinate to a restricted Codazzi projective structure is exact, then all are, and it makes sense to look for a subordinate AH structure for which the unweighted scalar curvature of the aligned representative with respect to a volume-normalized distinguished metric is critical. The resulting equations \eqref{yamabe} generalize the equations arising in the usual Yamabe problem, the usual scalar curvature being replaced by its difference with a constant times the norm of the cubic torsion. This shows that the analogue of the usual Yamabe problem makes sense for restricted Codazzi projective structures. Its study is left as a problem for the future.

In section \ref{bachsection} the problem of defining a Bach tensor for $4$-dimensional Codazzi projective structures is discussed and some partial results are described. One conclusion is that objects such as Q-curvature, Paneitz operators, and the Fefferman-Graham ambient metric should admit generalizations in this setting, and that a more sophisticated point of view along these lines will be necessary, as in its absence the computations become unmanageable.

\subsubsection{}
There is an analogy between AH structures and K\"ahler structures which is more fruitful than it may at first appear. In the closely related but more restricted contexts of the K\"ahler affine metrics considered in \cite{Cheng-Yau-realmongeampere} or the special geometry of e.g.  \cite{Bershadsky-Cecotti-Ooguri-Vafa}, \cite{Freed}, and \cite{Strominger}, this relation is more than an analogy, there being for example a K\"ahler tube domain canonically associated to a K\"ahler affine metric. %
In this analogy the projective structure plays the role of the complex structure and the conformal structure plays the role of the K\"ahler metric. The present paper focuses on the conformal properties of AH structures, and the problem of constructing an Einstein AH structure given a conformal structure and a cubic tensor of the appropriate sort is somehow analogous to the problem of finding on a symplectic manifold a complex structure such that with a given metric it forms a K\"ahler Einstein metric having in its K\"ahler class the given symplectic form. The projective point of view on AH structures has been here mostly ignored outside of section \ref{convexprojectivesection}. There the basic problem is to find a conformal structure which forms with a given a projective structure (in particular a flat projective structure) an Einstein AH structure and such that the conformal metric is in some sense within a K\"ahler class; what this means exactly was already discussed some in \cite{Cheng-Yau-realmongeampere}, and there should be a formulation in which \cite{Cheng-Yau-realmongeampere} can be seen as resolving the `negative Chern class' case of the problem for flat projective structures. In such a formulation the K\"ahler class in the usual cohomology theory has to be replaced by an analogous object determined by the metric and expressed in terms of the generalized BGG sequences associated to projective structures as in \cite{Cap-Slovak-Soucek} and \cite{Calderbank-Diemer}. The related complexes constructed by Calabi in \cite{Calabi-constantcurvature} are also relevant. Some idea of what is meant can be obtained from considering Lemma \ref{buahlemma}, Theorem \ref{butheorem}, and Theorem \ref{convextheorem}, in which such a formulation is implicit. Theorem \ref{convextheorem} gives a result in this direction, but its proof passes through deep results associating to a convex flat projective structure a properly embedded affine hypersphere, and it is desirable to find a direct intrinsic proof of Theorem \ref{convextheorem}, in which the aforementioned analogy is brought to the surface.

\subsubsection{}\label{lcksection}
It should be mentioned that there is a good notion of a locally conformally K\"ahler (LCK) AH structure, which directly generalizes the case of two-dimensional AH structures and the so called \textit{K\"ahler Weyl} structures (see e.g. \cite{Calderbank-Pedersen}). Namely, this is an AH structure $(\en, [h])$ equipped with an almost complex structure $J_{i}\,^{j}$ such that the aligned representative $\nabla \in \en$ satisfies $\nabla_{[i}J_{j]}\,^{k} = 0$ and $\nabla_{i}\Omega_{jk} = 0$ in which $\Omega_{ij} \defeq J_{i}\,^{p}H_{pj}$ is the density-valued analogue of the K\"ahler form. Holomorphicity yields stronger results in this setting, and this will be treated in detail in another paper. Here these structures are mentioned only to point out another connection with structures already studied: a flat, exact LCK AH structure is a \textbf{special K\"ahler structure} in the sense of \cite{Freed} (formalizing a notion from the physics literature, e.g. \cite{Bershadsky-Cecotti-Ooguri-Vafa} or \cite{Strominger}).

\subsubsection{}
Generalizations of the Einstein equations with a similar spirit have been made before, and some of them may be related to the proposals here more directly than by simple analogy. In this regard, the \textbf{non-Hermitian Yang-Mills} equations studied by M. Verbitsky and D. Kaledin in \cite{Verbitsky-Kaledin} seem relevant. It seems likely that at least some special cases of the Einstein AH equations are related to the Yang-Mills-Higgs formalism, and this deserves further exploration.

\subsubsection{}
For any definition of a new class of mathematical structures there needs to be shown on the one hand that the class contains interesting examples, and on the other that it is sufficiently limited as to be amenable to study. Einstein Weyl structures and convex flat real projective structures are extensive families of interesting examples of Einstein AH structures with self-conjugate curvature, and it is already interesting that there is a common framework which includes both. On the other hand, the results given here show that not all such AH structures arise in these families. This is taken as evidence that the abstract notion of AH structures generalizes both in a non-trivial way. Results about the scalar curvature and Gauduchon gauge of Einstein Weyl structures and cubic torsion of affine hyperspheres mostly extend to Einstein AH structures, and such results suggest that the class of such structures is sufficiently limited as to admit fruitful study. More examples are needed to fully justify the formalism. The situation will be satisfactory once there is in hand a theorem proving the existence on some class of compact manifolds of dimension at least $4$ of Einstein AH structures having self-conjugate curvature and negative scalar curvature but which are not projectively flat.

\subsubsection{}
Despite its length, this article has a preliminary character. The main technical tools beyond tensor calculus are standard elliptic PDE techniques in their most elementary forms, namely the maximum principle, the Bochner technique, and Weitzenb\"ock type formulas. More refined and stronger results should be obtainable with more sophisticated methods. In particular the equations of Theorem \ref{eigentheorem} should be solvable under appropriate negativity conditions on the curvature. This project, which almost certainly requires studying an associated flow, seems to be the line of inquiry which needs to be pursued next. %

More sophisticated machinery for treating projective and conformal structures, e.g. the associated Cartan connections or the Fefferman-Graham ambient metric, has been avoided. In this regard it should be mentioned that in \cite{Armstrong-einstein} S. Armstrong has proposed a notion of Einstein equations for parabolic geometries, and the specialization of his notion to the context of projective structures appears to be related to that proposed here. This will be examined elsewhere.

\subsubsection{}
The main ideas here are motivated by studying the papers \cite{Calabi-improper, Calabi-bernsteinproblems, Calabi-completeaffine, Calabi-affinelyinvariant, Calabi-affinemaximal} of E. Calabi and \cite{Cheng-Yau-maximalspacelike, Cheng-Yau-mongeampere, Cheng-Yau-realmongeampere, Cheng-Yau-affinehyperspheresI} of S.Y. Cheng and S.T. Yau. While an effort has been made to reference related work, some has undoubtedly been overlooked. In particular the literature on geometry of hypersurfaces in affine space is enormous. The basics can be found in the book \cite{Nomizu-Sasaki} of K. Nomizu and T. Sasaki. What is most relevant here is the work related to affine hyperspheres going back to E. Calabi, S.Y. Cheng and S.T. Yau, K. J\"orgens, and A. Pogorelov, which focused on Bernstein type problems and the study of Monge-Amp\`ere equations. In this regard the current situation is represented by the work of N. Trudinger and X. Wang; see e.g. \cite{Wang-icm}. My personal understanding of affine hyperspheres and convex flat projective structures has been greatly influenced by the papers \cite{Loftin-cubic,  Loftin-Yau-Zaslow, Loftin-Yau-Zaslow-erratum, Loftin-affinespheres, Loftin-affinekahler,  Loftin-riemannian} of J. Loftin and his collaborators. Y. Benoist's survey \cite{Benoist-survey} is recommended as an introduction to the subject of convex flat projective structures. Discussion of open problems related to these topics can be found in the survey \cite{Yau-perspectives}. W. Goldman's \cite{Goldman-convex} and F. Labourie's \cite{Labourie-flatprojective} are basic for the two-dimensional case. As will be evident from section \ref{scalarcurvaturesection}, my understanding of Einstein Weyl structures owes a lot to the papers of D. Calderbank, M. Eastwood, P. Gauduchon, H. Pedersen, P. Tod, and their collaborators, e.g. \cite{Calderbank-mobius, Calderbank-faraday, Calderbank-Pedersen, Eastwood-Tod, Gauduchon, Pedersen-Tod, Tod-compact}.

\section{AH structures}\label{ahsection}
In this section are given the basic definitions. Claims not here justified are, as a rule, verifiable by straightforward, though sometimes lengthy, computations.

\subsection{Preliminaries}\label{preliminariessection}

\subsubsection{}
The word \textit{smooth} means always infinitely differentiable. All manifolds are smooth and without boundary, and $M$ is always a manifold, and its dimensions is always written $n$. The dimension $n$ is always at least $2$, although, because attention will often be restricted to the $n > 2$ case, sometimes $n \geq 2$ will be stated explicitly, to emphasize the inclusion of the $n = 2$ case. 

The result of applying to $X \in T_{p}M$ the differential of a smooth map $f:M \to N$ is written $Tf(p)(X)$; the letter $D$ is used to denote a covariant derivative. Smooth sections of a vector bundle $E$ are written $\Ga(E)$, although in the special case of $\Ga(TM)$ occasionally there is written instead $\vect(M)$, when it is desired to emphasize the structure as a Lie algebra.

\subsubsection{}
Tensors are indicated using the abstract index notation, so that, for instance, $a_{ij}$ indicates a covariant two tensor. Enclosure of indices in square brackets (resp. parentheses) indicates complete skew-symmetrization (resp. complete symmetrization), so that for example $a^{ij}= a^{(ij)} + a^{[ij]}$ indicates the decomposition of a contravariant two-tensor into its symmetric and skew-symmetric parts. Inclusion of an index between vertical bars $|\,|$ indicates its omission from an indicated symmetrization; for example $2a_{[i|jk|l]} = a_{ijkl} - a_{ljki}$. The summation convention is always in effect in the following form: indices are in either \textit{up} position or \textit{down} and a label appearing as both an up index and a down index indicates the trace pairing. Since polynomials on the vector space $\ste$ are tautologically identified with symmetric tensors on the dual vector space $\ste^{\ast}$, the index $i$ in $\tfrac{\pr}{\pr y_{i}}$ has to be regarded as an \textit{up} index. The horizontal position of indices is always maintained when they are raised or lowered. A line-bundle valued tensor is refered to as \textit{weighted}. In particular, by a $\la$-density is meant a section of $|\det \ctm|^{\la}$. A non-degenerate weighted covariant two-tensor $h_{ij}$ determines a contravariant two-tensor $h^{ij}$ of complementary weight defined uniquely by $h^{ip}h_{jp} = \delta_{j}\,^{i}$, in which here, as always, $\delta_{i}\,^{j}$ is the tautological $\binom{1}{1}$-tensor determined by the pairing of vectors with covectors. Note that $h_{ij}$ is not assumed symmetric, so this establishes a convention regarding the ordering of the indices. In particular, if $h_{[ij]} = h_{ij}$ is a symplectic form then $h^{ip}h_{pj} = -\delta_{j}\,^{i}$. The wedge product is defined consistently with the convention that $X \wedge Y = X \tensor Y - Y \tensor X$ for vector fields $X$ and $Y$, so $(X \wedge Y)^{ij} = 2X^{[i}Y^{j]}$.

The trace-free symmetric parts of covariant two tensors will appear so frequently that it will be convenient to have for them some notation; given a non-degenerate symmetric tensor $h_{ij}$ write $\mr{Z}_{ij} \defeq Z_{(ij)} - \tfrac{1}{n}h^{pq}Z_{pq}h_{ij}$ for the trace-free symmetric part of the covariant two-tensor $Z_{ij}$. For a vector bundle $E$, $S^{k}(E)$ denotes the $k$th symmetric power of $E$, and if there is given a fiberwise metric on $E$ then $S^{k}_{0}(E)$ denotes the subbundle of $S^{k}(E)$ comprising completely symmetric elements completely trace-free with respect to the given metric.

\subsubsection{}
To a Young diagram the boxes of which are labeled with distinct indices corresponds the irreducible $GL(n, \rea)$ module comprising tensors skew-symmetric in the indices in a given column of the Young diagram and vanishing when skew-symmetrized over the indices in a given column and any index in any box to the right of the given column. The irreducible representations of the subgroup $CO(h)$ of $GL(n, \rea)$ acting conformally with respect to a fixed metric $h$ on $\rea^{n}$ are described in \cite{Weyl}. The submodules of the irreducible $GL(n, \rea)$ representations comprising tensors completely trace-free with respect to $h$ are representations of the subgroup $CO(h)$ of $GL(n, \rea)$ acting conformally with respect to $h$ (they are usually, but not always, again irreducible). Lemma \ref{weylcriterion} will be invoked repeatedly.

\begin{lemma}[\cite{Weyl}, Theorem $5.7.$A]\label{weylcriterion}
The $CO(h)$-modules of covariant trace-free tensors on $\rea^{n}$ having symmetries corresponding to Young diagrams the sums of the lengths of the first two columns of which are greater than $n$ are trivial. 
\end{lemma}

For instance, Lemma \ref{weylcriterion} implies that the usual conformal Weyl tensor of a Riemannian metric (the completely trace-free part of the Riemann curvature tensor) vanishes identically on a manifold of dimension at most $3$. The following lemmas cause various identities to simplify in the two-dimensional case.
\begin{lemma}\label{twodcubicformlemma}
Let $h_{ij}$ be a constant non-degenerate symmetric tensor on $\rea^{2}$, and for $k > 1$ let $A_{i_{1}\dots i_{k}}$ and $B_{i_{1}\dots i_{k}}$ be completely symmetric, completely $h$-trace free tensors on $\rea^{2}$. Then 
\begin{align}\label{twodab}
2A_{a_{1}\dots a_{k-1}(i}B_{j)}\,^{a_{1}\dots a_{k-1}} = A_{a_{1}\dots a_{k}}B^{a_{1}\dots a_{k}}h_{ij}, 
\end{align}
in which indices are raised and lowered with $h_{ij}$ and its inverse $h^{ij}$.
\end{lemma}
\begin{proof}
Let $A, B \in S^{k}_{0}(\rea_{2})$. It is easily verified that the tensor
\begin{align*}
\om_{ija_{1}\dots a_{k}} = h_{i(a_{1}}A_{a_{2}\dots a_{k})j} + h_{j(a_{1}}A_{a_{2}\dots a_{k})i} - h_{ij}A_{a_{1}\dots a_{k}} - h_{(a_{1}a_{2}}A_{a_{3}\dots a_{k})ij}
\end{align*}
is completely trace-free. As $\om_{ija_{1}\dots a_{k}} = \om_{(ij)(a_{i}\dots a_{k})}$ and $\om_{i(ja_{1}\dots a_{k})} = 0$, it follows from Lemma \ref{weylcriterion} that $\om = 0$. Hence $2A_{a_{1}\dots a_{k}(i}B_{j)}\,^{a_{1}\dots a_{k}} - A_{a_{1}\dots a_{k}}B^{a_{1}\dots a_{k}}h_{ij} = B^{a_{1}\dots a_{k}}\om_{ija_{1}\dots a_{k}} = 0$.
\end{proof}

\begin{lemma}\label{twodxblemma}
Let $h_{ij}$ be a non-degenerate symmetric tensor on $\rea^{2}$, and for $k > 1$ let $B_{i_{1}\dots i_{k}} = B_{(i_{1}\dots i_{k})}$ be completely $h$-trace free and let $X^{i}$ be a vector. Then $2|i(X)B|^{2}_{h} = |X|_{h}^{2}|B|_{h}^{2}$. In particular if $h_{ij}$ has definite signature the equations $|B|_{h}^{2}|X|_{h}^{2} = 0$ and $X^{p}B_{pi_{1}\dots i_{k-1}} = 0$ are equivalent.
\end{lemma}
\begin{proof}
By Lemma \ref{twodcubicformlemma} there holds $2|i(X)B|^{2}_{h} = 2X^{p}B_{pi_{1}\dots i_{k-1}}X^{q}B_{q}\,^{i_{1}\dots i_{k-1}} = |X|_{h}^{2}|B|_{h}^{2}$.%
\end{proof}

\subsubsection{}
Two affine connections are \textbf{projectively equivalent} if they have the same unparameterized geodesics in the sense that the image of any geodesic of one connection is the image of a geodesic of the other connection. This is the case if and only if the symmetric part of their difference tensor is pure trace. A \textbf{projective structure} is an equivalence class of projectively equivalent affine connections. For a torsion-free affine connection $\nabla$ the \tbf{projective Weyl} and \tbf{projective Cotton} tensors $B_{ijk}\,^{l}$ and $C_{ijk}$ are defined by 
\begin{align}\label{bijkl}
&B_{ijk}\,^{l} \defeq R_{ijk}\,^{l} + 2\delta_{[i}\,^{l}P_{j]k}  - 2\delta_{k}\,^{l}P_{[ij]},& &C_{ijk} \defeq 2\nabla_{[i}P_{j]k},
\end{align}
in which $P_{ij}\defeq \tfrac{1}{1-n}R_{(ij)} - \tfrac{1}{n+1}R_{[ij]}$, %
is the \tbf{projective Schouten tensor}. The projective Weyl tensor does not depend on the choice of $\nabla \in \en$. When $n = 2$ the projective Weyl tensor is identically zero and the projective Cotton tensor does not depend on the choice of $\nabla \in \en$. Tracing the differential Bianchi identity yields $\nabla_{p}R_{ijk}\,^{p} = 2\nabla_{[i}R_{j]k}$ and so the algebraic Bianchi identity yields $0 = \nabla_{p}R_{[ijk]}\,^{p} = 2\nabla_{[i}R_{jk]} = 2(n+1)\nabla_{[i}P_{jk]} = (n+1)C_{[ijk]}$, showing that $C_{[ijk]} = 0$. In a similar fashion the trace-free parts of the Bianchi identities for $\nabla$ yield:
\begin{align}\label{projectivebianchi}
&B_{[ijk]}\,^{l} = 0,& & C_{[ijk]} = 0,&
&\nabla_{[i}B_{jk]l}\,^{p} = -\delta_{[i}\,^{p}C_{jk]l},& &\nabla_{p}B_{ijk}\,^{p} = (2-n)C_{ijk}.
\end{align}
The Ricci identity and the algebraic Bianchi identity yield $\nabla_{[i}C_{jk]l} = 2\nabla_{[i}\nabla_{j}P_{k]l} = - R_{[ij|l|}\,^{p}P_{k]p} =  - B_{[ij|l|}\,^{p}P_{k]p}$, or $\nabla_{[i}C_{jk]l} + P_{[i|p|}B_{jk]l}\,^{p} = 0$.

\subsubsection{}
A conformal structure $[h]$ means a pseudo-Riemannian metric determined up to multiplication by a positive function. A conformal structure is \textbf{Riemannian} if any representative metric is positive definite. A conformal structure $[h]$ is identified with its \textbf{normalized representative} $H_{ij}\defeq |\det h|^{-1/n}h_{ij}$ which takes values in the bundle of $-2/n$ densities (a $1$-density is the absolute value of a volume form). Note that $\det H_{ij} = 1$, so that if $\nabla$ is a torsion-free connection there holds $H^{pq}\nabla_{i}H_{pq} = 0$. 

The default convention throughout the paper is that indices are raised and lowered using $H_{ij}$ and the dual symmetric bivector $H^{ij}$ defined by $H^{ip}H_{pj} =\delta_{j}\,^{i}$. The only systematic exception is that if $h_{ij}$ is a metric, then $h^{ij}$ always denotes the dual bivector, and not $H^{ia}H^{jb}h_{ab}$. 

For conformal metrics $\tilde{h}_{ij} = fh_{ij}$ the Levi-Civita connections are written $\tD$ and $D$, and their difference tensor is written $\tD - D =   2\si_{(i}\delta_{j)}\,^{k} - h_{ij}h^{kp}\si_{p}$ in which $2\si_{i} = d\log{f}_{i}$. The curvature of $D$ is written $\sR_{ijk}\,^{l}$. Objects corresponding to $\tD$ are written with the same notations as those corresponding to $D$, but decorated with a $\tilde{\,}$. 

\subsubsection{}\label{coclosedsection}
Suppose $h$ is a pseudo-Riemannian metric with signature $(s, n-s)$. An $h$-orthonormal frame means a basis of vectors $X_{1}, \dots, X_{n}$ which are pairwise orthogonal, and such that the first $s$ have $h$-norm $-1$, and the last $(n-s)$ have $h$-norm $1$. If $\al_{i_{1}\dots i_{k}}^{j_{1}\dots j_{l}}$ is any tensor then $|\al|_{h}^{2}$ always denotes the quadratic form associated to the pairing of tensors $\lb \al, \be \ra_{h}$ given by complete contraction, $\lb \al, \be \ra = \al_{i_{1}\dots i_{k}}^{j_{1}\dots j_{l}}\be_{a_{1}\dots a_{k}}^{b_{1}\dots b_{l}}h^{i_{1}a_{1}}\dots h^{a_{k}i_{k}}h_{j_{1}b_{1}}\dots h_{j_{l}b_{l}}$. Here, as it will often be, the subscript indicating dependence on $h$ has been omitted. The conventions are such that in Riemannian signature the norm square of the wedge product of $k$ orthonormal one-forms equals $k!$, and not $1$.

If $M$ is oriented, the pseudo-Riemannian volume element $\ep_{i_{1}\dots i_{n}}$ is defined by the requirement that if $X_{1},\dots, X_{n}$ is an $h$-orthonormal frame then $X_{1}^{i_{1}}\dots X_{n}^{i_{n}}\ep_{i_{1}\dots i_{n}} = 1$. The dual $n$-vector $\ep^{i_{1}\dots i_{n}}$ is defined by raising the indices of $\ep_{i_{1}\dots i_{n}}$, so satisfies $\ep^{a_{1}\dots a_{n-k} i_{1} \dots i_{k}}\ep_{a_{1}\dots a_{n-k} j_{1}\dots j_{k}} = (-1)^{s}(n-k)!k!\delta_{[j_{1}}^{[i_{i}}\dots \delta_{j_{k}]}^{i_{k}]}$. The Hodge star operator on $k$-forms is defined by $k! \al \wedge \star \be = \lb \al, \be \ra \ep$ and is given explicitly by $k! \star \al_{i_{1}\dots i_{n-k}} = \al^{j_{1}\dots j_{k}}\ep_{j_{1}\dots j_{k}i_{1}\dots i_{n-k}}$. The factor $k!$ is forced by the normalization $\star \ep = 1$. The Hodge star operators $\hsth$ and $\hsh$ of the conformal metrics  $\tilde{h}_{ij} = fh_{ij}$ are related on $k$-forms by $\hsth = f^{(n-2k)/2}\hsh$.

Define $\dad_{h}$ on $k$-forms by $\dad_{h}= -(-1)^{nk + n + s}\star d \star$, so that for a $k$-form $\al$ and a $(k+1)$-form $\be$ there holds $\tfrac{1}{(k+1)!}\int_{M}\lb d\al,  \be\ra\ep = \int_{M} d\al \wedge \star \be = \int_{M} \dad\be \wedge \star \al = \tfrac{1}{k!}\int \lb \al, \dad_{h}\be \ra\ep$. Explicitly 
\begin{align}\label{daddefined}
\dad_{h}\al_{i_{1}\dots i_{k-1}} = - D^{p}\al_{p i_{1}\dots i_{k-1}}
\end{align}
If $M$ is not oriented \eqref{daddefined} still makes sense and is taken as the definition of $\dad_{h}$. On any possibly density-valued tensor the operator $\lap_{h}$ is defined by $\lap_{h}\al_{i_{1}\dots i_{p}}^{j_{1}\dots j_{q}} = D^{a}D_{a}\al_{i_{1}\dots i_{p}}^{j_{1}\dots j_{q}}$, and is referred to as the \textbf{Laplacian}. The \textbf{Hodge Laplacian} $\hodge$ on $p$-forms is defined by $\hodge = d \dad_{h} + \dad_{h}d$. The conventions are such that for a $k$-form $\al$, $\hodge \al + \lap_{h}\al$ is given by the action of the curvature tensor on $\al$; there is needed only the $k = 1$ case, $\hodge\al_{i} + \lap_{h}\al_{i} = \sR_{ip}\al^{p}$. The conventions are such that for $\al \in \Ga(\ext^{k}(\ctm))$,
\begin{align}\label{hodgebyparts}
\int_{M}\lb \al, \hodge\al\ra\ep = \tfrac{1}{k+1}\int_{M}|d\al|_{h}^{2}\ep + k \int_{M}|\dad_{h}\al|_{h}^{2}\ep,
\end{align}
which will annoy some readers. Observe that when $M$ is not orientable $\hodge$ is defined via \eqref{daddefined} and \eqref{hodgebyparts} makes sense if $\ep$ is replaced by the volume measure $|\ep| = |\det h|^{1/2}$. These conventions will mostly matter in section \ref{bochnervanishingsection}.

\subsection{Definition of AH structures}

\subsubsection{Tensors of difference tensor type}\label{difftensorsection}
Let $\rea_{n} = (\rea^{n})^{\ast}$. Define $\A = \tensor^{3}\rea_{n}$ and 
define subspaces
\begin{align*}
\A_{1} & \defeq \{A_{ijk} \in \A: A_{ijk} = A_{(ijk)}\} ,&
\A_{2} & \defeq \{A_{ijk} \in \A: A_{(ij)k} = 0 =  A_{[ijk]}\},\\
\A_{3} & \defeq \{A_{ijk} \in \A: A_{i(jk)} = 0 = A_{[ijk]}\},& 
\A_{4} & \defeq \{A_{ijk} \in \A: A_{ijk} = A_{[ijk]}\},
\end{align*}
and linear operators $\P_{i}:\A \to \A_{i}$
\begin{align*}
&\P_{1}(A)_{ijk}  = A_{(ijk)},& &\P_{2}(A)_{ijk} =  \tfrac{2}{3}\left(A_{i(jk)} - A_{j(ik)}\right),&\\
&\P_{3}(A)_{ijk} =  \tfrac{2}{3}\left(A_{(ij)k} - A_{(ki)j}\right), & &\P_{4}(A)_{ijk} = A_{[ijk]}.
\end{align*}
The operators $\P_{i}$ are orthogonal projection operators, which means that $\A_{i} = \P_{i}(\A)$ and $\P_{i}\circ \P_{j} = \delta_{ij}\P_{i}$, from which it follows that $\A = \A_{1} \oplus \A_{2} \oplus \A_{3} \oplus \A_{4}$ is a direct sum decomposition into irreducible $GL(n, \rea)$ modules. This decomposition is given explictly by
\begin{align}\label{decom}
A_{ijk} & = A_{(ijk)} + \tfrac{2}{3}\left(A_{i(jk)} - A_{j(ik)}\right) + \tfrac{2}{3}\left(A_{(ij)k} - A_{(ki)j}\right) + A_{[ijk]}.
\end{align}
Other (isomorphic) decompositions of $\A$ into irreducibles are possible; the decomposition \eqref{decom} is simply convenient for present purposes. The map $\inter:\A \to \A$ defined by $\inter(A)_{ijk} =  A_{kji}$ intertwines the projections $\P_{2}$ and $\P_{3}$ in the sense that $\P_{3}\circ \inter = \inter \circ \P_{2}$ so restricts to $\A_{2}$ to realize explicitly an isomorphism $\A_{2} \simeq \A_{3}$. %

Given a \tbf{metric} (a non-degenerate symmetric two-tensor) $h_{ij}$ and any module $\B$ of tensors define $\tf^{h}:\B \to \B$ to be the linear projection onto the subspace comprising completely $h$-trace free elements. The superscript $h$ will sometimes be dropped when it is clear from context. The projection $\tf^{h}$ depends only on the conformal class of $h$ and commutes with any $O(h)$-equivariant projection operator, e.g. each of the $\P_{i}$. In particular, it makes sense to restrict $\tf^{h}$ to any $O(h)$-submodule of $\A$. For example
\begin{align}
\label{tfp}
\begin{split}
\tf^{h}\circ\P_{1}(A)_{ijk} &  = A_{(ijk)} - \tfrac{1}{n+2}h^{pq}\left(h_{ij}A_{(kpq)} + h_{jk}A_{(ipq)} + h_{ki}A_{(jpq)}\right),\\
\tf^{h}\circ\P_{2}(A)_{ijk}  &  = \tfrac{2}{3}\left( A_{i(jk)} - A_{j(ik)}  + \tfrac{1}{n-1}\left( 2h_{k[i}A_{j]p}\,^{p} - A_{p[j}\,^{p}h_{i]k} - A_{p}\,^{p}\,_{[j}h_{i]k} \right)\right),
\end{split}
\end{align}
in which $h^{ik}h_{kj} = \delta_{j}^{i}$, $h^{ij} = h^{ji}$, and indices are raised and lowered using $h$.

\subsubsection{}
Let $\nabla$ be an affine connection, and $h_{ij}$ be a pseudo-Riemannian metric. All the constructions of section \ref{difftensorsection} apply fiberwise to sections of $\tensor^{3}\ctm$, in particular to $\nabla_{i}h_{jk}$. Define 
$\L_{\nabla}(h) = \tf^{h}\circ \P_{1}(\nabla h)$ and $\ct_{\nabla}(h) = \tf^{h}\circ \P_{2}(\nabla h)$, so that $\tf^{h}(\nabla h)_{ijk} = \bt_{\nabla}(h)_{ijk} + \ct_{\nabla}(h)_{i(jk)}$. Explicit expressions are given by \eqref{tfp}. Observe that $\ct_{\nabla}$ and $\bt_{\nabla}$ are conformally invariant in the sense that $\bt_{\nabla}(e^{f} h) = e^{f} \bt_{\nabla}(h)$ (and similarly for $\ct_{\nabla}$) for $f \in \cinf(M)$.

\subsubsection{}
If $[h]$ is a conformal structure, two torsion-free affine connections $\bnabla$ and $\nabla$ are \textbf{$[h]$-conformal projectively equivalent} if for any representative $h \in [h]$ the completely $h$-trace free part of their difference tensor is $0$, that is $\tf^{h}(\bnabla - \nabla) = 0$. Explicitly this means that given $h \in [g]$ there are a one-form $\al_{i}$ and a vector field $\be^{i}$ such that $\bnabla - \nabla = 2\al_{(i}\delta_{j)}\,^{k} + h_{ij}\be^{k}$. In particular, the Levi-Civita connections of conformal pseudo-Riemannian metrics are conformal projectively equivalent, and projectively equivalent connections are conformal projectively equivalent. This accounts for the terminology. A \textbf{conformal projective structure} $(\enb, [h])$ is a conformal structure $[h]$ and an equivalence class $\enb$ of $[h]$-conformal projectively equivalent affine connections. A connection with torsion is a \textbf{representative} of a given conformal projective structure if its torsion-free part is. The formalism of conformal projective structures will be needed principally in section \ref{curvaturecodazzisection}.

\begin{lemma}\label{sametorsion}
Torsion-free connections $\bnabla$ and $\nabla$ are $[h]$-conformal projectively equivalent if and only if  $\tf^{h}(\nabla h) = \tf^{h}(\bnabla h)$ for every possibly weighted representative $h \in [h]$.
\end{lemma}
\begin{proof}
For any covariant $3$-tensor there holds 
\begin{align}\label{lcidentity}
A_{ijk} = A_{i(jk)} + A_{j(ik)} - A_{k(ij)} + A_{[ij]k} + A_{[ki]j} - A_{[jk]i},
\end{align}
which identity is familiar from the construction of the Levi-Civita connection. Write $\bnabla - \nabla = \Pi_{ij}\,^{k}$. Then $A_{ijk} = \Pi_{ij}\,^{p}h_{pk}$ satisfies $\Pi_{[ij]k} = 0$. From $\bnabla_{i}h_{jk} - \nabla_{i}h_{jk} = -2A_{i(jk)}$ it follows that $\tf^{h}(\bnabla h - \nabla h) = -2\tf^{h}(A)_{i(jk)}$, and the claim follows from this in conjunction with \eqref{lcidentity}.
\end{proof}

Lemma \ref{sametorsion} shows that $\tf^{h}(\bnabla h) = \tf^{h}(\nabla h)$ for $[h]$-conformal projectively equivalent connections $\bnabla$ and $\nabla$, so that it makes sense to write $\L_{\en}$ and $\ct_{\en}$, or even $\L_{\enb}$ and $\ct_{\enb}$. Hence equations such as $\L_{\en}([h]) = 0$ and $\ct_{\enb}([h]) = 0$ have sense. The \textbf{cubic torsion} $\bt_{ij}\,^{k} \defeq h^{kp}\L_{\nabla}(h)_{ijp}$, and the \textbf{conformal torsion} $\ct_{ij}\,^{k} \defeq h^{kp}\ct_{\nabla}(h)_{ijp}$ depend only on the conformal class of $h$ and the conformal projective equivalence class of $\nabla$. Given a conformal structure $[h]$, and a conformal projective equivalence class, define density valued tensors $\bt_{ijk} = \bt_{ij}\,^{p}H_{pk}$ and $\ct_{ijk} \defeq \ct_{ij}\,^{p}H_{pk}$. Note that $\bt_{ijk}$ and $\bt_{\nabla}(h)_{ijk}$, are not the same. The conformal and cubic torsions have the symmetries $\ct_{[ij]k} = \ct_{ijk}$, $\ct_{[ijk]} = 0$, and $\ct_{ip}\,^{p} = 0$, and $\bt_{ijk} = \bt_{(ijk)}$ and $\bt_{ip}\,^{p} = 0$. 

\subsubsection{}
A \textbf{conformal projective pair} (\textbf{CP pair}) is a pair $(\en, [h])$ comprising a projective structure $\en$ and a (pseudo-Riemannian) conformal structure $[h]$. The CP pair \textbf{generated} by $\nabla$ and $[h]$ is $(\en, [h])$, where $\en$ is the projective structure generated by $\nabla$. A CP pair $(\en, [h])$ is \textbf{subordinate} to a conformal projective structure $(\enb, [h])$ if $\en \subset \enb$; one says that $(\en, [h])$ \textbf{generates} $(\enb, [h])$. The \textbf{cubic torsion} and \textbf{conformal torsion} of a CP pair are the cubic torsion $\bt_{ij}\,^{k} \defeq H^{kp}\nabla_{(i}H_{jk)}$ and conformal torsion $\ct_{ij}\,^{k} \defeq \tfrac{4}{3}H^{kp}\nabla_{[i}H_{j]k}$ of the underlying conformal projective structure. A CP pair is a \textbf{generalized affine hypersurface structure} (\textbf{AH structure}) if the conformal torsion vanishes. A CP pair for which there vanish both $\bt_{ij}\,^{k}$ and $\ct_{ij}\,^{k}$ is a \textbf{Weyl structure}. A CP pair $(\en, [h])$ is \textbf{Riemannian} if $[h]$ is Riemannian.

\subsubsection{}
A torsion-free affine connection $\nabla$ and a metric $h_{ij}$ determine a pair of one-forms, namely $h^{pq}\nabla_{i}h_{pq} = |\det h|^{-1}\nabla_{i}|\det h|$, which measures the difference between the connections induced by $\nabla$ and the Levi-Civita connection of $D$ on bundles of densities of non-trivial weights, and $h^{pq}\nabla_{p}h_{qi} = -h_{ij}\nabla_{p}h^{jp}$, which is the one-form metrically dual to the divergence of the bivector $h^{ij}$. A torsion-free affine connection $\nabla$ is \textbf{aligned} with respect to the metric $h_{ij}$ if there holds $h^{pq}\nabla_{i}h_{pq} = nh^{pq}\nabla_{p}h_{qi}$. Lemma \ref{special} shows that
\begin{enumerate}
\item A torsion-free affine connection is aligned with respect to a metric $h_{ij}$ if and only if it is aligned with respect to any conformally equivalent metric $e^{f}h_{ij}$.
\item A torsion-free affine connection is projectively equivalent to a unique connection aligned with respect to a given metric.
\end{enumerate}
Consequently it makes sense to speak of an affine connection being aligned with respect to a conformal structure $[h]$, and a CP pair admits a unique \textbf{aligned} representative $\nabla \in \en$ characterized by the requirement that for an arbitrary representative $h \in [h]$ there holds $h^{pq}\nabla_{i}h_{pq} = nh^{pq}\nabla_{p}h_{qi}$. Lemma \ref{special} gives also various equivalent formulations of the alignment condition. While this is routine the proof is included because these will be used constantly in the rest of the paper, often without comment.

\begin{lemma}\label{special}
Given a conformal structure $[h]$, a projective structure $\en$ admits a unique torsion-free representative $\nabla \in \en$ such that there holds each of the following equivalent conditions.
\begin{enumerate}
\item For any $h \in [h]$ there holds $nh^{pq}\nabla_{p}h_{qi} = h^{pq}\nabla_{i}h_{pq}$.
\item For any $h \in [h]$ there holds $\nabla_{i}h_{jk}  = 2\ga_{i}h_{jk} + \bt_{\nabla}(h)_{ijk} + \ct_{\nabla}(h)_{i(jk)}$ with $2n\ga_{i} = h^{pq}\nabla_{i}h_{pq}$.
\item $\nabla_{i}H_{jk} = \bt_{ijk} + \ct_{i(jk)}$. 
\item $\bt_{ijk} = \nabla_{(i}H_{jk)}$ and $\ct_{ijk} = \tfrac{4}{3}\nabla_{[i}H_{j]k}$.
\item $H^{pq}\nabla_{p}H_{qi} = 0$. That is, $\nabla_{i}H_{jk}$ is completely trace-free.
\item $\nabla_{p}H^{ip} = 0$.
\end{enumerate}
\end{lemma}

\begin{proof}
Part of the claim is that in $(1)$ and $(2)$ the connection $\nabla$ does not depend on the choice of $h$ in $[h]$. This follows from tracing $\nabla_{i}(e^{2f}h_{jk}) = e^{2f}(\nabla_{i}h_{jk} + 2df_{i}h_{jk})$ in two different ways. Say that metric $h$ with values in $|\Det \ctm|^{-\la/(n+1)}$ represents $[h]$ if there is a positive section $\mu$ of $|\Det \ctm|^{\la/(n+1)}$ such that $\mu \tensor h \in [h]$. For torsion-free $\nabla \in \en$ there are one-forms $\ga_{i}$ and $\be_{i}$ such that
\begin{align}\label{nablah}
\nabla_{i}h_{jk} & =2\ga_{i}h_{jk} + 2\be_{(j}h_{k)i} + \tf^{h}(\nabla h)_{ijk} = 2\ga_{i}h_{jk} + 2\be_{(j}h_{k)i} + \L_{\nabla}(h)_{ijk} + \ct_{\nabla}(h)_{i(jk)}.
\end{align} 
Tracing \eqref{nablah} in two different ways using the dual weighted bivector $h^{ij}$ shows that 
\begin{align}\label{albeh}\begin{split}
2&(n-1)(n+2)\ga_{i} = h^{pq}((n+1)\nabla_{i}h_{pq} - 2\nabla_{p}h_{qi}), \\ & (n-1)(n+2)\be_{i} = h^{pq}(n\nabla_{p}h_{qi} - \nabla_{i}h_{pq}),
\end{split}
\begin{split}
&2n\ga_{i} + 2\be_{i} = h^{pq}\nabla_{i}h_{pq},\\ &2\ga_{i} + (n+1)\be_{i} = h^{pq}\nabla_{p}h_{qi}.
\end{split}
\end{align}
Let $\bar{\ga}_{i}$ and $\bar{\be}_{i}$ be the the one-forms defined as in \eqref{nablah} with $\bnabla = \nabla + 2\si_{(i}\delta_{j)}\,^{k} \in \en$ in place of $\nabla$ in \eqref{nablah}. From
\begin{align}\label{tnablah}
\tnabla_{i}h_{jk} = \nabla_{i}h_{jk} + (\la - 2)\si_{i}h_{jk} - 2\si_{(j}h_{k)i},
\end{align}
it follows that $2(\bar{\ga}_{i} - \ga_{i}) = (\la - 2)\si_{i}$ and $\bar{\be}_{i} - \be_{i} = -\si_{i}$. Whatever is $\la$, imposing $\be_{i} = 0$ determines $\si_{i}$ uniquely as $\si_{i} = - \bar{\be}_{i}$. This shows there is a unique torsion-free $\nabla \in\en$ such that in \eqref{nablah} there holds $\be_{i} = 0$. In this case \eqref{albeh} shows that any one of equations $\be_{i} = 0$,  $2n\ga_{i} = h^{pq}\nabla_{i}h_{pq}$, and $nh^{pq}\nabla_{p}h_{qi} = h^{pq}\nabla_{i}h_{pq}$, implies the other two. This shows the equivalence of $(1)$ and $(2)$. The equivalence of $(1)$ (for an unweighted metric $h$) and $(5)$ follows from $nH^{pq}\nabla_{p}H_{qi} = nh^{pq}\nabla_{p}h_{qi} - h^{ip}h^{ab}\nabla_{p}h_{ab}$. The equivalence of $(3)$ and $(4)$ is obvious. The remaining conditions follow by specializing the preceeding discussion with $\la = 2$. Because by definition $|\det H| = 1$, there holds $H^{pq}\nabla_{i}H_{pq} = 0$, and so \eqref{albeh} shows $n\ga_{i} = -\be_{i}$. Substitued into \eqref{nablah} this gives
\begin{align}\label{nablaH}
\nabla_{i}H_{jk} & = -\tfrac{2}{n}\be_{i}H_{jk} + 2\be_{(j}H_{k)i} + \L_{ijk} + \ct_{i(jk)},
\end{align} 
in which $(n+2)(n-1)\be_{i} = nH^{pq}\nabla_{p}H_{qi} = -nH_{iq}\nabla_{p}H^{pq}$. As was shown in the preceeding paragraph there is a unique torsion-free $\nabla \in \en$ for which $\be_{i} = 0$, and the vanishing of $\be$ is equivalent to each of the conditions $(3)$, $(5)$, and $(6)$. 
\end{proof}

The aligned representative of a Weyl structure is what is usually called a \textit{Weyl connection}.

\subsubsection{}
From now on except where stated otherwise, $\nabla$ denotes the aligned representative of a CP pair and indices are raised and lowered using $H_{ij}$ and $H^{ij}$. While it may seem strange to bother speaking of the projective structure $\en$ if one works only with a distinguished representative $\nabla \in \en$, later developments, e.g. section \ref{convexprojectivesection}, will show that the perspective is useful.

\subsubsection{}\label{anpara}
The basic example of an AH structure is the following. A hypersurface immersion in flat affine space is \textbf{non-degenerate} if its second fundamental form (which takes values in the normal bundle) is non-degenerate. If the immersion is also co-oriented the second fundamental form determines a conformal structure on the hypersurface. A choice of subbundle transverse to the immersion induces on the hypersurface a torsion-free affine connection, and there is a unique choice of transverse subbundle such that the induced connection is aligned with respect to the conformal structure determined by the second fundamental form and the co-orientation. This choice of transverse subbundle is the \textbf{affine normal subbundle}. These claims are explained in detail in section \ref{affinehyperspheresection}.

\subsubsection{}
An AH structure $(\en, [h])$ determines the \textbf{pencil} of AH structure $(\pen, [h])$ defined as follows. The aligned representative $\pnabla$ of $(\pen, [h])$ is defined by $\pen = \nabla + t\bt_{ij}\,^{k}$.  Because the conformal torsion vanishes, $\pnabla$ is torsion free, and it is straightforward to check that $\pnabla_{i}H_{jk} = (1-2t) \nabla_{i}H_{jk}$, which is trace-free, so the projective structure generated by $\pnabla$ forms with $[h]$ an AH structure with cubic torsion $\brt{\bt}_{ij}\,^{k} = (1-2t)\bt_{ij}\,^{k}$. When $t = 1$ write $\bnabla$ and $\ben$ in place of $\pnabla$ and $\pen$. The AH structure \textbf{conjugate} to the AH structure $(\en, [h])$ is $(\ben, [h])$, and its cubic torsion is $\bar{\bt}_{ij}\,^{k} = -\bt_{ij}\,^{k}$. When $t = 1/2$ there are written $\lcp$ and $[\lcp]$ in place of $\pnabla$ and $\pen$, and $([\lcp], [h])$ is called the \textbf{underlying Weyl structure} of the AH structure. 

\subsubsection{}\label{conormalpara}
The conormal Gau\ss\, map of a non-degenerate co-oriented immersed hypersurface in flat affine space sends a point of the hypersurface to the annihilator of the space tangent to the hypersurface at the point. The pullback of the flat projective structure on the projectivization of the dual to the flat affine space via this conormal Gau\ss\, map forms with the conformal structure determined by the second fundamental form and the co-orientation the AH structure conjugate to that determined by the affine normal subbundle. These claims are explained in detail in section \ref{conormalsection}.

\subsubsection{}
The \textbf{curvature} of a CP pair is defined to be the curvature $R_{ijk}\,^{l}$ of the aligned representative $\nabla$, the convention being that $2\nabla_{[i}\nabla_{j}X^{k} = R_{ijp}\,^{k}X^{p}$. The \textbf{Ricci curvature} is defined by $R_{ij} = R_{pij}\,^{p}$. It is necessary to consider also the trace $Q_{ij} \defeq R_{ip}\,^{p}\,_{j}$. In general the trace-free symmetric parts of $Q_{ij}$ and $R_{ij}$ are independent. The \textbf{weighted scalar curvature} $R$ is the density $R \defeq R_{p}\,^{p} =  H^{pq}R_{pq} = Q_{p}\,^{p}$. Often the qualifier \textit{weighted} will be omitted, and $R$ will be called simply the \textit{scalar curvature}. It does not make sense to speak of the numerical value of $R$ because $R$ takes values in the line bundle $|\det \ctm|^{1/n}$; however it does make sense to speak of the vanishing of $R$ and because $|\det \ctm|^{1/n}$ is oriented, to speak of the positivity or negativity of $R$. A CP pair is \textbf{proper} if its weighted scalar curvature is non-vanishing.

\subsubsection{}
A CP pair is \textbf{exact} if there is a representative $h \in [h]$ such that $\nabla_{i}\det h = 0$ for the aligned representative $\nabla \in \en$. If there is such an $h$ it is determined uniquely up to positive homothety (on each connected component of $M$). Such an $h$ will be called a \textbf{distinguished representative} of the CP pair. A CP pair is exact if and only if there is a global $\nabla$-parallel non-vanishing density of non-trivial weight, for if there is such a density, then some power of it is a non-vanishing density $\mu$ such that $h_{ij} = \mu \tensor H_{ij} \in [h]$ verifies $\nabla |\det h| = \nabla \mu^{n} = 0$ (the converse is obvious). 

For example, as is explained in section \ref{equiaffinesection}, the AH structure induced on a non-degenerate co-oriented hypersurface in flat affine space is exact. 

The \textbf{Faraday form} $F_{ij}$ of a CP pair $(\en, [h])$ is the curvature of the covariant derivative induced on the line bundle of $-1/n$-densities by the aligned representative $\nabla \in \en$. If $R_{ijk}\,^{l}$ is the curvature of $\nabla$, then by definition and the traced algebraic Bianchi identity there hold
\begin{align}\label{faradayexplicit}
nF_{ij} = R_{ijp}\,^{p} = -2R_{[ij]}. 
\end{align} 
The CP pair $(\en, [h])$ is \textbf{closed} if $F_{ij} \equiv 0$. That a CP pair be exact (resp. closed) is not the same as that $F_{ij}$ be exact (resp. closed); in particular, because $F_{ij}$ is the curvature of a trivial line bundle, it is always exact.

The \textbf{Faraday primitive} $\ga_{i}$ associated to $h \in [h]$ is the one-form $\ga_{i}$ defined by $2n\ga_{i} = h^{pq}\nabla_{i}h_{pq}$. From the Ricci identity follows $d\ga_{ij} = 2\nabla_{[i}\ga_{j]} = -F_{ij}$, so that $\ga_{i}$ is a primitive for $-F_{ij}$. If $\tilde{h}_{ij} = e^{2f}h_{ij} \in [h]$ then the corresponding one-form $\tilde{\ga}_{i}$ differs from $\ga_{i}$ by an exact one-form, $\tilde{\ga}_{i} = \ga_{i} + df_{i}$. The equivalence class $\{\ga\}$ of one-forms so determined is the \tbf{equivalence class of Faraday primitives} \textbf{induced by $(\en, [h])$}. 

Evidently a proper CP pair is exact if its weighted scalar curvature is parallel. For a proper CP pair there holds $\nabla_{[i}\nabla_{j]}R = - F_{ij}R$, and so the one-form $2\ga_{i} = -R^{-1}\nabla_{i}R$ satisfies $d\ga_{ij} = 2\nabla_{[i}\ga_{j]} = -R^{-1}\nabla_{[i}\nabla_{j]}R = F_{ij}$.

\subsubsection{}
The Faraday primitives of the AH structures $(\en, [h])$ and $(\pen, [h])$ associated to $h \in [h]$ are the same, and so these AH structures determine the same equivalence class of Faraday primitives; in particular one is closed (resp. exact) if and only if the other is closed (resp. exact). This observation has the consequence that most properties of the Faraday curvature of Weyl structures hold equally for AH structures. For example, Lemma \ref{fcoclosedlemma} generalizes (trivially) Theorem $2.5$ of \cite{Calderbank-faraday}.
\begin{lemma}\label{fcoclosedlemma}
A definite signature CP pair on a manifold of dimension $n \neq 4$ is closed if and only if $\nabla^{p}F_{ip} = 0$.
\end{lemma}
\begin{proof}
For any CP pair there holds
\begin{align}\label{fcoclosed}
2\nabla_{[p}\nabla_{q]}F^{pq} =  -2R_{[pq]}F^{pq}- \tfrac{4}{n}R_{pqa}\,^{a}F^{pq} = -2R_{[pq]}F^{pq} - 4F_{pq}F^{pq} = (n-4)F_{pq}F^{pq}.
\end{align}
If $n \neq 4$ then \eqref{fcoclosed} shows that $\nabla^{p}F_{ip} = 0$ implies $F_{ij}$ is $[h]$-null; in the definite signature case this holds if and only if $F_{ij} \equiv 0$.
\end{proof}

\subsubsection{}
Lemma \ref{conformalprojectivediff} gives further motivation for the name \textit{conformal projective structure}, as it shows that such a structure can be viewed as a conformal class of projective structures; the difference tensor of the aligned representatives of any two projective structures generating the conformal projective structure has the form of the difference tensor of the Levi-Civita connections of two metrics representing a conformal structure.
\begin{lemma}\label{conformalprojectivediff}
If $(\ten, [h])$ and $(\en, [h])$ are CP pairs subordinate to a given conformal projective structure $(\enb, [h])$, then the difference tensor $\tnabla - \nabla$ of the aligned reprentatives of $(\ten, [h])$ and $(\en, [h])$ has the form $2\al_{(i}\delta_{j)}\,^{k} - \al^{k}H_{ij}$, in which $-\al_{i} = \tilde{\ga}_{i} - \ga_{i}$ for the Faraday primitives $\tilde{\ga}_{i}$ and $\ga_{i}$ associated to $h \in [h]$. In particular, the Faraday curvatures of AH structures generating the same conformal projective structure are cohomologous. %
\end{lemma}
\begin{proof}
Let $h \in [h]$. Because $\tnabla \in \ten$ and $\nabla\in \en$ are by assumption conformal projectively equivalent their difference tensor has the form $\Pi_{ij}\,^{k} = 2\al_{(i}\delta_{j)}\,^{k} + H_{ij}\be^{k}$ for some one-forms $\al_{i}$ and $\be_{i}$, and from Lemma \ref{special} it follows easily that if $\tnabla$ and $\nabla$ are aligned then $\be_{i} = -\al_{i}$. In this case the difference $\tnabla_{i}h_{jk} - \nabla_{i}h_{jk}$ is $-2\al_{i}h_{jk}$, from which it follows that $-\al_{i} = \tilde{\ga}_{i} - \ga_{i}$. 
\end{proof}

\subsubsection{}
If $(\en, [h])$ is a CP pair and $h \in [h]$ has corresponding Faraday primitive $\ga_{i}$ then the Levi-Civita connection $D$ of $h$ is related to the aligned representative $\nabla$ by
\begin{align}\label{dnabladiff}
D = \nabla + \tfrac{1}{2}\bt_{ij}\,^{k} - \ct^{k}\,_{(ij)} + 2\ga_{(i}\delta_{j)}\,^{k} - H_{ij}\ga^{k} .
\end{align}
Equation \eqref{dnabladiff} shows how to build from a metric and a one-form a CP pair with specified torsion. 

\subsubsection{}
It makes sense to restrict an AH structure to a non-degenerate submanifold. If $[h]$ is a conformal structure on $M$, an immersion $i:N \to M$ is \textbf{non-degenerate} if $Ti(p)(T_{p}N)$ is never coisotropic for any $p \in N$. In this case $i^{\ast}([h]) \defeq [i^{\ast}(h)]$ is a conformal structure on $N$; it makes sense to define $i^{\ast}(H)_{ij}$ to be the normalized metric of $i^{\ast}([h])$. The CP pair \tbfs{induced}{CP pair} on $N$ by the immersion $i$ is denoted $i^{\ast}([h], \en)$ and equals the CP pair $(i^{\ast}([h]), \ben$ with $\ben$ the projective structure generated by a connection $\bnabla$ defined as follows. For each $h \in [h]$ let $\ga_{i}$ be the corresponding Faraday primitive. Then $\bar{h} \defeq i^{\ast}(h)$ and $\bar{\ga} \defeq i^{\ast}(\ga)$ are defined, as are the tensors $\bar{\bt} \defeq i^{\ast}(\bt_{\nabla}(h))$ and $\bar{\ct} \defeq i^{\ast}(\ct_{\nabla}(h))$. Let $\bD$ be the Levi-Civita connection of $\bar{h}$ and define $\bnabla$ so that there holds \eqref{dnabladiff} with the barred tensors in place of the unbarred tensors. If $\tilde{h} = fh$ then $\tilde{\ga}_{i} = \ga_{i} + \si_{i}$ where $2\si_{i} = d\log{f}_{i}$, while the Levi-Civita connections of $i^{\ast}(\tilde{h})$ and $i^{\ast}(h)$ differ by $2i^{\ast}(\si)_{(i}\delta_{j)}\,^{k} - \bar{h}_{ij}\bar{h}^{kp}i^{\ast}(\si)_{p}$, from which it follows that $\bnabla$ so defined is independent of the choice of $h$. It is straightforward to check that $\bnabla$ is the aligned representative of the CP pair $(i^{\ast}([h]), \ben)$, and so it makes sense to define $(i^{\ast}([h]), \ben)$ to be the induced CP pair. 

\subsubsection{}
Structures similar to AH structures have been studied intensively. Here are made only some limited remarks. If $\nabla$ is a torsion-free affine connection a section $a_{i_{1}\dots i_{k}} \in \Ga(S^{k}(\ctm))$ is a \textbf{Codazzi tensor} if $\nabla_{[i}a_{j]i_{1}\dots i_{k-2}} =0$, and if $h_{ij}$ is a metric, a pair $(\nabla, h)$ is a \textbf{Codazzi structure} if $h$ is non-degenerate and Codazzi. By the curvature of a Codazzi structure $(\en, h)$ is meant the curvature of $\nabla$ (and not that of $h$), and a Codazzi structure is \textbf{flat} if $\nabla$ is flat. A flat Codazzi structure is called a \textbf{Hessian structure} or, following \cite{Cheng-Yau-realmongeampere}, an \textbf{affine K\"ahler structure}. Much of what is known about Hessian structures is recounted in H. Shima's book \cite{Shima}. A Hessian structure $(\nabla, h_{ij} = \nabla_{i}\nabla_{j}f)$ for which $\Pi_{ijk} = \nabla_{i}h_{jk} = \nabla_{i}\nabla_{j}\nabla_{k}f$ and which satisfies the \textbf{associativity equations} (or \textbf{WDVV equations}) $\Pi_{pl[i}\Pi_{j]k}\,^{p} = 0$ is a special kind of \textbf{Frobenius manifold} (see e.g. \cite{Manin-frobenius} or \cite{Chen-Kontsevich-Schwarz}). 

Note that a Codazzi structure $(\nabla, h)$ generates an AH structure but that $\nabla$ need \textit{not} be the aligned representative, because it need not be that the $h$-volume density is preserved by $\nabla$. A natural example of a Codazzi structure $(\nabla, h)$ for which $\nabla$ is in general not aligned with respect to $h$ is given in Theorem \ref{inducedcodazzitheorem}. The point of view taken here is that AH structures are to Codazzi structures as projective is to affine.

A pseudo-Riemannian metric $h$ on $M$ determines the map $\flat:TM \to \ctm$ define by $\flat(X) = i(X)h$. The following geometric description of Codazzi structures is Theorem $2$ of P. Delano\"e's \cite{Delanoe}.
\begin{theorem}[Theorem $2$ of \cite{Delanoe}]\label{delanoetheorem}
A torsion free affine connection $\nabla$ on $M$ forms with a pseudo-Riemannian metric $h_{ij}$ a Codazzi structure if and only if the image of the horizontal subbundle of $TTM$ determined by $\nabla$ under the differential $T\flat$ of the map $\flat:TM \to \ctm$ determined by $h$ is a Lagrangian subbundle of $T\ctm$ with respect to the tautological symplectic structure on $\ctm$.
\end{theorem}

AH structures admit a similar geometric description. Let $\rho:\ctm \to \projp(\ctm)$ be the canonical projection, and let $C$ be the canonical contact structure on $\projp(\ctm)$, which is the image under $T\rho$ of the kernel of the tautological one-form $\al$ on $\ctm$. Let $\tilde{\flat}:TM \to \ctm$ be the map determined by the conformal metric $\tilde{h} = fh$. 
If $H$ is the homogeneous horizontal subbundle on $TM$ determined by $\nabla$, then the images $T\tilde{\flat}(H)$ and $T\flat(H)$ are not the same, but their images under $T\rho$ determine the same rank $n$ subbundle of $T\projp(\ctm)$, and so the images under $T\rho$ of their intersections with $\ker \al$ determine the same rank $(n-1)$ subbundle of $C$ on $\projp(\ctm)$. By the rank $(n-1)$ subbundle $L$ of $C$ associated to the CP pair $(\en, [h])$ is meant the subbundle determined by the aligned representative $\nabla$ together with $[h]$.

\begin{theorem}\label{ahcharacterizationtheorem}
The rank $(n-1)$ subbundle $L$ of the contact structure $E$ on $\projp(\ctm)$ determined by a CP pair $(\en, [h])$ is Legendrian if and only if the CP pair is AH. 
\end{theorem}
Since no use will made of Theorem \ref{ahcharacterizationtheorem} its proof is omitted. Together theorems \ref{delanoetheorem} and \ref{ahcharacterizationtheorem} give some credibility to the point of view taken here that the notion of AH structure is the projectivization of the notion of Codazzi structure.

\subsection{Curvature of AH structures}
\subsubsection{}\label{curvaturetensorsection}
Let $\rea_{n}$ denote $(\rea^{n})^{\ast}$. Elements of the $GL(n, \rea)$-module $\curvmod \defeq \{B_{ijkl} \in \tensor^{4}\rea_{n}: B_{(ij)kl} = 0 = B_{[ijk]l}\}$ will be referred to as \tbf{algebraic curvature tensors}. The space $\curvmod$ decomposes as a direct sum $\curvmod = \curvmod_{1}\oplus\curvmod_{2}\oplus\curvmod_{3}$ in which
\begin{align*}
\curvmod_{1} &= \{B_{ijkl} \in \curvmod: B_{ij(kl)} = B_{ijkl}\},&
\curvmod_{2} &= \{B_{ijkl} \in \curvmod: B_{ij[kl]} = B_{ijkl}\},\\
\curvmod_{3} &= \{B_{ijkl} \in \curvmod: 8B_{ijkl} = 9B_{[ij|k|l]} - 3B_{[ij|l|k]}\},.
\end{align*}
The projections $\P_{i}:\curvmod \to \curvmod_{i}$ onto the summands are given for $B_{ijkl} \in \curvmod$ by
\begin{align*}
&\P_{1}(B)_{ijkl} = \tfrac{3}{4}\left(B_{i(jkl)} - B_{j(ikl)}\right),&
&\P_{2}(B)_{ijkl} = \tfrac{1}{2}\left(B_{ij[kl]} + B_{kl[ij]}\right),&\\
&\P_{3}(B)_{ijkl} = \tfrac{3}{8}\left(3B_{[ij|k|l]} - B_{[ij|l|k]}\right).
\end{align*}
These are orthogonal in the sense that $\P_{i}\circ \P_{j} = 0$ if $i \neq j$. The decomposition of $B_{ijkl} \in \curvmod$ into $GL(n, \rea)$-irreducible components is given explicitly by
\begin{align}\label{curvaturedecom}
B_{ijkl}   = \tfrac{3}{4}\left(B_{i(jkl)} - B_{j(ikl)}\right) + \tfrac{1}{2}\left(B_{ij[kl]} + B_{kl[ij]}\right) +\tfrac{3}{8}\left(3B_{[ij|k|l]} - B_{[jk|l|i]}\right),
\end{align}
Note that $\P_{3}(B)_{[ijk]l} = \P_{3}(B)_{[ijkl]} = 0$, while $\P_{3}(B)_{[ij|k|l]} = B_{[ij|k|l]}$; the last equality implies $\P_{3}(B)_{ijkl}$ is the result of applying $\Psi$ to the tensor $B_{[ij|k|l]}$.

\subsubsection{}
Given a metric $H_{ij}$, raise and lower indices using $H^{ij}$. From $\P_{i}\circ \P_{j} = 0$ if $i \neq j$ it follows that the subspaces $\curvmod_{i}$ are pairwise orthogonal with respect to $H_{ij}$. Each of the subspaces $\curvmod_{i}$ can be decomposed further by taking traces. Since $-2B_{ijp}\,^{p} = B_{p[ij]}\,^{p}$ for $B_{ijkl} \in \curvmod$, all traces of such a $B_{ijkl}$ can be expressed as linear combinations of $B_{pij}\,^{p}$ and $B_{ip}\,^{p}\,_{j}$. Define the \tbf{Ricci trace} of $B_{ijkl} \in \curvmod$ to be $B_{ij} \defeq B_{pij}\,^{p}$ and the \tbf{scalar trace} to be $B_{p}\,^{p} = -B_{pq}\,^{pq}$. In fact, as will be explained now, for $B_{ijkl} \in \curvmod_{i}$, $i = 1, 2, 3$, it is always the case that $B_{ip}\,^{p}\,_{j}$ is a scalar multiple of $B_{ij}$. For $B_{ijkl} \in \curvmod_{1}$, $B_{ip}\,^{p}\,_{j} = - B_{pij}\,^{p}$, while for $B_{ijkl} \in \curvmod_{2}$, $B_{ip}\,^{p}\,_{j} = B_{pij}\,^{p}$. For $B_{ijkl} \in \curvmod$, there holds
\begin{align}\label{p3bsym}
8\P_{3}(B)_{i(jk)l} = 3(B_{li(jk)} + B_{i(jk)l} - B_{l(jk)i}) = 12\P_{3}(B)_{li(jk)},
\end{align} 
so that for $B_{ijkl} \in \curvmod_{3}$ there holds $2B_{i(jk)l} = 3B_{li(jk)}$. Tracing this in $il$ shows $B_{p(ij)}\,^{p} = 0$, while tracing it in $jk$ and tracing $8B_{ijkl} = 9B_{[ij|k|l]} - 3B_{[ij|l|k]}$ in $kl$ gives $2B_{ip}\,^{p}\,_{j} = -3B_{ijp}\,^{p} = 6B_{p[ij]}\,^{p} = 6B_{pij}\,^{p}$, from which it is evident that all traces of $B_{ijkl}$ are determined by $B_{ij}$, which is in this case skew-symmetric. The relations of the Ricci traces of $\P_{i}(B)$ to the traces of $B_{ijkl} \in \curvmod$ are as follows.
\begin{align}\label{pbtraces}
\begin{split}
\P_{1}(B)_{ij} = \tfrac{1}{2}B_{p(ij)}\,^{p} - \tfrac{1}{2}B_{(i|p|}\,^{p}\,_{j)}& + \tfrac{3}{4}B_{p[ij]}\,^{p} - \tfrac{1}{4}B_{[i|p|}\,^{p}\,_{j]},\\
\P_{2}(B)_{ij} = \tfrac{1}{2}B_{p(ij)}\,^{p} + \tfrac{1}{2}B_{(i|p|}\,^{p}\,_{j)}, &\qquad
\P_{3}(B)_{ij} = \tfrac{1}{4}B_{p[ij]}\,^{p} + \tfrac{1}{4}B_{[i|p|}\,^{p}\,_{j]}.
\end{split}
\end{align}
It follows from \eqref{pbtraces} that the scalar trace of $B_{ijkl}$ is given by $-B_{pq}\,^{pq} = B_{p}\,^{p} = \P_{2}(B)_{p}\,^{p}$, while $\P_{1}(B)_{p}\,^{p} = 0 = \P_{3}(B)_{p}\,^{p}$.

\subsubsection{}\label{fgtraces}
Define a linear map $\r: \ext^{2}(\rea_{n}) \to \C$ by $\r(F)_{ijkl} \defeq F_{ij}H_{kl} - F_{k[i}H_{j]l} + F_{l[i}H_{j]k}$. Applying to $\r(F)_{ijkl}$ the decomposition \eqref{curvaturedecom}, there result $\P_{2}(\r(F))_{ijkl} = 0$ and 
\begin{align}\label{fvdecom}
\r(F)_{ijkl} = F_{ijkl} + G_{ijkl}, 
\end{align}
in which the tensors $F_{ijkl}$ and $G_{ijkl}$ are defined by
\begin{align}
F_{ijkl} \defeq \P_{1}(\r(F))_{ijkl} & = \tfrac{3}{4}(F_{i(j}H_{kl)} - F_{j(i}H_{kl)}) = \tfrac{1}{2}F_{ij}H_{kl} - \tfrac{1}{2}F_{k[i}H_{j]l} - \tfrac{1}{2}F_{l[i}H_{j]k},\\
G_{ijkl} \defeq \P_{3}(\r(F))_{ijkl} & = \tfrac{3}{4}(3F_{[ij}H_{l]k} - F_{[ij}H_{k]l}) = \tfrac{1}{2}F_{ij}H_{kl} - \tfrac{1}{2}F_{k[i}H_{j]l} + \tfrac{3}{2}F_{l[i}H_{j]k}.
\end{align}
There hold $-2\r(F)_{pij}\,^{p} = nF_{ij}$ and 
From \eqref{pbtraces} there follow
\begin{align}
\label{prftraces1} &-(n+2)F_{ij} = 4F_{pij}\,^{p} = 4\P_{1}(\r(F))_{ij},& &(2-n)F_{ij}= 4G_{pij}\,^{p} = 4\P_{3}(\r(F))_{ij},\\
\label{prftraces2} &nF_{ij} = -2\r(F)_{pij}\,^{p},& & (4-n)F_{ij} = 2\r(F)_{[i|p|}\,^{p}\,_{j]}. 
\end{align}
It follows from \eqref{prftraces2} that $\r$ is injective, so an isomorphism onto its image.

\subsubsection{}\label{fdecomsection}
Let $(\en, [h])$ be an AH structure. The constructions of section \ref{curvaturetensorsection} apply fiberwise on $\ctm$. Thus there are defined the (weighted) tensor $\r(F)_{ijkl} \defeq F_{ij}H_{kl} - F_{k[i}H_{j]l} + F_{l[i}H_{j]k}$, which is a section of the module $\curvmod(\ctm)$ of algebraic curvature tensors, and the tensors $F_{ijkl}$ and $G_{ijkl}$ determined from it as in \eqref{fvdecom}, via the the decomposition \eqref{curvaturedecom}. All traces of $F_{ijkl}$ are determined by $-4F_{pij}\,^{p} = (n+2)F_{ij}$, from which it follows that the CP pair is closed if and only if $F_{ijkl} = 0$. All traces of $G_{ijkl}$ are determined by $4G_{pij}\,^{p} = (2-n)F_{ij}$, from which it follows that on a manifold of dimension $n >2$ a CP pair is closed if and only if $G_{ijkl} = 0$.

\subsubsection{}
Let $(\en, [h])$ be an AH structure. The Ricci identity implies
\begin{align}\label{skewnablah}
 -R_{ij(kl)} + F_{ij}H_{kl} = \nabla_{[i}\nabla_{j]}H_{kl} = \nabla_{[i}\bt_{j]kl}.
\end{align}
Skew-symmetrizing \eqref{skewnablah} in $ijk$ gives $R_{[ij|l|k]}  = 2F_{[ij}H_{k]l}$, and from this and the definition of $V_{ijkl}$ there results $\P_{3}(R)_{ijkl}  = G_{ijkl}$. Define $\uf_{ijkl} \defeq \P_{1}(R - \r(F))_{ijkl}$ and $T_{ijkl} \defeq \P_{2}(R)_{ijkl}$. Since $\P_{2}(\r(F)) = 0$, the decomposition \eqref{curvaturedecom} yields $R_{ijkl} = \uf_{ijkl} + T_{ijkl} + F_{ijkl} + G_{ijkl}$. Note the orthogonality relations such as $\uf_{ijkl}G^{ijkl} = \uf_{ijkl}T^{ijkl}  = 0$. The factors involving $\r(F)$ are subtracted off because the resulting tensors transform nicely under conjugation of the AH structure, as will be seen below in section \ref{curvatureahpencilsection}.

Tracing \eqref{skewnablah} in $il$ and relabeling yields
\begin{align}
\label{qrdiff} Q_{ij}- R_{ij}& = 2F_{ij} + \nabla_{p}\bt_{ij}\,^{p}  - \bt_{ij}.
\end{align}
Taking the skew-symmetric part of \eqref{qrdiff}, and using \eqref{faradayexplicit} to simplify the result yields $2Q_{[ij]}  = (4-n)F_{ij}$. All traces of $T_{ijkl}$ and $\uf_{ijkl}$ are expressible in terms of their Ricci traces, the expressions of which in terms of $R_{ij}$ and $Q_{ij}$ follow from \eqref{pbtraces}:
\begin{align}
&T_{ij}  = T_{pij}\,^{p} = \tfrac{1}{2}R_{(ij)} + \tfrac{1}{2}Q_{(ij)},& 
&\uf_{(ij)}  = \uf_{p(ij)}\,^{p} = \tfrac{1}{2}R_{(ij)} - \tfrac{1}{2}Q_{(ij)},&  
\end{align}
and
\begin{align}
\begin{split}\label{uijskew}
\uf_{[ij]} & = \tfrac{3}{4}R_{[ij]} - \tfrac{1}{4}Q_{[ij]} + \tfrac{(n+2)}{4}F_{ij} = \tfrac{4-n}{8}F_{ij} - \tfrac{1}{4}Q_{[ij]} = 0.
\end{split}
\end{align}
Taking the symmetric part of \eqref{qrdiff} and using \eqref{uijskew} shows
\begin{align}
\label{uijsym2}& \uf_{ij} = \uf_{(ij)} = -\tfrac{1}{2}\nabla_{p}\bt_{ij}\,^{p} +\tfrac{1}{2} \bt_{ij},& 
\end{align}
which will play an important role in deducing consequences of the Einstein equations. Tracing \eqref{qrdiff} shows $Q_{p}\,^{p} = R$, so that $T_{p}\,^{p} = R$ and $\uf_{p}\,^{p} = 0$.

\subsubsection{}
The decompositions of $T_{ijkl}$ and $\uf_{ijkl}$ into conformally irreducible components are as follows.
\begin{align}
\label{aijkldefined} 
\begin{split}
&T_{ijkl}  = A_{ijkl} - 2H_{l[i}A_{j]k} + 2H_{k[i}A_{j]l}, \\
&\uf_{ijkl} = E_{ijkl} - 2H_{l[i}E_{j]k} - 2H_{k[i}E_{j]l},
\end{split}
\end{align}
in which $A_{ijkl}$ and $E_{ijkl}$ are completely $H$-trace free and
\begin{align}
\label{aij}
\begin{split}
&A_{ij}  = A_{(ij)} \defeq \tfrac{1}{2-n}\left(T_{ij} + \tfrac{R}{2(1-n)}H_{ij}\right), 
\qquad A_{p}\,^{p} = \tfrac{1}{2(1-n)}R,\\
& E_{ij} \defeq -\tfrac{1}{n}\uf_{ij}, \qquad \qquad \qquad E_{p}\,^{p}  = 0.\\
\end{split}
\end{align}
The tensors $T_{ijkl}$ and $A_{ijkl}$ have the algebraic symmetries of a metric curvature tensor, while $\uf_{ijkl}$ and $E_{ijkl}$ have the algebraic symmetries of a symplectic curvature tensor.

\subsubsection{}
It is convenient to define some tensors quadratic in the cubic torsion.
\begin{align}\label{tbtdefined}\begin{split}
&\bt_{ijkl} \defeq 2\bt_{k[i}\,^{p}\bt_{j]lp}, \qquad \bt_{ij} \defeq \bt_{ip}\,^{q}\bt_{jq}\,^{p} = \bt_{pij}\,^{p}, \qquad \nbt \defeq \bt_{p}\,^{p} = \bt^{ijk}\bt_{ijk},\\ 
&\tbt_{ij} \defeq \tfrac{1}{2-n}\left(\bt_{ij} + \tfrac{1}{2(1-n)}\nbt H_{ij} \right),\\
&\tbt_{ijkl} \defeq \bt_{ijkl} + \tfrac{2}{2-n}(H_{l[i}\bt_{j]k} - H_{k[i}\bt_{j]l})+ \tfrac{2}{(n-1)(n-2)}\nbt H_{l[i}H_{j]k}\\
 & \qquad = \bt_{ijkl} + 2H_{l[i}\tbt_{j]k} - H_{k[i}\tbt_{j]l}.
\end{split}
\end{align}
Observe that $\bt_{ijkl}$ has the algebraic symmetries of a metric curvature tensor. The completely trace-free part of $\bt_{ijkl}$ is $\tbt_{ijkl}$, and the second equality of its definition shows that it has the form of a conformal Weyl tensor, and that $\tbt_{ij}$ has the form of the conformal Schouten tensor.

\subsubsection{}
AH structures are the curved generalizations of the projective analogues of Frobenius manifolds (see e.g. \cite{Manin-frobenius} or \cite{Hitchin-frobenius}). The cubic torsion $\bt_{ij}\,^{k}$ can be regarded as making the space of sections $\Ga(TM)$ into a commutative, non-associative algebra with the multiplication $(X\mlt Y)^{i} \defeq X^{p}Y^{q}\bt_{pq}\,^{i}$. The tensor $\bt_{ijkl}$ (resp. $\tbt_{ijkl}$) will be called the \textbf{(resp. conformal) non-associativity tensor} of the AH structure, because the algebra $(\Ga(TM), \mlt)$ is associative if and only if $\bt_{ijkl} = 0$. This point of view will be useful in section \ref{cubicformsection}.

\subsubsection{}
By Lemma \ref{weylcriterion}, the tensors $A_{ijkl}$ and $\tbt_{ijkl}$ vanish identically in dimension at most $3$, and $E_{ijkl}$ vanishes in dimension $2$.

\subsubsection{}
Applying to \eqref{skewnablah} the projection $\P_{1}$ and using \eqref{skewnablah} yields
\begin{align}\label{nbtskew}
\begin{split}
R_{ij(kl)} - F_{ij}H_{kl} = \nabla_{[i}\bt_{j]kl} &= -\P_{1}(R)_{ijkl} + F_{ijkl} \\
&= -\uf_{ijkl} = -E_{ijkl} + 2H_{l[i}E_{j]k}  + 2H_{k[i}E_{j]l}.
\end{split}
\end{align}
Together \eqref{uijsym2} and \eqref{nbtskew} show that $-E_{ijkl}$ is the completely $H$-trace free part of $\nabla_{[i}\bt_{j]kl}$. From \eqref{nbtskew} there follows
\begin{align}\label{nablabtupdown}
2\nabla_{[i}\bt_{j]k}\,^{l} - \bt_{ijk}\,^{l} = 2H^{lp}\nabla_{[i}\bt_{j]kp} = -2\uf_{ijk}\,^{l}.%
\end{align}
Using \eqref{nablabtupdown}, equations \eqref{uijsym2} and \eqref{nbtskew} can be rewritten in the perhaps more natural forms
\begin{align}\label{eijup}
&\nabla^{p}\bt_{pi}\,^{j} = 2nE_{i}\,^{j},&&  \nabla^{[i}\bt^{j]}\,_{kl} + \tfrac{2}{n}\nabla^{p}\bt_{p(k}\,^{[i}\delta_{l)}\,^{j]} = -E^{ij}\,_{kl}.
\end{align}
Straightforward computation using $2\bt^{abc}\nabla_{i}\bt_{abc} = \nabla_{i}\bt + 3 \bt_{i}\,^{pq}\bt_{pq}$, the definition of $\bt_{ij}$, \eqref{nbtskew}, and \eqref{eijup} shows 
\begin{align}
\label{divbtij}2\nabla^{p}\bt_{ip} & = \nabla_{i}\bt + \bt_{i}\,^{pq}\bt_{pq} + 4\bt^{abc}E_{iabc}  + 4(n-2)\bt_{i}\,^{pq}E_{pq}.
\end{align}

\subsubsection{Curvature of the AH pencil}\label{curvatureahpencilsection}
Let $(\pen, [h])$ be the pencil of AH structures generated by the AH structure $(\en, [h])$, and let $(\ben, [h])$ be the conjugate AH structure. Decorate with a $\,\brt{\,}\,$ (resp. $\,\bar{\,}\,$) the tensors derived from the curvature $\brt{R}_{ijkl}$ of $(\pen, [h])$ (resp. $\bar{R}_{ijkl}$ of $(\ben, [h])$). By definition $\brt{R}_{ijk}\,^{l} - R_{ijk}\,^{l} = 2t\nabla_{[i}\bt_{j]k}\,^{l} - t^{2}\bt_{ijk}\,^{l}$, so by \eqref{nbtskew} there holds
\begin{align}
\label{conjugateblaschkecurvature2}
\begin{split}
\brt{R}_{ijkl} &= R_{ijkl} + 2t\nabla_{[i}\bt_{j]kl} + t(1-t)\bt_{ijkl} 
 = R_{ijkl} - 2t\uf_{ijkl} + t(1-t)\bt_{ijkl} .
\end{split}
\end{align} 
Substituting into \eqref{conjugateblaschkecurvature2} the expression for $\nabla_{[i}\bt_{j]kl}$ given by \eqref{skewnablah} yields 
\begin{align}
\label{conjugateblaschkecurvature}&\brt{R}_{ijkl} = (1-t)R_{ijkl} - tR_{ijlk} + 2tF_{ij}H_{kl}  + t(1-t)\bt_{ijkl}.
\end{align}
Tracing \eqref{conjugateblaschkecurvature2} in $kl$ and using $\uf_{[ij]} = 0$ shows $\brt{R}_{ijp}\,^{p} = R_{ijp}\,^{p}$, so that $\brt{F}_{ij} = F_{ij}$ and therefore also $\brt{F}_{ijkl} = F_{ijkl}$ and $\brt{G}_{ijkl} = G_{ijkl}$. Hence an AH structure $(\en, [h])$ is closed if and only if for every $t$ the AH structure $(\pen, [h])$ is closed. With this information, decomposing \eqref{conjugateblaschkecurvature2} by symmetries yields 
\begin{align}
\label{brttijkl}&\brt{T}_{ijkl} = T_{ijkl} + t(1-t)\bt_{ijkl},&&\brt{\uf}_{ijkl} = (1-2t)\uf_{ijkl},\\
\label{brtaijkl}&\brt{A}_{ijkl} = A_{ijkl} +  t(1-t)\tbt_{ijkl},& &\brt{E}_{ijkl} = (1-2t)E_{ijkl},&
\end{align}
while tracing \eqref{conjugateblaschkecurvature} yields
\begin{align}
\label{blconjricci} &\brt{R}_{(ij)} = (1-t)R_{(ij)} + tQ_{(ij)} + t(1-t)\bt_{ij},& &\brt{R} = R + t(1-t)\nbt,&\\
&\brt{Q}_{(ij)} = (1-t)Q_{(ij)} + tR_{(ij)} + t(1-t)\bt_{ij},& &\brt{\uf}_{ij} = (1-2t)\uf_{i)},&\\
\label{brteij}&\brt{T}_{ij}  = T_{ij} + t(1-t)\bt_{ij},&  &\brt{E}_{ij} = (1-2t)E_{ij},&\\ 
\label{brtaij}&\brt{A}_{ij} = A_{ij} + t(1-t)\tbt_{ij},&
&\brt{F}_{ij} = F_{ij},&\\
\label{nbscalinv}&\brt{\nabla}_{i}\brt{R} = \nabla_{i}R + t(1-t)\nabla_{i}\nbt,& &\brt{\nabla}^{p}\brt{F}_{ip} = \nabla^{p}F_{ip}.
\end{align}

\subsubsection{}
The equality $\brtop\bnabla =\, \pnabla$ means that expressions for the curvature of an AH structure in terms of the curvature of its conjugate can be obtained by interchanging barred and unbarred tensors, and replacing $t$ by $(1-t)$.

Let $(\en, [h])$ and $(\ben, [h])$ be conjugate AH structures. Then \eqref{brttijkl}, \eqref{brteij}, and \eqref{brtaij} show that $\bar{A}_{ij} = A_{ij}$ and $\bar{A}_{ijkl} = A_{ijkl}$ while $\bar{E}_{ij} = -E_{ij}$ and $\bar{E}_{ijkl} = -E_{ijkl}$. Tensors such as $T_{ijkl}$ and $A_{ijkl}$ (resp. $\uf_{ijkl}$ and $E_{ijkl}$) unchanged under conjugacy (multiplied by $-1$ under conjugacy) will be called \tbf{self-conjugate} (resp. \tbf{anti-self-conjugate}). %

It seems that every interesting class of AH structures defined by some condition on the curvatures of its constituents has the property that the defining conditions are preserved under conjugacy.

\begin{remark}
For $4$-dimensional AH structures there should be considered also the usual notion of (anti)-self-duality for $A_{ijk}\,^{l}$. Although this will not be discussed here, it seems likely that some results for self-dual Weyl structures will generalize to this setting. The first step will be reinterpreting AH structures in terms of the usual twistor formalism. %
\end{remark}

\subsubsection{}
Equation \eqref{conjugateblaschkecurvature2} shows that
\begin{align}\label{rijlk}
R_{ijlk} + \bar{R}_{ijkl} = 2F_{ij}H_{kl},
\end{align}
so that the anti-self-conjugate part of $R_{ijkl}$ is $R_{ij(kl)} - F_{ij}H_{kl}$, which by \eqref{nbtskew} is equal to $\uf_{ijkl}$, and so it follows from \eqref{skewnablah} that if $(\en, [h])$ is Weyl then the curvature is self-conjugate. With \eqref{qrdiff}, \eqref{uijsym2}, and \eqref{nbtskew} this implies that for a Weyl structure there hold
\begin{align}
\label{rfh}& R_{ij(kl)}  = F_{ij}H_{kl},& & Q_{ij} = R_{ij} + 2F_{ij},& &E_{ijkl} = 0 = E_{ij}.
\end{align}
The class of AH structures with self-conjugate curvature tensor is considerably larger than the class of Weyl structures. In section \ref{convexprojectivesection} it will be shown that given a convex flat real projective structure $\en$ there is a conformal structure $[h]$ such that $(\en, [h])$ is an AH structure, and in general these AH structures are not Weyl, though their curvature is self-conjugate.

\subsubsection{Relation with underlying conformal structure}\label{underlyingsection}
This section describes how the various curvature tensors associated to an AH structure, e.g. $A_{ijkl}$ and $E_{ijkl}$, relate to the curvature tensors associated to the underlying conformal structure $[h]$, e.g. the usual conformal Weyl tensor.

Fix an AH structure $(\en, [h])$ and a representative $h \in [h]$ with associated Faraday primitive $\ga_{i}$ and Levi-Civita connection $D$, which is related to $\nabla$ by \eqref{dnabladiff}. Because of the convention of raising and lowering indices with the normalized metric $H^{ij}$, some notation is needed when indices are raised using instead $h^{ij}$. If $\ga_{i}$ is a one-form then $\ga^{\sharp\,i} \defeq h^{ip}\ga_{p}$, so, for example, $h_{ij}\ga^{\sharp\,k} = H_{ij}\ga^{k}$. Were it necessary to refer to $\tilde{h}^{ip}\ga_{p}$, this would be written $\ga^{\tilde{\sharp}\,i}$. Let $\sR_{ijk}\,^{l}$ be the curvature of $D$, and write the tensors derived from it in the same script. When there is chosen a representative $h \in [h]$ there will be written $\uR_{h}$ (or simply $\uR$ if the dependence on $h$ is clear from context) for the unweighted scalar curvature $h^{ij}R_{ij}$. It is important not to confuse $\uR_{h}$ with the scalar curvature of (the Levi-Civita connection of) $h$, which is written $\sR_{h}$. Let $\mu \defeq |\det h|^{-1/n}$, $\ga^{\sharp\,i}\defeq h^{ip}\ga_{p}$, $|\ga|^{2}_{h} = \ga^{\sharp\,p}\ga_{p}$, and $|\bt|_{h}^{2} = h^{ij}\bt_{ij} = h^{ij}\bt_{ip}\,^{q}\bt_{jq}\,^{p}$. Given a one-form $\al_{i}$, it will be convenient to write 
\begin{align}\label{alijshort}
&\al_{ij} \defeq \nabla_{i}\al_{j}- \al_{i}\al_{j} + \tfrac{1}{2}H_{ij}\al_{p}\al^{p}, &\al_{p}\,^{p} = \nabla_{p}\al^{p} + \tfrac{n-2}{2}\al_{p}\al^{p}.
\end{align}
Let $\ga_{ij}$ be defined by \eqref{alijshort} with $\ga_{i}$ in place of $\al_{i}$. 
There hold
\begin{align}
&\label{gd1}D_{i}\ga_{j} = \nabla_{i}\ga_{j} - \tfrac{1}{2}\bt_{ij}\,^{p}\ga_{p} - 2\ga_{i}\ga_{j} + H_{ij}\ga_{p}\ga^{p},& 
& D^{p}\ga_{p}  = \nabla^{p}\ga_{p} + (n-2)\ga^{p}\ga_{p},\\
\label{gd2}&\ga_{ij} = D_{i}\ga_{j} + \tfrac{1}{2}\bt_{ij}\,^{p}\ga_{p} + \ga_{i}\ga_{j} - \tfrac{1}{2}\ga^{p}\ga_{p}H_{ij},& &\ga_{p}\,^{p} = D^{p}\ga_{p} + \tfrac{2-n}{2}\ga^{p}\ga_{p}. 
\end{align}
There hold
\begin{align}
\label{dga3}
\begin{split}&D_{i}\bt_{jk}\,^{l} = \\ &\nabla_{i}\bt_{jk}\,^{l} - \tfrac{1}{2}\bt_{ij}\,^{p}\bt_{kp}\,^{l} - \bt_{k[i}\,^{p}\bt_{j]p}\,^{l} + \delta_{i}\,^{l}\ga_{p}\bt_{jk}\,^{p} + 2H_{i(j}\bt_{k)p}\,^{l}\ga^{p} - 3\ga_{(i}\bt_{jk)}\,^{l} - \bt_{ijk}\ga^{l} ,
\end{split}\\
& D_{[i}\bt_{j]k}\,^{l} = \nabla_{[i}\bt_{j]k}\,^{l} - \bt_{k[i}\,^{p}\bt_{j]p}\,^{l} + \ga_{p}\delta_{[i}\,^{l}\bt_{j]k}\,^{p} + H_{k[i}\bt_{j]p}\,^{l}\ga^{p},\\
\label{dbtskew}&D_{[i}\bt_{j]kl} = \nabla_{[i}\bt_{j]kl} + H_{k[i}\bt_{j]l}\,^{p}\ga_{p} + H_{l[i}\bt_{j]k}\,^{p}\ga_{p},
\end{align}
the last of which follows from \eqref{nablabtupdown}. From \eqref{uijsym2} and \eqref{dga3} there follows
\begin{align}\label{ddivbt}
2nE_{ij}  =  \nabla_{p}\bt_{ij}\,^{p} - \bt_{ij} =D_{p}\bt_{ij}\,^{p} - n\ga_{p}\bt_{ij}\,^{p}.
\end{align}
From \eqref{fvdecom}, \eqref{nbtskew}, and \eqref{dbtskew} there follows
\begin{align}\label{uplusv}
D_{[i}\bt_{j]kl} + \uf_{ijkl}  = H_{k[i}\bt_{j]l}\,^{p}\ga_{p} + H_{l[i}\bt_{j]k}\,^{p}\ga_{p}. 
\end{align}
Calculating $\sR_{ijkl} - R_{ijkl}$ using $F_{ij} = -d\ga_{ij}$ and \eqref{uplusv} yields
\begin{align}\label{confcurvijkl}\sR_{ijkl} &=   R_{ijkl} + D_{[i}\bt_{j]kl} - \tfrac{1}{2}\bt_{lp[i}\bt_{j]k}\,^{p} - 2H_{k[i}\bt_{j]l}\,^{p}\ga_{p}  - F_{ijkl} - G_{ijkl}\\ \notag &+\ga_{(jl)}H_{ik} - \ga_{(il)}H_{jk} - \ga_{(jk)}H_{il} + \ga_{(ik)}H_{jl}\\
\notag & = T_{ijkl} +\tfrac{1}{4}\bt_{ijkl}   + H_{l[i}\bt_{j]k}\,^{p}\ga_{p} - H_{k[i}\bt_{j]l}\,^{p}\ga_{p}\\ \notag &+ \ga_{(jl)}H_{ik} - \ga_{(il)}H_{jk} - \ga_{(jk)}H_{il} + \ga_{(ik)}H_{jl}.
\end{align}
From \eqref{confcurvijkl} there follow
\begin{align}
\label{confric}&\sR_{ij}  = T_{ij} + \tfrac{1}{4}\bt_{ij} + (2-n)\ga_{(ij)} - H_{ij}\ga_{p}\,^{p} + \tfrac{n-2}{2}\ga_{p}\bt_{ij}\,^{p}\\
\notag & \quad\,\,\, = T_{ij} + \tfrac{1}{4}\bt_{ij} + (2-n)\left(D_{(i}\ga_{j)} + \ga_{i}\ga_{j} - \ga^{p}\ga_{p}H_{ij}\right) - D^{p}\ga_{p}H_{ij}\\
\notag & \quad\,\,\, = R_{(ij)} + \tfrac{1}{2}\left(D_{p}\bt_{ij}\,^{p} - n\ga_{p}\bt_{ij}\,^{p}\right)+ \tfrac{1}{4}\bt_{ij} + (2-n)\left(D_{(i}\ga_{j)} + \ga_{i}\ga_{j} - |\ga|_{h}^{2}h_{ij}\right) +\dad_{h}\ga h_{ij},\\
\label{confscal}&\sR \defeq h^{ij}\sR_{ij}  = |\det h|^{-1/n}\left(R + \tfrac{1}{4}\nbt  + 2(1-n)\ga_{p}\,^{p}\right),\\
\notag & \quad\, = \uR_{h} + \tfrac{1}{4}|\bt|_{h}^{2}  + 2(n-1)\dad_{h}\ga + (n-1)(n-2)|\ga|_{h}^{2},\\
\label{confricfree}&\mr{\sR}_{ij}  = \mr{T}_{ij} + \tfrac{1}{4}\mr{\bt}_{ij}%
 + (2-n)\left(\ga_{(ij)} - \tfrac{1}{n}\ga_{p}\,^{p}H_{ij} -\tfrac{1}{2}\ga_{p}\bt_{ij}\,^{p} \right)\\
\notag & \qquad \,\, = \mr{T}_{ij} + \tfrac{1}{4}\mr{\bt}_{ij} + (2-n)\left(D_{(i}\ga_{j)} + \tfrac{1}{n}(\dad_{h}\ga)h_{ij} + \ga_{i}\ga_{j} - \tfrac{1}{n}|\ga|_{h}^{2}h_{ij}\right)\\
\notag & \qquad\,\,\, = \mr{R}_{ij} + \tfrac{1}{4}\mr{\bt}_{ij} + (2-n)\mr{\ga\tensor\ga}_{ij} + \tfrac{1}{2}\left(D_{p}\bt_{ij}\,^{p} - n\ga_{p}\bt_{ij}\,^{p}\right)  + (2-n)\mr{D\ga}_{ij}.
\end{align}
The reason for writing an equation such as \eqref{confric} in several ways is that one or other expression is useful depending on whether one is supposing given the AH structure and deducing properties of it, or one is given a metric and a cubic form and is trying to construct an AH structure with desired properties. For instance, later it will be shown that for a Riemannian signature Einstein AH structure on a compact manifold and a particular choice of $h$ there vanish the last two terms of \eqref{confricfree}. Using $\nabla_{i}R = D_{i}R + 2\ga_{i}R$ and \eqref{gd1} it is straightforward to check 
\begin{align}
\label{preconserve} 
\begin{split}
\nabla^{p}F_{ip} &= D^{p}F_{ip} + (4-n)\ga^{p}F_{ip},\\
 \nabla_{i}R + n\nabla^{p}F_{ip} &= D_{i}R + 2\ga_{i}R  + nD^{p}F_{ip} + n(4-n)\ga^{p}F_{ip}.
\end{split}\\
\label{conserve}
\begin{split}
\mu(\nabla_{i}R + n\nabla^{p}F_{ip})  & = D_{i}\uR_{h}  + 2\ga_{i}\uR_{h} + nh^{pq}D_{p}F_{iq} + n(4-n)\ga^{\sharp\,p}F_{ip}\\
& = D_{i}\uR_{h}  + 2\ga_{i}\uR_{h} - n\dad_{h}d\ga_{i} + n(4-n)\ga^{\sharp\,p}d\ga_{pi}.
\end{split}
\end{align}

\section{Codazzi projective structures}\label{codazziprojectivesection}

\subsection{Curvature of Codazzi projective structures}\label{curvaturecodazzisection}
This section begins with an analysis of how are related the curvatures of two CP pairs inducing the same conformal projective structure. 

\subsubsection{}
A conformal projective structure is a \textbf{Codazzi projective structure} if its conformal torsion vanishes. By Lemma \ref{conformalprojectivediff} a Codazzi projective structure can be viewed as a conformal class of AH structures.

\subsubsection{}\label{cweightsection}
Let $\cmf$ be the line bundle $|\Det \ctm|^{-1/2n}$ and say that a section of $\cmf^{\la}$ has \textbf{c-weight} $\la$. The terminology \textit{c-weight} abbreviates \textit{conformal weight}. If $\tD = D + 2\si_{(i}\delta_{j)}\,^{k} - h_{ij}\si^{\sharp\,k}$ are Levi-Civita connections of conformal metrics and $u \in \Ga(\cmf^{\la})$ then $\tD_{i}u - D_{i}u = \la \si_{i} u$. For this reason when analyzing the effects of conformal changes on density-valued tensors it is convenient to speak of c-weights.

\subsubsection{}
Temporarily write $\A_{k}(\la)$ for the vector space of completely trace-free, completely symmetric tensors $\om_{i_{1}\dots i_{k}}$ having c-weight $\la$, and $\B_{k}(\la)$ for the vector space of tensors $\si_{iji_{1}\dots i_{k-1}}$ having c-weight $\la$ and satisfying $\si_{iji_{1}\dots i_{k-1}} = \si_{[ij](i_{1}\dots i_{k-1})}$ and $\si_{[ijk]i_{1}\dots i_{k-2}} = 0$. Let $\B_{k}^{0}(\la)$ comprise the trace-free elements of $\B_{k}(\la)$. For example, $\bt_{ijk} \in \A_{3}(2)$ and $E_{ijkl} \in \B_{3}^{0}(2)$.

\begin{lemma}\label{eopinvariantlemma}
Given an AH structure $(\en, [h])$ the associated differential operator $\klie:\A_{k}(k-1) \to \B_{k}(k-1)$ defined by
\begin{align}\label{colanczosdefined}
\klie(\om)_{iji_{1}\dots i_{k-1}} \defeq - H_{ia}H_{jb}\left(\nabla^{[a}\om^{b]}\,_{i_{1}\dots i_{k-1}} + \tfrac{k-1}{n+k-3}\nabla^{p}\om_{p(i_{1}\dots i_{k-2}}\,^{[a}\delta_{i_{k-1})}\,^{b]} \right),
\end{align}
is invariant in the sense that if $(\ten, [h])$ is another AH structure generating the same Codazzi projective structure as does $(\en, [h])$ then the operator $\tilde{\klie}$ associated to $(\ten, [h])$ is equal to $\klie$. 
\end{lemma}
\begin{proof}
By Lemma \ref{conformalprojectivediff} the difference tensor $\tnabla - \nabla$ has the form $2\al_{(i}\delta_{j)}\,^{k} - \al^{k}H_{ij}$ for some one-form $\al_{i}$. It is straightforward to verify that for trace-free $\om_{i_{1}\dots i_{k}} \in \A_{k}(\la)$ there hold
\begin{align}\label{eomdiff}
\begin{split}
\tnabla_{i}\om_{i_{1}\dots i_{k}}          & = \nabla_{i}\om_{i_{1}\dots i_{k}}  + (\la - k)\al_{i}\om_{i_{1}\dots i_{k}} - k\al_{(i_{1}}\om_{i_{2}\dots i_{k})i} + kH_{i(i_{1}}\om_{i_{2}\dots i_{k})p}\al^{p},\\
\tnabla^{[i}\om^{j]}\,_{i_{1}\dots i_{k-1}} & = \nabla^{[i}\om^{j]}\,_{i_{1}\dots i_{k-1}} - (\la - k +1)\om_{i_{1}\dots i_{k-1}}\,^{[i}\al^{j]} - (k-1)\al^{p}\om_{p(i_{1}\dots i_{k-2}}\,^{[i}\delta_{i_{k-1})}\,^{j]},\\
\tnabla^{p}\om_{i_{1}\dots i_{k-1}p}       & = \nabla^{p}\om_{i_{1}\dots i_{k-1}p}  + (n-2+\la)\al^{p}\om_{i_{1}\dots i_{k-1}p}.
\end{split}
\end{align}
from which the claim follows. 
\end{proof}
Taking the trace-free part of $\klie(\om)$ determines an invariant operator $\klie^{0}:\A_{k}(k-1) \to B^{0}_{k}(k-1)$. The difference $\klie^{0}(\om) - \klie(\om)$ comprises terms involving the cubic torsion. The complicated explicit expression is not needed here, so omitted. %
The reason for introducing $\klie$ is the following lemma.
\begin{lemma}\label{ascweyllemma}
For a Codazzi projective structure $(\enb, [h])$ with cubic torsion $\bt_{ijk}$ there holds $\klie^{0}(\bt)_{ijkl} = \klie(\bt)_{ijkl} = E_{ijkl}$.
\end{lemma}
\begin{proof}
For $\om_{ijk} \in \A_{3}(2)$ there holds $\klie(\om)_{ijkl} = - H_{ia}H_{jb}\left(\nabla^{[a}\om^{b]}\,_{kl} + \tfrac{2}{n}\nabla^{p}\omega_{p(k}\,^{[a}\delta_{l)}\,^{b]}\right)$,
and comparing this with \eqref{eijup} shows that $\klie(\bt)_{ijkl} = E_{ijkl}$, which is trace-free.
\end{proof}

\subsubsection{}
For an AH structure $(\en, [h])$, define $W_{ijkl} \defeq A_{ijkl} + E_{ijkl}$ and 
\begin{align}
\label{wijdefined} W_{ij} \defeq & A_{ij} + E_{ij} + \tfrac{1}{2}F_{ij}
 = \tfrac{1}{2-n}\left(\tfrac{n-1}{n}R_{(ij)} + \tfrac{1}{n}Q_{(ij)} + \tfrac{R}{2(1-n)}H_{ij} + \tfrac{2-n}{2}F_{ij} \right).
\end{align}
From the definition of $W_{ij}$, \eqref{brteij}, and \eqref{brtaij} there follow:
\begin{align}\label{rintermsofw}
R_{ijkl} & = W_{ijkl} - 2H_{l[i}W_{j]k} + 2H_{k[i}\bar{W}_{j]l} + (W_{[ij]} + \bar{W}_{[ij]})H_{kl}.\\
\label{ricintermsofw}
R_{ij} &= (1-n)W_{ij} + \bar{W}_{ij} + \tfrac{1}{2(n-1)}RH_{ij} - F_{ij}.
\end{align}
By Theorem \ref{aeinvariant}, $W_{ijkl}$ can be regarded as the generalization to Codazzi projective structures of the usual conformal Weyl tensor of a conformal structure, and this justifies calling $A_{ijkl}$ and $E_{ijkl}$ the \textbf{self-conjugate Weyl tensor} and the \textbf{anti-self-conjugate Weyl tensor}. 
\begin{theorem}\label{aeinvariant}
 The self-conjugate and anti-self-conjugate Weyl tensors $A_{ijkl}$ and $E_{ijkl}$ of an AH structure $(\en, [h])$ depend only on the Codazzi projective structure $(\enb, [h])$ generated by $(\en, [h])$ in the sense that if $(\ten, [h])$ is another AH structure generating the same Codazzi projective structure the corresponding tensors $\tilde{A}_{ijkl}$ and $\tilde{E}_{ijkl}$) equal $A_{ijkl}$ and $E_{ijkl}$, respectively.
\end{theorem}

\begin{proof}
The decomposition via the projections $\P_{i}$ on $\curvmod$ is useful for organizing what would otherwise be a tremendous mess. The sorts of computations to be made will be familiar to anyone who has computed explicitly conformally invariant tensors. The invariance of $E_{ijkl}$ for Codazzi projective structures proved in Theorem \ref{aeinvariant} is immediate from Lemma \ref{ascweyllemma}; a direct computational proof will also be given in what follows.

Let $\tnabla$ and $\nabla$ be the aligned representatives of CP pairs generating the same conformal projective structure. By Lemma \ref{conformalprojectivediff} the difference tensor $\tnabla - \nabla$ has the form $2\al_{(i}\delta_{j)}\,^{k} - \al^{k}H_{ij}$ for some one-form $\al_{i}$. Recall the definition \eqref{alijshort} of $\al_{ij}$. Via conditions $(4)$ and $(5)$ of Lemma \ref{special} the alignment condition implies
\begin{align}
\label{aliupj}&\al_{i}\,^{j} = H^{jk}\al_{ik} = \nabla_{i}\al^{j} - \al_{i}\al^{j} + \tfrac{1}{2}\al^{p}\al_{p}\delta_{i}\,^{j} + \al^{p}\bt_{pi}\,^{j}.&
\end{align}
Decorate with a $\tilde{\,\,}$ the curvature of $\tnabla$ and all tensors derived from it. The Faraday primitives $\tilde{\ga}_{i}$ and $\ga_{i}$ associated to $h \in [h]$ are related by $\tilde{\ga}_{i} = \ga_{i} - \al_{i}$ and so $\tilde{F}_{ij} - F_{ij} = 2\nabla_{[i}\al_{j]} = d\al_{ij} = 2\al_{[ij]}$. It follows immediately that 
\begin{align}
\tilde{F}_{ijkl} + \tilde{G}_{ijkl} - F_{ijkl} - G_{ijkl} = d\al_{ij}H_{kl} - H_{l[i}d\al_{j]k} + H_{k[i}d\al_{j]l}.
\end{align}
Straightforward computations show
\begin{align}\label{rijkldiff}
\begin{split}
\tilde{R}_{ijkl} & - R_{ijkl} =  d\al_{ij}H_{kl} - H_{l[i}d\al_{j]k} + H_{k[i}d\al_{j]l} - 2H_{l[i}\mu_{j]k} + 2H_{k[i}\mu_{j]l} - 2H_{k[i}\bt_{j]lp}\al^{p},
\end{split}
\end{align}
so that, writing $\mu_{ij} = \al_{(ij)}$,
\begin{align}\label{rijkldiff2}
\begin{split}
(\tilde{\uf}_{ijkl} &+ \tilde{T}_{ijkl}) - (\uf_{ijkl} - T_{ijkl})
=   - 2H_{l[i}\mu_{j]k} + 2H_{k[i}\mu_{j]l} - 2H_{k[i}\bt_{j]lp}\al^{p}.
\end{split}
\end{align}
The transformation rules for $\uf_{ijkl}$ and $T_{ijkl}$ can now be found by applying to the righthand side of \eqref{rijkldiff2} the projections $\P_{1}$ and $\P_{2}$. In order to do this, it is convenient to observe that the image under $\P_{2}$ of $2H_{k[i}\si_{j]l} - 2H_{l[i}\si_{j]k}$ is itself or $0$ according to whether $\si_{ij}$ is symmetric or anti-symmetric. Applying $\P_{1}$ and $\P_{2}$ to \eqref{rijkldiff2} and simplifying there result 
\begin{align}
\label{uijkldiff}
\tilde{\uf}_{ijkl} & - \uf_{ijkl} = - H_{k[i}\bt_{j]lp}\al^{p} - H_{l[i}\bt_{j]kp}\al^{p},\\
\label{tijkldiff}
\tilde{T}_{ijkl} & - T_{ijkl} = - 2H_{l[i}\mu_{j]k} + 2H_{k[i}\mu_{j]l} + H_{l[i}\bt_{j]kp}\al^{p} - H_{k[i}\bt_{j]lp}\al^{p}.
\end{align}
Taking traces of \eqref{rijkldiff}, \eqref{uijkldiff}, and \eqref{tijkldiff} in various ways gives
\begin{align}
&\label{rijsymdiff}\tilde{R}_{(ij)} - R_{(ij)}  = (2-n)\mu_{ij} - H_{ij}\al_{p}\,^{p}  - \al^{p}\bt_{ijp},\\
&\tilde{Q}_{(ij)} - Q_{(ij)}  = (2-n)\mu_{ij} - H_{ij}\al_{p}\,^{p}  + (n-1)\al^{p}\bt_{ijp},\\
&\tilde{T}_{ij} - T_{ij}  = (2-n)\mu_{ij} - H_{ij}\al_{p}\,^{p}  + \tfrac{n-2}{2}\al^{p}\bt_{ijp}, &&\label{cpdiffscalar}\tilde{R} - R  = 2(1-n)\al_{p}\,^{p},\\
&\label{aijdiff}\tilde{A}_{ij} - A_{ij}  = \mu_{ij} - \tfrac{1}{2}\bt_{ijp}\al^{p},&
&\tilde{E}_{ij} - E_{ij} = \tfrac{1}{2}\bt_{ijp}\al^{p}.
\end{align}
From \eqref{aijdiff} and \eqref{tijkldiff} there follows $\tilde{A}_{ijkl} = A_{ijkl}$, while from \eqref{aijdiff} and \eqref{uijkldiff} there follows $\tilde{E}_{ijkl} = E_{ijkl}$.
\end{proof}

\subsubsection{}

Next is defined the analogue of the usual conformal Cotton tensor. Define
\begin{align}
\label{aijkdefined} 
\begin{split}
w_{ijk} & = 2\nabla_{[i}W_{j]k},\\
a_{ijk} & \defeq 2\nabla_{[i}A_{j]k} + \nabla_{[i}F_{j]k}- \bt_{k[i}\,^{p}A_{j]p}- \tfrac{1}{2}\bt_{k[i}\,^{p}F_{j]p} + \bt_{k[i}\,^{p}E_{j]p},\\
e_{ijk}& \defeq 2\nabla_{[i}E_{j]k} + \bt_{k[i}\,^{p}A_{j]p}     + \tfrac{1}{2}\bt_{k[i}\,^{p}F_{j]p}     - \bt_{k[i}\,^{p}E_{j]p}.
\end{split}
\end{align}
Let $W_{i} = w_{ip}\,^{p}$, $A_{i} = a_{ip}\,^{p}$, and $E_{i} = e_{ip}\,^{p}$ and define trace-free tensors
\begin{align}
\label{wijkdefined2} &W_{ijk}  \defeq w_{ijk} - \tfrac{2}{n-1}W_{[i}H_{j]k},& & A_{ijk}  \defeq a_{ijk} - \tfrac{2}{n-1}A_{[i}H_{j]k}, & & E_{ijk}  \defeq e_{ijk} - \tfrac{2}{n-1}E_{[i}H_{j]k}.
\end{align}
Because $\nabla_{[i}F_{jk]} = 0$, there hold $W_{[ijk]} = A_{[ijk]} = E_{[ijk]} = 0$. There holds $w_{ijk} = a_{ijk} + e_{ijk}$, so also $W_{i} = A_{i} + E_{i}$ and $W_{ijk} = A_{ijk} + E_{ijk}$. Write also $W = W_{p}\,^{p} = \tfrac{1}{2(1-n)}R$. For a Weyl structure, $A_{ijk} = a_{ijk} = W_{ijk}$ is the usual conformal Cotton tensor.

That $A_{ij}$ and $A_{ijkl}$ (resp. $E_{ij}$ and $E_{ijkl}$) are the self-conjugate (resp. anti-self-conjugate) parts of $W_{(ij)}$ and $W_{ijkl}$ suggests Lemma \ref{barabarelemma}.

\begin{lemma}\label{barabarelemma}
Let $(\en, [h])$ and $(\ben, [h])$ be conjugate AH structures. Then $a_{ijk}$, $A_{i}$, and $A_{ijk}$ (resp. $e_{ijk}$, $E_{i}$, and $E_{ijk}$) are the self-conjugate (resp. anti-self-conjugate) parts of $w_{ijk}$, $W_{i}$, and $W_{ijk}$, respectively. 
\end{lemma}
\begin{proof}
From \eqref{brtaij} and \eqref{brteij} there follows $\bar{W}_{ij} = W_{ij} -2E_{ij}$. Skewing $\bnabla_{i}\bar{W}_{jk} = \nabla_{i}\bar{W}_{jk} - \bt_{ij}\,^{p}\bar{W}_{pk} - \bt_{ik}\,^{p}\bar{W}_{jp}$ yields $w_{ijk} + \bar{w}_{ijk} = 2a_{ijk}$ and $w_{ijk} - \bar{w}_{ijk} = 2e_{ijk}$. The remaining claims are immediate upon taking traces.
\end{proof}

\begin{remark}\label{conjugateremark}
Observe that the aligned representative at time $t$ of the pencil of AH structure generated by $(\ten, [h])$ is obtained from the aligned representative of the AH structure $(\pen, [h])$ by adding $2\al_{(i}\delta_{j)}\,^{k}- H_{ij}\al^{k}$. Hence tensors decorated by $\brt{\,}$ formed from $(\pen, [h])$ transform under conformal projective equivalence just as do the corresponding undecorated tensors formed from $(\en, [h])$. In particular these remarks apply when $t = 1$, in which case the decoration is by bars; that is the aligned representative of the AH structure conjugate to $(\ten, [h])$ is obtained from the aligned representative of the conjugate AH structure $(\ben, [h])$ by adding $2\al_{(i}\delta_{j)}\,^{k}- H_{ij}\al^{k}$, and barred tensors formed from $(\ben, [h])$ transform under conformal projective equivalence just as do the corresponding unbarred tensors formed from $(\en, [h])$. 

However, in making use of this observation it is important to keep in mind that $\al_{ij}$ depends on $\nabla$, and that $\bar{\al}_{ij}$, defined as is $\al_{ij}$, but in terms of $\bnabla$ is not equal to $\al_{ij}$; rather,
\begin{align}\label{alijconjugate}
\bar{\al}_{ij} = \al_{ij} - \bt_{ij}\,^{p}\al_{p}.
\end{align}
In particular, the self-conjugate part of $\al_{ij}$ is $\al_{ij} - \tfrac{1}{2}\bt_{ij}\,^{p}\al_{p}$.
\end{remark}

\begin{theorem}\label{codazzitransformtheorem}
Let $(\en, [h])$ and $(\ten, [h]$ be AH structures on a manifold of dimension $n > 2$ generating the same Codazzi projective structure.
There hold
\begin{align}
\label{codazzitransform} &\tilde{W}_{ijk}  = W_{ijk} - \al^{p}W_{ijkp},&  &\tilde{A}_{ijk} = A_{ijk} -\al^{p}A_{ijkp}, &%
&\tilde{E}_{ijk} = E_{ijk} - \al^{p}E_{ijkp},
\end{align}
and $\tilde{A}_{i} = A_{i}$ depends only on the underlying Codazzi projective structure. As $A_{ijkl}$ vanishes identically when $n=3$, in this case $A_{ijk}$ depends only on the underlying Codazzi projective structure. 
\end{theorem}
\begin{proof}
Let $\tnabla$ and $\nabla$ be aligned connections representing AH structures generating the same Codazzi projective structure. Preparatory to the proof observe
\begin{align*}
\begin{split}
&\nabla_{[i}\al_{j]k} = -\tfrac{1}{2}R_{ijk}\,^{p}\al_{p} + \al_{[i}\nabla_{j]}\al_{k} - \al_{k}\al_{[ij]} - \al^{p}H_{k[i}\nabla_{j]}\al_{p} + \tfrac{1}{2}H_{k[i}\bt_{j]pq}\al^{p}\al^{q},\\
&\al_{[i}\al_{j]k} = \al_{[i}\nabla_{j]}\al_{k} - \tfrac{1}{2}H_{k[i}\al_{j]}\al_{p}\al^{p},\qquad 
\al^{p}H_{k[i}\al_{j]p} = \al^{p}H_{k[i}\nabla_{j]}\al_{p} -\tfrac{1}{2}H_{k[i}\al_{j]}\al_{p}\al^{p},\\
&\nabla_{[i}\al_{j]k} - \al_{[i}\al_{j]k} + \tfrac{1}{2}\al_{k}d\al_{ij}+ \al^{p}H_{k[i}\al_{j]p} = -\tfrac{1}{2}R_{ijkp}\al^{p} + \tfrac{1}{2}H_{k[i}\bt_{j]pq}\al^{p}\al^{q}.
\end{split}
\end{align*}
From \eqref{aijdiff} there follows
\begin{align}\label{wijdiff}
&\tilde{W}_{ij} - W_{ij} = \al_{ij}.
\end{align} 
It follows straightforwardly that
\begin{align}
\notag \begin{split}
& \tilde{w}_{ijk} - w_{ijk} \\
&= -\al^{p}\left(2H_{p[i}W_{j]k} - 2H_{k[i}W_{j]p} + F_{ij}H_{kp}\right)
+ 2\left(\nabla_{[i}\al_{j]k} - \al_{[i}\al_{j]k} + \al_{k}\al_{[ij]}+ \al^{p}H_{k[i}\al_{j]p}\right)\\
 &= -\al^{l}\left(W_{ijkl} - 4H_{k[i}E_{j]l}\right) + H_{k[i}\bt_{j]pq}\al^{p}\al^{q},
\end{split}\\
\label{witransform}&\tilde{W}_{i} - W_{i} = (1-n)\left(2E_{ip}\al^{p} + \tfrac{1}{2}\bt_{ipq}\al^{p}\al^{q}\right),
\end{align}
from which follows the first equality of \eqref{codazzitransform}. Together Remark \ref{conjugateremark} and Lemma \ref{barabarelemma} show that the first equality of \eqref{codazzitransform} implies the other two equalities in \eqref{codazzitransform}. 
From \eqref{witransform} and Lemma \ref{barabarelemma} there follows
\begin{align*}
\tilde{A}_{i} - A_{i} = \tfrac{1}{2}(\tilde{W}_{i} + \tilde{\bar{W}}_{i} - W_{i} - \bar{W}_{i}) = (1-n)\al^{p}(E_{ip} + \bar{E}_{ip}) +\tfrac{1}{4}\al^{p}\al^{q}(\bt_{ipq} + \bar{\bt}_{ipq}) = 0, 
\end{align*}
showing the invariance of $A_{i}$.
\end{proof}

The invariance of $A_{i}$, which vanishes for a Weyl structure, came as a surprise. 

\begin{definition}
An AH structure (or a Codazzi projective structure) on a manifold of dimension at least $3$ is \textbf{conservative} if $A_{i} = 0$. 
\end{definition}
The terminology is motivated by thinking of the constancy of the scalar curvature of an Einstein metric (in the usual sense) as a conservation law in conjunction with \eqref{ebtw} below. It will become evident that any reasonable definition of Einstein AH structures must include the vanishing of $A_{i}$.

\subsection{Differential Bianchi identities}
Lemma \ref{diffbianchilemma} records the consequences of the differential Bianchi identity.
\begin{lemma}\label{diffbianchilemma}
For an AH structure $(\en, [h])$ of dimension $n > 2$ there hold
\begin{align}
\label{wf6} \nabla_{[p}A_{ij]kl} %
 & = H_{k[p}a_{ij]l} - H_{l[p}a_{ij]k}  + \tfrac{1}{2}\bt_{l[p}\,^{q}W_{ij]kq} -\tfrac{1}{2}\bt_{k[p}\,^{q}W_{ij]lq },\\
\label{wf7} \nabla_{[p}E_{ij]kl} %
& = -H_{k[p}e_{ij]l} - H_{l[p}e_{ij]k} + \tfrac{1}{2}\bt_{l[p}\,^{q}W_{ij]kq} + \tfrac{1}{2}\bt_{k[p}\,^{q}W_{ij]lq } ,\\
\label{adiffbianchi}\nabla^{p}A_{ijkp} %
& = (3-n)a_{ijk} -2A_{[i}H_{j]k} + \bt_{[i}\,^{pq}W_{j]pqk} - \tfrac{1}{2}\bt_{k}\,^{pq}E_{ijpq}\\
\notag & = (3-n)A_{ijk} + \tfrac{4(2-n)}{n-1}A_{[i}H_{j]k} + \bt_{[i}\,^{pq}W_{j]pqk}  - \tfrac{1}{2}\bt_{k}\,^{pq}E_{ijpq},\\
\label{ediffbianchi}\nabla^{p}E_{ijkp} %
& = (1-n)E_{ijk}  - \bt_{[i}\,^{pq}W_{j]pqk} + \tfrac{1}{2}\bt_{k}\,^{pq}E_{ijpq},\\
\label{wdiffbianchi}\begin{split}\nabla^{p}W_{ijkp} & = (3-n)A_{ijk} + (1-n)E_{ijk} + \tfrac{4(2-n)}{n-1}A_{[i}H_{j]k},\\%(3-n)a_{ijk} -2A_{[i}H_{j]k},\\
                                                    & = (3-n)W_{ijk} - 2E_{ijk} + \tfrac{4(2-n)}{n-1}A_{[i}H_{j]k},\end{split}\\
\label{divAtrace} &H^{jk}\nabla^{p}A_{ijkp} = 0 ,\\
\label{ebtw} & \bt^{abc}E_{iabc} = H^{jk}\nabla^{p}E_{ijkp} = 2(2-n)A_{i}\\
\notag &= (n-2)\left(2\nabla^{p}\mr{A}_{ip} - \bt_{i}\,^{pq}\mr{W}_{pq}+ \nabla^{p}F_{ip} +\tfrac{1}{n}\nabla_{i}R \right).
\end{align}
\end{lemma}

\begin{proof}
Let $(\ben, [h])$ be the conjugate AH structure. There hold
\begin{align}
\label{bach3}\nabla_{p}\bar{W}_{ijkl} &= \bnabla_{p}\bar{W}_{ijkl} - 2\bt_{p[i}\,^{q}\bar{W}_{j]qkl} +\bt_{pk}\,^{q}\bar{W}_{ijql} + \bt_{pl}\,^{q}\bar{W}_{ijkq},\\
\label{bach3b}  2\nabla_{[i}\bar{W}_{j]k} &= \bar{w}_{ijk} + 2\bt_{k[i}\,^{p}\bar{W}_{j]p}.
\end{align}
Differentiating \eqref{rintermsofw} gives
\begin{align}\label{fullwbi}
\begin{split}
\nabla_{p}R_{ijkl} & = \nabla_{p}W_{ijkl} -2H_{l[i}\nabla_{|p|}W_{j]k} + 2H_{k[i}\nabla_{|p|}\bar{W}_{j]k} \\  &- 2\bt_{pl[i}W_{j]k} + 2\bt_{pk[i}\bar{W}_{j]l} + H_{kl}\nabla_{p}F_{ij} + F_{ij}\bt_{pkl}.
\end{split}
\end{align}
Skew-symmetrizing \eqref{fullwbi} in $pij$, using the consequence $\nabla_{[p}R_{ij]kl} = \bt_{ql[p}R_{ij]k}\,^{q}$ of the differential Bianchi identity $\nabla_{[p}R_{ij]k}\,^{l} = 0$, and substituting into the result \eqref{bach3b} yields
\begin{align}\label{wf1}\begin{split}
\bt_{l[p}\,^{q}R_{ij]kq} &= \nabla_{[p}W_{ij]kl} + 2H_{l[p}\nabla_{i}W_{j]k}  - 2H_{k[p}\nabla_{i}\bar{W}_{j]l}  + F_{[ij}\bt_{p]kl}\\
& = \nabla_{[p}W_{ij]kl} + H_{l[p}w_{ij]k}  - H_{k[p}\bar{w}_{ij]l} - 2H_{k[p}\bt_{|l|i}\,^{q}\bar{W}_{j]q}  + F_{[ij}\bt_{p]kl}.
\end{split}
\end{align}
On the other hand, from \eqref{rintermsofw} there follows
\begin{align}\label{wf2}
\bt_{l[p}\,^{q}R_{ij]kq} =\bt_{l[p}\,^{q}W_{ij]kq} - 2H_{k[p}\bt_{|l|i}\,^{q}\bar{W}_{j]q} + F_{[ij}\bt_{p]kl}.
\end{align}
Together \eqref{wf1} and \eqref{wf2} give 
\begin{align}
\label{wf3} \nabla_{[p}W_{ij]kl} & = H_{k[p}\bar{w}_{ij]l} - H_{l[p}w_{ij]k} + \bt_{l[p}\,^{q}W_{ij]kq}.
\end{align}
The identity conjugate to \eqref{wf3} is
\begin{align}\label{wf4}
 \bnabla_{[p}\bar{W}_{ij]kl} & = H_{k[p}w_{ij]l} - H_{l[p}\bar{w}_{ij]k} - \bt_{l[p}\,^{q}\bar{W}_{ij]kq}.
\end{align}
Together \eqref{wf4} and the skew-symmetrization in $pij$ of \eqref{bach3} give
\begin{align}\label{wf5}
\begin{split}
\nabla_{[p}\bar{W}_{ij]kl} & = H_{k[p}w_{ij]l} - H_{l[p}\bar{w}_{ij]k} + \bt_{k[p}\,^{q}\bar{W}_{ij]ql}\\
 & = H_{k[p}w_{ij]l} - H_{l[p}\bar{w}_{ij]k} - \bt_{k[p}\,^{q}W_{ij]lq}.
\end{split}
\end{align}
Taking half the sum and the difference of \eqref{wf3} and \eqref{wf5} yields \eqref{wf6} and \eqref{wf7}. Tracing \eqref{wf6} and \eqref{wf7} in $pl$ yields \eqref{adiffbianchi} and \eqref{ediffbianchi}, and summing these gives \eqref{wdiffbianchi}. In taking traces one has to be careful; for example $3H^{pl}\nabla_{[p}A_{ij]kl} = \nabla^{p}A_{ijkp} - 2\bt_{[i}\,^{pq}A_{j]pqk}$. Equation \eqref{divAtrace} follows from $H^{jk}\nabla^{p}A_{ijkp} = - A_{ijk}\,^{p}\nabla_{p}H^{jk} = \bt^{abc}A_{i(abc)} = 0$. A similar computation gives the first equality of \eqref{ebtw} while tracing \eqref{adiffbianchi} and using \eqref{divAtrace} gives the second equality of \eqref{ebtw}. Tracing \eqref{aijkdefined} gives
\begin{align}
 -2A_{i} & = \tfrac{1}{n-1}\nabla_{i}R + 2\nabla^{p}A_{ip} + \nabla^{p}F_{ip} - \bt_{i}\,^{pq}(A_{pq} + E_{pq})\\
\notag & =\tfrac{1}{n}\nabla_{i}R + 2\nabla^{p}\mr{A}_{ip} + \nabla^{p}F_{ip} - \bt_{i}\,^{pq}\mr{W}_{pq}.
\end{align}
which gives the last equality in \eqref{ebtw}.
\end{proof}

Note that \eqref{ebtw} gives an alternative proof of the invariance of $A_{i}$. In examples it is usually easier to check the vanishing of $\bt^{abc}E_{iabc}$ than it is to check directly the vanishing of $A_{i}$.

\begin{corollary}
If $(\en, [h])$ is an AH structure with self-conjugate curvature on a manifold of dimension $n > 2$ then
\begin{align}
\begin{split}
&\nabla_{[p}A_{ij]kl} = H_{k[p}A_{ij]l} - H_{l[p}A_{ij]k} + \tfrac{1}{2}\bt_{l[p}\,^{q}A_{ij]kq} -  \tfrac{1}{2}\bt_{k[p}\,^{q}A_{ij]lq},\\
&\nabla^{p}A_{ijkp} = (3-n)A_{ijk} + \bt_{[i}\,^{pq}A_{j]pqk} = (3-n)A_{ijk} + (1-n)E_{ijk},\\
&(1-n)E_{ijk} = \bt_{[i}\,^{pq}A_{j]pqk}.
\end{split}
\end{align}
\end{corollary}
\begin{proof}
Immediate from Lemma \ref{diffbianchilemma}.
\end{proof}

\subsubsection{A non-conservative AH structure}\label{aiexample}
Here is given an example of an AH structure which is not conservative. On $\rea^{n}$ with $n > 2$ let $\delta_{ij}$ be the standard flat Euclidean metric, let $D$ be its Levi-Civita connection, and raise and lower indices using $\delta_{ij}$ and the dual bivector $\delta^{ij}$. Let $\al_{i}$ and $\be_{i}$ be non-zero covectors (constant one-forms) and which satisfy $\al^{i}\be^{j}\delta_{ij} = 0$. Define a trace-free completely symmetric tensor by $a_{ijk} \defeq \al_{i}\al_{j}\al_{k} - \tfrac{3}{n+2}|\al|^{2}\al_{(i}\delta_{jk)}$, and define $a_{ij} \defeq a_{i}\,^{pq}a_{jpq}$, and observe
\begin{align}
&|a|^{2} = \tfrac{n-1}{n+2}|\al|^{6},&  &a_{ij} =\tfrac{(n-2)(n+1)}{(n+2)^{2}}|\al|^{4}\al_{i}\al_{j} + \tfrac{2}{(n+2)^{2}}|\al|^{6}\delta_{ij},&
&\be^{p}a_{ijp} = -\tfrac{2}{n+2}|\al|^{2}\al_{(i}\be_{j)}.
\end{align}
The maps $e_{ijkl} \to \om_{ijkl} \defeq e_{i(jkl)}$ and $\om_{ijkl} \to e_{ijkl} \defeq \tfrac{3}{2}\om_{[ij]kl}$ are inverse isomorphisms between the space $\A$ of trace-free tensors $e_{ijkl}$ satisfying $e_{ijkl} = e_{[ij]kl}$, $e_{[ijk]l} = 0$, $e_{ij[kl]} = 0$, and the space $\B$ of trace-free tensors $\om_{ijkl}$ satisyfing $\om_{ijkl} = \om_{i(jkl)}$ and $\om_{(ijkl)} = 0$. Define $\si_{ijkl} \defeq \be_{i}a_{jkl} - \be_{(i}a_{jkl)}$, $\si_{ij} \defeq \tfrac{1}{2}\be^{p}a_{ijp}$ and note that $\si_{ij} = \si_{pij}\,^{p} = -\si_{ijp}\,^{p}$. Let 
\begin{align*}
\begin{split}
\om_{ijkl} &\defeq \si_{ijkl} + \tfrac{6}{n}\left(\si_{i(j}\delta_{kl)} - \si_{(ij}\delta_{kl)} \right) \\
& = \si_{ijkl} + \tfrac{1}{n}\left(\si_{ij}\delta_{kl}+ \si_{ik}\delta_{jl} + \si_{il}\delta_{jk} - \delta_{ij}\si_{kl}-  \delta_{ik}\si_{jl} - \delta_{il}\si_{jk} \right),
\end{split}
\end{align*}
and observe that $\om_{ijkl}$ is in $\B$. Let $e_{ijkl} \defeq \tfrac{3}{2}\om_{[ij]kl} = \tfrac{3}{2}\si_{[ij]kl}  + \tfrac{3}{n}\left( \si_{k[i}\delta_{j]l} + \si_{l[i}\delta_{j]k}\right)$. Let $x^{i}$ be coordinates on $\rea^{n}$ such that $dx^{i}$ is a parallel frame and define $L_{ijk} = a_{ijk} - \tfrac{3}{2}x^{p}\om_{pijk}$, so that $D_{i}L_{jkl} = -\tfrac{3}{2}\om_{ijkl}$. Let $\nabla = D - \tfrac{1}{2}L_{ij}\,^{k}$, so that $\nabla$ is the aligned representative of the exact AH structure $(\en, [\delta])$ which it generates with the conformal class $[\delta]$ of $\delta$. By \eqref{ddivbt} there holds $2nE_{ij} = D_{p}L_{ij}\,^{p} = -\tfrac{3}{2}\om_{pij}\,^{p} = 0$. With \eqref{uplusv} this shows $E_{ijk}\,^{l} = -D_{[i}L_{j]k}\,^{p} =  \tfrac{3}{2}\om_{[ij]k}\,^{l}= e_{ijk}\,^{l}$. Hence $\bt^{abc}E_{iabc} = L^{abc}e_{iabc}$, and at the origin this equals $a^{abc}e_{iabc}$, and a bit of computation shows that $a^{jkl}e_{ijkl} = \tfrac{3(n-2)(n+1)}{4n(n+2)}|\al|^{6}\be_{i}$. By \eqref{ebtw} this shows $A_{i} \neq 0$ in a neighborhood of the origin.

\subsubsection{}
All traces of $\nabla_{p}A_{ijk}\,^{l}$ and $\nabla_{p}E_{ijk}\,^{l}$ are expressible as linear combinations of $\nabla_{p}A_{ijk}\,^{p}$, $H^{ab}\nabla_{i}A_{jabk}$, $\nabla_{p}E_{ijk}\,^{p}$, $\nabla^{p}E_{p(ijk)}$, and $H^{ab}\nabla_{i}E_{jabk}$. Evidently $H^{ab}\nabla_{i}A_{jabk} = \bt_{i}\,^{ab}A_{jabk}$ and $H^{ab}\nabla_{i}E_{jabk} = \bt_{i}\,^{ab}E_{jabk}$ are invariants of the underlying Codazzi projective structure. In general, there are many invariant tensors which can be built in this way, particularly when $E_{ijk}\,^{l}$ or $A_{i}$ is not zero. While these tensors are not very interesting because their geometric meaning is completely obscure, it seems worth mentioning that they are particularly abundant in dimension $4$. 
\begin{lemma}
For an AH structure on a manifold of dimension $4$ the tensors $\nabla^{p}E_{pijk} + 2E_{i(jk)}$, $\nabla^{p}E_{p(ijk)}$, and $\nabla^{p}A_{p}$ are invariants of the underlying Codazzi projective structure.
\end{lemma}
\begin{proof}
In any dimension there hold
\begin{align}\label{nablaeijkltransform}
\begin{split}
&\tnabla_{p}A_{ijk}\,^{p}  = \nabla_{p}A_{ijk}\,^{p} + (n-3)\al_{p}A_{ijk}\,^{p}, \quad \tnabla_{p}E_{ijk}\,^{p}  = \nabla_{p}E_{ijk}\,^{p} + (n-1)\al_{p}E_{ijk}\,^{p},\\
&\tnabla^{p}E_{pijk}  = \nabla^{p}E_{pijk} + (n-4)\al^{p}E_{pijk} - 2\al^{p}E_{i(jk)p}.
\end{split}
\end{align}
From \eqref{codazzitransform} and \eqref{nablaeijkltransform} there follows
\begin{align}
\label{nablaonee}&\tnabla^{p}E_{pijk} + 2\tilde{E}_{i(jk)} = \nabla^{p}E_{pijk} + 2E_{i(jk)} + (n-4)\al^{p}E_{pijk}.
\end{align}
There hold
\begin{align}\label{nablaatrace}
&\tnabla_{i}A_{j} = \nabla_{i}A_{j} - 3\al_{i}A_{j} - \al_{j}A_{i} + \al^{p}A_{p}H_{ij},&
&\tnabla^{p}A_{p} = \nabla^{p}A_{p} + (n-4)\al^{p}A_{p}.
\end{align}
The claims are evident from \eqref{nablaonee} and \eqref{nablaatrace}.
\end{proof}
It has not been shown that $\nabla^{p}E_{pijk} + 2E_{i(jk)}$, $\nabla^{p}E_{p(ijk)}$, and $\nabla^{p}A_{p}$ need not vanish identically, although this seems likely.

\subsubsection{Projectively and conjugate projectively flat AH structures}
An AH structure is \textbf{projectively flat} if $\en$ is projectively flat. If the conjugate AH structure $(\ben, [h])$ is projectively flat, then $(\en, [h])$ is \textbf{conjugate projectively flat}.

Although the projective Weyl tensor $B_{ijk}\,^{l}$ is trace-free it is not $H$-trace free. The trace-free tensor $B_{ij} \defeq B_{ip}\,^{p}\,_{j}$ satisfies 
\begin{align}
\label{middlebtraces} &B_{(ij)}  = \mr{Q}_{ij} + \tfrac{1}{n-1}\mr{R}_{ij} 
= \tfrac{n(n-2)}{n-1}\left(E_{ij} - \mr{A}_{ij}\right)= \tfrac{n(2-n)}{n-1}\mr{\bar{W}}_{ij}, &  B_{[ij]}  = \tfrac{(4-n^{2})}{2(n+1)}F_{ij}.
\end{align}
Substituting \eqref{bijkl} into \eqref{rintermsofw} and using \eqref{ricintermsofw} to write $\mr{P}_{ij} = \mr{W}_{ij} + \tfrac{1}{1-n}\mr{\bar{W}}_{ij}$ gives  
\begin{align}\label{bintermsofw}
B_{ijkl} & = W_{ijkl} + \tfrac{2}{1-n}H_{l[i}\mr{\bar{W}}_{j]k} + 2H_{k[i}\mr{\bar{W}}_{j]l} + \tfrac{1}{n+1}F_{ij}H_{kl} - \tfrac{1}{n+1}H_{l[i}F_{j]k} + H_{k[i}F_{j]l}.
\end{align}

\begin{lemma}\label{projflatahlemma}
For an AH structure $(\en, [h])$ on a manifold of dimension $n > 2$ the following are equivalent:
\begin{enumerate}
\item $(\en, [h])$ is projectively flat.
\item $F_{ij} = 0$, $A_{ijkl} = 0$, $E_{ijkl} =0$, and $\mr{\bar{W}}_{ij} = 0$.
\item There holds $R_{ijkl} = -2H_{l[i}\mr{A}_{j]k}  - 2H_{l[i}E_{j]k}+ \tfrac{2}{n(n-1)}RH_{l[i}H_{j]k}$.
\end{enumerate}
\end{lemma}
\begin{proof}
The equivalence of the three conditions in the fourth part of $(2)$ is immediate from \eqref{middlebtraces}. From \eqref{bintermsofw} it is evident that $(2)$ implies $(1)$. If $(\en, [h])$ is projectively flat then by \eqref{middlebtraces} there hold $\mr{\bar{W}}_{ij} =0$ and $F_{ij} = 0$, in which case \eqref{bintermsofw} shows $W_{ijkl} = 0$; thus $(1)$ implies $(2)$. If there holds $(1)$ then $\mr{\bar{W}}_{ij} = 0$ and $F_{ij} = 0$, and so \eqref{rintermsofw} yields $(3)$. On the other hand, if there holds $(3)$, then tracing $(3)$ shows $nR_{ij} = n(1-n)\mr{W}_{ij} + RH_{ij}$, so that $F_{ij} = 0$ and $P_{ij} = \mr{W}_{ij} + \tfrac{R}{n(1-n)}H_{ij}$; substituting this last expression into $(3)$ shows $B_{ijkl} = 0$, so that $\en$ is projectively flat.
\end{proof}
Note that $\mr{\bar{W}}_{ij} = 0$ is equivalent to either of the conditions $E_{ij} = \mr{A}_{ij}$ or $\mr{R}_{ij} = (1-n)\mr{Q}_{ij}$.
\begin{corollary}\label{conjprojflatae}
AH structure $(\en, [h])$ on manifold of dimension at least $3$ is conjugate projectively flat if and only if $F_{ij} = 0$; $A_{ijkl} = 0$; $E_{ijkl} = 0$;  and $\mr{W}_{ij} = 0$.
\end{corollary}

As will be shown in section \ref{affinehyperspheresection}, the AH structure induced on a non-degenerate co-oriented hypersurface in flat affine space is exact and conjugate projectively flat, and, conversely, a conjugate projectively flat AH structure is locally given in this way.

\subsubsection{Two-dimensional AH structures}\label{twodimensionalsection}
Here some of the special features of the two-dimensional case are briefly recounted (this case is treated in detail in \cite{Fox-2dahs}). This material is needed in the discussion in section \ref{einsteinsection} about the possibility of including the self-conjugacy of the curvature in the definition of the Einstein condition.

By Lemma \ref{weylcriterion}, the conformal torsion of a two-dimensional CP pair must vanish, and so it must be an AH structure, and, again by Lemma \ref{weylcriterion}, the tensors $A_{ijkl}$ and $E_{ijkl}$ must vanish, so its curvature is completely determined by the tensors $R$, $F_{ij}$, $\mr{R}_{ij}$, and $\mr{Q}_{ij}$. Lemma \ref{twodcubicformlemma} (or \eqref{confricfree}) implies $2\bt_{ij} = \nbt H_{ij}$, and by \eqref{middlebtraces} there holds  $\mr{Q}_{ij} = - \mr{R}_{ij}$, and so the local invariants of $(\en, [h])$ are expressible in terms of $R$, $F_{ij}$, $E_{ij}$, and $\bt$. In two-dimensions the projective Cotton tensor $C_{ijk}$ is the complete obstruction to local projective flatness, and, again by Lemma \ref{weylcriterion}, $C_{ijk} = 2C_{[i}H_{j]k}$ in which $C_{i} \defeq C_{ip}\,^{p}$, so it is convenient to work with the weighted one-form $C_{i}$ in lieu of $C_{ijk}$.  Substituting $E_{ij}$ into \eqref{bijkl} and using $F_{[ij}H_{k]l} = 0$ gives $R_{ijkl} = -4H_{l[i}E_{j]k} + RH_{l[i}H_{j]k} + F_{ij}H_{kl}$. However, because $H_{k[i}E_{j]l} - H_{l[i}E_{j]k}$ is trace-free and has symmetries corresponding to the Young diagram of the partition $(22)$ it vanishes by Lemma \ref{weylcriterion}, and so 
\begin{align}\label{rijkl}
& R_{ijkl} = -2H_{k[i}E_{j]l} -2H_{l[i}E_{j]k} + RH_{l[i}H_{j]k} + F_{ij}H_{kl}.
\end{align}
The following are easily verified.
\begin{align}\label{projcotton}
\begin{split}
&C_{i}   = - \tfrac{1}{2}\nabla_{i}R - \tfrac{1}{3}\nabla^{p}F_{ip} -2 \nabla^{p}E_{ip} +2\bt_{i}\,^{pq}E_{pq},\quad
 \bar{C}_{i}  = - \tfrac{1}{2}\nabla_{i}R - \tfrac{1}{3}\nabla^{p}F_{ip} +2\nabla^{p}E_{ip},\\
&\tfrac{1}{2}(C_{i} + \bar{C}_{i})  = - \tfrac{1}{2}\nabla_{i}R - \tfrac{1}{3}\nabla^{p}F_{ip} + \bt_{i}\,^{pq}E_{pq},\qquad
\tfrac{1}{2}(C_{i} - \bar{C}_{i})  = -2\nabla^{p}E_{ip} + \bt_{i}\,^{pq}E_{pq}.
\end{split}
\end{align}
On an oriented surface a Riemannian signature conformal structure determines a complex structure. It turns out that a two-dimensional AH structure can be seen as a special case of the locally conformally K\"ahler AH structures mentioned in section \ref{lcksection} of the introduction. Because of the holomorphic structure, in particular the Riemann-Roch theorem, more refined results are possible in the two-dimensional case than are possible in general, and this in part justifies giving this case a separate treatment. For example, it follows from \eqref{projcotton} that for a two-dimensional AH structure with self-conjugate projective Cotton tensor, the divergence $D^{p}E_{ip}$ is zero for $D$ the Levi-Civita connection of a Gauduchon metric, and this implies that $E_{ij}$ is the real part of a holomorphic quadratic differential. See the discussion in section \ref{labourieloftinsection} for a more relevant example.

\section{Einstein AH structures}\label{einsteinahsection}

In this section the Einstein equations are defined, and some examples are given of Einstein AH structures.

\subsection{Definition of Einstein equations}\label{einsteinsection}

\begin{definition}\label{einsteindefinition}
An AH structure on a manifold of dimension $n \geq 2$ is \textbf{naive Einstein} if there vanishes the trace-free symmetric part of every rank two trace of its curvature, and an AH structure is \textbf{Einstein} if it is naive Einstein and satisfies 
\begin{align}\label{einstein}
\nabla_{i}R + n\nabla^{p}F_{ip} = 0.
\end{align}
\end{definition}
\noindent
That $(\en, [h])$ be naive Einstein means explicitly that $\mr{R}_{ij} = 0$ and $\mr{Q}_{ij} = 0$, or, equivalently, $\mr{T}_{ij} = 0$ and $E_{ij} = 0$, or, etc. Usually it is easiest to check $\mr{R}_{ij} = 0$ and $E_{ij} = 0$. By \eqref{uijsym2} that an AH structure be Einstein is equivalent to the equations
\begin{align}\label{einsteinequations}
&R_{(ij)} = \tfrac{R}{n}H_{ij}, & &\nabla_{p}\bt_{ij}\,^{p} = \bt_{ij},& &\nabla_{i}R + n\nabla^{p}F_{ip} = 0.
\end{align}
An example of a naive Einstein AH structure which is not Einstein is given in section \ref{naiveexample}. 

\begin{lemma}
An AH structure $(\en, [h])$ is naive Einstein or Einstein if and only if the conjugate AH structure $(\ben, [h])$ has the same property.
\end{lemma}
\begin{proof}
Immediate from \eqref{blconjricci}-\eqref{nbscalinv}.
\end{proof}

\begin{lemma}\label{einsteinconservativelemma}
On a manifold of dimension $n > 2$ a naive Einstein AH structure $(\en, [h])$ is Einstein if and only if it is conservative, and a naive Einstein AH structure with self-conjugate curvature is Einstein. On a manifold of dimension $n \geq 2$ an Einstein AH structure has parallel scalar curvature if and only if its projective Cotton tensor is trace-free.
\end{lemma}
\begin{proof}
For a naive Einstein AH structure on a manifold of dimension $n>2$
\begin{align}
&\label{ebcontracted}2n(2-n)A_{i} = (n-2)\left(\nabla_{i}R + n\nabla^{p}F_{ip}\right) = nE_{ijkl}\bt^{jkl},
\end{align}
follows from the last equality of \eqref{ebtw}, and the first two claims are immediate from \eqref{ebcontracted}. For a naive Einstein AH structure on a manifold of dimension $n \geq 2$ there hold
\begin{align}
\label{ebcipp}&C_{i} =  -\tfrac{1}{n}\nabla_{i}R - \tfrac{n}{2(n+1)}\nabla^{p}F_{ip} = 2A_{i} + \tfrac{n+2}{2(n+1)}\nabla^{p}F_{ip}  = \tfrac{n}{n+1}A_{i} - \tfrac{(n+2)}{2n(n+1)}\nabla_{i}R,
\end{align}
in which the first equality follows directly from the definitions, and the last two equalities make sense only when $n > 2$ in which case they follow from \eqref{ebcontracted}. If $(\en, [h])$ is moroever Einstein then \eqref{ebcipp} implies $2n(n+1)C_{i} = -(n+2)\nabla_{i}R$, from which the last claim follows. 
\end{proof}

It is shown in section \ref{s3examplerevisited} that the AH structure on $S^{3}$ constructed in section \ref{s3example} is Einstein but does not have self-conjugate curvature. 

Because of Lemma \ref{einsteinconservativelemma} it is convenient to refer to \eqref{einstein} as the \textbf{conservation condition}, and to say that a naive Einstein AH structure is \textbf{conservative} if it satisfies \eqref{einstein}. This convention has content only when $n = 2$, and it should be noted that meaning has been given to the phrase \textit{conservative AH structure} only when $n > 2$.

\begin{lemma}\label{parallelexactlemma}
For an Einstein AH structure $(\en ,[h])$ on an $n$-manifold the weighted scalar curvature is parallel if and only if either it vanishes identically or $(\en, [h])$ is proper and exact. If $n \neq 4$ these conditions imply $F_{ij}$ is $[h]$-null. 
\end{lemma}
\begin{proof}
If $\nabla_{i}R = 0$ then either $R$ vanishes identically or $R$ vanishes nowhere. In the latter case $RH_{ij}$ is a distinguished metric, so $(\en, [h])$ is proper and exact. Conversely, if $(\en, [h])$ is exact then it is closed so the conservation condition implies $\nabla_{i}R = 0$. If $\nabla_{i}R = 0$ then $\nabla^{p}F_{ip} = 0$ and if $n \neq 4$ the proof of Lemma \ref{fcoclosedlemma} shows that $F_{ij}$ is $[h]$-null. 
\end{proof}

\subsubsection{}
The usual Einstein Weyl equations are defined when $n > 2$ as $\mr{R}_{ij} = 0$. Since the curvature of a Weyl structure is necessarily self-conjugate, the Einstein AH equations for Weyl structure are the usual Einstein Weyl equations. When $n = 2$ the naive Einstein AH equations are just $E_{ij} = 0$, which is vacuous for a Weyl structure. In \cite{Calderbank-mobius}, Calderbank defined a Weyl structure to be Einstein Weyl if $\nabla_{i}R + 2\nabla^{p}F_{ip} = 0$. The Einstein AH equations were defined so that for $n = 2$ they specialize for Weyl structures to Calderbank's. Moreover, comparison of the two-dimensional Einstein Weyl case with \eqref{ebcontracted} motivated including the equation \eqref{einstein} in the definition of the $n$-dimensional Einstein AH equations. By the traced differential Bianchi identity the usual Einstein equations for a metric imply the constancy of the scalar curvature. Similarly, (by \eqref{ebcontracted}) the usual Einstein Weyl equations imply \eqref{einstein} when $n > 2$. In this sense the Einstein AH equations more closely resemble the usual Einstein equations for a metric than do the naive Einstein AH equations. 

\subsubsection{}
Let $(\en, [h])$ be an AH structure with aligned representative $\nabla$ and let $D = \nabla + \tfrac{1}{2}\bt_{ij}\,^{k} + 2\ga_{(i}\delta_{j)}\,^{k} - H_{ij}\ga^{k}$ be the Levi-Civita connection of $h \in [h]$. By \eqref{ddivbt} the equation $E_{ij} = 0$ is equivalent to either of the equations
\begin{align}
\label{dga0} &\nabla_{p}\bt_{ij}\,^{p} = \bt_{ij},& &D_{p}\bt_{ij}\,^{p} = n\ga_{p}\bt_{ij}\,^{p},
\end{align}
while by \eqref{confricfree} the equation $\mr{R}_{ij} = 0$ is equivalent to
\begin{align}\label{dga1}
\mr{\sR}_{ij} = \tfrac{1}{4}\mr{\bt}_{ij} + (2-n)\mr{D\ga}_{ij} + (2-n)\mr{\ga\tensor\ga}_{ij}.%
\end{align}
Equations \eqref{dga0} and \eqref{dga1} express the naive Einstein condition in terms of the underlying conformal structure. The meaning of \eqref{einstein} can be seen from \eqref{conserve}.

\subsubsection{}
Lemma \ref{projprojflateinsteinlemma} relates the Einstein equations and projective flatness.
\begin{lemma}\label{projprojflateinsteinlemma}
For an AH structure $(\en, [h])$ of dimension at least $3$ any two of the following conditions imply the third.
\begin{enumerate}
\item $(\en, [h])$ is projectively flat.
\item $(\en, [h])$ is conjugate projectively flat.
\item $(\en, [h])$ is naive Einstein.
\end{enumerate}
These conditions imply $(\en, [h])$ is Einstein with self-conjugate curvature.
\end{lemma}
\begin{proof}
By Lemma \ref{projflatahlemma} a projectively and conjugate projectively flat AH structure has $\mr{W}_{ij} = 0 = \mr{\bar{W}}_{ij}$, so is naive Einstein. By \eqref{bintermsofw} and projectively flat naive Einstein AH structure is closed with $W_{ijkl} = 0$, and by the conjugated version of \eqref{bintermsofw} this implies that it is conjugate projectively flat. By duality conjugate projectively flat and naive Einstein implies projectively flat. By Lemma \ref{conjprojflatae} a projectively flat or conjugate projectively flat AH structure has $E_{ijkl} = 0$, and so by \eqref{ebcontracted} satisfies \eqref{einstein}, so an AH structure satisfying the given conditions is Einstein with self-conjugate curvature.
\end{proof}

When $n = 2$ it follows from \eqref{projcotton} that naive Einstein and projectively flat implies conjugate projectively flat, but it is no longer the case that projectively flat and conjugate projectively flat necessarily implies naive Einstein. In this case the claim analogous to Lemma \ref{projprojflateinsteinlemma} is the following immediate consequence of \eqref{projcotton}.
\begin{lemma}\label{2deinsteinlemma}
For a two-dimensional Einstein AH structure $(\en, [h])$, any one of the following statements implies the other two.
\begin{enumerate}
\item $(\en, [h])$ is projectively flat.
\item $(\en, [h])$ is conjugate projectively flat.
\item The weighted scalar curvature is parallel.
\end{enumerate}
In particular, if $(\en, [h])$ is proper then it is projectively flat and conjugate projectively flat.
\end{lemma}

\subsubsection{}
In dimensions greater than $3$ it is arguable that self-conjugacy of the curvature should be included in the definition of Einstein. In section \ref{einsteinestimatesection} it will be shown that estimates on the growth of the cubic form of an affine hypersphere partly carry over to the Einstein AH structures with self-conjugate curvature, and there will be other indications that the subclass of Einstein AH structures having self-conjugate curvature is more amenable to study than is the class of Einstein AH structures. For example, Lemma \ref{projprojflateinsteinlemma} could be interpreted in this way. On the other hand, in two-dimensions the naive Einstein equations could be phrased simply as that the curvature of the AH structure be self-conjugate, and the Einstein equations have to be imposed on top of this, which suggests regarding the self-conjugacy and Einstein conditions as distinct. It not yet being clear whether self-conjugacy of the curvature should be regarded as an essential constituent of the Einstein condition, or as a conceptually distinct complementary condition, it seems advisable to maintain a willingness to take this stronger condition as the definition of Einstein should further evidence suggest so doing. For the time being, as it seems possible to base a tractable theory on the weaker conservation condition, this is what has been done as far as is possible. The sort of result that would demand including self-conjugacy in the definition of Einstein would be a variational description that yielded self-conjugate Einstein structures but not the more general class of Einstein structures, or a result showing something like finite-dimensionality of the local moduli of self-conjugate Einstein AH structures, but not of Einstein AH structures.

\subsubsection{}
The Einstein equations are preserved under conjugacy; a stronger condition is to ask that they hold for the entire pencil $(\pen, [h])$.

\begin{lemma}\label{naiveforalltlemma}
If $(\en, [h])$ is naive Einstein the associated pencil $(\pen, [h])$ is naive Einstein for all $t$ if and only if $\mr{\bt}_{ij} = 0$. If $(\en, [h])$ is Einstein the associated pencil $(\pen, [h])$ is Einstein for all $t$ if and only if $\mr{\bt}_{ij} = 0$ and $\nabla_{i}\bt = 0$. 
\end{lemma}
\begin{proof}
By \eqref{brtaij} and \eqref{brteij} there hold $\mr{\brt{T}}_{ij} = \mr{T}_{ij} + t(1-t)\mr{\bt}_{ij}$ and $\mr{\brt{E}}_{ij} = (1-2t)\mr{E}_{ij}$, so that if $(\en, [h])$ is naive Einstein, then $(\pen, [h])$ is naive Einstein for all $t$ if and only if $\mr{\bt}_{ij} = 0$. Since $\brt{F}_{ij} = F_{ij}$, using \eqref{brteij} one finds $\pnabla_{i}\brt{R} + n\pnabla^{p}\brt{F}_{ip} = \nabla_{i}R + n\nabla^{p}F_{ip} + t(1-t)\nabla_{i}\nbt$, so that $(\pen, [h])$ is Einstein for all $t$ if and only if $\mr{\bt}_{ij} = 0$ and $\nabla_{i}\nbt = 0$.
\end{proof}

A non-trivial example of a naive Einstein AH structure with $\mr{\bt}_{ij} = 0$ is given in section \ref{naiveexample}.

\begin{lemma}\label{mrbtparallellemma}
Let $(\en, [h])$ be an Einstein AH structure on a manifold of dimension $n$. If $n > 2$ and $\mr{\bt}_{ij} = 0$ then $\nabla_{i}\nbt = 0$. %
If $n = 2$ then $\mr{\bt}_{ij} = 0$.
\end{lemma}
\begin{proof}
By \eqref{divbtij} and \eqref{ebtw} the Einstein condition implies $2\nabla^{p}\bt_{ip} = \nabla_{i}\bt + \bt_{i}\,^{pq}\bt_{pq}$, while that $\mr{\bt}_{ij} = 0$ implies $n\nabla^{p}\bt_{ip} = \nabla_{i}\bt$ and $\bt_{i}\,^{pq}\bt_{pq} = 0$. If $n > 2$ these imply $\nabla_{i}\bt = 0$. 
Lemma \ref{twodcubicformlemma} implies that when $n = 2$ the condition $\mr{\bt}_{ij} = 0$ is automatic.
\end{proof}

\begin{definition}\label{stronglyeinsteindefinition}
An Einstein AH structure is \textbf{strongly Einstein} if the associated pencil $(\pen, [h])$ is Einstein for all $t$. 
\end{definition}

\begin{lemma}\label{stronglyeinsteinlemma}
An Einstein AH structure $(\en, [h])$ on a manifold of dimension $n > 2$ is strongly Einstein if and only if $\mr{\bt}_{ij} = 0$. An Einstein AH structure $(\en, [h])$ on a manifold of dimension $2$ is strongly Einstein if and only if $\nabla_{i}\bt = 0$. 
\end{lemma}
\begin{proof}
This  follows from Lemma \ref{naiveforalltlemma} and Lemma \ref{mrbtparallellemma}.
\end{proof}

\begin{lemma}
A strongly Einstein Riemannian AH structure $(\en, [h])$ on a manifold of dimension $n \geq 2$ which is not Weyl is exact. 
\end{lemma}

\begin{proof}
By Lemma \ref{naiveforalltlemma}, $\nabla_{i}\bt = 0$, so that either $\bt$ vanishes identically or it vanishes nowhere. That $(\en, [h])$ not be Weyl means $\bt_{ij}\,^{k}$ is not everywhere $0$, and in Riemannian signature this means $\nbt$ is not everywhere zero. By the preceeding, $\bt$ is nowhere vanishing, so $(\en, [h])$ is exact.
\end{proof}

An example of a strongly Einstein Riemannian AH structure on $S^{3}$ is given in section \ref{s3example}.

\subsubsection{}
Given an AH structure, define an operator $\C$ on $c$-weight $\la$ two-forms $\omega_{ij}$ by
\begin{align}
\C(\om)_{ijk} = \nabla_{k}\omega_{ij} - \nabla_{[i}\omega_{jk]} + \tfrac{2}{n-1}\nabla^{p}\omega_{p[i}H_{j]k} + 2\bt_{k[i}\,^{p}\om_{j]p}.
\end{align}
If $(\ten, [h])$ and $(\en, [h])$ generate the same Codazzi projective structure, then
\begin{align*}
\tilde{\C}(\om)_{ijk} = \C(\om)_{ijk} + 2(\la - 3)\left(\tfrac{1}{3}\al_{k}\om_{ij} - \tfrac{1}{3}\al_{[i}\om_{j]k} + \tfrac{2}{n-1}\al^{p}\om_{p[i}H_{j]k}\right),
\end{align*}
where $\tnabla - \nabla = 2\al_{(i}\delta_{j)}\,^{k}- \al^{k}H_{ij}$, and so $\C(\om)$ is invariant if $\la = 3$. A straightforward computation using \eqref{aijkdefined} shows that for an Einstein AH structure there holds $\C(F)_{ijk} = -2A_{ijk}$, an observation which in the Weyl case is due to D. Calderbank in section $5$ of \cite{Calderbank-faraday}. 
\begin{lemma}\label{faradaycottonlemma}
For a closed Einstein AH structure there hold $A_{ijk} = 0$, $E_{ijk} = 0$, and
\begin{align}\label{einsteinharmonic}
\nabla_{p}R_{ijk}\,^{p} = \nabla_{p}W_{ijk}\,^{p} = 0.
\end{align}
\end{lemma}
\begin{proof}
That $\C(F)_{ijk} = -2A_{ijk}$ implies $A_{ijk} = 0$ for a closed Einstein AH structure. Alternatively, that both $A_{ijk}$ and $E_{ijk}$ vanish for a closed Einstein AH structure is easily checked directly from the definitions \eqref{aijkdefined}. The first equality of \eqref{einsteinharmonic} is immediate from \eqref{rintermsofw}, and the second follows from \eqref{wdiffbianchi} and the vanishing of $A_{ijk}$, $E_{ijk}$, and $A_{i}$. Alternatively, because $(\en, [h])$ is closed, $R$ is parallel, so $\nabla_{[i}R_{j]k} = 0$, and $\nabla_{p}R_{ijk}\,^{p} = 0$ follows from the differential Bianchi identity.
\end{proof}

Lemma \ref{faradaycottonlemma} shows that the curvature tensor of a closed Einstein AH structure satisfies a sort of harmonicity condition. It would be interesting to explain this. In this regard, observe that for an Einstein AH structure for which the curvature is not self-conjugate it has not been proved that $\nabla^{p}A_{ijkp}$ vanishes, and this is probably not true.

\subsection{Affine hyperspheres}\label{affinehyperspheresection}
In this section the claims made in paragraphs \ref{anpara} and \ref{conormalpara} are justified and it is proved that the AH structure induced on an affine hypersurface is Einstein if and only if the hypersurface is an affine hypersphere. 

It will be shown that a conformal projective structure is induced on a non-degenerate immersed hypersurface in a manifold equipped with projective structure, and that the induced conformal projective structure is a Codazzi structure when the ambient projective structure is projectively flat. The underlying conformal structure of the induced conformal projective structure is that determined by the second fundamental form of the immersion. Each torsion-free affine connection representative of the ambient projective structure determines an affine normal subbundle and an induced projective structure. When this representative ambient affine connection is projectively flat the connection induced by the affine normal subbundle forms with the second fundamental form an AH structure. In the case of a non-degenerate hypersurface immersion in flat affine space, the conjugate AH structure of the AH structure comprising the second fundamental form and the projective structure induced by the affine normal subbundle is that comprising the second fundamental form and the flat projective structure induced via the conormal Gau\ss\, map (because there is no metric structure on the ambient affine space, the normal and conormal Gau\ss\, maps must be distinguished).

Calabi's articles \cite{Calabi-affinelyinvariant}, and \cite{Calabi-affinemaximal} are recommended for introduction to the geometry of hypersurfaces in flat affine space. The introduction of \cite{Cheng-Yau-affinehyperspheresI} derives the structure equations of affine hypersurfaces following Chern using the method of moving frames. The material in this section is mostly standard, although the presentation is not and has been made so as to fit neatly into the viewpoint of the present article. In the case of hypersurfaces in flat affine space something equivalent to it can be found also in textbooks, e.g. \cite{Nomizu-Sasaki}. Certain points are presented with somewhat different emphasis than is usual. In particular the ambient space (the target of the immersion) is not assumed to carry a parallel volume form; the choice of an ambient volume form is viewed as an extra piece of data which can be included or not. %

\subsubsection{}
The smooth immersion $i:M \to N$ is \textbf{co-orientable} if the normal bundle $\nmb(M) \defeq i^{\ast}(TN)/Ti(TM)$ is orientable, and \textbf{co-oriented} if an orientation of $\nmb(M)$ is fixed. The \textbf{conormal bundle} $\ann Ti(TM) \subset \ctn$ is an exact Lagrangian submanifold of $\ctn$. Define the \textbf{second fundamental form} $\Pi^{\nabla}(i) \in \Ga(S^{2}(\ctm)\tensor \nmb(M))$ of $i$ relative to the torsion-free affine connection $\nabla$ on $N$ by $\Pi^{\nabla}(i)(X,Y) = \nmp(\nabla_{X} Ti(Y))$ for $X, Y \in \Ga(TM)$. Evidently $\Pi^{\nabla}(i)(X,Y)$ depends only on the projective equivalence class $\en$ of $\nabla$, and so it makes sense to speak of \textit{the second fundamental form of an immersion relative to a projective structure} on the image manifold. A subbundle $W$ of $i^{\ast}(TN)$ transverse to $TM$ is a \textbf{transverse subbundle} (along $M$). A non-vanishing (local) section of some such $W$ is a (local) \textbf{transverse vector field}, or simply a (local) transversal. %

\subsubsection{}\label{inducedconnectionsection}
Let $i:M \to N$ be a smooth immersion. The connection $\hnabla$ induced on $i^{\ast}(TN)$ by an affine connection $\hnabla$ on $N$ can be split using a transverse subbundle $W$. Namely, the \tbsf{connection}{induced} \textbf{$\nabla$ induced on $M$ by $\hnabla$ and $W$} is defined by setting $Ti(\nabla_{X}Y)$ equal to the image of $\hnabla_{X}Ti(Y)$ under the projection of $i^{\ast}(TN)$ onto $Ti(TM)$ along $W$, and the \tbf{connection $\bnabla$ induced on $W$ by $\hnabla$} is defined by setting $\bnabla_{X}s$ equal to the projection onto $W$ along $Ti(TM)$ of $\bnabla_{X}s$ for $s \in \Ga(W)$. The connection induced on $M$ is torsion free if $\hnabla$ is torsion free. Evidently the connections induced on $M$ by projectively equivalent connections on $N$ and a fixed transversal are projectively equivalent.

\subsubsection{}
A \tbf{hypersurface immersion} means an immersion $i:M \to N$ with rank one normal bundle. Since in the remainder of this section all immersions will be hypersurface immersions, they will just be called \textit{immersions}. A hypersurface immersion is \textbf{non-degenerate} (with respect to $\hnabla$) if its second fundamental form $\Pi = \Pi^{\hnabla}$ is non-degenerate. While the rank of $\Pi$ is well-defined, its signature is not unless a co-orientation of $i(M)$ in $N$ is specified, although even if $i$ is not co-orientable it makes sense to say that $\Pi$ has or has not split signature. If the signature $(p, n-p)$ is not split and $M$ is co-orientable, a co-orientation of $M$, called \tbfs{positive}{co-orientation} (resp. \tbfs{negative}{co-orientation}) is distinguished by the requirement that $p > n-p$ (resp. $p < n-p$). A co-orientable non-degenerate immersion equipped with the positive co-orientation will be called a \tbf{positively co-oriented non-degenerate immersion} (when this phrase is used in the split signature case, it is to be understood that an arbitrary choice of co-orientation has been fixed). A local section of the normal bundle consistent or not with a given co-orientation is called \tbf{positive} or \tbf{negative}. So a normal vector field can be positive with respect to the negative co-orientation, in which case it is negative for the positive co-orientation.

When $i:M \to (N, \hnabla)$ is non-degenerate and co-orientable, the phrase \textit{the induced conformal structure}, or something similar, will refer to the conformal structure induced by the positive co-orientation, unless explicitly indicated otherwise. Alternatively, one can speak of the \textit{positive} and \textit{negative} conformal classes induced by a non-degenerate co-oriented immersion, and if not indicated otherwise, there will be considered only the positive conformal class. Discussion assuming a positive co-orientation applies in the split signature case provided a particular co-orientation is fixed.

\subsubsection{}
An ellipsoid in affine space divides the space into two regions, one of which is compact and convex, and one of which is neither compact nor convex; a normal field positive with respect to the positive co-orientation points into the non-compact non-convex region. The hyperboloid $z^{n+1} = \sqrt{\sum_{i = 1}^{n}(z^{i})^{2} + 1}$ divides affine space into two non-compact regions, one of which is convex and one of which is not convex; a normal field positive with respect to the positive co-orientation points into the convex region. 

A (connected) co-orientable hypersurface in flat affine space divides the space into two regions. A choice of co-orientation determines an outside and an inside: a normal field positive for the given co-orientation points into the exterior region. Reversing the co-orientation interchanges the notions of inside and outside. 

A co-orientable non-degenerate hypersurface is \tbf{convex} (resp. \tbf{concave}) at a point $x$ if the intersection of the plane tangent to the hypersurface at $x$ with a small ball centered at $x$ lies on the same side (resp. opposite side) of the hypersurface as does any transverse vector at $x$ positive with respect to the positive co-orientation. These conventions are such that in flat affine space an ellipsoid is convex, and the exterior region determined by the positive co-orientation is the region complementary to that bounded by the ellipsoid (the ellipsoid's exterior in the colloquial sense), while the upwards opening hyperboloid $z^{n+1} = \sqrt{\sum_{i = 1}^{n}(z^{i})^{2} + 1}$ is concave, and the exterior region determined by the positive co-orientation is the convex region bounded from below by the hypersurface (which is the hyperboloid's interior in a naive sense).

\subsubsection{}
Let $i:M \to (N, \hnabla)$ be a hypersurface immersion with induced connection $\nabla$ on $M$. A (local) transverse vector field $\nm$ induces: an identification of the second fundamental form $\Pi^{\hnabla}(i)$ with a tensor $h \in \Ga(S^{2}(\ctm))$; the \tbf{shape operator} $S \in \Ga(\ctm \tensor TM)$; and the $1$-form $\tau$, all defined by 
\begin{align}\label{affineinduced2}
&\hnabla_{X}Ti(Y) = Ti(\nabla_{X}Y) + h(X, Y)\nm, &&\hnabla_{X}\nm = -Ti(S(X)) + \tau(X)\nm.
\end{align}
in which, as usual, there is required interpretation in terms of vector fields defined near $i(M)$.

If $\nm$ is replaced by $\tilde{\nm} = f(\nm + \tilde{X})$ with $f \neq 0$ and $\tilde{X} = Ti(X)$ along $i(M)$, then $\nabla$, $h$, $S$, and $\tau$ are replaced by $\tnabla$, $\tilde{h}$, $\tilde{S}$, and $\tilde{\tau}$ satisfying
\begin{align}
\label{chh} 
\begin{split}
&\tilde{h}_{ij} = f^{-1}h_{ij},\qquad  \tilde{S}_{i}\,^{j} = f\left(S_{i}\,^{j} - \nabla_{i}X^{j} + \tau_{i}X^{j} + h_{ip}X^{p}X^{j} \right),\\
&\tnabla - \nabla = - h_{ij}X^{k},\qquad   \tilde{\tau}_{i} = \tau_{i} + d\log|f|_{i} + h_{ip}X^{p}.
\end{split}
\end{align}
The positive induced conformal structure is denoted $[h]$ and comprises those $h$ induced by positive transversals. Via the trivialization of the normal bundle given by $\nm$ the induced connection on $W \simeq \nmb(M)$ determines a one-form on $M$, and this one-form is $\tau$. %
If a transversal $\nm$ is fixed and $\si$ is a one-form on $N$ the projectively equivalent connection $\hnabla + \si \tensor \delta + \delta \tensor \si$ induces on $M$ the connection $\tnabla$, and the tensors $\tilde{h}$, $\tilde{S}$, and $\tilde{\tau}$ related to $\nabla$, $h$, $S$, and $\tau$ by
\begin{align}
\label{pchh}
&\tilde{h}_{ij} = h_{ij},&
&\tilde{S}_{i}\,^{j}  = S_{i}\,^{j} - \si(\nm)\delta_{i}\,^{j},& 
&\tnabla -  \nabla =  2i^{\ast}(\si)_{(i}\delta_{j)}\,^{k},&
&\tilde{\tau}_{i}= \tau_{i} + i^{\ast}(\si)_{i}.
\end{align}
Equation \eqref{pchh} shows  that projective change of $\hnabla$ induces a projective change of the induced connection, and leaves unchanged the trace-free part of the shape operator. From \eqref{chh} and \eqref{pchh} there follows

\begin{lemma}
A non-degenerate co-oriented immersion of a smooth manifold $M$ as a hypersurface in a manifold $N$ equipped with projective structure $[\hnabla]$ induces on $M$ a conformal projective structure.
\end{lemma}

The metric $h_{ij}$ and shape operator $S_{i}\,^{j}$ induced by a positive normal vector field $\nm$ change when the vector field is rescaled; for this reason it is sometimes more convenient to work with the \tbf{scale-free shape operator} $\sh_{ij}$ defined by $\sh_{ij} = S_{i}\,^{p}h_{pj}$. By \eqref{chh}, $\sh_{ij}$ depends only on the normal bundle, and not on the scaling of the normal vector field. Its trace-free symmetric part $\mr{\sh}_{ij}$ depends only on the projective class of $\hnabla$.

\subsubsection{}
Write $\hat{R}_{IJK}\,^{L}$ for the curvature tensor of $\hnabla$. Fix a transversal $\nm$. Letting $P_{TM W}$ and $P_{W TM}$ denote the projections of $i^{\ast}(TN)$ onto $TM$ along $W$ and onto $W$ along $TM$, define on $M$ tensors $\hat{R}_{ijk}\,^{l}$, $\hat{R}_{ijk}\,^{\infty}$, $\hat{R}_{ij\infty}\,^{k}$, and $\hat{R}_{ij\infty}\,^{\infty}$ (the $\infty$'s should be regarded as dummy place-holders or as abstract indices on the tensor powers of the normal bundle) by $P_{\nm TM}(\hat{R}(X, Y)Z)^{l} = X^{i}Y^{j}Z^{k}\hat{R}_{ijk}\,^{l}$, $P_{TM \nm}(\hat{R}(X, Y)Z) = X^{i}Y^{j}Z^{k}\hat{R}_{ijk}\,^{\infty}\nm$, $P_{\nm TM}(\hat{R}(X, Y)\nm)^{l} = X^{i}Y^{j}\hat{R}_{ij\infty}\,^{l}$, and $P_{TM \nm}(\hat{R}(X, Y)\nm) = X^{i}Y^{j}\hat{R}_{ij\infty}\,^{\infty}\nm$. Then 
\begin{align}
\label{projhrijkl} &\hat{R}_{ijk}\,^{l}  = R_{ijk}\,^{l} - 2S_{[i}\,^{l}h_{j]k},&
&\hat{R}_{ijk}\,^{\infty}  = 2\nabla_{[i}h_{j]k} + 2\tau_{[i}h_{j]k},\\
\label{projhrijnk} &\hat{R}_{ij\infty}\,^{k} = -2\nabla_{[i}S_{j]}\,^{k} + 2\tau_{[i}S_{j]}\,^{k},&
&\hat{R}_{ij\infty}\,^{\infty} = d\tau_{ij} + 2\sh_{[ij]}.
\end{align}
A non-degenerate hypersurface immersion has \tbf{trace-free normal curvature} if $\Pi^{jk}\hat{R}_{ijk}\,^{\infty} = 0$, in which $\Pi^{ij}$ is the $\nmb(M)^{-1}$-valued bivector dual to the second fundamental form. If $\nm$ is fixed then by \eqref{projhrijkl}, the completely $h$-trace free part of $2\nabla_{[i}h_{j]k}$ equals the completely $h$-trace free part of $\hat{R}_{ijk}\,^{\infty}$, so that 
\begin{align}\label{khatr}
3\ct_{\nabla}(h)_{ijk} = 2\left(\hat{R}_{ijk}\,^{\infty} + \tfrac{2}{n-1}h^{pq}h_{k[i}\hat{R}_{j]pq}\,^{\infty}\right). 
\end{align}
As already observed, when $\nm$ is changed the induced connection changes conformal projectively, so that the conformal torsion of the induced connection does not change, and the conformal torsion of the induced conformal projective structure is therefore equal to the $[h]$-trace-free part of $\hat{R}_{ijk}\,^{\infty}$ for any choice of $\nm$. Let $\hat{B}_{IJK}\,^{L}$ be the projective Weyl tensor of the projective structure $[\hnabla]$ generated by $\hnabla$. Since there is not available a useful expression for the Ricci tensor of $\hnabla$, most of the components of $\hat{B}_{IJK}\,^{L}$ are not readily computable; however, $\hat{B}_{ijk}\,^{\infty} = \hat{R}_{ijk}\,^{\infty}$, so that if $[\hnabla]$ is projectively flat then by \eqref{projhrijkl} there holds $\nabla_{[i}h_{j]k} = -\tau_{[i}h_{j]k}$ for any choice of normal $\nm$, in which case the induced conformal projective structure is a Codazzi projective structure.

\begin{lemma}
A non-degenerate co-oriented immersion of a smooth manifold $M$ as a hypersurface in a manifold $N$ equipped with a flat projective structure $[\hnabla]$ induces on $M$ a Codazzi projective structure.
\end{lemma}

This example motivates the definitions of conformal projective and Codazzi projective structures. 

\subsubsection{}
Let $\hat{P}_{IJ}$ be the modified Ricci tensor of $\hnabla$. Tracing the first equation of \eqref{projhrijkl} in $kl$ and rewriting the second equation of \eqref{projhrijnk} gives
\begin{align}
& \hat{B}_{ijp}\,^{p} + 2(n+1)\hat{P}_{[ij]} = R_{ijp}\,^{p} - 2\sh_{[ij]}, & & \hat{B}_{ij\infty}\,^{\infty} + 2\hat{P}_{[ij]} = d\tau_{ij} + 2\sh_{[ij]}.
\end{align}
Taking a linear combination so as to cancel the terms involving $\hat{P}_{[ij]}$, and noting $\hat{B}_{ijp}\,^{p} + \hat{B}_{ij\infty}\,^{\infty} = 0$ yields
\begin{align}
\label{hatbdoublei}
&R_{ijp}\,^{p} + (n+2)\hat{B}_{ij\infty}\,^{\infty} = (n+1)d\tau_{ij} + 2(n+2)\sh_{[ij]},\\
\label{hatprdtau}&\hat{R}_{ijQ}\,^{Q} = -2 \hat{R}_{[ij]} = 2(n+2)\hat{P}_{[ij]} = R_{ijp}\,^{p} + d\tau_{ij}.
\end{align}
\begin{lemma}
Let $i:M \to (N, \hnabla)$ be a non-degenerate positively co-oriented immersion. For any choice of transversal the two-form $-d\tau$ equals the difference of the pullback via $i$ of the curvature of the principal connection induced on $\Det T^{\ast}N$ by $\hnabla$ with the curvature of the principal connection induced on the density bundle $\dens \defeq \Det \ctm \to M$ by the connection $\nabla$ induced on $M$ by the given transversal. In particular, if $\hnabla$ has symmetric Ricci tensor then $d\tau$ is the curvature of the connection induced on $\dens$ by $\nabla$.
\end{lemma}
\begin{proof}
In symbols the first claim is just \eqref{hatprdtau}. From \eqref{hatprdtau} it follows that if $\hnabla$ is Ricci symmetric then the curvature of the connection induced on the normal bundle by the choice of a transverse bundle agrees with the curvature of the connection induced on $\Det \ctm$ by the connection $\nabla$ induced by that same transverse bundle. %
\end{proof}

\begin{theorem}\label{affinenormalbundle}
Let $(N, \hnabla)$ be an $(n+1)$-manifold with a torsion-free affine connection $\hnabla$ and let $i:M \to N$ be a non-degenerate hypersurface immersion. There is a unique transverse subbundle $W$ such that the connection $\nabla$ induced on $M$ by $W$ is aligned with respect to the metric $h$ induced on an open $U \subset M$ by a non-vanishing local section $\nm \in \Ga(U, W)$. If, moreover, $\hnabla$ is projectively flat, then the induced CP pair is an AH structure.
\end{theorem}
\begin{proof}
Let $W$ be any transverse subbundle, let $U \subset M$ be an open subset such that there is a non-vanishing $\nm \in \Ga(U, W)$, and let $\nabla$ and $h$ be the connection and metric induced by $\nm$ on $U$. If $\nm$ is replaced by $\tilde{\nm}^{\al} = \nm^{\al} + Ti_{p}\,^{\al}X^{p}$ then 
\begin{align}
\begin{split}
-(n+2)(n-1)X^{p}h_{pi} &= \tilde{h}^{pq}\left(\tnabla_{i}\tilde{h}_{pq} - n\tnabla_{p}\tilde{h}_{qi} \right) - h^{pq}\left(\nabla_{i}h_{pq} - n\nabla_{p}h_{qi} \right),
\end{split}
\end{align}
Using \eqref{projhrijkl} to rewrite the term $h^{pq}\nabla_{p}h_{qi}$, it follows that if $\nm$ is given, then $X^{i}$ defined by 
\begin{align}\label{normaldetermined}
\begin{split}
(n+2)(n-1)X^{p}h_{pi} &= -n(n-1)\left(\tau_{i} + \tfrac{1}{n}h^{pq}\nabla_{i}h_{pq} + \tfrac{1}{1-n}\hat{R}_{iab}\,^{\infty}h^{ab} \right),
\end{split}
\end{align}
is the unique vector field such that the induced $\tnabla$ is aligned on $U$. If on $U$ the condition of the theorem is satisfied by $\nm$, replacing $\nm$ by $f\nm$ with $f \neq 0$ does not affect the statement. It follows that the span $W$ of $\nm$ is uniquely determined over $U$. Since around each point of $M$ there can be found a neighborhood $U$ on which there are non-vanishing transversals, and since by the uniqueness just proved the $W$'s determined on overlapping neighborhoods of this sort must agree on the overlaps, the transverse subbundle $W$ is globally defined, whether or not $i$ is co-oriented. In the event that $i$ is co-oriented there can be chosen a global transverse vector field $\nm$ and the induced connection is aligned with respect to the conformal class of the metric which it induces. If $\hnabla$ is projectively flat then \eqref{khatr} shows the conformal torsion vanishes, so the induced CP pair is an AH structure.
\end{proof}

The distinguished transverse subbundle $W$ associated by Theorem \ref{affinenormalbundle} to a non-degenerate immersion $i:M \to N$ and a torsion-free connection $\hnabla$ on $N$ will be called the \tbf{affine normal bundle}, and a non-vanishing local section $\nm$ of $W$ will be called an \tbf{affine normal vector field}. 
Lemma \ref{affprojvarylemma} shows how the affine normal changes when the connection on $N$ is varied projectively.

\begin{lemma}\label{affprojvarylemma}
Let $\hnabla$ be a torsion-free affine connection on $N$ and let $i:M \to (N, \hnabla)$ be a non-degenerate hypersurface immersion. If $\nm$ is a (local) affine normal vector field for $\hnabla$ and $h$ is the induced metric then the affine normal bundle associated to $\hnabla + \si \tensor \delta + \delta \tensor \si$ is spanned locally by $\tilde{\nm} = \nm + \si^{\sharp}$ where $\si^{\sharp}$ is any vector field equal to $h^{ip}i^{\ast}(\si)_{p}$ along $i(M)$. 
\end{lemma}
\begin{proof}
Because $\hat{B}_{ijk}\,^{\infty} = \hat{R}_{ijk}\,^{\infty}$ does not change when $\hnabla$ is replaced by $\hnabla + \si \tensor \delta + \delta \tensor \si$, it follows from \eqref{pchh} that the righthand side of \eqref{normaldetermined} transforms by the addition of $(n-1)(n+1)i^{\ast}(\si)_{i}$ when $\nabla$ is replaced by $\hnabla + \si \tensor \delta + \delta \tensor \si$. The claim follows.
\end{proof}

The curvature $d\tau$ of the connection induced on the normal bundle by the affine normal splitting does not depend on the choice of affine normal vector field. 

\begin{lemma}\label{alignedtaulemma}
Let $\hnabla$ be a torsion-free affine connection on $N$ and let $i:M \to (N, \hnabla)$ be a non-degenerate hypersurface immersion. If $\hnabla$ has trace-free normal curvature then a transverse bundle is the affine normal bundle if and only if for the induced connection $\nabla$ and the tensors $h_{ij}$ and $\tau_{i}$ induced by any choice of transverse vector field there hold 
\begin{align}
\label{alignedtau}
&n\tau_{i} = - h^{pq}\nabla_{i}h_{pq}, & &\text{and}& & nd\tau_{ij} = 2R_{ijp}\,^{p},
\end{align}
or, what is the same, $\tau$ is the connection one-form of the connection induced on $\dens^{2(n+1)/n}$ by $\nabla$ corresponding to the local trivialization of $\dens^{2(n+1)/n}$ given by $|\det h|^{-1/n}$.
\end{lemma}

\begin{proof}
The second equality of \eqref{alignedtau} follows from the first by $nd\tau_{ij} = 2n\nabla_{[i}\tau_{j]} = 2R_{ijp}\,^{p}$, so it suffices to consider the first equality of \eqref{alignedtau}. Because $\hnabla$ has trace-free normal curvature, \eqref{projhrijkl} implies $\nabla_{[i}h_{j]k} = -\tau_{[i}h_{j]k}$ so that $(1-n)\tau_{i} = h^{pq}\nabla_{i}h_{pq} - h^{pq}\nabla_{p}h_{qi}$. The connection $\nabla$ induced by a transverse bundle is independent of the choice of transverse vector, and by definition $\nabla$ is aligned if and only if $nh^{pq}\nabla_{p}h_{qi} = h^{pq}\nabla_{i}h_{pq}$, or, what is the same, $n\tau_{i} = -h^{pq}\nabla_{i}h_{pq}$. The first equality of \eqref{alignedtau} just says $\nabla_{i}|\det h|^{-1/n} = |\det h|^{-1/n}\tau_{i}$ which holds if and only if the connection induced on $\dens^{2(n+1)/n}$ by $\nabla$ has connection one-form $\tau$ with respect to the local trivialization of $\dens^{2(n+1)/n}$ given by $|\det h|^{-1/n}$.
\end{proof}

\begin{lemma}\label{tauclosed}
Let $i:M \to (N, \hnabla)$ be a non-degenerate positively co-oriented immersion with affine normal bundle $W$ and suppose $\hnabla$ has trace-free normal curvature. If $\hat{B}_{ij\infty}\,^{\infty} = 0$ (in particular if $\hnabla$ is projectively flat) then $d\tau_{ij} = -4\sh_{[ij]}$. If $\hnabla$ is also Ricci symmetric then $d\tau_{ij} = 0 = \sh_{[ij]}$, so that in this case the induced CP pair is closed. 
\end{lemma}
\begin{proof}
If $\hnabla$ has trace-free normal curvature and $\hat{B}_{ij\infty}\,^{\infty} = 0$ then \eqref{hatbdoublei} and \eqref{alignedtau} imply $\tau_{ij} = -4\sh_{[ij]}$. If $\hnabla$ is also Ricci symmetric, then \eqref{hatprdtau} and \eqref{alignedtau} together imply $d\tau_{ij} = 0 = R_{ijp}\,^{p}$.
\end{proof}

\subsubsection{Equiaffine immersions}\label{equiaffinesection}
A volume density $\Psi$ on $N$ means a section of $|\Det \ctn|$. On a sufficiently small neighborhood $U$ in $N$ of a point $p \in i(M)$, $\Psi$ agrees with a volume form $\bar{\Psi}$, and the interior multiplication of a transversal $\nm$ to $i(M)$ with $\Psi$ can be defined on $i^{-1}(U \cap i(M))$ to be the volume density $|i^{\ast}(i(\nm)\bar{\Psi})|$ obtained by taking the absolute value of the restriction $i^{\ast}(i(\nm)\bar{\Psi})$. Covering a neighborhood of $i(M)$ with such small neighborhoods, it is clear that the resulting densities patch together to give on $M$ a globally defined volume density, the \tbf{restriction to $M$ of $\Psi$}, to be denoted $i^{\ast}(i(\nm)\Psi)$.

\begin{lemma}\label{equiaffinelemma}
Let $i:M \to (N, \hnabla)$ be a co-oriented non-degenerate hypersurface immersion. A  volume density $\Psi$ on $N$ determines on $M$ a unique affine normal vector field $\nm$ consistent with the co-orientation and satisfying the equality $i^{\ast}(i(\nm)\Psi) = \mu_{h}$. If $\Psi$ is $\hnabla$ parallel then this distinguished affine normal is called the \tbf{equiaffine normal} and satisfies $\nabla \mu_{h} = \tau \tensor \mu_{h}$. If $\hnabla$ has trace-free normal curvature (in particular if $\hnabla$ is projectively flat), then for the equiaffine normal vector field $W$, the connection form $\tau$ is identically zero and $\nabla \mu_{h} = 0$.
\end{lemma}

\begin{proof} 
Let $\nm$ be an affine normal consistent with the given co-orientation. Replacing $\nm$ by $\tilde{\nm} = f\nm$ with $f > 0$ rescales $i^{\ast}(i(\nm)\Psi)$ by $f$ and $\mu_{h}$ by $f^{-n/2}$, and so the normalization $i^{\ast}(i(\tilde{\nm})\Psi) = \mu_{\tilde{h}}$ is realized uniquely by $f = (\mu_{h}/i^{\ast}(i(\nm)\Psi))^{2/(n+2)}$. If $\Psi$ is any volume density on $N$ then $\nabla i^{\ast}(i(\nm)\Psi) = i^{\ast}(i(\nm)\hnabla \Psi) + \tau \tensor i^{\ast}(i(\nm)\Psi)$, so that if $\hnabla \Psi = 0$ then $\nabla i^{\ast}(i(\nm)\Psi) =\tau \tensor i^{\ast}(i(\nm)\Psi)$. Hence for the affine normal $\nm$ induced by $\Psi$ there holds $2\tau_{i} = 2\mu_{h}^{-1}\nabla_{i}\mu_{h} = |\det h|^{-1}\nabla_{i}|\det h| = h^{pq}\nabla_{i}h_{pq}$. On the other hand, by construction of the affine normal bundle there holds $n(n-1)\tau_{i} = (1-n)h^{pq}\nabla_{i}h_{pq} + nh^{pq}\hat{R}_{ipq}\,^{\infty}$, so that $(n+2)(n-1)\tau_{i} =  nh^{pq}\hat{R}_{ipq}\,^{\infty}$. If $\hnabla$ has trace-free normal curvature then $h^{pq}\hat{R}_{ipq}\,^{\infty} = 0$, so that $\tau = 0$.
\end{proof}
The metric $h$ induced by the equiaffine normal is called the \textbf{equiaffine metric} or \textbf{Berwald-Blaschke metric} induced by $\hnabla$ and $\Psi$. By construction the equiaffine metric is invariant under volume-preserving automorphisms of $(N, \hnabla, \Psi)$. It is determined up to positive homotheties corresponding to rescaling $\Psi$. More generally, if $\hnabla$ has trace-free normal curvature and symmetric Ricci tensor, then the induced CP pair is exact, and the affine normal vector fields inducing distinguished metrics are called \textbf{equiaffine}. This can be seen as follows; that the Ricci tensor be symmetric means that locally on $N$ there exist parallel volume densities, and on a connected open set two such differ by multiplication by a non-zero real number. The assumed orientation of $\nmb(M)$ and the orientations of $|\Det \ctm|$ and $i^{\ast}(|\Det \ctn|)$ mean that there can be chosen a global non-vanishing section of $i^{\ast}(|\Det \ctn|)$ parallel with respect to the connection induced on $i^{\ast}(|\Det \ctn|)$ by $\hnabla$, and this is all that is needed to make the argument of the proof of Lemma \ref{equiaffinelemma}.

In the event that $N$ is orientable, $i$ is co-oriented, and $\Psi$ is a volume \textit{form}, the second fundamental form induces an identification of $\Psi$ with a non-vanishing section of $\nmb(M)^{-n-2}$, and hence via the given orientation of $\nmb(M)$, an orientation of $W \simeq \nmb(M)$. The resulting co-orientation need not be the positive co-orientation. To be clear: the convention here is that the conformal structure induced on a co-oriented non-degenerate hypersurface is that determined by the positive co-orientation, which co-orientation is determined by a requirement on the signature of the second fundamental form, while a distinguished affine normal vector field is determined by an equality of volume \textit{densities}. Perhaps a better way to describe the normalization as follows. Say two affine normal vector fields are (positively) homothetic if one is a (positive) non-zero real multiple of the other. A positive homothety class of affine normal fields induces a CP pair $(\en, [h])$ with aligned representative $\nabla$ and a positive homothety class of metrics contained in the positive conformal class $[h]$; the equiaffine normalization requires that the induced CP pair be exact and that this positive homothety class of metrics be the positive homothety class of distinguished metrics of the CP pair. This induced CP pair will be called the \tbf{equiaffine CP pair}, or, when $\hnabla$ is projectively flat, the \tbf{equiaffine AH structure}.

\subsubsection{}\label{conormalsection}

The \tbf{conormal Gau\ss\, map} $\gss(i):M \to \proj(\sted)$ of a hypersurface immersion $i:M \to (\ste, \hnabla)$ in flat affine space, is the smooth map to the dual flat affine space $(\sted, \hnabla)$ defined by $\gss(i)(p) = \ann Ti(p)(T_{p}M)$, which associates to $p \in M$ the annihilator of the tangent space to $i(M)$ at $i(p)$. Because all the discussion in this section is local, it will be assumed that $i$ is co-oriented, and then $\gss(i)$ can be regarded as taking values in $\projp(\sted)$; to $p$ is associated the ray in $\ann Ti(p)(T_{p}M)$ consistent with the co-orientation. Equip $\projp(\sted)$ with the flat projective structure induced by the flat affine connection $\hnabla$ on $\sted$ induced by duality. In the interest of readability, the duality pairing $\ste \times \sted \to \rea$ is denoted by $\lb \dum, \dum \ra$. An affine normal field $\nm$ determines a factorization $\gss(i) = \rho \circ \hgss(i)$, in which $\rho:\stedz \to \projp(\sted)$ is the canonical projection and $\hgss(i):M \to \stedz$ is defined by the conditions $\lb \hgss(i)(p), \nm_{i(p)}\ra = 1$ and $\lb \hgss(i)(p), Ti(p)(X)\ra = 0$. Supposing $\nm$ is an equiaffine normal field, differentiating these conditions yields for $X \in \Ga(TM)$
\begin{align}\label{hgssid}
&0 = \lb T\hgss(i)(X), \nm\ra,& &-h(X, Y) = \lb T\hgss(i)(X), Ti(Y)\ra.
\end{align}
(The dependence on the basepoint is suppressed for readability). The map $\hgss(i)$ associated to an arbitary affine normal field will be called the \tbss{lifted}{conormal Gau\ss\, map} associated to the affine normal field $\nm$; when $\nm$ is an equiaffine normal field, then $\hgss(i)$ will be called the \tbss{centroaffine}{Gau\ss\, map}. In what follows $\hgss(i)$ will refer always to a centroaffine Gau\ss\, map; this simplifies computations. Let $\rad$ denote the radial vector field on $\stedz$. Then $\rad_{\hgss{i}(p)}$ is the position vector, which could also be written simply as $\hgss(i)(p)$, with which vector it is identified by parallel translation to the origin. This vector spans $\ker T\rho(\hgss(i)(p))$, and so $T\gss(i)(p)(X) = 0$ if and only if $T\hgss(i)(p)(X) \wedge \rad_{\hgss(i)(p)} = 0$. From the first equality of \eqref{hgssid} it follows that this last condition holds if and only if $T\hgss(i)(p)(X) = 0$, and from the second equality of \eqref{hgssid} it follows that this is equivalent to the vanishing of $h(X, \dum)$. This proves
\begin{lemma}
A hypersurface immersion in flat affine space is non-degenerate if and only if its conormal Gau\ss\, map is an immersion.
\end{lemma}
Now suppose $i$ is non-degenerate. Pullback of the flat projective structure on $\projp(\sted)$ via $\gss(i)$ yields a flat projective structure on $M$. It follows from the preceeding that $\rad$ is transverse to $\hgss(i)(M)$. Let $\tnabla$ and $\tilde{P}$ be the connection and second fundamental form of $\hgss(i)(M)$ induced by the transverse field $\rad$. By definition $\hnabla_{X}T\hgss(i)(Y) = T\hgss(i)(\tnabla_{X}Y) + \tilde{P}(X, Y)\rad_{\hgss(i)}$. Pairing this with $\nm$ and using \eqref{hgssid} gives 
\begin{align*}
\begin{split}
\tilde{P}(X, Y) & = \lb \hnabla_{X}T\hgss(i)(Y), \nm\ra = \hnabla_{X}\lb T\hgss(i)(Y), \nm\ra + \lb T\hgss(i)(Y), Ti(S(X))\ra \\ &= -h(S(X), Y) = -\sh(X, Y).
\end{split}
\end{align*}
This proves
\begin{lemma}\label{scalefreelemma}
The scale-free shape operator of a non-degenerate hypersurface immersion in flat affine space is the negative of the second fundamental form of the centroaffine Gau\ss\, map of the immersion relative to the radial vector field on the dual affine space. 
\end{lemma}
View $\hgss(i)$ as a one-form on $\ste$ and consider the two-tensor $\hnabla \hgss(i)$. From \eqref{hgssid} and the definition of $\hgss(i)$ there follow
\begin{align*}
&\lb \hnabla_{Ti(X)}\hgss(i), Ti(Y)\ra = -\lb \hgss(i), \hnabla_{X}Ti(Y)\ra = -h(X, Y),& &\lb \hnabla_{Ti(X)}\hgss(i), \nm\ra = \lb \hgss(i), Ti(S(X))\ra =0,
\end{align*}
which prove
\begin{lemma}
The pullback via a hypersurface immersion $i$ into flat affine space $\ste$ of the covariant derivative of the centroaffine Gau\ss\, map, considered as a one-form on $\ste$ is the negative of the second fundamental form of $i$ relative to the equiaffine normal field determining the centroaffine Gau\ss\, map.
\end{lemma}
Computing in two ways $\lb \hnabla_{X}T\hgss(i)(Y), Ti(Z)\ra$ yields
\begin{align}\label{conjhg0}
\begin{split}
\lb \hnabla_{X}T\hgss(i)(Y), Ti(Z)\ra & = X \lb T\hgss(i)(Y), Ti(Z)\ra - \lb T\hgss(i)(Y), Ti(\nabla_{X}Z) + h(X, Z)\nm\ra \\ &= -X(h(Y, Z)) + h(Y, \nabla_{X}Z),
\end{split}\\
\label{conjhg}\begin{split}
\lb \hnabla_{X}T\hgss(i)(Y), Ti(Z)\ra & = \lb T\hgss(i)(\tnabla_{X}Y) - \sh(X, Y)\rad_{\hgss(i)}, Ti(Z)\ra = -h(\tnabla_{X}Y, Z),
\end{split}
\end{align}
In \eqref{conjhg} the first equality uses Lemma \ref{scalefreelemma} and the second equality uses the second equality of \eqref{hgssid}. Together \eqref{conjhg0} and \eqref{conjhg} show that $\tnabla$ is the $h$-conjugate connection of $\nabla$. Since $\nabla$ is the aligned representative of the induced AH structure $(\en, [h])$ and $h$ is a distinguished metric, this implies that $\tnabla$ is the aligned representative of the conjugate AH structure $(\ben, [h])$.
\begin{theorem}\label{conormaltheorem}
The flat projective structure induced on $M$ by the conormal Gau\ss\, map of a hypersurface immersion $i$ of $M$ in flat affine space forms with the conformal structure determined by the second fundamental form of $i$ the AH structure conjugate to the AH structure induced on $M$ via the affine normal. The connection induced on $M$ via the centroaffine Gau\ss\, map associated to an equiaffine normal field is the aligned representative of this AH structure, and is the conjugate of the connection induced on $M$ by the equiaffine normal field with respect to the second fundamental form induced by that same equiaffine normal field. 
\end{theorem}

It would be possible to extend the preceeding discussion to the case of an immersion into a projectively flat manifold, but it seems more straightforward and technically easier to work directly with the AH structure conjugate to that induced via the affine normal, which is defined instrinsically.

\subsubsection{}\label{projflathypersurfacesection}
Suppose that $i:M \to (N, \hnabla)$ is a positively co-oriented non-degenerate immersion and $\hnabla$ is projectively flat. Let $\nm^{I}$ be an affine normal vector field inducing on $M$ the metric $h_{ij}$ and the connection $\nabla$. Let $\hat{P}_{IJ}$ be the modified Ricci tensor of $\hnabla$, and let $\hat{P}_{ij}$ and $\hat{P}_{i}$ be the restrictions to $M$ of $\hat{P}_{IJ}$ and $\hat{P}_{IQ}\nm^{Q}$, respectively. Let $H_{ij}$ be the normalized metric of the induced AH structure $(\en, [h])$, and raise and lower indices with $H^{ij}$. Define $\sh = \sh_{p}\,^{p} = H^{ij}\sh_{ij} = |\det h|^{1/n}S_{p}\,^{p}$. From \eqref{projhrijkl} there follows
\begin{align}\label{shcurv}
R_{ijkl} = -2H_{l[i}\hat{P}_{j]k} + 2H_{kl}\hat{P}_{[ij]} - 2H_{k[i}\sh_{j]l},
\end{align}
From Lemmas \ref{alignedtaulemma} and \ref{tauclosed} there follows $F_{ij} = -2\sh_{[ij]}$. Tracing \eqref{shcurv} in two ways, decomposing by symmetries, and using $F_{ij} = -2\sh_{[ij]}$ gives
\begin{align}
\label{shric} & R_{(ij)} = \sh H_{ij} - \sh_{(ij)} + (1-n)\hat{P}_{(ij)}, & &Q_{(ij)} = (n-1)\sh_{(ij)} - \hat{P}H_{ij} + \hat{P}_{(ij)},\\
\label{shscal} &2\hat{P}_{[ij]} = F_{ij} = -2\sh_{[ij]},& & R= (n-1)(\sh - \hat{P}),
\end{align}
in which $\hat{P} \defeq H^{ij}\hat{P}_{ij}$. From \eqref{shric} and \eqref{shscal} there follow
\begin{align}
\label{sheij} &2E_{ij} = \mr{\hat{P} + \sh}_{ij}, & &2A_{ij} = \hat{P}_{(ij)} - \sh_{(ij)},\\
\label{shw}   &W_{ij} = \hat{P}_{(ij)} - \tfrac{1}{2n}(\hat{P} + \sh)H_{ij} + \tfrac{1}{2}F_{ij}, && \bar{W}_{ij} = -\sh_{(ij)} +\tfrac{1}{2n}(\hat{P} + \sh)H_{ij} + \tfrac{1}{2}F_{ij}.
\end{align}
Substituting the first equation of \eqref{shscal} into \eqref{shcurv} and using \eqref{shw} to simplify the result gives
\begin{align}\label{curvprojflathypersurface}
\begin{split}
R_{ijkl} & = -H_{li}\hat{P}_{(jk)} + H_{lj}\hat{P}_{(ik)} - H_{ki}\sh_{(jl)} + H_{kj}\sh_{(il)} + F_{ijkl} + G_{ijkl}\\
         & = -2H_{l[i}W_{j]k} + 2H_{k[i}\bar{W}_{j]l} + F_{ij}H_{kl}.%
\end{split}
\end{align}
With \eqref{rintermsofw} this shows $A_{ijkl} = 0$ and $E_{ijkl} = 0$. If $\hnabla$ is moreover Ricci symmetric, then Lemma \ref{tauclosed} shows $F_{ij} = \hat{P}_{[ij]} = \sh_{[ij]} = d\tau_{ij} = 0$. Differentiating the definition of $\sh_{ij}$ and using \eqref{alignedtau} gives
\begin{align}\label{sh1}
\nabla_{i}\sh_{jk} = h_{pk}\nabla_{i}S_{j}\,^{p} - \tau_{i}\sh_{jk} + \bt_{ik}\,^{p}\sh_{jp},
\end{align}
and skewing \eqref{sh1} and using \eqref{projhrijnk} gives
\begin{align}\label{shnabla}
2\hat{P}_{[i}h_{j]k} = h_{pk}\hat{R}_{ij\infty}\,^{p} = -2\nabla_{[i}\sh_{j]k} + 2\bt_{k[i}\,^{p}\sh_{j]p}.
\end{align}
Tracing \eqref{shnabla} and using \eqref{shscal} gives
\begin{align}\label{shdiv}
(n-1)|\det h|^{1/n}\hat{P}_{i}= - \nabla_{i}\sh + \nabla^{p}\sh_{ip} = \tfrac{1-n}{n}\nabla_{i}\sh + \nabla^{p}\mr{\sh}_{ip} - \tfrac{1}{2}\nabla^{p}F_{ip}.
\end{align}

\begin{lemma}
If $i:M \to (N, \hnabla)$ is a positively co-oriented non-degenerate hypersurface immersion into a flat $(n+1)$-dimensional manifold then the AH structure $(\en, [h])$ induced on $M$ is conjugate projectively flat. If $n > 2$ then $(\en, [h])$ is projectively flat if and only if $\mr{\sh}_{ij} = 0$. If $\mr{\sh}_{ij} = 0$ then $\nabla_{i}\sh = 0$ and $R$ is parallel. If $n = 2$ then $\mr{\sh}_{ij} = 0$ implies that $(\en, [h])$ is projectively flat. 
\end{lemma}
\begin{proof}
The preceeding shows that $A_{ijkl} = 0 = E_{ijkl}$, $F_{ij} = 0$, and \eqref{shw} shows $\mr{W}_{ij} = 0$. If $n > 2$ then by Lemma \ref{projflatahlemma}, $(\en, [h])$ is conjugate projectively flat. Also by Lemma \ref{projflatahlemma}, $(\en, [h])$ is projectively flat if and only if moreover $\mr{\bar{W}}_{ij} = 0$, which by \eqref{shw} is equivalent to $\mr{\sh}_{ij} = 0$. That $\mr{\sh}_{ij} = 0$ implies $\nabla_{i}\sh = 0$ is immediate from \eqref{shdiv}, and with \eqref{shscal} this shows $R$ is parallel. If $n = 2$, computing using \eqref{projcotton}, \eqref{shscal}, \eqref{sheij}, and \eqref{shdiv} shows that $\bar{C}_{i} = 0$ and $C_{i} = \nabla_{i}\sh + \bt_{i}\,^{pq}\mr{\sh}_{pq}$. In particular $(\en, [h])$ is conjugate projectively flat.  It follows that when $n = 2$ the condition $\mr{\sh}_{ij} = 0$ implies $C_{i} = 0$, so $(\en, [h])$ is projectively flat.
\end{proof}

It has been shown that the the AH structure induced on a non-degenerate co-oriented affine hypersurface is necessarily exact and conjugate projectively flat. The distinguished metrics are the equiaffine metrics induced by the possible choices of a determinant function on $\ste$ consistent with the given co-orientation. The scalar curvature $\uR_{h}$ with respect to a distinguished metric is a constant multiple of the trace of the shape operator (the affine mean curvature). The numerical value of this constant has no absolute meaning, but its sign has meaning, as has the comparison of its numerical values at different points.

\begin{theorem}
An AH structure is conjugate projectively flat if and only if it has the property that around each point of $M$ there is an open neighborhood $U$ and a non-degenerate co-oriented immersion $i:U \to \ste$ into flat affine space such that the AH structure via $i$ coincides with the given AH structure.
\end{theorem}
\begin{proof}
If an AH structure is conjugate projectively flat, then it is closed by \eqref{middlebtraces}, and so its restriction to a small enough neighborhood of any point is exact. Either Theorem $8.2$ of \cite{Nomizu-Sasaki} or the results of \cite{Dillen-Nomizu-Vranken} can be used to deduce the claim.
\end{proof}

\subsubsection{}
Let $i:M \to (N, \hnabla)$ be a positively co-oriented non-degenerate hypersurface immersion into a flat $(n+1)$-dimensional manifold. If there is a $\hnabla$-parallel volume form $\Psi$, by Lemma \ref{equiaffinelemma} (the homothety class of) $\Psi$ induces on $M$ the (homothety class of the) equiaffine representative $h_{ij} \in [h]$ which is distinguished by $\nabla_{i}|\det h| = 0$. The difference tensor of the Levi-Civita connection $D$ of $h_{ij}$ with $\nabla$ is $D - \nabla = \tfrac{1}{2}\bt_{ij}\,^{k}$ and, denoting the curvature of $D$ by $\sR_{ijk}\,^{l}$ as in Section \ref{underlyingsection}, from \eqref{confcurvijkl}, \eqref{confric}, \eqref{confscal}, and $ T_{ijkl}  = \sh_{l[i}H_{j]k}   - \sh_{k[i}H_{j]l}$, there follow
\begin{align}
\label{c1}\begin{split}&\sR_{ijk}\,^{p}H_{pl} = H_{k[j}\sh_{i]l}  + H_{l[i}\sh_{j]k}  + \tfrac{1}{4}\bt_{ijkl},\quad %
\sR_{ij}  = \tfrac{\sh}{2} H_{ij} +\tfrac{n-2}{2} \sh_{ij}+ \tfrac{1}{4}\bt_{ij} ,\\
&\sR = h^{ij}\sR_{ij}  =(n-1)h^{pq}S_{pq} + h^{pq}\bt_{pq}.
\end{split}
\end{align}
For an immersion into flat affine space $(\hnabla, \ste)$ these recover equations $3.15$, $3.18$, and $3.19$ of \cite{Calabi-affinelyinvariant} (Calabi's $A_{ij}\,^{k}$ is $\tfrac{1}{2}\bt_{ij}\,^{k}$). 

\subsubsection{}
A non-degenerate immersed hypersurface in flat affine space is a \textbf{proper affine hypersphere} if its affine normal bundles meet in a point and an \tbf{improper affine hypersphere} if the affine normal subspaces are all parallel. The point at which the affine normals of a proper affine hypersphere meet is called its \tbf{center}. If the induced conformal structure is Riemannian, a proper affine hypersphere is said to be \tbf{elliptic} or \tbf{hyperbolic} according to whether its center is on its inside or its outside. An improper affine hypersphere is also called a \tbf{parabolic} affine hypersphere and is said to have \tbf{center at infinity}. For the basic facts about affine hyperspheres good references are \cite{Calabi-affinelyinvariant}, \cite{Calabi-completeaffine}, \cite{Cheng-Yau-affinehyperspheresI}, and \cite{Nomizu-Sasaki}. Theorems \ref{affhypnormals} and \ref{quadrictheorem} are standard.
 
\begin{theorem}\label{affhypnormals}
A non-degenerate co-oriented hypersurface immersion into flat affine space is an affine hypersphere if and only if its shape operator is pure trace. An affine hypersphere is improper if and only if its shape operator vanishes at some point, in which case its shape operator vanishes identically. 
\end{theorem}

\begin{theorem} [Maschke-Pick-Berwald Theorem]\label{quadrictheorem}
The induced AH structure on a non-degenerate co-oriented hypersurface in flat affine space is a Weyl structure if and only if the hypersurface is an open subset of a quadric.
\end{theorem}

\subsubsection{}
The example of affine hypersurfaces gives essential motivation for the definition of the Einstein equations for AH structures. %
\begin{theorem}\label{sphereeinsteintheorem}
Let $n > 2$. For a non-degenerate positively co-oriented hypersurface immersion into flat $(n+1)$-dimensional affine space the following are equivalent:
\begin{enumerate}
\item The image of the immersion is an affine hypersphere.
\item The AH structure induced on the hypersurface is Einstein.
\item The induced AH structure is projectively flat. 
\end{enumerate}
\end{theorem}
\begin{proof}
The equivalence of $(1)$ and $(2)$ follows from Theorem \ref{affhypnormals}, the equations \eqref{shric}, the observation that the scalar curvature of an affine hypersphere is parallel, and the observation that the induced AH structure on a non-degenerate affine hypersurface is closed. When $n  > 2$ the equivalence of $(2)$ and $(3)$ follows from Lemma \ref{projprojflateinsteinlemma}, because the induced AH structure is necessarily conjugate projectively flat. 
\end{proof}

\begin{remark}
If $n = 2$ then the equivalence of $(1)$ and $(2)$ holds, and, because the AH structure induced on an affine hypersurface is closed with parallel weighted scalar curvature and conjugate projectively flat, it follows from \eqref{projcotton} that $C_{i} = -2\nabla^{p}E_{ip}$, so that if the induced AH structure is Einstein, then it is projectively flat. On the other hand, it need not be the case that projective flatness implies $(1)$ and $(2)$.
\end{remark}

In section \ref{cubicformsection} there will be given examples of exact Riemannian signature Einstein AH structures with self-conjugate curvature which are neither projectively nor conjugate projectively flat.

\subsection{Convex flat real projective structures}\label{convexprojectivesection}
In this section it is explained that together the fact that the induced AH structure on an affine hypersphere be Einstein and a theorem of Cheng-Yau imply that on a manifold with a compact convex flat real projective $\en$ there is an Einstein AH structure $(\en, [h])$. This gives a large class of examples of compact Einstein AH structures which are not Einstein Weyl. 

A full development of this interesting example would take more space than is appropriate here, because it would require introducing formalism, for instance as the Thomas connection or standard tractor connection (associated to the canonical Cartan connection) of a projective structure. The approach taken is as naive as possible and some details are only sketched. Complete references are not given, but the surveys \cite{Benoist-survey} and \cite{Loftin-survey} are excellent starting points.

\subsubsection{}
Let $M$ be an $n$-manifold. In what follows $\emf^{\la} \to M$ will denote a line bundle sections of which transform as $-\la/(n+1)$ densities and on which any affine connection induces a covariant derivative. Making this precise requires a digression of several paragraphs. Since in the case in which it mainly will be needed here, the relevant bundle will be trivial (although not canonically so) and will admit an elementary description, the reader could skip to section \ref{projopsection}, thinking of $\emf^{1}$ as something such as the dual of the tautological bundle on the oriented projectivization $\projp(\ste)$ of an $(n+1)$-dimensional vector space $\ste$ or the bundle $|\Det \ctm|^{-1/(n+1)}$ of $-1/(n+1)$ densities. A section of $\emf^{\la}$ will be said to have \textbf{p-weight} $\la$. 

\subsubsection{}\label{markingsection}
Let $\dens \defeq \Det \ctm \setminus \{0\} \to M$ be the bundle of top forms with the zero section deleted, regarded as a $\reat$ principal bundle. Any continuous representation $\si:\reat \times \rea \to \rea$ of $\reat$ on $\rea$ has the form $\si(r) \cdot t = \sile(r)\cdot t \defeq \sign(r)^{\ep}|r|^{\lambda}t$ for some $(\lambda, \epsilon) \in \rea \times \zmodtwo$, in which $\sign(r) \defeq |r|^{-1}r$. Assume given a $\reat$ principal bundle $\F \to M$ and a principal bundle morphism $\Q:\F \to \dens$ such that there is $(\al, \ep) \in \reat \times \zmodtwo$ such that $\Q(u\cdot r) = \si^{\al, \ep}(r)\cdot \Q(u)$ for all $u \in \F$. Given a $\reat$ principal bundle $\F \to M$, let $|\F|$ be the $\reat$ principal bundle obtained by applying to the transition functions defining $\F$ the homomorphism $\si^{1, 0}$. By assumption $\Q$ induces an isomorphism between $|\F|^{\al}$ and $|\dens|$. Let $\emf^{\la, \ep} = \F \times_{\si^{-\al\la/(n+1), \ep}} \rea$ be the associated real line bundle, so that smooth sections $u$ of $\emf^{\la, \ep}$ correspond bijectively to functions $\tilde{u} \in \cinf(\F)$ which have positive homogeneity $\al\lambda(n+1)$ and are even or odd according to the parity of $\ep$. Precisely, $\emf^{\la, \ep}$ is the quotient of $\F \times \rea$ by the equivalence relation $(u, t) \sim (ru, \si^{\al\la/(n+1), \ep}(r)t)$; the equivariant funtion $u \in \cinf(\rho^{-1}(U))$ corresponding to the local section $u \in \Ga(U, \emf^{\la, \ep})$ is defined by $u(x) = [p, \tilde{u}(p)]$ for $x \in M$ and any $p \in \rho^{-1}(x)$. An example of $\F$ which is typical is the defining bundle $\ste \setminus \{0\}$ over the projectivization $\proj(\ste)$ of an $(n+1)$-dimensional vector space, viewed as the bundle of frames in the tautological line bundle $\taut$. The determinant on $\ste$ induces a canonical identification $\Det T^{\ast} \proj(\ste)$ with $\taut^{n+1}$ which gives the required $\Q$ with $\al = n+1$.

By construction $\Q$ induces an identification of $\emf^{\la, 0}$ with $|\Det \ctm|^{-\la/(n+1)}$. Since $\emf^{\la, \ep}$ differs from $\emf^{\la, 0}$ only by topological twisting, this gives sense to saying that sections of $\emf^{\la, \ep}$ transform like $-\la/(n+1)$ densities. An affine connection $\nabla$ induces a principal connection on $\dens$, and the pullback via $\Q$ of the induced principal connection multiplied by $1/\al$ is a principal connection $\be$ on $\F$, so induces on each $\emf^{\la, \ep}$ a covariant derivative, which will be denoted also by $\nabla_{X}u$ for $X \in \Ga(TM)$ and $u \in \Ga(\emf^{\la, \ep})$. By definition the equivariant function corresponding to $\nabla_{X}u$ is the result of applying $d\tilde{u}$ to the $\be$-horizontal lift of $X$ on $\F$. This gives sense to the statement that an affine connection induces a covariant derivative on $\emf^{\al, \ep}$.

In what follows there will be supposed fixed $\F$, $\Q$, and $\ep$, and the last will be omitted from the notation, there being written simply $\emf^{\la}$. Saying that $\emf^{1}$ is oriented means that $\emf^{1, \ep}$ is orientable and there is fixed an idenfitication with $\emf^{1, 0} \simeq |\Det \ctm|^{-1/(n+1)}$. In most of the applications there will be considered only $\emf^{1}$, and it will be oriented, but some of the results make sense in more generality.

\subsubsection{}\label{projopsection}
A differential operator defined by a connection representing a projective structure $\en$ is \textbf{projectively invariant} if the resulting operator does not depend on the choice of representative connection. It is straightforward to check the projective invariance of the following differential operators on tensors having the specified p-weights
\begin{align*}
\B_{1}(u)_{ij} & \defeq \nabla_{i}\nabla_{j}u - P_{ij}u = \nabla_{(i}\nabla_{j)}u - P_{(ij)}u& &u \in \Ga(\emf^{1}),\\
\B_{1}^{1}(a)_{ijk} &\defeq \nabla_{[i}a_{j]k}& &a_{ij}\in \Ga(S^{2}(\ctm)\tensor \emf^{1}),\\
\B_{2}(v)_{ijk} & \defeq \nabla_{i}\nabla_{j}\nabla_{k}v - 2P_{jk}\nabla_{i}v - 2P_{i(j}\nabla_{k)}v - 2(\nabla_{i}P_{jk})v,& &v \in\Ga( \emf^{2}).
\end{align*}
From the Ricci identity there follows 
\begin{align}\label{bijskewk}
&\B_{2}(v)_{[ij]k} = -\tfrac{1}{2}B_{ijk}\,^{p}\nabla_{p}v - C_{ijk}v, &\B_{2}(v)_{i[jk]} = 0,& 
\end{align}
In particular, \eqref{bijskewk} shows that for a flat projective structure $\B_{2}(v)_{ijk}$ is completely symmetric. For any torsion-free $\bnabla \in \en$ and any $u \in \Ga(\emf^{1})$ there hold
\begin{align}\label{precubu}
&\tfrac{1}{2}\B_{2}(u^{2})_{ijk} = u\bnabla_{i}\B_{1}(u)_{jk} +\B_{1}(u)_{jk} \bnabla_{i}u + \B_{1}(u)_{ik} \bnabla_{j}u+ \B_{1}(u)_{ij}\bnabla_{k}u,\\
\label{obstruct1}
&\B^{1}_{1}(\B_{1}(u))_{ijk} = \bnabla_{[i}\B_{1}(u)_{j]k}= -\tfrac{1}{2}B_{ijk}\,^{p}\bnabla_{p}u - \tfrac{1}{2}C_{ijk}u.
\end{align}
While identities such as \eqref{bijskewk}, \eqref{precubu}, and \eqref{obstruct1} can be demonstrated by direct computation, they are better understood by working with the Thomas connection or, what is equivalent, the standard tractor connection); for example solutions of $\B_{2}(u) = 0$ correspond to parallel tractors. For these formalisms see for instance \cite{Bailey-Eastwood-Gover} and \cite{Cap-Gover}. The definitions of $\B_{1}$, $\B^{1}_{1}$, and $\B_{2}$ may appear \textit{ad hoc}; in fact $B_{1}$ and $\B_{1}^{1}$ are the first two operators appearing in the generalized BGG sequences associated to the standard representation of $\sll(n+1, \rea)$ and $\B_{2}$ is the first operator in the BGG sequence for a different representation. For what are BGG sequences, consult \cite{Cap-Slovak-Soucek} and \cite{Calderbank-Diemer}. From \eqref{obstruct1} it is evident that $\B^{1}_{1}(\B_{1}(u))_{ijk} = 0$ is flat, which is a special case of a general fact about BGG sequences. However, while such points of view are powerful and conceptually clarifying, they are omitted because to adequately develop the necessary formalism would require more preparation than would be justified by the limited use which would be made of them here.

\subsubsection{}
A $u \in \Ga(\emf^{1})$ is \textbf{non-degenerate} (with respect to $\en$) if it satisfies
\begin{align}\label{ubu}
&\det \B_{1}(u) \neq 0.
\end{align}
A non-degenerate $u$ is \textbf{convex} (resp. \textbf{concave} or \textbf{indefinite}) if $\B_{1}(u)$ is positive definite (resp. negative definite or indefinite). Note that the definition permits that a non-degenerate $u$ have zeroes, although it will turn out that in many cases of interest such a $u$ does not. For $u \in \Ga(\emf^{1})$, $\det \B_{1}(u)$ has p-weight $-n-2$, and so $u^{n+2}\det \B_{1}(u)$ is a function (even when $\emf^{1}$ is topologically non-trivial), and hence for any $f \in \cinf(M)$ it makes sense to consider the projectively invariant Monge-Amp\'ere equation 
\begin{align}
\monge(u) \defeq u^{n+2}\det \B_{1}(u) = f.
\end{align}
An equation of this sort was probably first studied as such in \cite{Loewner-Nirenberg}. Of particular interest are non-degenerate solutions to the equation $\monge(u) = \ka$ with $\ka \in \reat$.

The operator $\monge(u)$ is usually encountered in the following setting. There is given a convex domain $\Omega$ in an affine subspace of $\proj(\ste)$, and $\pr$ is the flat affine connection on this affine subspace representing the flat projective structure on $\proj(\ste)$. With respect to flat affine coordinates $x^{i}$ the operator $\B_{1}(u)$ is given by the Hessian $\tfrac{\pr^{2}\tilde{u}}{\pr x^{i} \pr x^{j}}$ where $u = \tilde{u}|dx^{1}\wedge \dots dx^{n}|^{-1/(n+1)}$. Written in terms of $\tilde{u}$ and a flat affine connection on $\rea^{n}$, the equation $\monge(u) = \ka$ has been studied intensively, some of the most important references being \cite{Loewner-Nirenberg}, \cite{Jorgens}, \cite{Pogorelov}, \cite{Calabi-improper}, \cite{Calabi-completeaffine}, and \cite{Cheng-Yau-mongeampere}.

In what follows it is useful to keep in mind the simplest examples, among which are the following
\begin{align}
\label{ballexample}&\tilde{u} = -(1 - |x|^{2})^{1/2} \,\,\text{ on the ball} \,\, \Omega = \{x \in \rea^{n}: u < 0\},\\
\label{sphereexample}&\tilde{u} = (1 + |x|^{2})^{1/2} \,\,\text{ on}\,\, \rea^{n},\\
\label{orthantexample}&\tilde{u} = -\sqrt{n+1}(\prod_{i}x^{i})^{1/(n+1)} \,\,\text{on the orthant}\,\,\{x \in \rea^{n}: \prod_{i}x^{i} > 0\} .
\end{align}
Examples \eqref{ballexample} and \eqref{orthantexample} solve $\monge(u) = (-1)^{n+2}$, while \eqref{sphereexample} solves $\monge(u) = 1$. The last example is due to E. Calabi, \cite{Calabi-completeaffine}.

\subsubsection{}
Since it makes sense to speak of a metric taking values in an oriented line bundle as a representative of a conformal structure, if $\emf^{1}$ is oriented and $u \in \Ga(\emf^{1})$ is non-degenereate, it makes sense to speak of the conformal struture generated by the p-weight $1$ tensor $\B_{1}(u)_{ij}$. For this conformal structure the unweighted metric $h_{ij} = |\det \B_{1}(u)|^{1/(n+2)}\B_{1}(u)_{ij}$ is a representative and $H_{ij} = |\det \B_{1}(u)|^{-1/n}\B_{1}(u)_{ij}$ is the normalized representative. To $u$ which is moreover nowhere vanishing there is associated also the unweighted metric $\sign(u)u^{-1}\B_{1}(u)_{ij} = |\monge(u)|^{-1/(n+2)}h_{ij}$.

\begin{lemma}\label{buahlemma}
Suppose $\emf^{1}$ is oriented. If a non-degenerate $u \in \Ga(\emf^{1})$ solves $\B^{1}_{1}(\B_{1}(u))_{ijk} = \si_{[i}\B_{1}(u)_{j]k}$ for some one-form $\si_{i}$ then the CP pair $(\en, [h])$ determined by $\en$ and the conformal structure $[h]$ generated by the weighted metric $\B_{1}(u)_{ij}$ is an AH structure for which the Faraday primitive $\ga_{i}$ associated to the metric $h_{ij} = |\det \B_{1}(u)|^{1/(n+2)}\B_{1}(u)_{ij}$ satisfies $(n+2)\ga_{i} =(n+1)\si_{i}$. In particular, if $\si_{i}$ is $0$ then $(\en, [h])$ is exact, $h_{ij}$ is a distinguished metric, and $|\det \B_{1}(u)|$ is parallel.
\end{lemma}

In particular if $\en$ is flat, by \eqref{obstruct1} the hypothesis of Lemma \ref{buahlemma} holds with $\ga \equiv 0$ for any non-degenerate $u \in \Ga(\emf^{1})$.

\begin{proof}
By \eqref{obstruct1} the hypothesis implies $\bnabla_{[i}\B_{1}(u)_{j]k}= \si_{[i}\B_{1}(u)_{j]k}$ for any $\bnabla \in \en$ and so $n\bnabla_{[i}H_{j]k} = -(n+2)\tau_{[i}H_{j]k}$ for $(n+2)\tau_{i} = \bnabla_{i}\log |\det \B^{1}(u)| - n\si_{i}$. This implies that $\en$ generates with the conformal structure determined by $h = |\det \B_{1}(u)|^{1/(n+2)}\B_{1}(u)_{ij}$ an AH structure for which the aligned representative $\nabla \in \en$ has the form $\nabla = \bnabla + 2\tau_{(i}\delta_{j)}\,^{k}$. By construction $\nabla$ satisfies $\nabla_{i}|\det \B_{1}(u)| = \bnabla_{i}|\det \B_{1}(u)| - (n+2)\tau_{i}|\det \B_{1}(u)| = n\si_{i}$, and since $|\det h| = |\det \B_{1}(u)|^{2(n+1)/(n+2)}$, the Faraday primitive $\ga_{i}$ associated to $h$ is given by $2n \ga_{i} = |\det h|^{-1}\nabla_{i}|\det h| = \tfrac{2n(n+1)}{(n+2)} \si_{i}$, so that $(n+2)\ga_{i} = (n+1)\si_{i}$. If $\si_{i} = 0$ then $(\en, [h])$ is exact and $h_{ij}$ is a distinguished metric. %
\end{proof}

It is suggestive to think of $u$ as in Lemma \ref{buahlemma} as an analogue of a K\"ahler potential, the non-degeneracy of $\B_{1}(u)_{ij}$ being analogous to the condition that the K\"ahler form be symplectic, and the condition $\B^{1}_{1}(\B_{1}(u))_{ijk} = 0$ being analogous to the condition that it be closed. In this analogy, projective structures correspond to almost complex structures, and flat projective structures to complex structures. Whatever the worth of this analogy, it motivates calling $u \in \Ga(\emf^{1})$ an \textbf{almost AH potential} if $\B^{1}_{1}(\B_{1}(u))_{ijk} = 0$, and an \textbf{AH potential} if moreover $u$ is non-degenerate.

\subsubsection{}
Given $\la \in \rea$ let $\hat{\la} = (n-1)\la - 2(n+1)$. There is a canonical map $\adj:\Ga(S^{2}(\ctm)\tensor \emf^{\la}) \to \Ga(S^{2}(TM) \tensor \emf^{\hat{\la}})$ determined uniquely by the requirement that there hold $\adj(a)^{jp}a_{pi} = (\det a)\delta_{i}\,^{j}$ for all $a_{ij} \in \Ga(S^{2}(\ctm)\tensor \emf^{\la})$. The tensor $\adj(a)^{ij}$ is called the \textbf{adjugate tensor} of $a_{ij}$. When $a_{ij}$ is non-degenerate the weighted symmetric bivector $a^{ij}$ defined by $a^{jp}a_{ip} = \delta_{i}\,^{j}$ is given by $a^{ij} = (\det a)^{-1}\adj(a)^{ij}$. The adjugate appears most naturally in the computation of the derivative of a determinant. If $\nabla$ is a torsion-free affine connection then $\nabla_{i}\det a = \adj(a)^{pq}\nabla_{i}a_{pq}$.

Theorem \ref{butheorem} is the principal technical result of this section.

\begin{theorem}\label{butheorem}
Let $\en$ be a projective structure on a connected manifold and let $u$ be a section of $\emf^{1}$ which is not identically zero. The following conditions are equivalent.
\begin{enumerate}
\item There is a $\kappa \in \rea$ such that $u$ solves $\monge(u) = \ka$.
\item $u^{n+1}\adj(\B_{1}(u))^{pq}\B_{2}(u^{2})_{ipq} =0 $.
\end{enumerate}
If $u$ is an almost AH potential (e.g. $\en$ is flat) then $(1)$ and $(2)$ are equivalent to
\begin{enumerate}
\setcounter{enumi}{2}
\item  $u^{n+1}\adj(\B_{1}(u))^{pq}\B_{2}(u^{2})_{(ipq)} = 0$.
\end{enumerate}
If $\emf^{1}$ is oriented and $u$ is an AH potential, then $(1)$ and $(2)$ are equivalent to 
\begin{enumerate}
\setcounter{enumi}{3}
\item There holds $\nabla u = 0$ for the unique representative $\nabla \in \en$ aligned with respect to the conformal structure $[h]$ generated by $\B_{1}(u)_{ij}$,
\end{enumerate}
If $\emf^{1}$ is oriented, $\en$ is flat, and $u$ is non-degenerate and satisfies $(1)$ and $(2)$ then $u$ is non-vanishing, $\ka \neq 0$, and $(\en, [h])$ is proper Einstein with scalar curvature having the same sign as has $u$.
\end{theorem}

\begin{proof}
Tracing \eqref{precubu} with the adjugate tensor $\adj(\B_{1}(u))^{jk}$ gives
\begin{align}\label{nablaimonge}
&\tfrac{1}{2}u^{n+1}\adj(\B_{1}(u))^{pq}\B_{2}(u^{2})_{ipq} = u^{n+2}\bnabla_{i}\det \B_{1}(u) +(n+2)u^{n+1}\det \B_{1}(u)\bnabla_{i}u = \bnabla_{i}\monge(u),
\end{align}
from which the equivalence of $(1)$ and $(2)$ is evident. Skewing \eqref{precubu} and using \eqref{obstruct1} show
\begin{align*}
&\B_{2}(u^{2})_{[ij]k} =  2u\nabla_{[i}\B_{1}(u)_{j]k} = -uB_{ijk}\,^{p}\nabla_{p}u - C_{ijk}u^{2} = 2u\B^{1}_{1}(\B_{1}(u))_{ijk},
\end{align*}
so that if $u$ is an almost AH potential then $\B_{2}(u^{2})_{[ij]k} = 0$, and there holds $\B_{2}(u^{2})_{ijk} = \B_{2}(u^{2})_{(ijk)}$. Consequently $(2)$ and $(3)$ are equivalent if $u$ is an almost AH potential. If $\emf^{1}$ is oriented and $u$ is an AH potential, let $(\en, [h])$ be the exact AH structure determined by $\B^{1}(u)_{ij}$ as in Lemma \ref{buahlemma}, and let $\nabla$ be its aligned representative. By Lemma \ref{buahlemma}, $\nabla \det \B_{1}(u) = 0$. From the first equality of \eqref{nablaimonge} it follows that $\nabla u = 0$ if and only if there holds $(2)$, so with these hypotheses $(4)$ and $(2)$ are equivalent. 

If $\emf^{1}$ is oriented, $\en$ is flat, and $u$ is non-degenerate then $(4)$ and the assumption that $u$ is not identically zero imply $u$ is non-vanishing. Since $u$ is non-vanishing and non-degenerate, $(1)$ implies that $\ka \neq 0$. Because $(\en, [h])$ is exact, the Ricci curvature $R_{ij}$ is symmetric. By $(4)$, $(n-1)R_{ij}u = -P_{ij}u = \B_{1}(u)_{ij}$, from which $\mr{R}_{ij} = 0$ is immediate. Because $\en$ is projectively flat it follows from \eqref{middlebtraces} that $\mr{Q}_{ij} = 0$. The weighted scalar curvature is $R = H^{ij}R_{ij} = \tfrac{1}{n-1}u^{-1}H^{ij}h_{ij} = \tfrac{n}{n-1}u^{-1}|\det \B_{1}(u)|^{1/n}$, which is $\nabla$-parallel. Since $(\en, [h])$ is exact this shows it is Einstein. Because $u$ is non-vanishing and non-degenerate, it is proper, and from the explicit expression for $R$ it is evident that its scalar curvature has the same sign as has $u$.
 \end{proof}

Theorem \ref{butheorem} shows that proper Einstein AH structures can be constructed by finding a non-degenerate solution to $\monge(u) = \ka$. Lemma \ref{buuniqlemma} shows that when it exists the solution is unique up to positive rescaling.

\begin{lemma}\label{buuniqlemma}
Let $\en$ be a projective structure on a connected manifold and suppose $\emf^{1}$ is oriented. If AH potentials $u$ and $v$ generate conformally equivalent metrics $\B_{1}(u)$ and $\B_{1}(v)$ and solve $\monge(u) = \ka_{u}$ and $\monge(v) = \ka_{v}$ for constants $\ka_{u}, \ka_{v} \in \reat$, then there is $c \in \reap$ such that $v = cu$ and $\ka_{v} = c^{2(n+1)}\ka_{u}$.
\end{lemma}
\begin{proof}
Because $u$ and $v$ solve $\monge(u) = \ka_{u}$ and $\monge(v) = \ka_{v}$ they are nowhere vanishing. By $(4)$ of Theorem \ref{butheorem}, $\nabla u = 0 = \nabla v$, so there is $c \in \reat$ such that $v = cu$. Since $\B_{1}(v) = c\B_{1}(u)$ is by assumption in the conformal class of $\B_{1}(u)$, it must be $c$ is positive. Then $c^{-n-2}\ka_{v}u^{-n-2} = \ka_{v}v^{-n-2} = \det \B_{1}(v) = c^{n}\det \B_{1}(u) = c^{n}\ka_{u}u^{-n-2}$, from which there follows $\ka_{v} = c^{2(n+1)}\ka_{u}$.
\end{proof}

\begin{theorem}\label{potentialtheorem}
Let $(\en, [h])$ be a proper closed Einstein AH structure with weighted scalar curvature $R$ on a connected manifold $M$. There is a non-vanishing non-degenerate section $u \in |\Det \ctm|^{-1/(n+1)}$, unique up to positive homothety, such that $\sign(R)u^{-1}\B_{1}(u)$ generates $[h]$, and $u$ satisfies $(1)$, $(2)$, and $(4)$ of Theorem \ref{butheorem}. 
\end{theorem}

\begin{proof}
By definition a closed Einstein AH structure has parallel scalar curvature, so a proper closed Einstein AH structure is necessarily exact with non-vanishing parallel weighted scalar curvature $R$. The desired $u$ is $u = |R|^{-n/2(n+1)}$. 
Since $\nabla_{i}u = 0$, there holds $n(n-1)\B_{1}(u)_{ij} = n(1-n)P_{ij}u = nuR_{ij} = uRH_{ij}$, which shows that $\sign(R)u^{-1}\B_{1}(u)_{ij}$ generates the given conformal structure. As $n(n-1)\nabla_{[i}\B_{1}(u)_{j]k} = uR\nabla_{[i}H_{j]k} = 0$ and $u$ is evidently non-degenerate, $u$ is an AH potential. As $(4)$ of Theorem \ref{butheorem} is satisfied, and $u$ is an AH potential, there hold $(1)$ and $(2)$ of Theorem \ref{butheorem}. The uniqueness up to positive homothety follows from Lemma \ref{buuniqlemma}.
\end{proof}

\subsubsection{}
For background on convex subsets of $\projp(\ste)$ and convex projective structures, Y. Benoist's survey \cite{Benoist-survey} is an excellent starting point, and the terminology used there is adopted here. A subset $\Sigma \subset\projp(\ste)$ is \textbf{convex} if its intersection with every geodesic line (great circle) is connected. A convex subset $\Sigma \subset \projp(\ste)$ is \textbf{properly convex} if its closure contains no pair of antipodal points, and a properly convex $\Sigma$ is \textbf{strictly convex} if its boundary does not contain a non-empty open subset of a geodesic line. Finally, a properly convex subset $\Sigma \subset \projp(\ste)$ is \textbf{divisible} if the group of projective automorphisms of $\Sigma$ contains a discrete subgroup which acts on $\Sigma$ properly and cocompactly. An open properly convex set is contained in the complement of its image under the antipodal map from which it follows that it can be identified with a bounded convex domain in an affine space.

The restriction to a properly convex set $\Omega \subset \projp(\ste)$ of the dual of the tautological line bundle is an orientable line bundle $\emf^{1}$ the $(n+1)$st power of which is identified canonically with $\Det T \Omega$. The following deep theorem of Cheng-Yau is fundamental.

\begin{theorem}[Cheng-Yau, Theorem $6$ of \cite{Cheng-Yau-mongeampere}]\label{cyconvextheorem}
Let $\Omega \subset \projp(\ste)$ be a properly convex domain in the oriented projectivization of an $(n+1)$-dimensional vector space $\ste$, equipped with the restriction $\en$ of the standard flat projective structure. Then there is a unique negative, convex $u \in \Ga(\Omega, \emf^{1})$ vanishing continuously on the boundary of $\Omega$ and solving $\monge(u) = (-1)^{n}$.
\end{theorem}

A locally flat projective structure on $M$ is \textbf{convex} if its developing map $\dev:\tilde{M} \to \projp(\ste)$ is a diffeomorphism from its universal cover $\tilde{M}$ onto a properly convex subset of $\projp(\ste)$.

\begin{theorem}\label{convextheorem}
Let $\en$ be a convex flat projective structure on a connected $n$-manifold $M$. Then there exists a unique conformal structure $[h]$ such that $(\en, [h])$ is an exact Riemannian signature Einstein AH structure with negative scalar curvature.
\end{theorem}
\begin{proof}
By assumption $\dev:\tilde{M} \to \projp(\ste)$ is a diffeomorphism onto a properly convex domain $\Omega \in \proj(\ste)$. By Theorem \ref{cyconvextheorem} there is a unique negative, convex $u \in \Ga(\Omega, \emf^{1})$ solving $\monge(u) = (-1)^{n}$. For a deck transformation $g$ of $\tilde{M}$ there holds $\dev \circ g = \hol(g)\circ \dev$, in which $\hol$ is the holonomy representation. Since $\hol(g)$ is a projective automorphism of $\Omega$, the uniqueness in Theorem \ref{cyconvextheorem} implies $\hol(g)^{\ast}(u) = u$, so that $g^{\ast}\dev^{\ast}(u) = \dev^{\ast}\hol(g)^{\ast}(u) = \dev^{\ast}(u)$. The projective structure on $\tilde{M}$ is the pullback of the flat projective structure on $\Omega$ via the developing map, and it follows that $\tilde{v} = \dev^{\ast}(u)$ solves the equation $\monge(\tilde{v}) = (-1)^{n}$ on $\tilde{M}$ and is left unchanged by deck transformations. Hence it descends to a solution $v$ of $\monge(v) = (-1)^{n}$ for the given projective structure on $M$. By Theorem \ref{butheorem} the conformal class $[h]$ of $\B_{1}(v)$ determines with the flat projective structure on $\Omega$ an exact Riemannian signature Einstein AH structure with negative scalar curvature. 

If $(\en, [h])$ is an exact Einstein AH structure having negative scalar curvature and with the given underlying convex flat projective structure $\en$, then by Theorem \ref{potentialtheorem} there is a non-vanishing and non-degenerate $w \in \Ga(M;\emf^{1})$ such that $\B_{1}(w)$ represents $[h]$ and such that $w$ solves $\monge(w) = \ka$ for some $\ka \in \reat$ having the same sign as has $w$. By Lemma \ref{buuniqlemma} some positive constant multiple of $W$ is equal to the $v$ constructed in the previous paragraph.
\end{proof}

An interesting question, the answer to which is not clear when $n > 2$, is the following: let $M$ be a compact manifold admitting at least one convex flat projective structure; for which conformal structures $[h]$ on $M$ is there a convex flat projective structure $\en$ such that $(\en, [h])$ is an Einstein AH structure? For $n = 2$ a complete answer is given by the theorem of Labourie-Loftin mentioned in section \ref{labourieloftinsection}. However, the situation in higher dimensions is unclear.

\subsubsection{}
Note that in deducing Theorem \ref{convextheorem}, while essential use has been made of the solvability of the Monge-Amp\`ere equation on a properly convex domain, no use has been made of the deep results, also due in essence to Cheng-Yau, relating the solution to a properly embedded hyperbolic affine hypersphere. The existence of the Einstein AH structure could also be deduced via this correspondence, since it has been shown that the conjugate AH structures induced on such an affine hypersphere are Einstein. However, the preceeding intrinsic approach has the virtue that it deduces Theorem \ref{convextheorem} directly from Theorem \ref{cyconvextheorem}, and does not require use of results relating the completeness of the affine metric and the properness of the immersion of an affine hypersphere. Here, the relation with affine hyperspheres is briefly recounted in order to clarify the remark just made, and to motivate the attention paid to affine hyperspheres in Section \ref{affinehyperspheresection}. The formulations given follow those in the papers of Loftin, of which \cite{Loftin-survey} is a suitable summary.

\begin{theorem}\label{affinehyperspheretheorem} 
Let $\ste$ be flat $(n+1)$-dimensional affine space with its standard parallel volume form.
\begin{enumerate}
\item In $\ste$ let $K$ be an open convex cone which contains no affine lines. For each negative real number $-r < 0$ there is a unique convex properly embedded hyperbolic affine hypersphere $L_{r}$ of affine mean curvature $-r$ which has as its center the vertex of $K$ and is asymptotic to the boundary $\pr K$ and $\textit{int}\, K = \cup_{r > 0}L_{r}$. 
\item An immersed hyperbolic affine hypersphere $L$ in $\ste$ is properly immersed if and only if the equiaffine metric is complete. In this case $L$ is properly embedded and asymptotic to the boundary of the cone obtained by forming the convex hull of $L$ with respect to its center.
\end{enumerate}
\end{theorem}
Theorem \ref{affinehyperspheretheorem} was conjectured in a quite precise form by E. Calabi in \cite{Calabi-completeaffine}. The theorem has two difficult parts. The affine hypersphere of part $(1)$ is constructed using the solution of the Monge-Amp\`ere equation yielding Theorem \ref{cyconvextheorem}. The properness of the embedding in part $(2)$ is deduced from the completeness of the equiaffine metric via a gradient estimate. The resolution of each is due to Cheng-Yau, although the proof they published of the second part has a small mistake corrected by  A-M. Li \cite{Li-1}, \cite{Li-2}. Similarly, the verification that the affine hypersphere constructed using Theorem \ref{cyconvextheorem} has all the claimed properties seems never to have been published by Cheng and Yau, and was described in detail independently by T. Sasaki \cite{Sasaki} and S. Gigena \cite{Gigena}. Consult \cite{Loftin-survey} for a more complete discussion of the history of Theorem \ref{affinehyperspheretheorem}. 

For a convex flat projective structure on $M$, the universal cover is identified with a properly convex domain $\Omega$, and there is induced on the affine hypersphere asymptotic to the cone over $\Omega$ given by Theorem \ref{affinehyperspheretheorem}, and hence also on $\Omega$, a pair of conjugate Einstein AH structures, $(\en, [h])$ and $(\ben, [h])$, the former induced via the affine normal and the latter induced via the conormal Gau\ss map. These evidently descend to $M$ by an argument like that in the proof of Theorem \ref{convextheorem}. It is not hard to see that it is the AH structure induced on $M$ by $(\ben, [h])$ that coincides with that constructed in Theorem \ref{convextheorem}, although this will not be explained in detail.

\subsubsection{}
Because a projectively flat Einstein AH structure is closed, it necessarily has parallel scalar curvature, so if $\emf^{1}$ is orientable, Theorem \ref{potentialtheorem} and Lemma \ref{buuniqlemma} imply that a proper projectively flat Einstein AH structure determines a non-degenerate $u \in \emf^{1}$, unique up to positive homothety, satisfying $(1)$-$(4)$ of Theorem \ref{butheorem} with $\ka \neq 0$. 

\begin{lemma}\label{linearlemma}
Let $(\en, [h])$ be a projectively flat proper Einstein AH structure with weighted scalar curvature $R$ on a connected $n$-manifold $M$ and suppose $\emf^{1}$ is orientable. If there is $v \in \Ga(\emf^{1})$ such that $\B_{1}(v) = 0$ then for $u$ solving $\monge(u) = \ka \in \reat$ the function $f \defeq u^{-1}v$ solves $(1-n)\lap_{h}f = \uR_{h}f$, in which $h$ is a distinguished metric.
\end{lemma}

\begin{proof}
Because $u$ is non-vanishing, $f$ is defined. By $(3)$ of Theorem \ref{potentialtheorem} the aligned representative $\nabla \in \en$ satisfies $\nabla u = 0$. Hence $df_{i} = \nabla_{i}(u^{-1}v) = u^{-1}\nabla_{i}v$. The Levi-Civita connection of $h$ is $D = \nabla + \tfrac{1}{2}\bt_{ij}\,^{k}$ so 
\begin{align*}\begin{split}
D_{i}df_{j} &= \nabla_{i}df_{j} - \tfrac{1}{2}\bt_{ij}\,^{p}df_{p} = u^{-1}\left(\nabla_{i}\nabla_{j}v - \tfrac{1}{2}\bt_{ij}\,^{p}\nabla_{p}v\right) \\
& = u^{-1}\left(P_{ij}v - \tfrac{1}{2}\bt_{ij}\,^{p}\nabla_{p}v\right) = -u^{-1}\left(\tfrac{1}{n(n-1)}\uR_{h}vh_{ij} + \tfrac{1}{2}\bt_{ij}\,^{p}\nabla_{p}v\right). 
\end{split}
\end{align*}
Hence $(1-n)\lap_{h}f = \uR_{h}f$.
\end{proof}

Let $u$ solve $\monge(u) = \ka$ on the properly convex domain $\Omega$ and let $x^{i}$ be flat affine coordinates. Taking as $v$ any one of the densities $|dx^{1}\wedge \dots \wedge dx^{n}|^{-1/(n+1)}$ and $x^{i}|dx^{1}\wedge \dots \wedge dx^{n}|^{-1/(n+1)}$ and writing $u = \tilde{u}|dx^{1}\wedge \dots \wedge dx^{n}|^{-1/(n+1)}$ there result the $(n+1)$ functions $\tilde{u}^{-1}$ and $\tilde{u}^{-1}x^{i}$. The image of these $(n+1)$ functions describes a non-degenerate hypersurface in $\ste$ called the \textbf{radial graph} of $u$. For instance, the radial graph of example \ref{ballexample} is one of the sheets of the two-sheeted hyperboloid for the standard Lorentz metric. Lemma \ref{linearlemma} shows that the position vector of the radial graph has the property that its Laplacian with respect to the metric $h$ is a particular positive multiple of itself. As the Laplacian in the equiaffine metric of the position vector of a non-denerate hypersurface is $n$ times the affine normal vector field (see Theorem $6.5$ of \cite{Nomizu-Sasaki}; this is one way to define the affine normal), this last property shows that the radial graph is an affine hypersphere. In fact, the radial graph is the affine hypersphere associated to the cone over $\Omega$ as in $(1)$ of Theorem \ref{affinehyperspheretheorem}, although justifying this is left to the reader, as it is not germane to the main thrust of the present article. To conclude the section there is mentioned an easy consequence of Lemma \ref{linearlemma}.

\begin{corollary}\label{noparalleltractorscorollary}
Let $\en$ be a convex flat real projective structure on a compact $n$-manifold. If $v \in \Ga(\emf^{1})$ solves $\B_{1}(v) = 0$, then $v$ is identically $0$.
\end{corollary}
\begin{proof}
Let $(\en, [h])$ be the exact Riemannian signature Einstein AH structure with negative scalar curvature given by Theorem \ref{convextheorem}, and let $u$ be the solution of $\monge(u) = \ka$ given by Theorem \ref{butheorem}, which is non-vanishing. If $\B_{1}(v) = 0$ then by Lemma \ref{linearlemma} the function $f = u^{-1}v$ solves $(1-n)\lap_{h}f = \uR_{h}f$, where $h \in [h]$ is a distinguished metric. Since $\uR_{h}$ is negative the maximum principle implies $f$ is identically $0$.
\end{proof}

Although the claim will not be proved here, Corollary \ref{noparalleltractorscorollary} means that over a convex flat projective manifold there are no parallel standard tractors.

\subsection{Left invariant Einstein AH structures on semisimple Lie groups}\label{leftinvariantsection}
In this section there are constructed on $S^{3}$ and on all noncompact semisimple Lie groups examples of left-invariant proper Einstein AH structures for which the anti-self-conjugate Weyl tensor does not vanish, and so which, by Lemma \ref{projflatahlemma}, are neither projectively nor conjugate projectively flat.

\subsubsection{}
This paragraph recalls some well known facts about left-invariant affine connections on Lie groups. Proofs can be found in chapter II of \cite{Helgason}. Let $G$ be an $n$-dimensional semisimple Lie group with Lie algebra $\g$. Let $L_{g}$ denote left multiplication by $g \in G$. Define for each $a \in \g$ a left invariant vector field $\lf^{a}$ by $\lf^{a}_{g} = \tfrac{d}{dt}_{|t = 0}g\cdot \exp(ta)$. The map $a \to \lf^{a}$ is a Lie algebra homomorphism from $\g$ to $\vect(G)$. 
An affine connection $\nabla$ on $G$ is \textbf{left-invariant} if every left translation $L_{g}$ is a $\nabla$-affine transformation, i.e. $L_{g}^{\ast}(\nabla) = \nabla$ for all $g \in G$. For $A \in \g^{\ast} \tensor \g^{\ast} \tensor \g$ the affine connection $\nabla$ defined by $\nabla_{\lf^{a}}\lf^{b} = \lf^{A(a, b)}$ is left-invariant, and this association gives a bijection between elements of $\g^{\ast} \tensor \g^{\ast} \tensor \g$ and left-invariant affine connections on $G$. Theorem II.1.1.4 of \cite{Helgason} shows that if $\nabla$ is the connection corresponding to $A \in \g^{\ast} \tensor \g^{\ast} \tensor \g$, the geodesic through the identity of $G$ in the direction of $a \in \g$ is a one-parameter subgroup if and only if $A(a, a) = 0$. It follows that the left-invariant affine connection $\nabla$ corresponding to $A$ has the property that its geodesics are the left translates of one-parameter subgroups of $G$ if and only if $A \in \ext^{2}(\g^{\ast})\tensor \g$. The torsion of $\nabla$ is $\tau(\lf^{a}, \lf^{b}) = \lf^{A(a, b) - A(b, a) - [a, b]}$, so that $\nabla$ is torsion free if and only if $A(a, b) = \tfrac{1}{2}\ad(a)(b) + \Gamma$ with $\Gamma \in S^{2}(\g^{\ast})\tensor \g$. Thus the affine space of left-invariant torsion-free affine connections on $G$ is modeled on $S^{2}(\g^{\ast}) \tensor \g$. Among the geodesics of a torsion free $\nabla$ are the left translates of one-parameter subgroups if and only if $A$ is skew-symmetric, which implies that $\Gamma = 0$, so that the unique torsion-free left-invariant affine connection on $G$ admitting among its geodesics the left translates of one-parameter subgroups is that defined by $A = \tfrac{1}{2}\ad$.

\subsubsection{}
Let $\nabla$ be a left-invariant torsion-free connection on $G$. Let $e_{i}$ be a basis of $\g$, write $[e_{i}, e_{j}] = c_{ij}\,^{k}e_{k}$ and $A(e_{i}, e_{j}) = A_{ij}\,^{k}e_{k}$, and express components of tensors using the left-invariant frame $E_{i} = \lf^{e_{i}}$. The Jacobi identity is $c_{[ij}\,^{p}c_{k]p}\,^{l} = 0$. 
Let $h \in S^{2}(\g^{\ast})$ be the left-invariant form determined by the \textit{negative} of the Killing form, $h(\lf^{a}, \lf^{b}) \defeq - B(a, b)$, or, equivalently $h_{ij} = -c_{ip}\,^{q}c_{jq}\,^{p}$. The invariance of the Killing form means that $h$ is in fact bi-invariant. Because $G$ is assumed semisimple, $h$ is a pseudo-Riemannian metric. Let $h^{ij}$ be the bivector dual to $h_{ij}$ and in the remainder of this section raise and lower indices with $h^{ij}$ and $h_{ij}$.

The invariance of the Killing form implies $c_{i(j}\,^{p}h_{k)p} = 0$, so that $c_{ijk} \defeq c_{ij}\,^{p}h_{pk}$ is completely antisymmetric, $c_{ijk} = c_{[ijk]}$. Define $A_{ijk} = A_{ij}\,^{p}h_{pk}$ and write $A_{ijk} = \tfrac{1}{2}c_{ijk} + \Ga_{ijk}$; that $\nabla$ be torsion-free is equivalent to $\Ga_{[ij]k} = 0$. Then $\nabla_{i}h_{jk} = 2A_{i(jk)} = 2\Ga_{i(jk)}$. A sufficient condition for the Ricci tensor of $\nabla$ to be symmetric is that $\nabla \det h = 0$, and, as $h^{jk}\nabla_{i}h_{jk} = 2A_{ip}\,^{p} = 2\Ga_{ip}\,^{p}$, this condition holds if and only if $\Ga_{ijk}$ is completely $h$-trace-free. From $\nabla_{[i}h_{j]k} = \Ga_{k[ij]}$ it follows that $(\en, [h])$ is an AH structure if and only if $\Ga_{k[ij]}$ is pure $h$-trace, or, what is the same, $\Ga_{ijk} = \Ga_{(ijk)} + \tfrac{2}{3(n-1)}\left(\Ga_{p[i}\,^{p}h_{j]k} - \Ga_{p}\,^{p}\,_{[i}h_{j]k}\right)$. If $(\en, [h])$ is AH, then $\nabla$ is its aligned representative if and only if $n\Ga_{p}\,^{p}\,_{i} = (2-n)\Ga_{ip}\,^{p}$, or, what is the same, $\Ga_{ijk} = \Ga_{(ijk)} + \tfrac{4}{3n}\Ga_{p[i}\,^{p}h_{j]k}$. For $(\en, [h])$ to be exact AH with aligned representative $\nabla$ and distinguished metric $h_{ij}$, it suffices that $\Ga_{(ijk)} = \Ga_{ijk}$ be completely $h$-trace free.

Again suppose only that $\nabla$ is left-invariant and torsion-free. The curvature of $\nabla$ is $R(\lf^{a}, \lf^{b})\lf^{c} = \lf^{A(a, A(b, c)) - A(b, A(a, c)) - A([a, b], c)}$. %
With respect to the frame $E_{i} = \lf^{e_{i}}$ the components of the curvature of $\nabla$ have the forms 
\begin{align}
\label{liecurv}
\begin{split}
R_{ijk}\,^{l} &= A_{ip}\,^{l}A_{jk}\,^{p} - A_{jp}\,^{l}A_{ik}\,^{p} - c_{ij}\,^{p}A_{pk}\,^{l} \\
& = -\tfrac{1}{4}c_{ij}\,^{p}c_{pk}\,^{l} + 2\Ga_{p[i}\,^{l}\Ga_{j]k}\,^{p}  + \left(c_{p[j}\,^{l}\Ga_{i]k}\,^{p} + c_{k[i}\,^{p}\Ga_{j]p}\,^{l} - c_{ij}\,^{p}\Ga_{pk}\,^{l} \right),\\
R_{ij} %
&= \tfrac{1}{4}h_{ij} -\Ga_{ip}\,^{q}\Ga_{jq}\,^{p} + \Ga_{pq}\,^{q}\left(\tfrac{1}{2}c_{ij}\,^{p} + \Ga_{ij}\,^{p}\right) -c_{pq(i}\Ga_{j)}\,^{[pq]} ,\\
Q_{ij} & \defeq R_{iab}\,^{p}H_{pj}H^{ab}= \tfrac{1}{4}h_{ij} -\Ga_{pqj}\Ga_{i}\,^{qp} + \Ga_{q}\,^{qp}\left(\Ga_{pij} - \tfrac{1}{2}c_{pij}\right) - \tfrac{1}{2}c_{pqj}\Ga_{i}\,^{[pq]},\\
\uR_{h} & = \tfrac{n}{4} - \Ga_{ijk}\Ga^{ikj} + \Ga_{pq}\,^{q}\Ga_{a}\,^{ap}.
\end{split}
\end{align}
Now suppose $\Ga_{ijk}$ is completely symmetric and completely $h$-trace free, so that $(\en, [h])$ is an exact AH structure with aligned representative $\nabla$ and distinguished metric $h$ and cubic torsion $\bt_{ij}\,^{k} = 2\Ga_{ij}\,^{k}$. Then $R_{ij} = Q_{ij} = \tfrac{1}{4}h_{ij} - \Ga_{ip}\,^{q}\Ga_{jq}\,^{p} = \tfrac{1}{4}(h_{ij} - \bt_{ij})$, so that $E_{ij} = 0$, and $\uR_{h} = \tfrac{n}{4} - \Ga_{ijk}\Ga^{ijk}$. Let $\mu = |\det h|^{-1/n}$. From $\mu^{-1} R_{i(jkl)} = -2c_{pi(j}\Ga_{kl)}\,^{p}$ and $E_{ij} =0$ there follows 
\begin{align}\label{lieeijkl}
\mu^{-1}E_{ijkl} = \mu^{-1} \uf_{ijkl} = \Ga_{k[i}\,^{p}c_{j]lp} + \Ga_{l[i}\,^{p}c_{j]kp} - c_{pij}\Ga_{kl}\,^{p}. 
\end{align}
Combining \eqref{lieeijkl} and \eqref{liecurv} shows
\begin{align}\label{liecurv2}
R_{ijk}\,^{l} = -\tfrac{1}{4}c_{ij}\,^{p}c_{pk}\,^{l} + 2\Ga_{p[i}\,^{l}\Ga_{j]k}\,^{p} + E_{ijk}\,^{l} = T_{ijk}\,^{l} + E_{ijk}\,^{l}.
\end{align}
It is evident from \eqref{lieeijkl} that $\bt^{jkl}E_{ijkl} = 0$, so that $(\en, [h])$ is conservative. Hence for $(\en, [h])$ to be Einstein, it suffices that it be naive Einstein, and this is the case if and only if there holds
\begin{align}\label{liecrit}
\Ga_{ip}\,^{q}\Ga_{jq}\,^{p} = \tfrac{1}{n}h^{ab}\Ga_{ap}\,^{q}\Ga_{bq}\,^{p}h_{ij}.
\end{align}
Equivalently, $(\en, [h])$ is Einstein if and only if $\mr{\bt}_{ij} = 0$, or, what is the same, $(\en, [h])$ is strongly Einstein.

\subsubsection{}
Let $x, y, z \in \g$ be Killing orthogonal and extend them globally as left-invariant vector fields $X^{i}$, $Y^{i}$, $Z^{i}$. Write $\ep_{x} = x^{i}x^{j}h_{ij}$, and similarly for $y$ and $z$, and suppose $\ep_{x}, \ep_{y}, \ep_{z} \in \{0, \pm 1\}$. Let $t \in \rea$. Then $\Ga_{ijk} \defeq t x_{(i}y_{j}z_{k)}$ satisfies
\begin{align}\label{liethree}
&18\Ga_{ip}\,^{q}\Ga_{jq}\,^{p} =  t^{2}\left(\ep_{y}\ep_{z}x_{i}x_{j} + \ep_{z}\ep_{x}y_{i}y_{j} + \ep_{x}\ep_{y}z_{i}z_{j}\right),&
&6\Ga_{pqr}\Ga^{pqr} = t^{2}\ep_{x}\ep_{y}\ep_{z}, && \Ga_{ip}\,^{p} =0.
\end{align}

\subsubsection{}\label{s3example}
Let $G = S^{3}$. Then $h_{ij}$ is positive definite. Choose $x, y, z$ to satisfy the bracket relations $[x, y] = z/\sqrt{2}$, $[y, z] = x/\sqrt{2}$, and $[z, x] = y/\sqrt{2}$ so that $\{x, y, z\}$ is an $h$-orthonormal basis of $\g$ (for which $c_{ijk} = \sqrt{2}\left(x_{[i}y_{j]}z_{k} + y_{[i}z_{j]}x_{k} + z_{[i}x_{j]}y_{k} \right)$) and $h_{ij} = X_{i}X_{j} + Y_{i}Y_{j} + Z_{i}Z_{j}$. 
 Define $\Ga_{ijk} =  t X_{(i}Y_{j}Z_{k)}$ for $t \in \reap$. Then \eqref{liethree} shows that $\Ga_{ip}\,^{p} = 0$ and $18\Ga_{ip}\,^{q}\Ga_{jq}\,^{p} = t^{2} h_{ij}$, so that $\Ga_{ij}\,^{k}$ solves \eqref{liecrit}, and hence $(\en, [h])$ is a strongly Einstein Riemannian AH structure with scalar curvature $\uR_{h} = \tfrac{3}{4} - \tfrac{t^{2}}{6}$, which is positive for $t < 3/\sqrt{2}$, negative for $t > 3/\sqrt{2}$, and zero for $t = 3/\sqrt{2}$. Since the dimension is $3$ there holds $A_{ijk}\,^{l} = 0$. There hold $\nabla_{X}Y = (\tfrac{t}{6} + \tfrac{1}{2\sqrt{2}})Z$ and $\nabla_{Y}X + (\tfrac{t }{6} - \tfrac{1}{2\sqrt{2}})Z$, and their cyclic permutations, and $\nabla_{X}X = 0$, etc. Using \eqref{lieeijkl} one finds
\begin{align*}
&X^{i}Y^{j}Z^{k}Z^{l}E_{ijkl} = 0,& &X^{i}Y^{j}X^{k}X^{l}E_{ijkl} = 0,& &\mu^{-1} X^{i}Y^{j}X^{k}Y^{l}E_{ijkl} = - t\sqrt{2}/6. 
\end{align*}
The last of these shows that $E_{ijkl}$ is not zero, so, by Lemma \ref{projflatahlemma}, $(\en, [h])$ is neither projectively flat nor conjugate projectively flat. In fact it can be checked that 
\begin{align*}
E_{ijkl} = -\tfrac{2\sqrt{2} t}{3}\left(X_{[i}Y_{j]}X_{(k}Y_{l)} + Y_{[i}Z_{j]}Y_{(k}Z_{l)} + Z_{[i}X_{j]}Z_{(k}X_{l)} \right),
\end{align*}
from which follows $|E|_{h}^{2} = 2t^{2}/3$, though these computations are left to the reader. 

This example will be revisited in Section \ref{s3examplerevisited}, where it will be used to show the necessity in Theorem \ref{cubictorsiontheorem} of the self-conjugacy of the curvature.

\subsubsection{}
For $G = SL(2, \rea)$ a construction just like that of Subsection \eqref{s3example} yields an indefinite signature Einstein AH structure. In this case one chooses $x, y, z \in \g$ satisfying $[x, y] = z/\sqrt{2}$, $[y, z] = -x/\sqrt{2}$, and $[z, x] = -y/\sqrt{2}$, so that $h(x, x) = -1 = h(y, y)$ and $h(z, z) = 1$. Defining $\Ga_{ijk}$ as before yields a signature $(1, 2)$ Einstein AH structure, the scalar curvature of which can be positive, negative, or zero, depending on $t$.

\subsubsection{}
Suppose $G$ is semisimple of noncompact type. Since $h_{ij}$ is indefinite there is a non-zero $h$-null vector $x_{i} \in \g$. Then $\Ga_{ijk} = X_{i}X_{j}X_{k}$ solves \eqref{liecrit}, and so $(\en, [h])$ is an Einstein AH structure with $4R_{ij} = h_{ij}$, and so which is proper. If $G = SL(2, \rea)$ then the null elements of $\g$ are exactly the nilpotent elements, and so the construction just described associates to each nilpotent element of $\sll(2, \rea)$ an Einstein AH structure on $SL(2, \rea)$. From \eqref{liecurv2} there follow $4\mu^{-1} T_{ijkl} =  - c_{ij}\,^{p}c_{pk}\,^{l}$ and
\begin{align}\label{nileijkl}
\begin{split}
&2\mu^{-1} E_{ijkl} = X^{p}\left(X_{i}X_{k}c_{pjl} - X_{j}X_{k}c_{pil} - X_{j}X_{l}c_{pik} + X_{i}X_{l}c_{pjk} - 2X_{k}X_{l}c_{pij} \right).
\end{split}
\end{align}
If for some $x$ it can be shown that $E_{ijkl}$ is not identically $0$, then by Lemma \ref{conjprojflatae} the Einstein AH structure determined by $x$ is neither projectively nor conjugate projectively flat. To prove this there is needed:
\begin{lemma}\label{tdslemma}
If $\g$ is a finite-dimensional semisimple Lie algebra over a field of characteristic zero and $0 \neq e \in \g$ is a nilpotent element then there is an $\sll(2, \rea)$ triple $\{f, h, e\}$ containing $e$ as the nilpositive element.
\end{lemma}
\begin{proof}
The claim means there is a  triple $\{e, f, h\} \subset \g$ such that $[e, f] = h$, $[h, e] = 2e$, and $[h, f] = -2f$. If $\g$ is a finite-dimensional semisimple Lie algebra over a field of characteristic zero and $0 \neq e \in \g$ is a nilpotent element then $e \in \im \ad(e)^{2}$. This is proved over the complex field as Lemma $3.3$ of \cite{Kostant-triples}, but the proof given there works over any field of characteristic $0$. Because $e \in \im\ad(e)^{2}$, there can be chosen $y \in \g$ such that $h = [e, y]$ satisfies $[h, e] = 2e$. Morozov's Lemma (see e.g. Lemma $\text{III}.11.7$ of \cite{Jacobson}) shows that there is an $\sll(2, \rea)$ triple $\{f, h, e\}$ containing $e$ as nilpositive element and $h$ as semisimple element.
\end{proof}
By invariance of the Killing form, for any $\sll(2, \rea)$ triple $\{f, h, e\}$ there hold $B(e, e) = B(f, f) = B(e, h) = B(f, h) =0$ (e.g. $B(e, h) = B(e, [e, f]) = B([e, e], f) = 0$). In particular the nilpositive element is null so generates an Einstein AH structure as above. Were $B(e, f) = 0$ then $B(h, h) = B(h, [e ,f]) = B([h, e], f) = 2B(e, f) = 0$, and the restriction to $\h \defeq \spn\{f, h, e\}$ of $B$ would be null. As $h$ is a semisimple element of $\h$, and its adjoint action on $\g$ has integer eigenvalues, $h$ is semisimple as an element of $\g$; standard arguments show $h$ is contained in a Cartan subalgebra of $\g$ and $B(h, h) \neq 0$. Hence $B(e, f) \neq 0$ also. Temporarily write $e = x$ and let $\{e, f, h\} \subset \g$ be an $\sll(2, \rea)$ triple containing $e$ as the nilpositive element. From \eqref{nileijkl} there follows $\mu^{-1}H^{i}F^{j}F^{k}F^{l}E_{ijkl} = 2B(e, f)^{3} \neq 0$, so that $E_{ijkl}$ is not identically zero. By lemma \ref{projflatahlemma}, this shows that these Einstein AH structures on $G$ are neither projectively nor conjugate projectively flat when $x \neq 0$. This proves constructively Theorem \ref{leftinvarianttheorem}.
\begin{theorem}\label{leftinvarianttheorem}
Let $G$ be a noncompact semisimple, finite-dimensional real Lie group with Lie algebra $\g$. To every non-zero nilpotent $x \in \g$ there is associated a left-invariant proper Einstein AH structure on $G$ for which the anti-self conjugate Weyl tensor is not identically zero, and so which is neither projectively nor conjugate projectively flat.
\end{theorem}

\section{Scalar curvature of Riemannian signature Einstein AH structures}\label{curvatureconditionssection}
In this section it is shown that many results about the scalar curvature of compact Riemannian signature Einstein Weyl structures generalize to compact Riemannian signature Einstein AH structures. The key technical result is Theorem \ref{todtheorem} which shows that for a Riemannian signature Einstein AH on a compact $n$-manifold the vector field dual to the Faraday primitive of a Gauduchon metric is conformal Killing. The principal structural result is Theorem \ref{negexacttheorem}. While the strategies of the proofs are generally the same as in the Weyl case, their realization often requires further argument. For example, there can be given a proof of Theorem \ref{todtheorem}, in which it is proved that the one-form associated to the Gauduchon gauge of an at least three-dimensional compact Riemannian signature Einstein AH structure is Killing for the Gauduchon metric, following exactly that of the proof given by K.P. Tod of its specialization for Einstein Weyl structures, but the computations require the full strength of the Einstein condition, in particular the vanishing of $\nabla_{i}R + n\nabla^{p}F_{ip}$, and become somewhat involved. Here the result is deduced from Theorem \ref{bochnertheorem}, essentially via a Bochner vanishing argument.

\subsection{Gauduchon metrics}\label{gauduchongaugesection}
Let $(\en, [h])$ be an AH structure, $\nabla \in \en$ the aligned representative, and $\{\ga\} = \{\ga + df: f\in \cinf(M)\}$ the associated equivalence class of Faraday primitves. A \textbf{Gauduchon metric} is a representative $h \in [h]$ such that the associated Faraday primitive $\ga_{i}$ is co-closed with respect to $h$; that is $\dad_{h}\ga = 0$. By \eqref{gd1}, $h$ is a Gauduchon gauge if and only if $\nabla^{p}\ga_{p} + (n-2)\ga^{p}\ga_{p} = 0$. Let $(\pen, [h])$ be the AH pencil generated by an AH structure $(\en, [h])$. Because the difference tensor of $\pnabla$ with $\nabla$ is trace-free, they induce on any line bundle of densities of non-trivial weight the same covariant derivative, and thus the Faraday primitive associated to the representative $h \in [h]$ with respect to $\pnabla$ does not depend on $t$. In particular, the equation $\pnabla^{p}\ga_{p} + (n-2)\ga^{p}\ga_{p} = 0$ does not depend on $t$. It follows that to prove the existence of a Gauduchon gauge for $(\en, [h])$, it suffices to prove it for the underlying Weyl structure. It is well known that for a Weyl structure on a compact manifold there exists a Gauduchon gauge unique up to positive homothety. The proof, which is a straightforward application of standard elliptic operator theory, is sketched in the first paragraph of page $503$ of \cite{Gauduchon} (this sketch is repeated in Appendix $1$ of \cite{Tod-compact} and in section $4$ of \cite{Calderbank-Pedersen}). The precise statement needed, which will be used in section \ref{yamabesection}, is
\begin{theorem}[P. Gauduchon, \cite{Gauduchon}]\label{gauduchontheorem}
Let $h_{ij}$ be a Riemannian metric on a compact manifold $M$ of dimension $n \geq 2$ and let $\ga_{i}$ be a one-form. There exists a positive $f \in \cinf(M)$, determined uniquely up to multiplication by a positive constant, such that the one-form $\tilde{\ga}_{i} \defeq \ga_{i} + \tfrac{1}{2}d\log{f}_{i}$ is co-closed with respect to the conformal metric $\tilde{h}_{ij} \defeq fh_{ij}$.
\end{theorem}
\begin{proof}
For convenience, the proof is repeated here. Let $2\si_{i} = d\log{f}_{i}$. Let $\tD$ be the Levi-Civita connection of $\tilde{h} = fh$ and let $\tilde{\ga} = \ga + \si$. If $\dad_{\tilde{h}}\tilde{\ga} = 0$ then $f$ must solve $-\dad_{h}(\ga + \si) + (n-2)(\si^{p}\ga_{p} +  |\ga|^{2}) = 0$. If $n =2$ writing $f = e^{-2\phi}$, $\phi$ must solve $\lap_{h}\phi = D^{p}\ga_{p}$. Since $D^{p}\ga_{p}$ integrates to $0$ this has always a solution, unique up to addition of a constant, and the claim is proved in the $n = 2$ case. Now suppose $n > 2$ and let $f = \phi^{2/(n-2)}$. The equation to be solved simplifies to
\begin{align}\label{gaud1}
L(\phi) \defeq \lap_{h}\phi + (2-n)\dad_{h}(\phi \ga) = -\dad_{h}(d\phi + (n-2)\phi \ga) = 0.
\end{align}
Because $L$ is elliptic it is Fredholm. Because the formal adjoint $L^{\ast}(u) = \lap_{h}u + (2-n)\ga^{p}D_{p}u$ of $L$ has no zeroth order term, the maximum principle shows its kernel comprises constants. Because $L$ and $L^{\ast}$ have the same principal symbol their analytic indices are the same, so the dimension of kernel of $L$ equals that of the kernel of $L^{\ast}$, which is one. Hence a solution $\phi$ of \eqref{gaud1} is determined up to addition of a constant. Because $L^{\ast}$ has no zeroth order term, by the maximum principle a function in its image must change sign. Since $\im L^{\ast}$ is orthogonal to $\ker L$, were $\phi$ to change signs there would be a positive function in $\im L^{\ast}$ orthogonal to $\phi$, which there is not. Hence $\phi$ is either non-negative or non-positive, and the maximum principle applied to $L$ shows that if $\phi$ is non-zero then it is either positive or negative. 
\end{proof}
Applying Theorem \ref{todtheorem} to an arbitrary $h \in [h]$ with associated Faraday primitive $\ga$ yields
\begin{corollary}\label{gdgauge}
For a Riemannian AH structure on a compact manifold of dimension $n \geq 2$ there is a Gauduchon gauge determined uniquely up to homothety.
\end{corollary}

\subsection{Properties of the Gauduchon metric of a Riemannian Einstein AH structure}\label{bochnervanishingsection}
\begin{theorem}\label{bochnertheorem}
Let $(\en, [h])$ be a Riemannian AH structure on a compact manifold $M$ of dimension $n \geq 2$. Let $h \in [h]$ be a Gauduchon metric, $\ga_{i}$ the corresponding Faraday primitive, and $D$ the Levi-Civita connection of $h$. In the statement of this theorem and in its proof raise and lower indices using $h_{ij}$ and $h^{ij}$. If $\al_{i}$ is an $h$-harmonic one-form then
\begin{align}\label{bochner5}
0 = \int_{M}|D\al|_{h}^{2} + \int_{M}\al^{i}\al^{j}\left(T_{ij} + \tfrac{1}{4}\bt_{ij}\right) + (n-2)\int_{M}\left(|\ga|^{2}_{h}|\al|_{h}^{2} - \lb \al, \ga \ra^{2}\right) , 
\end{align}
in which all integrals are with respect to the $h$ volume measure. Also,
\begin{align}\label{bochner5b}
\tfrac{1}{4}\int_{M}|d\ga|_{h}^{2} = \int_{M}|\mr{D\ga}|^{2}_{h} + \int_{M}\ga^{i}\ga^{j}\left(T_{ij} + \tfrac{1}{4}\bt_{ij}\right).
\end{align}
\end{theorem}

\begin{proof}
Let $h \in [h]$ be any representative and $D$ its Levi-Civita connection. Let $\square = dd^{\ast} + d^{\ast}d$ be the Hodge Laplacian (in the non-orientable case intepreted as explained in section \ref{coclosedsection}). For any one-form $\al_{i}$, substituting \eqref{confric} into the Bochner formula $\tfrac{1}{2}\lap_{h}|\al|^{2} + \al^{i}\square\al_{i} = \sR_{ij}\al^{i}\al^{j} + D^{i}\al^{j}D_{i}\al_{j}$ gives
\begin{align}\label{bochner1}
\tfrac{1}{2}\lap_{h}|\al|^{2}_{h} + \al^{i}\square\al_{i} &= |D\al|_{h}^{2} + \al^{i}\al^{j}\left(T_{ij} + \tfrac{1}{4}\bt_{ij} + (n-2)\left(|\ga|^{2}_{h}h_{ij} - \ga_{i}\ga_{j}\right)\right) \\
\notag &+ (2-n)\al^{i}\al^{j}D_{i}\ga_{j} + |\al|_{h}^{2}\dad_{h}\ga.
\end{align}
Integrating by parts several times yields
\begin{align}\label{bch3}
 \int_{M}|\al|_{h}^{2}\dad_{h}\ga &+ 2\int_{M}\al^{i}\al^{j}D_{i}\ga_{j} = 2\int_{M}\lb \al, \ga \ra \dad_{h}\al +  2\int_{M}\ga^{i}\al^{j}d\al_{ij}.
\end{align}
Integrating \eqref{bochner1} using \eqref{bch3} gives
\begin{align}\label{bch4}
\begin{split}
\int_{M}&\al^{i}\square\al_{i} + (n-2) \int_{M} \lb \al, \ga \ra \dad_{h}\al + (n-2)\int_{M}\ga^{i}\al^{j}d\al_{ij}  \\ 
&=  \int_{M}|D\al|_{h}^{2} + \int_{M}\al^{i}\al^{j}\left(T_{ij} + \tfrac{1}{4}\bt_{ij} + (n-2)\left(|\ga|^{2}_{h}h_{ij} - \ga_{i}\ga_{j}\right)\right) + \tfrac{n}{2}\int_{M}|\al|_{h}^{2}\dad_{h}\ga.
\end{split}
\end{align}
If $\al$ is harmonic, then it is closed and co-closed, and so there vanish the first three terms of \eqref{bch4} and so if $h$ is Gauduchon there results \eqref{bochner5}. Taking $\al_{i}$ to be $\ga_{i}$ in \eqref{bch4}, using \eqref{hodgebyparts},
and writing $\int_{M}|D\ga|^{2} = \tfrac{1}{4}\int_{M}|d\ga|^{2} + \int_{M}|\mr{D\ga}|^{2}_{h} - \tfrac{1}{n}\int_{M}(\dad_{h}\ga)^{2}$ gives
\begin{align}\label{bchgauge}
\tfrac{n-4}{2}\int_{M}|\ga|_{h}^{2}\dad_{h}\ga + \tfrac{n+1}{n}\int_{M}(\dad_{h}\ga)^{2} + \tfrac{1}{4}\int_{M}|d\ga|_{h}^{2} = \int_{M}|\mr{D\ga}|^{2}_{h} + \int_{M}\ga^{i}\ga^{j}\left(T_{ij} + \tfrac{1}{4}\bt_{ij}\right).
\end{align}
If $\ga$ is the Gauduchon gauge, then \eqref{bchgauge} becomes \eqref{bochner5b}, and the last term of \eqref{bch4} vanishes. 
\end{proof}

The following theorem generalizes to Einstein AH structures a property of the Gauduchon gauge for Einstein Weyl structures shown by P. Tod as Theorem $2.2$ of \cite{Tod-compact}.
\begin{theorem}\label{todtheorem}
Let $(\en, [h])$ be a Riemannian signature Einstein AH structure on a compact manifold of dimension $n \geq 2$. If $h \in [h]$ is the Gauduchon gauge and $\ga_{i}$ the corresponding one-form, then the vector field $h^{ip}\ga_{p}$ is $h$-Killing and for the Levi-Civita connection $D$ of $h$ there holds $D_{p}\bt_{ij}\,^{p} = n\ga_{p}\bt_{ij}\,^{p} = 0$.
\end{theorem}

\begin{proof}[Proof of Theorem \ref{todtheorem}]
Let $h \in [h]$ be a Gauduchon metric with Levi-Civita connection $D$ and associated Faraday primitive $\ga_{i}$. Integrating by parts gives 
\begin{align}\label{ttinter}
n\int_{M}|F|_{h}^{2}\,d\vol_{h} & = -2n\int_{M}\lb D\ga, F\ra\,d\vol_{h} = -2n\int_{M}\ga^{\sharp\,i}\dad_{h}F_{i}\,d\vol_{h}.
\end{align}
Contracting \eqref{conserve} with $\ga^{\sharp\,i}$ gives $n\ga^{\sharp\,i}\dad_{h}F_{i} = -\ga^{\sharp\,i}D_{i}\uR_{h} - 2|\ga|_{h}^{2}\uR_{h}$. Substituting this into \eqref{ttinter} and integrating by parts using $\dad_{h}\ga = 0$ shows
\begin{align}\label{einsteinf}
4\int_{M}|\ga|_{h}^{2}\uR_{h} \, d\vol_{h} = n\int_{M}|F|_{h}^{2}\, d\vol_{h}.
\end{align}
Substituting \eqref{einsteinf} in \eqref{bochner5b} gives $4\int_{M}|\mr{D\ga}|^{2}_{h} + \int_{M}\ga^{i}\ga^{j}\bt_{ij} = 0$. This implies $D_{(i}\ga_{j)} = 0$ and $\ga_{p}\bt_{ij}\,^{p} = 0$, and so, by \eqref{dga0}, there follows $D_{p}\bt_{ij}\,^{p} = n\ga_{p}\bt_{ij}\,^{p} = 0$. 
\end{proof}

\begin{remark}
The proof of Theorem \ref{todtheorem} uses in an essential way the conservation condition which distinguishes the Einstein condition from the weaker naive Einstein condition. This can be taken as motivation for imposing the conservation condition.
\end{remark}

In \cite{Calderbank-mobius}, Calderbank classified the Riemannian signature Einstein Weyl structures on compact Riemann surfaces. The key result is Theorem $3.7$ of \cite{Calderbank-mobius}, part of which is directly generalized by the two-dimensional case of Theorem \ref{classtheorem}

\begin{theorem}\label{classtheorem}
For a Riemannian signature Einstein AH structure $(\en, [h])$ on a compact manifold of dimension at least $2$ if $h \in [h]$ is a Gauduchon metric there are satisfied the equations
\begin{align}\label{confein1} &D_{p}\bt_{ij}\,^{p} = 0, \qquad \ga_{p}\bt_{ij}\,^{p} = 0,\qquad D_{(i}\ga_{j)} = 0,\\
\label{confein2}&\mr{\sR}_{ij} = \tfrac{1}{4}\mr{\bt}_{ij} + (2-n)\mr{\ga\tensor \ga}_{ij},\\
\label{confein3} &D_{i}(\sR - \tfrac{1}{4}|\bt|_{h}^{2} - (n+2)|\ga|_{h}^{2}) = D_{i}\left(\uR_{h} + n(n-4)|\ga|_{h}^{2}\right) = 0,\\
\label{confein4} &\ga^{\sharp\,i}D_{i}\uR_{h} = 0, \qquad \ga^{\sharp\,i}D_{i}|\bt|_{h}^{2} = 0.
\end{align}
If $n =2$ then \eqref{confein2} is vacuous and $|\ga|_{h}^{2}|\bt|_{h}^{2} = 0$. Conversely, if on a manifold of dimension at least $2$ (not necessarily compact) there are a Riemannian metric $h$ (with Levi-Civita connection $D$), an $h$-Killing field $X^{i}$, and a completely symmetric, completely $h$-trace free tensor $B_{ijk} = B_{(ijk)}$, such that $\ga_{i}\defeq X^{i}h_{pi}$ and $\bt_{ij}\,^{k} \defeq h^{kp}B_{ijp}$ solve \eqref{confein1}-\eqref{confein3}, then $\nabla \defeq D - \tfrac{1}{2}\bt_{ij}\,^{k} - 2\ga_{(i}\delta_{j)}\,^{k} + h_{ij}X^{k}$ is the aligned representative of an Einstein AH structure $(\en, [h])$ with cubic torsion $\bt_{ij}\,^{k}$ and for which $h$ is a Gauduchon metric. 
\end{theorem}

\begin{proof}
The equations \eqref{confein1} follow immediately from \eqref{dga0} and Theorem \ref{todtheorem}. Substituting these into \eqref{dga1} gives \eqref{confein2}. Because $D_{(i}\ga_{j)} = 0$ there holds $d\ga_{ij} = 2D_{i}\ga_{j}$, and so $D_{i}|\ga|_{h}^{2} = \ga^{\sharp\,p}d\ga_{ip}$. By the Ricci identity, $\dad_{h}d\ga_{i} = -D^{p}d\ga_{pi} = 2\sR_{ip}\ga^{\sharp\,p}$. Together \eqref{confein1} and \eqref{confric} show that $n\ga^{\sharp\,p}\sR_{ip} = \uR_{h}\ga_{i}$. Substituting the preceeding three observations into \eqref{conserve} gives the second equality of \eqref{confein3}. The first equality in \eqref{confein3} is true for any AH structure in any gauge, by \eqref{confscal} (it is included only for convenience). If $n = 4$ then by \eqref{confein3}, $D_{i}\uR_{h} = 0$, so the first equality of \eqref{confein4} holds trivially; if $n \neq 4$ then by \eqref{confein1} and \eqref{confein3}, $\ga^{\sharp\,i}D_{i}\uR_{h} = 2n(4-n)\ga^{p}\ga^{q}D_{(p}\ga_{q)} = 0$, showing the first equality of \eqref{confein4}. Since $\ga^{\sharp\,i}$ is $h$-Killing, there holds $\ga^{\sharp\,p}D_{p}\sR = 0$, and with the first equality of \eqref{confein4} and \eqref{confein3} this shows the second equality of \eqref{confein4}. If $n =2$ Lemma \ref{twodxblemma} shows that $\ga_{p}\bt_{ij}\,^{p} = 0$ implies that $|\ga|_{h}^{2}|\bt|_{h}^{2} = 0$.

If given $(h, X, B)$ as in the statement of the theorem, then it is straightforward to check that $(\en, [h])$ is an AH structure with cubic torsion $\bt_{ij}\,^{k}$, aligned representative $\nabla$, and Gauduchon metric $h$. The curvatures of $\nabla$ and $D$ are related as in \eqref{confcurvijkl}, and there hold \eqref{ddivbt}, \eqref{confric}, \eqref{confscal}, and \eqref{conserve}. Together \eqref{ddivbt} and \eqref{confein1} show $E_{ij} = 0$. Together \eqref{confric}, \eqref{confscal}, \eqref{confein1}, and \eqref{confein2} show $\mr{T}_{ij} = 0$, and so show the naive Einstein equations. Finally, substituting $D_{i}|\ga|_{h}^{2} = \ga^{\sharp\,p}d\ga_{ip}$ and $D^{p}d\ga_{pi} = -2\sR_{ip}\ga^{\sharp\,p}$ into \eqref{conserve}, and using \eqref{confein2} shows that $\nabla_{i}R + 2\nabla^{p}F_{ip} = 0$.
\end{proof}

When $n = 2$, on an oriented surface a Riemannian conformal structure determines a complex structure, and that $\ga^{\sharp}$ be Killing and the second equation of \eqref{confein1} imply that $X$ is the real part of a holomorphic vector field and $L_{ijk}$ is the real part of a cubic holomorphic differential. In combination with the Riemann-Roch theorem these observations lead to stronger results which are reported in \cite{Fox-2dahs}. In particular, in two-dimensions the holomorphicity implies that the zeroes of $X$ and $L$ are isolated, and it follows from $\ga_{p}L_{ij}\,^{p} = 0$ that a Riemannian signature Einstein AH structure which is not exact is Weyl, and which is not Weyl is exact. In higher dimensions the zeroes of a Killing field need not be isolated (see e.g. \cite{Kobayashi-fixedpoints}), and it is not clear to what extent this last dichotomy carries over to higher dimensions.

\subsection{A theorem of Bochner vanishing type for AH structures}
Theorem $1.1$ of B. Alexandrov and S. Ivanov's \cite{Alexandrov-Ivanov} is a Bochner vanishing theorem for Weyl structures on a compact $n$-manifold satisfying the non-negativity condition that for the Faraday primitive associated to a Gauduchon metric there hold $R_{(ij)} - \tfrac{(n-2)(n-4)}{2}(|\ga|_{h}^{2}h_{ij} - \ga_{i}\ga_{j}) \geq 0$. Theorem \ref{ivanovtheorem} generalizes (with a similar proof) Theorem $1.1$ of \cite{Alexandrov-Ivanov} to AH structures, though with a non-negativity condition which is weaker even in the Weyl case. 

\begin{theorem}\label{ivanovtheorem}
Let $(\en, [h])$ be a Riemannian signature AH structure on a compact manifold $M$ of dimension $n \geq 2$ and having first Betti number $b_{1}$. Suppose there holds 
\begin{align}\label{bochnerass}
T_{ij} + \tfrac{1}{4}\bt_{ij} + (n-2)\left(|\ga|^{2}_{h}h_{ij} - \ga_{i}\ga_{j}\right) \geq 0,
\end{align}
for a Gauduchon metric $h\in [h]$ and the associated Gauduchon gauge $\ga_{i}$. Then 

\begin{enumerate}
\item Any $h$-harmonic one-form is parallel and $b_{1} \leq n$. There holds $b_{1} = n$ if and only if any Gauduchon metric is flat and the universal cover of $M$ is isometric to $\rea^{n}$ with the Euclidean metric.
\item If $b_{1} = 0$ and $(\en, [h])$ is closed then $(\en, [h])$ is exact.
\item If $b_{1}\geq 1$ the universal cover of $M$ with a Gauduchon metric is isometric to a product metric on $\rea \times N$ with $N$ simply-connected and complete in the induced metric $g$. %
If $n > 2$ and $(\en, [h])$ is not exact, then $(\en, [h])$ is closed, $b_{1} = 1$, $H^{1}_{\derham}(M;\rea)$ is generated by the Gauduchon one-form, and the Ricci tensor of the metric induced on $N$ is non-negative. If, moreover, the restriction to $\ker \ga$ of $T_{ij} + \tfrac{1}{4}\bt_{ij}$ is positive definite, then $N$ is compact; in particular if $n$ is $3$ or $4$ then $N$ is $S^{n-1}$ with a metric of constant curvature. 
\item If either at some point $\ga^{i}\ga^{j}(T_{ij} + \tfrac{1}{4}\bt_{ij}) > 0$ or the inequality \eqref{bochnerass} is strict somewhere on $M$, then $b_{1} = 0$. 
\end{enumerate}
\end{theorem}

From the non-negativity of each of $\bt_{ij}$ and $|\ga|_{h}^{2}h_{ij} - \ga_{i}\ga_{j}$, it follows that it is not stronger to assume \eqref{bochnerass} than it is to assume the non-positivity of something gauge independent such as $T_{ij} + \tfrac{1}{4}\bt_{ij}$. Hence the reader unhappy with imposing on an AH structure a condition which makes reference to a particular choice of representative metric can substitute for the main hypothesis of Theorem \ref{ivanovtheorem} the non-positivity of $T_{ij} + \tfrac{1}{4}\bt_{ij}$.

For Weyl structures the condition \eqref{bochnerass} is \textit{a priori} weaker than condition $(1.1)$ of \cite{Alexandrov-Ivanov}. When $(\en, [h])$ is exact and $M$ is compact, by \eqref{confric} the condition \eqref{bochnerass} means that the Ricci tensor of a Gauduchon metric is non-negative. For a Riemannian Einstein AH structure on a compact $n$-manifold it follows from \eqref{confric} and \eqref{confein1} that the lefthand side of \eqref{bochnerass} is simply the Ricci curvature of a Gauduchon metric.

\begin{proof}[Proof of Theorem \ref{ivanovtheorem}]
By \eqref{confric} if $n = 2$ then by \eqref{confric} the assumption \eqref{bochnerass} implies that the curvature of a Gauduchon metric is non-negative. In this case the classical Bochner argument shows $(1)$ and so $M$ is diffeomorphic to one of the sphere, projective plane, torus, or Klein bottle. 

The core of the argument is based on that of the proof of Theorem $1.1$ of \cite{Alexandrov-Ivanov}. If $\al$ is an $h$-harmonic one-form then \eqref{bochnerass} implies that the righthand side of \eqref{bochner5} is non-negative, so there must hold $D_{i}\al_{j} = 0$ and, when $n >2$, $\al \wedge \ga = 0$. As every de Rham cohomology class contains a unique harmonic representative and there are at most $n$ linearly independent parallel one-forms, $b_{1} \leq n$. If $b_{1} = n$ then there are $n$ parallel, non-vanishing one-forms, so $M$ with a Gauduchon metric is flat, and the universal cover of $M$ is $\rea^{n}$.

By \eqref{bochnerass} the righthand side of \eqref{bochner5b} is non-negative, and it follows that if $(\en, [h])$ is closed then the Gauduchon gauge must be parallel, $D_{i}\ga_{j} = 0$, and so, in particular, $\ga$ is harmonic. In this case, because $\ga$ is parallel, either it is identically zero or it is nowhere vanishing, so either $(\en, [h])$ is exact, or $b_{1} > 0$. If $b_{1} = 0$ and $(\en, [h])$ is closed then $\ga$ is harmonic, so must be identically zero, so $(\en, [h])$ is exact. 

If $b_{1} > 0$ then there is a non-trivial harmonic one-form $\al_{i}$, and because $\al_{i}$ is $D$-parallel it must be nowhere vanishing; the de Rham decomposition implies that the universal cover $\tilde{M}$ of $M$ is isometric to a direct product $\rea \times N$ for some simply-connected $(n-1)$-manifold $N$ and complete metrics on $\rea$ and $N$. A simply connected one-manifold is isometric to the line, so if $n = 2$ this means the universal cover of $M$ is isometric to a flat $\rea^{2}$. Suppose $n >2$. Then $\al \wedge \ga = 0$ so there is $f \in \cinf(M)$ such that $\ga = f\al$. Because $\al$ and $\ga$ are co-closed there holds $df(\al^{\sharp})  = 0$. Write $X^{i} = \la \al^{\sharp\,i} + Y^{i}$ with $\al_{i}Y^{i} = 0$. Contracting $0 = 2D_{[i}D_{j]}\al_{k} = -\al_{p}\sR_{ijk}\,^{p}$ gives $\al^{\sharp\,p}\sR_{ip} = 0$, so $Y^{i}Y^{j}\sR_{ij} = X^{i}X^{j}\sR_{ij}$. By \eqref{confric} and \eqref{bochnerass}, $Y^{i}Y^{j}\sR_{ij} = X^{i}X^{j}\sR_{ij} \geq (2-n)X^{i}X^{j}D_{i}\ga_{j} = (2-n)df(X)\al(X) = (2-n)\la |\al|_{h}^{2}df(Y)$. Since $|\al|_{h}^{2} > 0$ and this holds for arbitary real constants $\la \in \rea$, it must be $df(Y) = 0$. Since this holds for all $Y$ in the kernel of $\al$ and also $df(\al^{\sharp}) = 0$, $f$ must be constant, so $\ga$ is a constant multiple of $\al$. Either $(\en, [h])$ is exact, or every harmonic one-form is a constant multiple of the Gauduchon gauge, and so $b_{1} = 1$. Since $\ga$ is parallel, $(\en, [h])$ is closed. From \eqref{confric} and \eqref{bochnerass} it follows that the Ricci tensor of the metric induced on $N$ is non-negative. 

If, moreover, the restriction to $\ker \ga$ of $T_{ij} + \tfrac{1}{4}\bt_{ij}$ is positive definite then by \eqref{confric} the restriction to $\ker \ga$ of $\sR_{ij}$ is positive definite and so, because the splitting $\tilde{M} \simeq \rea \times N$ is isometric, the induced metric on $N$ has Ricci curvature bounded strictly away from zero. Because $N$ is complete, it is compact by Myers' Theorem. If $n = 3$ then $N$ is a compact, simply-connected surface admitting a metric of positive Ricci curvature, so must be diffeomorphic to a sphere; if $n = 4$ then $N$ is a compact, simply-connected three-manifold admitting a metric of positive Ricci curvature, so is diffeomorphic to a sphere by Hamilton's theorem, \cite{Hamilton}.

Suppose $n > 2$ and at some point there holds $\ga^{i}\ga^{j}(T_{ij} + \tfrac{1}{4}\bt_{ij}) > 0$. By the preceeding, were $b_{1} \geq 0$ then $\ga_{i}$ would be harmonic, but with \eqref{bochnerass} this would give a contradiction in \eqref{bochner5}; hence $b_{1} = 0$. If the inequality \eqref{bochnerass} is strict at $p \in M$ then \eqref{bochner5} implies $0 = \al^{i}\al^{j}(T_{ij} + \tfrac{1}{4}\bt_{ij}) > 0$ at $p$, so $\al_{i}$ vanishes at $p$; because $\al$ is parallel, it vanishes identically, so there is no non-trivial harmonic one-form and $b_{1} = 0$. 
\end{proof}

\subsection{Rough classification by scalar curvature of Riemannian signature Einstein AH structures on compact manifolds}\label{scalarcurvaturesection}
The simplest results about the scalar curvature of  Riemannian signature Einstein AH structure on compact manifolds are exactly as in the Einstein Weyl case. Such results depend fundamentally on the conservation condition; since covariant derivatives of densities are insensitive to the cubic torsion, the latter has no effect on the resulting formulas. In particular, Theorem $3.3$ of \cite{Calderbank-faraday} and the results of Section $6$ of \cite{Calderbank-faraday} carry over to AH structures with minor modifications, although the proofs given here are superficially different than those in \cite{Calderbank-faraday}. The principal result is Theorem \ref{negexacttheorem}.

\begin{lemma}
If $(\en, [h])$ is a Riemannian signature Einstein AH structure on a compact $n$-dimensional manifold and $h \in [h]$ is a Gauduchon metric with associated Faraday primitive $\ga_{i}$, then there hold
\begin{align}
\label{ne3}&\lap_{h}|\ga|^{2}_{h} = \tfrac{1}{2}|d\ga|_{h}^{2} - \tfrac{2}{n}\uR_{h}|\ga|_{h}^{2},& \\
\label{ne4}&n\lap_{h}\log |\ga| \geq -\uR_{h}, & & (\text{wherever}\,\, |\ga|^{2}_{h} > 0),\\
\label{scalar3}&\lap_{h}\uR_{h} - 2(n-4)|\ga|^{2}\uR_{h} = -2n(n-4)|D\ga|^{2}. 
\end{align}
\end{lemma}

\begin{proof}
Because by \eqref{confein1}, $D_{(i}\ga_{j)} = 0$ and $\ga_{p}\bt_{ij}\,^{p} = 0$, there follows from the Ricci identity and \eqref{confscal},
\begin{align*}
\begin{split}
\lap_{h}|\ga|_{h}^{2} &  = 2|D\ga|_{h}^{2} + 2\ga^{\sharp\,p}D^{i}D_{i}\ga_{p} = \tfrac{1}{2}|d\ga|_{h}^{2}  - 2\ga^{\sharp\,p}\ga^{\sharp\,q}\sR_{pq} = \tfrac{1}{2}|d\ga|_{h}^{2} - \tfrac{2}{n}\uR_{h}|\ga|_{h}^{2},
\end{split}
\end{align*}
which shows \eqref{ne3}. By Lemma \ref{katolemma} there holds the refined Kato inequality, $\tfrac{1}{4}|d\ga|_{h}^{2} = |D\ga|^{2}_{h} \geq 2|d|\ga||_{h}^{2}$ wherever $|\ga|^{2}_{h} \neq 0$. With \eqref{ne3} this implies that wherever $|\ga|_{h}^{2} \neq 0$ there holds 
\begin{align*}
\begin{split}
\lap_{h}\log|\ga|_{h}^{2} &= |\ga|_{h}^{-2}\lap_{h}|\ga|_{h}^{2} - |\ga|^{-4}_{h} = -\tfrac{2}{n}\uR_{h} + |\ga|_{h}^{-2}\left(\tfrac{1}{2}|d\ga|^{2}_{h} - 4|d|\ga||^{2}_{h} \right) \geq -\tfrac{2}{n}\uR_{h},
\end{split}
\end{align*}
which shows \eqref{ne4}. By \eqref{confein3}, $\uR_{h} + n(n-4)|\ga|^{2}$ is constant, and in \eqref{ne3} this yields \eqref{scalar3}.
\end{proof}

Theorem \ref{negexacttheorem} is the partial generalization to AH structures of Theorems $4.7-4.9$ of \cite{Calderbank-Pedersen}. 

\begin{theorem}\label{negexacttheorem}
If $(\en, [h])$ is a Riemannian signature Einstein AH structure on a compact $n$-manifold then there holds one of the following mutually exclusive possibilities.
\begin{enumerate}
\item $R$ is negative and $\nabla$-parallel and $(\en, [h])$ is exact.
\item $R = 0$ and $(\en, [h])$ is closed. The Faraday primitive associated to a Gauduchon metric is parallel with respect to the Gauduchon metric. The Gauduchon metric has non-negative Ricci curvature. Any $h$-harmonic one-form is parallel and the first Betti number $b_{1}$ of $M$ satisfies $b_{1} \leq n$. There holds $b_{1} = n$ if and only if any Gauduchon metric is flat and $M$ is diffeomorphic to a torus. %
If $n = 2$ and $(\en, [h])$ is not exact then $(\en, [h])$ is Weyl, and a Gauduchon metric is flat. If $n > 2$ and $(\en, [h])$ is not exact, then $b_{1} = 1$, the the Ricci curvature of the Gauduchon metric is positive definite on the kernel of $\ga$ and has positive scalar curvature, and the universal cover of $M$ equipped with the pullback of a Gauduchon metric is isometric to a product metric on $\rea \times N$ where $N$ is simply connected and compact and equipped with a metric $g$ which has positive Ricci curvature; if $3 \leq n \leq 4$, then $N$ is diffeomorphic to $S^{n-1}$. There is induced on $N$ an exact Riemannian Einstein AH structure which has positive weighted scalar curvature, and for which the induced metric $g$ is a distinguished metric.
\item $R > 0$ and either $(\en, [h])$ is exact with parallel scalar curvature or $(\en, [h])$ is not closed and its scalar curvature is not parallel. In either case a Gauduchon metric has positive Ricci curvature and $M$ has finite fundamental group.
\item $R$ is somewhere positive and somewhere non-positive, $n \leq 3$, $(\en, [h])$ is not closed, and for a Gauduchon metric $h \in [h]$ the quantity $\uR_{h} + n(n-4)|\ga|_{h}^{2}$ is a non-positive constant.
\end{enumerate}
\end{theorem}

\begin{proof}
Let $h \in [h]$ be a Gauduchon metric with Levi-Civita connection $D$ and associated Faraday primitive $\ga_{i}$. Integrating $\mu \ga^{\sharp\,i}\nabla_{i}R = \ga^{\sharp\,i}D_{i}\uR_{h} + 2|\ga|_{h}^{2}\uR_{h}$ by parts and using \eqref{einsteinf} gives $2\int_{M} \mu \ga^{\sharp\,i}\nabla_{i}R\,d\vol_{h} = 4\int_{M}|\ga|_{h}^{2}\uR_{h} \, d\vol_{h} = n\int_{M}|F|_{h}^{2}$, from which it is evident that if $\nabla_{i}R = 0$ then $F_{ij} = 0$ (the point is this holds also when $n = 4$); the converse is immediate from the conservation condition. It follows immediately from the definition and Lemma \ref{fcoclosedlemma} that a Riemannian signature Einstein AH structure on a compact $n$-manifold is closed if and only if its scalar curvature is parallel; the preceeding shows this is true when $n = 4$ as well. By Lemma \ref{parallelexactlemma}, if $(\en, [h])$ is Riemannian it is closed if and only if either $(\en, [h])$ is proper and exact, or $R$ is identically zero. This shows that for a Riemannian signature Einstein AH structure on a compact $n$-manifold there holds one of the following mutually exclusive possiblities:
\begin{itemize}
\item It is proper and exact with parallel scalar curvature.
\item Its scalar curvature is identically zero and it is closed.
\item It is not closed, and its scalar curvature is not parallel.
\end{itemize}
If $R$ is non-positive, then by \eqref{ne3} the function $|\ga|^{2}_{h}$ is subharmonic, so by the maximum principle is constant. In \eqref{ne3} this implies that either $R$ is identically $0$ and $d\ga = 0$, which with \eqref{confein1} shows that $D\ga = 0$, or that $\ga$ is identically $0$ and $R$ is somewhere negative. In the latter case $R$ must be parallel, and so $R$ must be everywhere negative. Thus if $R$ is non-positive there holds either $(1)$ or the first part of $(2)$. By the trichotomy established above, if $(\en, [h])$ does not satisfy either $(1)$ or $(2)$ then it must be that $(\en, [h])$ either is exact with positive parallel scalar curvature or $(\en, [h])$ is not closed with scalar curvature which is not parallel. If $R$ is positive, then by \eqref{confscal} and \eqref{confein2} the Ricci curvature $\sR_{ij}$ of a Gauduchon metric is positive definite, so by Myers' theorem $M$ has finite fundamental group. There remain to show that in the last case if $R$ is neither positive nor non-negative then it must be that there holds $(4)$, and the structural claims for the $R \equiv 0$ case.

If $n \geq 4$ then $-2(n-4)|\ga|^{2} \leq 0$, so by the maximum principle applied to \eqref{scalar3}, $\uR_{h}$ cannot have a non-positive minimum unless $\uR_{h}$ is constant. In particular, if $n \geq 4$ and $R$ is not everywhere non-positive then it is positive. The only remaining possibility is that $n \leq 3$, $(\en, [h])$ is not exact, and $R$ is somewhere positive and somewhere non-positive. By \eqref{confein3}, $\uR_{h} + n(n-4)|\ga|_{h}^{2}$ is a constant; since $n(n-4) < 0$ and the minimum of $\uR_{h}$ is non-positive, at a point where $\uR_{h}$ attains its minimum, $\uR_{h} + n(n-4)|\ga|_{h}^{2}$ must be non-positive.

Now suppose $R \equiv 0$. Since $\ga$ is parallel and $T_{ij} = 0$, the assumption \eqref{bochnerass} of Theorem \ref{bochnertheorem} is satisfied. This yields the first part of $(2)$. Suppose $(\en, [h])$ is not exact. Since $\ga$ is parallel it is nowhere vanishing and $|\ga|^{2}_{h}$ is a positive constant. Let $\tilde{M}$ be the universal cover of $M$ and $\rho:\tilde{M} \to M$ the covering projection. As in Theorem \ref{ivanovtheorem}, every harmonic one-form is a multiple of $\ga$ and $\rho^{\ast}(h)$ is isometric to a product metric on $\rea \times N$, where $N$ is simply-connected with a complete metric $g$. If $n = 2$ then by Lemma \ref{negexacttheorem} there holds $0 = |\bt|_{h}^{2}|\ga|_{h}^{2}$, so $|\bt|^{2}_{h} = 0$, and from \eqref{confscal} there follows $\sR = 0$, so that $(\en, [h])$ is Weyl and a Gauduchon metric is flat. If $n \geq 3$ then $(n-1)(n-2)|\ga|_{h}^{2} > 0$, so by \eqref{confscal} a Gauduchon metric has positive scalar curvature. By \eqref{confric} and \eqref{confein1}, for any $Y \in \Ga(TM)$ there holds $4Y^{i}Y^{j}\sR_{ij} = |i(Y)\bt|^{2}_{h} + 2(n-2)|\ga^{\sharp}\wedge Y|_{h}^{2} \geq 0$, with equality if and only if $Y$ is a multiple of $\ga^{\sharp}$, and so the Ricci curvature of $h$ is non-negative and its restriction to the kernel of $\ga$ is positive. The metric $\tilde{g}_{ij}$ on $N$ is the restriction to $N$ of the pullback to $N$ of $g_{ij} = h_{ij} - |\ga|^{-2}_{h}\ga_{i}\ga_{j}$. Because the splitting $\tilde{M}\simeq \rea \times N$ is isometric, the Ricci curvature $\tilde{\sR}_{ij}$ of $\tilde{g}$ equals the restriction to $N$ of the Ricci curvature of $\rho^{\ast}(h)$. By what was just shown $\tilde{\sR}_{ij}$ is positive definite on $N$. Because by the de Rham decomposition theorem the induced metric $\tilde{g}$ is complete, by Myers' theorem, $N$ is compact. In particular, if $3 \leq n \leq 4$ then $N$ must be a sphere, as in Theorem \ref{ivanovtheorem}.

Because $\rho$ is a local diffeomorphism the pullback $\rho^{\ast}(\nabla)$ is defined. Henceforth it will serve the interest of readability to indicate the pullbacks via $\rho$ of tensors on $M$ and the original tensors by identical notation. For example there will be written $\nabla$ in lieu of $\rho^{\ast}(\nabla)$, and there will be written $\bt_{ij}\,^{k}$; what the latter really means is $\rho^{\ast}(L)_{ijp}\rho^{\ast}(h)^{kp}$ where $L_{ijk} \defeq \bt_{ij}\,^{p}h_{pk}$. On $\tilde{M}$ define $\tnabla = \nabla - h_{ij}\ga^{\sharp\,k} = D - \tfrac{1}{2}\bt_{ij}\,^{k} - 2\ga_{(i}\delta_{j)}\,^{k}$. If $X^{i}\ga_{i} = 0$ then $\ga_{p}\tnabla_{i}X^{p} = \ga_{p}D_{i}X^{p} = -X^{p}D_{i}\ga_{p} = 0$, so that $\tnabla$ preserves $\ker \ga$ and hence induces a connection on $N$, to be denoted also $\tnabla$. There holds $\tnabla_{i}\ga_{j} = 2\ga_{i}\ga_{j}$ and it follows that 
\begin{align}\label{ng1}
\tnabla_{i}g_{jk} = \bt_{ij}\,^{p}h_{pk} + 2\ga_{i}g_{jk} + 2\ga_{(j}g_{k)i}
\end{align}
 on $\tilde{M}$. Since $\ga_{p}\bt_{ij}\,^{p} = 0$, $\bt_{ij}\,^{k}$ restricts to a tensor on $N$, to be denoted $\bt_{IJ}\,^{K}$. Here temporarily captial Latin indices are used to decorate tensors on $N$. For instance, \eqref{ng1} becomes $\tnabla_{I}g_{JK} = \bt_{IJ}\,^{Q}g_{QK}$, since by construction $\ga_{I} = 0$. This shows that the CP pair $([\tnabla], [g])$ generated on $N$ by $\tnabla$ and $g$ is an exact AH structure with cubic torsion $\bt_{IJ}\,^{K}$. Let $\tilde{D}$ be the Levi-Civita connection of $g_{IJ}$; because $\ga_{i}$ is $D$-parallel, $\tilde{D}$ is obtained from $D$ by orthogonal projection onto $N$ along $\rea$. Because the splitting $\tilde{M} = \rea \times N$ is isometric, the identity $D_{p}\bt_{ij}\,^{p} = 0$ implies $\tilde{D}_{Q}\bt_{IJ}\,^{Q} = 0$, and so by \eqref{ddivbt} and the exactness of $(\ten, [g])$, the curvature $\tilde{E}_{IJ}$ of $(\ten, [g])$ is zero. By \eqref{confric} there holds $\sR_{ij}  = \tfrac{1}{4}\bt_{ij} + (n-2)|\ga|^{2}_{h}g_{ij}$. The Ricci curvature $\tilde{\sR}_{IJ}$ of $g_{IJ}$ is the restriction to $N$ of $\sR_{ij}$ and so, by \eqref{confric}, the curvature $\tilde{T}_{IJ}$ of $(\ten, [g])$ satisfies $\tilde{T}_{IJ} = \tilde{\sR}_{IJ} - \bt_{IJ} = (n-2)|\ga|_{h}^{2}g_{IJ}$. This shows $(\ten, [g])$ is naive Einstein and exact with positive weighted scalar curvature $(n-1)(n-2)|\ga|_{h}^{2}|\det g|^{1/(n-1)}$. Because $\tnabla |\det g| = 0$, the scalar curvature is $\tnabla$ parallel; as $(\ten, [g])$ is exact with parallel scalar curvature, it is Einstein. 
\end{proof}

\subsubsection{}
There will be needed the following theorem of M. Eastwood and P. Tod, \cite{Eastwood-Tod}.
\begin{theorem}[\cite{Eastwood-Tod}]\label{eastwoodtodtheorem}
If $n\geq 3$, a Riemannian signature Einstein Weyl structure $(\en, [h])$ for which $[h]$ is locally conformally flat is closed. 
\end{theorem}
To prove Theorem \ref{eastwoodtodtheorem} Eastwood-Tod write the Einstein Weyl equations in terms of the one-form $\ga_{i}$ determined by the difference tensor of the Levi-Civita connection $D$ of an arbitrary choice of gauge $h \in [h]$ and the aligned representative $\nabla \in \en$. Differentiating these equations they obtain a closed system of linear partial differential equations in $\ga_{i}$, $F_{ij}$, and some tensors derived from them. Suprisingly involved algebra then yields $F_{ij} = 0$. %

\begin{corollary}
If $n \geq 3$ a proper Riemmannian signature Einstein Weyl structure $(\en, [h])$ for which $[h]$ is locally conformally flat is exact and a distinguished metric has constant sectional curvature.
\end{corollary}

\begin{proof}
By Theorem \ref{eastwoodtodtheorem}, $(\en, [h])$ is closed, and so by Theorem \ref{negexacttheorem} if it is not exact (in which case the claim is proved) then it has scalar curvature which is identically zero. 
\end{proof}

\section{Riemannian Einstein AH structures with self-conjugate curvature}\label{structuretheoremsection}
The goal of this section is to prove Theorem \ref{cubictorsiontheorem}, which is a sort of structural theorem for Riemannian signature Einstein AH structures with self-conjugate curvature on a compact manifold satisfying a certain geometric condition. The theorem and its proof are based directly on the description of affine hyperspheres, in particular the Bernstein type results for improper affine hyperspheres due to K. J\"orgens in \cite{Jorgens}; E. Calabi in \cite{Calabi-improper}, and finally A. V. Pogorelov in \cite{Pogorelov}; and the estimates of the growth of the cubic form of elliptic and hyperbolic affine hyperspheres made by Calabi in \cite{Calabi-completeaffine}.
All these results can be reduced to a maximum principle argument relying on a differential inequality for the Laplacian of the square of the norm of a completely symmetric tensor satisfying a Codazzi condition which is deduced using Weitzenb\"ock type formulas and refined Kato inequalities for such tensors, in a manner which follows a program probably first outlined in \cite{Calabi-bernsteinproblems}, and since realized in diverse contexts, of which \cite{Cheng-Yau-maximalspacelike} or \cite{Schoen-Simon-Yau} are perhaps representative.

While the results are needed only for completely trace-free symmetric cubic tensors, it is equally straightforward, and perhaps conceptually clarifying, to work them out for such tensors of any rank. In the process there are obtained some vanishing theorems for conformal Killing tensors and trace and divergence free Codazzi tensors. It would be unsurprising were these results already known to experts. They are analogous to the somewhat stronger vanishing theorems for symmetric tensors on K\"ahler manifolds obtained by S. Kobayashi in \cite{Kobayashi-holomorphicsymmetric} and \cite{Kobayashi-holomorphictensor}.

Throughout this section, in departure from the convention in effect in the rest of the paper, indices are raised and lowered using the metric $h_{ij}$, when such is given.

\subsection{Symmetric tensor algebra}\label{symmetrictensoralgebrasection}
The section begins with some background on symmetric tensors and harmonic polynomials, some of which will not be need until section \ref{cubicformsection}.

\subsubsection{}\label{symmetricharmonictensorsection}
In this section $\ste$ is a real $n$-dimensional vector space. Let $D$ be the standard flat affine connection on $\ste$ and let $x^{i}$ be global coordinates on $\ste$ such that the differentials $dx^{i}$ constitute a $D$-parallel coframe. A function on $\ste$ is a \textbf{polynomial} of \textbf{degree} $k$ if it is in the kernel of $D^{(j)}\defeq D \dots D$ ($j$ times) for some $j \geq 1$, and $k+1$ is the minimum such $j$. Let $\pol^{k}(\ste)$ be the vector space of degree $k$ polynomials which are also homogeneous (necessarily of degree $k$), and let $\pol(\ste) = \oplus_{k \geq 0}\pol^{k}(\ste)$ be the graded algebra of polynomial functions on $\ste$. Let $S(\sted) \defeq \oplus_{k \geq 0}S^{k}(\sted)$ be the graded vector space of finite linear combinations of completely symmetric covariant tensors on $\ste$. The vector spaces $\pol^{k}(\ste)$ and $S^{k}(\sted)$ are canonically isomorphic. The isomorphism and its inverse are given explicitly by
\begin{align*}
&\om \in S^{k}(\sted) \to P^{\om} \in \pol^{k}(\ste),& & P^{\om}(x) \defeq \om_{i_{1}\dots i_{k}}x^{i_{1}}\dots x^{i_{k}},\\
& P \in \pol^{k}(\ste) \to \om^{P} \in S^{k}(\sted),& &\om^{P} \defeq \tfrac{1}{k!}D^{(k)}P.
\end{align*}
Equip $S(\sted)$ with the structure of a graded algebra by pulling back the algebra structure on $\pol(\ste)$. That is, $P^{\al}P^{\be} = P^{\al\sprod \be}$, so that, for $\al \in S^{k}(\sted)$ and $\be \in S^{l}(\sted)$, the product $\al \sprod \be$ is simply the symmetrized tensor product $(\al \sprod \be)_{i_{1}\dots i_{k+l}} = \al_{(i_{1}\dots i_{k}}\be_{i_{k+1}\dots i_{k+l})}$.

\subsubsection{}
Let $h_{ij}$ be a constant non-degenerate symmetric tensor on $\ste$, or, what is the same, a pseudo-Riemannian metric parallel with respect to the standard flat affine connection $D$. Let $\lap$ be the $h$ Laplacian $D^{i}D_{i}$. Define $\tr(\om)_{i_{1} \dots i_{k-2}} = \om_{i_{1}\dots i_{k-2}p}\,^{p}$. Let $S^{k}_{0}(\sted) = S^{k}(\sted) \cap \ker \tr$ be the trace-free elements of $S^{k}(\sted)$. Then $\lap^{i} P^{\om} = k(k-1)\dots (k-2i+1)P^{\tr^{i}(\om)}$; in particular, $\om \in S^{k}_{0}(\sted)$ if and only if $\lap P^{\om} = 0$. Let $\har^{k}(\ste) \defeq \ker \lap \cap \pol^{k}(\ste)$ be the subspace of degree $k$ \textit{harmonic} (or \textit{$h$-harmonic}, for clarity) polynomials.  This terminology is used even for $h$ of indefinite signature. %
 
Observe that $\lb DP^{h}, DP^{\om}\ra = 2dP^{\al}(\eul)$ in which $\lb \dum, \dum \ra$ is the pairing induced by $h$, and $\eul$ is the radial vector field generating dilations by $e^{t}$, so $dP^{\al}(\eul) = kP^{\al}$ if $\al \in S^{k}(\sted)$. Induction shows $\lap (P^{h})^{i} = 2i(n+2(i-1))(P^{h})^{i-1}$. Using these observations it is straightforward to check that the harmonic part $\har(P^{\om})$ of $P^{\om}$ is
\begin{align}\label{harpart}
\begin{split}
\har(P^{\om}) &= P^{\om} - \sum_{i \geq 1} \tfrac{1}{2^{i}i!(n+2(k-2))\dots (n+2(k-i-1))}(P^{h})^{i}\lap^{i}P^{\om}.
\end{split}
\end{align}
Define $\met(\al) = h \sprod \al$. In terms of tensors \eqref{harpart} gives
\begin{align}
\begin{split}\label{tfom}
\tf(\om) & = \om - \sum_{i \geq 1}\tfrac{ k(k-1)\dots (k-2i+1)}{2^{i}i!(n+2(k-2))\dots (n+2(k-i-1))}\met^{i}(\tr^{i}(\om)).
\end{split}
\end{align}
For $\al \in S^{k}(\sted)$ and $\be \in S^{l}(\sted)$ define for $0 \leq j \leq \min\{k, l\}$,
\begin{align}
\con^{j}(\al, \be)_{i_{1}\dots i_{k+l-2j}} = \al_{q_{1}\dots q_{j}(i_{1}\dots i_{k-j}}\be_{i_{k-j+1}\dots j_{k+l-2j})}\,^{q_{1}\dots q_{j}}. 
\end{align}
In particular, if $k = l$ then $\con^{k}(\al, \be) = \lb \al, \be \ra$, and the conventions are $\con^{0}(\al, \be) = \al \sprod \be$ and $\con^{j}(\al, \be) = 0$ for $j > \min\{k, l\}$. Observe that $(k-j)!(l-j)!\lb D^{(j)}P^{\al}, D^{(j)}P^{\be}\ra = k!l!P^{\con^{j}(\al, \be)}$. In particular, for $\om \in S^{k}(\sted)$ there holds $|\hess P^{\om}|_{h}^{2} = k^{2}(k-1)^{2}P^{\con^{2}(\om, \om)}$. Expanding $\lap P^{\al\sprod \be} = \lap(P^{\al}P^{\be})$ yields
\begin{align}\label{tralbe}
\tbinom{k+l}{2}\tr(\al \sprod \be) = \tbinom{k}{2}\tr(\al)\sprod \be + \tbinom{l}{2}\al \sprod \tr(\be) + kl \con(\al, \be),
\end{align}
which is the $i = 0$ special case of
\begin{align}\label{conalbe}
\tbinom{k+l-2i}{2}\tr\con^{i}(\al, \be) = \tbinom{k-i}{2}\con^{i}(\tr\al,\be) + \tbinom{l-i}{2}\con^{i}(\al, \tr\be) + (k-i)(l-i) \con^{i+1}(\al, \be).
\end{align}
By an induction using \eqref{conalbe} it follows that if $\tr \al = 0 = \tr \be$ then 
\begin{align}\label{contr}
&2^{i}\tbinom{k+l-2i}{k-i}\con^{i}(\al, \be) = \tbinom{k+l}{k}\tr^{i}(\al \sprod \be), & &\al \in S^{k}_{0}(\sted), \be \in S^{l}_{0}(\sted)
\end{align}
and this implies
\begin{align}\label{powerlap}
&(k-i)!(l-i)!\lap^{i}P^{\al\sprod \be} = 2^{i}k!l!P^{\con^{i}(\al, \be)},& &\al \in S^{k}_{0}(\sted), \be \in S^{l}_{0}(\sted),
\end{align}
so that $\lap^{i}P^{\al\sprod \be} = 2^{i}\lb D^{(i)}P^{\al}, D^{(i)}P^{\be}\ra$. In particular 
\begin{align}\label{lapsquarepsquare}
\lap^{2}P^{\om \sprod \om} = 4|\hess P^{\om}|^{2}_{h},& &\text{for}\,\, \om \in S^{k}_{0}(\sted).
\end{align}
From $\con(h, \al) = \al$, $\met(h) = n$, and \eqref{tralbe} there results for $\al \in S^{k}(\sted)$ the commutation identity
\begin{align}\label{trhcommute}
\tbinom{k+2}{2}\tr \met(\al) = \tbinom{k}{2}\met(\tr \al) + (n+2k)\al.
\end{align}
Define operators $E$, $F$, and $H$ on $S^{k}(\sted)$ by $E(\al) = -\tfrac{k+1}{2}\met(\al)$, $F(\al) = \tfrac{k}{2}\tr(\al)$, and $H(\al) = (\tfrac{n}{2} + k)\al$. Then $[E, F] = H$, $[H, E] = 2E$, and $[H, F] = -2F$, so $\{H, E, F\}$ generate an action of $\sll(2, \rea)$ on $S(\sted)$. Using this observation, most of the claims made in this section and the next could be deduced from the representation theory of $\sll(2, \rea)$.

\subsubsection{} 
Let $S_{0}(\sted) = \oplus_{k \geq 0}S^{k}_{0}(\sted)$ which is a graded subspace of $S(\sted)$. The map $\tf:S(\sted) \to S_{0}(\sted)$ sending $\al \in S^{k}(\ste)$ to its trace-free part is a graded linear projection. If $\al \in S^{k}_{0}(\sted)$ then \eqref{trhcommute} shows $\al = \tfrac{(k+2)(k+1)}{2(n+2k)}\tr \met(\al)$. Hence $\tr^{i}\met(\al) =0$ for $i \geq 2$, and \eqref{tfom} gives $(\Id - \tf)(\met(\al)) = \met(\al)$. From this observation it is straightforward to deduce that the ideal $\ideal_{h}$ in $(S(\sted), \sprod)$ generated by $h$ is equal to $\ker \tf$, so there is a commutative algebra structure $\cprod$ on $S_{0}(\sted)$ induced via $\tf$ from that on $S(\sted)/\ideal_{h}$. This means that by definition $ \tf(\al)\cprod \tf(\be) = \tf(\al\sprod \be)$ for $\al \in S^{k}(\sted)$ and $\be \in S^{l}(\sted)$. That this is well-defined amounts to the easily verified identity $\tf(\tf(\al)\sprod \tf(\be)) = \tf(\al\sprod \be)$. For $\al \in S^{k}_{0}(\sted)$ and $\be \in S^{l}_{0}(\sted)$ this product $\al \cprod \be = \tf(\al \sprod \be )$ is called the \textbf{Cartan product}. It is straightforward, but not terribly useful, to write down an explicit formula for the Cartan product. For example, for $X \in \sted$ and $\om \in S^{k}_{0}(\sted)$,
\begin{align}
(X\cprod \om)_{i_{1}\dots i_{k+1}} = X_{(i_{1}}\om_{i_{2}\dots i_{k+1}} - \tfrac{k}{n+2(k-1)}h_{(i_{1}i_{2}}\om_{i_{3}\dots i_{k+1})p}X^{p}.
\end{align}

\subsection{Differential operators on trace-free symmetric tensors}\label{differentialoperatorssymmetricsection}

\subsubsection{}
If $E$ is a bundle of tensors on $M$, a metric $h_{ij}$ determines a pairing $\int_{M}\lb \al, \be \ra$ of sections $\al, \be \in \Ga(E)$, at least one of which is compactly supported, by integration of the complete contraction $\lb \al, \be \ra \defeq \al^{i_{1}\dots i_{k}}\be_{i_{1}\dots i_{k}}$. If $E$ and $F$ are bundles of tensors and $\Q:\Ga(E) \to \Ga(F)$ is a differential operator of order $p$, the formal adjoint $\Q^{\ast}$ of $\Q$ is defined to be the unique differential operator mapping $\Ga(F) \to \Ga(E)$ and satisfying $\int_{M}\lb \al, \Q(\be)\ra = (-1)^{p}\int_{M}\lb \Q^{\ast}(\al), \be \ra$.

\subsubsection{}
For any bundle $E$ of covariant tensors on $M$ the divergence operator $\div:\Ga(\ctm \tensor E) \to \Ga(E)$ associated to the metric $h$ is by definition the formal adjoint $D^{\ast}$ of the covariant derivative $D$ with respect to the pairing of sections determined by integration. Explicitly $\div(\om)_{i_{1}\dots i_{k}} \defeq D^{p}\om_{pi_{1}\dots i_{k}}$ for $\om \in \Ga(E)$.

\subsubsection{}
Because the fibers of $TM$ and $\ctm$ carry canonically dual flat centroaffine structures, when $M$ is equipped with a metric $h$, there carry over unchanged to $TM$ and $\ctm$ the constructions of section \ref{symmetrictensoralgebrasection} which do not involve infinite sums. By definition $S(TM)$ (resp $S_{0}(TM)$) is the graded algebra comprising finite linear combinations of completely (trace-free) symmetric tensors with the fiberwise multiplication $\sprod$ (resp. $\cprod$). By definition $\pol(\ctm)$ is the graded subalgebra of $\cinf(\ctm)$ comprising functions polynomial in the fibers of $\ctm$ and of globally bounded degree (it is this last condition which requires comment, at least when $M$ be not necessarily compact). One writes $\pol(\ctm) = \oplus_{k \geq 0}\pol^{k}(\ctm)$, $S_{0}(TM) = \oplus_{k\geq 0}\Ga(S^{k}_{0}(TM))$, etc., for the decompositions of these algebras into their graded pieces. 
 
Regarding $\pol(\ctm)$ as a subspace of $\cinf(\ctm)$ it acquires a Poisson structure from the tautological Poisson structure on $\ctm$. The result of tranporting this Poisson structure to $S(TM)$ is sometimes called the \textbf{symmetric Schouten bracket} and is defined, for any $X \in \Ga(S^{k}(TM))$ and $Y \in \Ga(S^{l}(TM))$, and any torsion-free affine connection $\nabla$ by
\begin{align}\label{schoutendefined}
\{X, Y\}^{i_{1}\dots i_{k+l-1}} = k X^{p(i_{1}\dots i_{k-1}}\nabla_{p}Y^{i_{k}\dots i_{k+l-1})} - l Y^{p(i_{1}\dots i_{l-1}}\nabla_{p}X^{i_{l}\dots i_{k+l-1})}.
\end{align}
If a Riemannian metric $h_{ij}$ is given then the connection in \eqref{schoutendefined} may be taken to be the Levi-Civita connection $D$ of $h_{ij}$, and there holds $\{h, X\} = 2D^{(i_{1}}X^{i_{2}\dots i_{k+1})}$ for any $X \in \Ga(S^{k}(TM))$ (in the bracket $\{h, X\}$ the notation $h$ refers to the dual bivector $h^{ij}$). Using $h$, $S(TM)$ and $S(\ctm)$ are identified by index raising and lowering, and for $X \in \Ga(S^{k}(TM))$ and $\om \in \Ga(S^{k}(\ctm))$ one defines $X^{\flat}_{i_{1}\dots i_{k}} = X^{j_{1}\dots j_{k}}h_{i_{1}j_{1}}\dots h_{i_{k}j_{k}} \in \Ga(S^{k}(\ctm))$ and defines $\om^{\sharp} \in \Ga(S^{k}(TM))$ dually. Since index raising and lowering induce symmetric algebra isomorphisms, there results on $S(\ctm)$ the Poisson bracket $\{\al, \be\}_{h}$ defined by $\{\al, \be\}_{h} = \{\al^{\sharp}, \be^{\sharp}\}^{\flat}$. The subscript $h$ is included to emphasize the dependence on $h$ of the bracket on $S(\ctm)$; this matters when a conformal change of metric is made.

\subsubsection{}
If $h_{ij}$ is a Riemannian metric with Levi-Civita connection $D$, Let $\tf:\Ga(\symk) \to \Ga(\symkt)$ be the $h$-orthogonal projection onto the completely trace-free part and define $\clie:\Ga(\symk) \to \Ga(\symkp)$ by $2\clie(\om) = \tf \{h, \om\}_{h}$. Explicitly, for $\om \in \Ga(\symk)$,
\begin{align}\label{cliedefined}
\clie(\om)_{i_{1}\dots i_{k+1}} = D_{(i_{1}}\om_{i_{2}\dots i_{k+1})} - \tfrac{1}{n+2(k-1)}\left(kh_{(i_{1}i_{2}}D^{p}\om_{i_{3}\dots i_{k+1})p} + \tbinom{k}{2}h_{(i_{1}i_{2}}D_{i_{3}}\om_{i_{4}\dots i_{k+1})p}\,^{p}\right).
\end{align}
It is easily checked that $\clie$ is the formal adjoint of the composition $\div \circ \tf$ with respect to the pairing of sections of $\symkp$ determined by integration. This means that for sections $\al$ of $\symk$ and $\be$ of $\symkp$, at least one of which has compact support, there holds $\int_{M}\lb \clie(\al), \be \ra \,d\vol_{h} = -\int \lb \al , \div \tf(\be) \ra\,d\vol_{h}$. If $X \in \Ga(TM)$, then $2\clie(X^{\flat})_{ij} = 2D_{(i}X_{j)} - \tfrac{2}{n}D_{p}X^{p} h_{ij} =\tf(\lie_{X}h)_{ij}$, which motivates the notation resembling that for the Lie derivative.

\subsubsection{}
For $\om \in \Ga(\symkt)$ define
\begin{align}\label{kliedefined}
\begin{split}
\klie(\om)_{ij i_{1}\dots i_{k-1}} &\defeq D_{[i}\om_{j]i_{1}\dots i_{k-1}} - \tfrac{1}{n-3+k}\sum_{s = 1}^{k-1}h_{i_{s}[i}D^{p}\om_{j]i_{1}\dots \hat{i}_{s}\dots i_{k-1}p}\\
& = D_{[i}\om_{j]i_{1}\dots i_{k-1}} - \tfrac{k-1}{2(n-3+k)}\left(h_{i(i_{1}}D^{p}\om_{i_{2}\dots i_{k-1})jp} -h_{j(i_{1}}D^{p}\om_{i_{2}\dots i_{k-1})ip}  \right),
\end{split}
\end{align}
which is the completely trace-free part of $ D_{[i}\om_{j]i_{1}\dots i_{k-1}}$. Verifying the equality of the two different expressions for the trace part of \eqref{kliedefined} is straightforward. The operator $\klie$ is a special case of the operator defined in \eqref{colanczosdefined}.

\subsubsection{}
The operator $\klie$ maps $\Ga(\symkt)$ into trace-free $(k+1)$-tensors having the symmetries determined by the Young projector given by symmetrization over the rows followed by anti-symmetrization over the columns of the Young diagram corresponding to the partition $(n1)$. Thus $\klie(\om)_{iji_{1}\dots i_{k-1}} = \klie(\om)_{[ij]i_{1}\dots i_{k-1}}$, $\klie(\om)_{iji_{1}\dots i_{k-1}} = \klie(\om)_{ij(i_{1}\dots i_{k-1})}$, and there vanishes the skew-symmetrization of $\klie(\om)_{iji_{1}\dots i_{k-1}}$ over $ij$ and any $i_{s}$. The formal adjoint $\kliea$ of $\klie$ takes values in $\Ga(\symkt)$ and is defined by $\int_{M}\phi^{iji_{1}\dots i_{k-1}}\klie(\om)_{iji_{1}\dots i_{k-1}} = -\int_{M}\kliea(\phi)_{i_{1}\dots i_{k}}\om_{i_{1}\dots i_{k}}$; here $\om \in \Ga(\symkt)$, $\phi_{iji_{1}\dots i_{k-1}} = \phi_{[ij]i_{1}\dots i_{k-1}}$, and $\phi_{iji_{1}\dots i_{k-1})} = \phi_{ij(i_{1}\dots i_{k-1})}$. Explicit computation shows that $\kliea(\phi)_{i_{1}\dots i_{k}} = D^{p}\phi_{p(i_{1}\dots i_{k})}$, so that $\kliea\klie(\om)_{i_{1}\dots i_{k}} = D^{p}\klie(\om)_{p(i_{1}\dots i_{k})}$.

\subsubsection{}
For $\om \in \Ga(\symkt)$ define
\begin{align}\label{tliedefined}
\tlie(\om)_{ii_{1}\dots i_{k}} \defeq \tfrac{2k}{k+1}\klie(\om)_{i(i_{1}\dots i_{k})},
\end{align}
which is completely trace-free and satisfies $\tlie(\om)_{i(i_{1}\dots i_{k})} = \tlie(\om)_{ii_{1}\dots i_{k}}$ and $\tlie(\om)_{(i_{1}\dots i_{k+1})} = 0$. Using that there vanishes the skew-symmetrization of $\klie(\om)_{iji_{1}\dots i_{k-1}}$ over the first two indices and any other index it can be verified that $\tlie(\om)_{[ij]i_{1}\dots i_{k-1}} = \klie(\om)_{iji_{1}\dots i_{k-1}}$. To the Young diagram determined by a partition is associated the Young projector given by symmetrization over the rows followed by anti-symmetrization over the columns. The opposite Young projector associated to the same Young diagram is given by anti-symmetrization over the rows followed by symmetrization over the columns. The representations determined by the Young projector associated to a partition and the opposite Young projector associated to the conjugate partition are isomorphic. Writing explicitly such an isomorphism for the Young projector associated to the partition $(n1)$ and the opposite Young projector associated to the conjugate partition $(21\dots 1)$ yields $\tlie(\om)$. 

The formal adjoint $\tliea$ is defined analogously to $\kliea$. If $\phi_{ii_{1}\dots i_{k}} = \phi_{i(i_{1}\dots i_{k})}$ and $\phi_{(ii_{1}\dots i_{k})} = 0$ let $\Psi(\phi)_{iji_{1}\dots i_{k-1}} = \tfrac{2k}{k+1}\phi_{[ij]i_{1}\dots i_{k-1}}$. Using that $\Psi(\phi)_{i(i_{1}\dots i_{k})} =  \phi_{ii_{1}\dots i_{k}}$ and the explicit expression for $\kliea$, it can be verified that $\kliea(\Psi(\phi)) = \div(\phi)$. From the definitions of formal adjoints it follows that $\tliea(\phi) = \kliea(\Psi(\phi))$, so $\tliea$ equals $\div$ (acting on the appropriate space of tensors). A straightforward computation shows that $\tliea \tlie(\om) = (2k/(k+1))\kliea\klie(\om)$.

\subsubsection{}\label{twoddivsection}
Let $\tilde{h}_{ij} = fh_{ij}$ be conformally related pseudo-Riemannian metrics. The Levi-Civita connections are related by $\tilde{D} = D + 2\si_{(i}\delta_{j)}\,^{k} - h_{ij}\si^{k}$ in which $2\si_{i} = d\log f_{i}$ and $\si^{i} = h^{ip}\si_{p}$. 
Write $\sbl_{\clie}(Z)(\phi)$ for the symbol of $\clie$ applied to the vector $Z^{i}$ and $\phi \in \Ga(\symkt)$, and similarly for $\klie$ and $\div$. The Cartan product arises naturally in the symbols of differential operators. For example, $\sbl_{\clie}(Z)(\phi) = Z\cprod \phi$. The transformations of $D\om$, $\klie(\om)$, etc. under a conformal change of metric are given by specializing \eqref{eomdiff} for a metric connection. For $\om \in \Ga(\symkt)$ there hold
\begin{align}
&\clie_{\tilde{h}}(\om) = \clie_{h}(\om) -2k\sbl_{\clie}(\si^{\sharp})(\om),&
&\klie_{\tilde{h}}(\om) = \klie_{h}(\om) + (1-k)\sbl_{\klie}(\si^{\sharp})(\om),
\end{align}
so that $\clie$ and $\klie$ are conformally invariant in the sense that for $0 < f \in \cinf(M)$ there hold
\begin{align}\label{lieconf}
&\clie_{\tilde{h}}(f^{k}\om) = f^{k}\clie_{h}(\om), && \klie_{\tilde{h}}(f^{(k-1)/2}\om) = f^{(k-1)/2}\klie_{h}(\om).
\end{align}
This means that $\clie$ and $\klie$ are conformally invariant if interpreted as operators on tensors of appropriate c-weights. In fact the operator $\K$ of Lemma \ref{eopinvariantlemma} is simply $\klie$ interpreted in this way, and this remark points to the reason for studying $\klie$ here.

Define $\clie^{\sharp}:\Ga(\symktv) \to \Ga(\symkptv)$ by $\clie^{\sharp}(X) = \clie(X^{\flat})^{\sharp}$.
Then \eqref{lieconf} is equivalent to the invariance $f\clie^{\sharp}_{\tilde{h}}(X) = \clie^{\sharp}_{h}(X)$, so that while $\lie^{\sharp}$ depends on $h$, the subspace $\ker \clie^{\sharp} \cap \Ga(\symktv)$ does not. Elements of $\ker \clie^{\sharp} \cap \Ga(\symktv)$ are \textbf{conformal Killing tensors} of rank $k$. Define $\ck(TM, [h]) = \oplus_{k \geq 0}\ker \clie^{\sharp} \cap \Ga(S^{k}_{0}(TM)) \subset \symat(TM)$ to be the subspace of $S_{0}(TM)$ comprising finite linear combinations of conformal Killing tensors. Except in ranks one and two conformal Killing tensors are not so well studied as are their skew-symmetric counterparts the conformal Killing forms, for which \cite{Semmelmann} is a good starting point. Probably their most natural occurence is as the symbols of symmetries of the Laplacian, for this and more background see M. Eastwood's \cite{Eastwood-laplacian}. Some other representative references are \cite{Woodhouse} and \cite{Rani-Edgar-Barnes}. Theorem \ref{conformalkillingtheorem} shows that $\ck(TM, [h])$ is a subalgebra of $S_{0}(TM)$ with respect to the Cartan multiplication. Although this will not be used in the sequel, it fits naturally into the discussion at this point. 
\begin{theorem}\label{conformalkillingtheorem}
Given a metric $h$ the associated operator $\clie$ is a derivation with respect to the Cartan multiplication on $S_{0}(\ctm)$, which means that for $\al \in \Ga(\symkt)$ and $\be \in \Ga(\symlt)$ there holds
\begin{align}\label{cprodleibniz}
\clie(\al \cprod \be) = \clie(\al)\cprod \be + \al \cprod \clie(\be).
\end{align}
The subspace $\ck(TM, [h]) \subset \symat(TM)$ is a subalgebra with respect to the Cartan multiplication.
\end{theorem} 
\begin{proof}
There is $\ga \in \Ga(S^{k+l-2}(\ctm))$ such that $(\Id - \tf)(\al \sprod \be) = \met(\ga)$, so $\tf \{h, \al \sprod \be - \tf(\al \sprod \be)\}_{h} = \tf \{h, h\sprod \ga\}_{h} = \tf( h\sprod \{h, \ga\}_{h}) = 0$. This gives the third equality in 
\begin{align*}
\begin{split}
2\clie(\al \cprod \be)&= \tf \{h, \al \cprod \be\}_{h}  = \tf \{h, \tf(\al\sprod \be)\}_{h} = \tf \{h, \al \sprod \be\}_{h} = \tf\left(\{h, \al\}_{h}\sprod \be + \al \sprod \{h, \be\}_{h}\right)\\
& = \tf (\{h, \al\}_{h})\cprod \be + \al \cprod \tf (\{h, \be\}_{h}) = 2\clie(\al)\cprod \be + 2\al \cprod \clie(\be),
\end{split}
\end{align*}
in which the penultimate equality holds because the graded linear map $\tf:(S(\ctm), \sprod) \to (S_{0}(\ctm), \cprod)$ is a homomorphism. This shows $\ck(\ctm, h) = \oplus_{k \geq 0}\ker \clie \cap \Ga(\symkt) \subset \symat(\ctm)$ is a subalgebra of $S_{0}(\ctm)$ under Cartan multiplicaiton. While $\ck(\ctm, h)$ depends on the choice of $h$, it is linearly isomorphic to $\ck(\ctm, e^{f}h)$ by the graded linear map sending $\om_{i_{1}\dots i_{k}}$ to $f^{k}\om_{i_{1}\dots i_{k}}$, and both are are identified with $\ck(TM, [h])$ via index raising, so $\clie^{\sharp}$ is a derivation of $S_{0}(TM)$ with kernel equal to the subalgebra $\ck(M, [h])$. 
\end{proof}

\subsubsection{}
Define $\ih:\Ga(\symkmt)\to \Ga(\ctm \tensor \symkt)$ by
\begin{align}
\ih(\om)_{ii_{1}\dots i_{k}} = \tfrac{k(n+2(k-2))}{(n-3 + k)(n+2(k-1))}h_{i(i_{1}}\om_{i_{2}\dots i_{k})} + \tfrac{k(1-k)}{(n-3 + k)(n+2(k-1))}h_{(i_{1}i_{2}}\om_{i_{3}\dots i_{k})i}.
\end{align}
The properties characterizing the injective linear map $\ih$ are that is image is contained in $\Ga(\ctm \tensor \symkt)$ (rather than $\Ga(\ctm \tensor \symk)$) and that the non-trivial traces of $\ih(\om)$ equal $\om$.

\subsubsection{}
If $\om \in \Ga(\symkt)$ then $D\om$ will have a pure trace part, and parts in the submodules of $\Ga(\symkpt)$ with symmetries corresponding to the partitions $(n+1)$ and $(n1)$. The following describes explicitly the decomposition of $D\om$ into these parts. For $\om_{i_{1}\dots i_{k}} \in \symk$ there holds
\begin{align}\label{dsym2}
D_{i}\om_{i_{1}\dots i_{k}} = D_{(i}\om_{i_{1}\dots i_{k})} + \tfrac{2}{k+1}\sum_{s = 1}^{k}D_{[i}\om_{i_{s}]i_{1}\dots \hat{i}_{s}\dots i_{k}}.
\end{align}
Substituting into \eqref{dsym2} the definitions of $\clie(\om)$ and $\klie(\om)$ and simplifying the trace terms yields
\begin{align}\label{domkl}
\begin{split}
D\om & = \clie(\om) + \tlie(\om) + \ih(\div(\om)).
\end{split}
\end{align}
It is straightforward to check that the righthand sides of \eqref{dsym2} and \eqref{domkl} are the same modulo pure trace terms. On the other hand from the properties characterizing $\ih$ it follows that the traces of the righthand sides of \eqref{dsym2} and \eqref{domkl} are the same. This verifies \eqref{domkl}. Contracting \eqref{domkl} with $D^{i}\om^{i_{1}\dots i_{k}}$ gives
\begin{align}\label{normdom}
\begin{split}
|D\om|^{2} & = |\clie(\om)|^{2} + |\tlie(\om)|^{2} + \tfrac{k(n+2(k-2))}{(n-3+k)(n+2(k-1))}|\div(\om)|^{2}\\
& = |\clie(\om)|^{2} + \tfrac{2k}{k+1}|\klie(\om)|^{2} + \tfrac{k(n+2(k-2))}{(n-3+k)(n+2(k-1))}|\div(\om)|^{2}.
\end{split}
\end{align}
It is immediate from \eqref{normdom} that $\ker D \cap \Ga(S^{k}_{0}(\ctm)) = \ker \clie\cap \ker \klie \cap \ker \div \cap \Ga(S^{k}_{0}(\ctm))$. The parallel tensors (those in $\ker D \cap \Ga(S^{k}_{0}(\ctm))$) play the role played by constant functions in the scalar elliptic theory.

\subsection{Curvature operator on trace-free symmetric tensors}\label{curvatureoperatorsection}
\subsubsection{}
Let $\sY_{ijkl}$ be a tensor of metric curvature type, with Ricci trace $\sY_{ij} = \sY_{pij}\,^{p}$, and assume $\sY_{[ij]} = 0$. Define the action of $\sY_{ijkl}$ on an arbitrary rank $k$ covariant tensor $\om_{i_{1}\dots i_{k}}$ by
\begin{align}
k\sY(\om)_{i_{1}\dots i_{k}} = \sum_{s = 1}^{k}\sY_{pi_{s}}\om_{i_{1}\dots i_{s-1}}\,^{p}\,_{i_{s+1}\dots i_{k}} - \sum_{r \neq s}\sY_{pi_{r}i_{s}q}\om_{i_{1}\dots i_{r - 1}}\,^{p}\,_{i_{r+1}\dots i_{s-1}}\,^{q}\,_{i_{s+1}\dots i_{k}}.
\end{align} 
The endomorphism $\sY$ is evidently symmetric in the sense that if $\om, \mu \in \symkt$ then $\lb \mu, \sY(\om)\ra = \lb \om, \sY(\mu)\ra$. When $k = 1$ then $\sY(\om)$ is simply the endomorphism of $\ctm$ induced by the Ricci trace, $\sY(\om)_{i} = \sY_{i}\,^{p}\om_{p}$. For $\om \in \symkt$ let $\qY(\om) \defeq \lb \om, \sY(\om)\ra$ be the quadratic form determined by $\sY(\om)$. 

It is proved on page $27$ of \cite{Lichnerowicz-propagateurs} that $\sY$ commutes with taking traces on any pair of indices (in \cite{Lichnerowicz-propagateurs} the proof is given for $\sR$, but the claim requires only that $\sY$ have metric curvature symmetries). If $\om \in \Ga(\symk)$ then
\begin{align}
\sY(\om)_{i_{1}\dots i_{k}} = \sY_{p(i_{1}}\om_{i_{1}\dots i_{k})}\,^{p} + (1-k)\sY_{p(i_{1}i_{2}}\,^{q}\om_{i_{3}\dots i_{k})q}\,^{p}.
\end{align}
Note also that $\sY(h) = 0$. This generalizes to the statement that $h$ commutes with $\sY$ in the following sense. 
\begin{lemma}
For $\om \in \Ga(\symk)$ let $h_{r, s}(\om)_{i_{1}\dots i_{k+2}} = h_{i_{r}i_{s}}\om_{i_{1}\dots \hat{i}_{r}\dots \hat{i}_{s}\dots i_{k+2}}$. Then
\begin{align}\label{hycommute}
(k+2)\sY(h_{r, s}(\om)) = kh_{r, s}(\sY(\om)).
\end{align}
\end{lemma}
\begin{proof}
Relabeling the indices it suffices to consider $r = 1$ and $s = 2$. Then 
\begin{align*}
\begin{split}
&\sum_{s = 1}^{k+2}\sY_{pi_{s}}h_{1, 2}(\om)_{i_{1}\dots i_{s-1}}\,^{p}\,_{i_{s+1}\dots i_{k+2}}= 2\sY_{i_{1}i_{2}}\om_{i_{3}\dots i_{k+2}} + h_{i_{1}i_{2}}\sum_{s = 3}^{k+2}\sY_{pi_{s}}\om_{i_{3}\dots i_{s-1}}\,^{p}\,_{i_{s+1}\dots i_{k+2}} ,\\
&\sum_{r \neq s}\sY_{pi_{r}i_{s}}\,^{q}h_{1, 2}(\om)_{i_{1}\dots i_{r - 1}}\,^{p}\,_{i_{r+1}\dots i_{s-1}}\,^{q}\,_{i_{s+1}\dots i_{k+2}} = - 2 \sY_{i_{1}i_{2}}\om_{i_{3}\dots i_{k+2}} \\
&+ \sum_{s = 3}^{k}\sY_{i_{2}i_{1}i_{s}}\,^{q}\om_{i_{3}\dots i_{s-1}qi_{s+1}\dots i_{k+2}} + \sum_{s = 3}^{k}\sY_{i_{1}i_{2}i_{s}}\,^{q}\om_{i_{3}\dots i_{s-1}qi_{s+1}\dots i_{k+2}} \\
&+ h_{i_{1}i_{2}}\sum_{r \neq s, r \geq 3, s \geq 3}\sY_{pi_{r}i_{s}q}\om_{i_{3}\dots i_{r - 1}}\,^{p}\,_{i_{r+1}\dots i_{s-1}}\,^{q}\,_{i_{s+1}\dots i_{k+2}}.
\end{split}
\end{align*}
which together show \eqref{hycommute}.
\end{proof}

\subsubsection{}
Any tensor $\sY_{ijkl}$ with symmetries of metric curvature type determines a symmetric endomorphism of $\Ga(S^{2}(\ctm))$ defined by $a_{ij} \to \sY_{ipjq}a^{pq}$. In general this endomorphism does not preserve the subspace $\Ga(S^{2}_{0}(\ctm))$. If $\sY_{(ij)} = \sY_{ij} = \sY_{pij}\,^{p}$ then the modified endomorphism $a_{ij} \to \sY(a)_{ij} \defeq a^{pq}(\sY_{ipjq} + \sY_{p(i}h_{j)q})$ restricts to a symmetric endomorphism of $S^{2}_{0}(\ctm)$.%

\subsubsection{}
Let $\sR_{ijk}\,^{l}$ be the curvature tensor of $D$. The curvature of the induced connection on $\symkt$ is given by $2D_{[i}D_{j]}\om_{i_{1}\dots i_{k}} = -k\sR_{ij(i_{1}}\,^{p}\om_{i_{2}\dots i_{k})p}$. Tracing this in $i$ and $i_{r}$ yields
\begin{align*}
\begin{split}
&D_{p}D_{i_{r}}\om_{i_{1}\dots i_{r-1}}\,^{p}\,_{i_{r+1}\dots i_{k}} - D_{i_{r}}D_{p}\om_{i_{1}\dots i_{r-1}}\,^{p}\,_{i_{r+1}\dots i_{k}} \\
&= \sR_{p i_{r}}\om_{i_{1}\dots i_{r-1}}\,^{p}\,_{i_{r+1}\dots i_{k}} - \sum_{s \neq r}\sR_{qi_{r}i_{s}p}\om_{i_{1}\dots i_{r-1}}\,^{p}\,_{i_{r+1}\dots i_{s-1}}\,^{q}\,_{i_{s+1} \dots i_{k}},
\end{split}
\end{align*}
and summing this in $r$ gives
\begin{align}\label{differentialr}
\begin{split}
&2\sum_{r = 1}^{k}D_{[p}D_{i_{r}]}\om_{i_{1}\dots i_{r-1}}\,^{p}\,_{i_{r+1}\dots i_{k}} = k\sR(\om)_{i_{1}\dots i_{k}}.
\end{split}
\end{align}
Since taking traces commutes with $D$, \eqref{differentialr} gives a simple proof that $\sR$ commutes with traces. %

\subsubsection{}
For a section $\om$ of a bundle $E$ of covariant tensors let $\qY(\om) \defeq \lb \om, \sY(\om)\ra$ be the quadratic form on $\Ga(E)$ corresponding to the symmetric endomorphism $\sY$. Say that \textbf{$\qY$ is positive, non-negative, zero, negative, non-negative, etc. on $E$} if $\qY(\om) > 0$, $\qY(\om) \geq 0$, etc. for all $\om \in \Ga(E)$. That $\qY$ be positive (etc.) on $\ctm$ means simply that $\sY_{ij}$ is positive (etc.) on $\ctm$.

\subsubsection{}
The tensor $\sH_{ijkl} \defeq \tfrac{2}{n(n-1)}h_{l[i}h_{j]k}$ is of metric curvature tensor type, and so determines an operator $\sH$ and an associated quadratic form $\qH$ defined in the same way as $\sY$ and $\qY$. For any $\om \in \Ga(\tensor^{k}\ctm)$ there holds
\begin{align}
k\sH(\om)_{i_{1}\dots i_{k}} = \tfrac{k(n+k-2)}{n(n-1)}\om_{i_{1}\dots i_{k}} - \tfrac{1}{n(n-1)}\sum_{r \neq s}h_{i_{r}i_{s}}\om_{i_{1}\dots i_{r-1}pi_{r+1}\dots i_{s-1}}\,^{p}\,_{i_{s+1}\dots i_{k}}.
\end{align}
It follows that $k\qH(\om) =  \tfrac{k(n+k-2)}{n(n-1)}|\om|^{2}_{h} - \tfrac{1}{n(n-1)}\sum_{r \neq s}|\tr_{rs}(\om)|^{2}_{h}$, in which $\tr_{rs}(\om)$ indicates the trace on the $r$ and $s$ slots. In particular, the quadratic form $\qH$ is positive on any bundle of completely trace-free tensors. In fact it seems likely that $\qH$ is positive on any bundle of tensors (it is not hard to check this for tensors of rank $2$ and $3$), but a proof of this seems to require an excursion into calculating weights that would take too much space to develop here.

That $h$ have constant sectional curvature $\tfrac{2\sR}{n(n-1)}$ is exactly the statement $\sR_{ijkl} = \sR \sH_{ijkl}$. Since $\qH(\om) = \tfrac{n+k-2}{n(n-1)}|\om|^{2}$ it follows that if $h$ has constant sectional curvature, then $\qR$ is positive, zero, or negative on any bundle of completely trace-free tensors according to whether the curvature of $h$ is positive, zero, or negative. When $n=2$ then $2\qR(\om) = k\sR |\om|^{2}$ for $\om \in \Ga(\symkt)$, so $\qR$ is positive, negative, etc. on $\symkt$ if and only if the scalar curvature has the same property.

\subsubsection{}
Let $X_{1}, X_{2}, X_{3}, X_{4} \in \Ga(TM)$. Then
\begin{align}\label{qy0}
\begin{split}
4\lb \sR(X_{1} \cprod X_{2}), X_{3} \cprod X_{4}\ra  =& -2\lb \sR(X_{1}, X_{3})X_{4}, X_{2} \ra  - 2 \lb \sR(X_{2}, X_{3})X_{4}, X_{1} \ra \\
&+ \sric(X_{3}, X_{1})\lb X_{4}, X_{2} \ra +  \sric(X_{3}, X_{2})\lb X_{4}, X_{1} \ra \\ &+ \sric(X_{4}, X_{1})\lb X_{3}, X_{2} \ra + \sric(X_{4}, X_{2})\lb X_{3}, X_{1} \ra.
\end{split}
\end{align}
The same identity holds with $\sR$ replaced by any metric curvature tensor $\sY$ having symmetric Ricci trace. For such a $\sY$ there follows from \eqref{qy0} that for $X, Y \in \Ga(TM)$, 
\begin{align}\label{qy1}
\begin{split}
\qY(X\cprod Y) &= -\tfrac{1}{2}X^{i}Y^{j}X^{k}Y^{l}\sY_{ijkl}  + \tfrac{1}{2}\lb X, Y \ra \sY_{pq}X^{p}Y^{q} + \tfrac{1}{4}|X|^{2}\sY_{pq}Y^{p}Y^{q} + \tfrac{1}{4}|Y|^{2}\sY_{pq}X^{p}X^{q}\\
 &= -\tfrac{1}{2}X^{i}Y^{j}X^{k}Y^{l}\sY_{ijkl}  + \tfrac{1}{2}(\lb X, Y \ra - |X||Y|)\sY_{pq}X^{p}Y^{q} + \tfrac{1}{4}\qY(|Y|X + |X|Y)\\ 
&= -\tfrac{1}{2}X^{i}Y^{j}X^{k}Y^{l}\sY_{ijkl}  + \tfrac{1}{2}(\lb X, Y \ra + |X||Y|)\sY_{pq}X^{p}Y^{q} + \tfrac{1}{4}\qY(|Y|X - |X|Y).
\end{split}
\end{align}
In particular, $\qY(X\cprod X) = |X|^{2}\sY_{pq}X^{p}X^{q} = |X|^{2}\qY(X)$. This shows that if $\qY$ is positive, non-negative, etc. on $S^{2}_{0}(\ctm)$, then it has the same property on $\ctm$. In particular, if $\qR$ is positive, etc. on $S^{2}_{0}(\ctm)$, then the Ricci curvature is positive, etc. If $X, Y \in \Ga(TM)$, let 
\begin{align*}
\ka(X, Y) \defeq -\tfrac{\lb\sR(X, Y)X, Y\ra}{|X|^{2}|Y|^{2} - \lb X, Y\ra} = -\tfrac{2\lb\sR(X, Y)X, Y\ra}{|X\wedge Y|^{2}},
\end{align*}
denote the sectional curvature of the span of $X$ and $Y$. For $\sY = \sR$, equation \eqref{qy1} can be rewritten as
\begin{align}
\begin{split}
4\qR(X\cprod Y) & = \ka(X, Y)|X \wedge Y|^{2} + |Y|^{2}\ric(X, X) + |X|^{2}\ric(Y, Y) + 2\lb X, Y \ra \ric(X, Y)\\
& = \ka(X, Y)|X \wedge Y|^{2} + \qR(|Y|X + |X|Y) - 2(|X||Y| - \lb X, Y \ra)\ric(X, Y)\\
& = \ka(X, Y)|X \wedge Y|^{2} + \qR(|Y|X - |X|Y) + 2(|X||Y| + \lb X, Y \ra)\ric(X, Y).
\end{split}
\end{align}
Similarly,
\begin{align}
\begin{split}
4\lb \sR(X\cprod X), Y \cprod Y\ra  & = -2\ka(X, Y)|X \wedge Y|^{2} + 4\lb X, Y \ra\ric(X, Y).
\end{split}
\end{align}
In particular, if $X$ and $Y$ are orthogonal and of unit norm, then
\begin{align}\label{qrsectional}
\begin{split}
&4\qR(X\cprod Y)  = 2\ka(X, Y) + \ric(X, X) + \ric(Y, Y),\\
&4\lb \sR(X\cprod X), Y \cprod Y\ra   = -4\ka(X, Y),\\
&4\qR(X\cprod X)   = 4\ric(X, X).
\end{split}
\end{align}
Let $\om$ be an arbitrary element of $\Ga(S^{2}_{0}(\ctm))$. There can be chosen an $h$-orthonormal local frame $X_{\al}$ in $TM$ with respect to which $\om$ has the form $\om = \sum_{\al = 1}^{n}\la_{\al}X_{\al}\cprod X_{\al}$ for some $\la_{\al} \in \rea$. Let $\ka_{\al\be} = \ka(X_{\al}, X_{\be})$. From \eqref{qrsectional} and $\ric(X_{\al}, X_{\al}) = \sum_{\be}\ka_{\al\be}$, there follows 
\begin{align}\label{qrsym2}
\begin{split}
2\qR(\om) &= 2\sum_{\al}\la_{\al}^{2}\qR(X_{\al}\cprod X_{\al}) + 2\sum_{\al \neq \be}\la_{\al}\la_{\be}\lb \sR(X_{\al}\cprod X_{\al}), X_{\be}\cprod X_{\be}\ra\\
& = 2\sum_{\al}\la_{\al}^{2}\sum_{\be}\ka_{\al\be} - 2\sum_{\al, \be}\la_{\al}\la_{\be}\ka_{\al\be} = \sum_{\al \neq \be}(\la_{\al} - \la_{\be})^{2}\ka_{\al\be}. 
\end{split}
\end{align}
The conclusions are all easily deduced from this equality and the following observation. Note $nX_{\al}\cprod X_{\al} = (n-1)X_{\al}\tensor X_{\al} - \sum_{\ga \neq \al}X_{\ga}\tensor X_{\ga}$. It follows that $\om = \sum_{\al = 1}^{n}\la_{\al}X_{\al}\cprod X_{\al} = \sum_{\al = 1}^{n}\mu_{\al}X_{\al}\tensor X_{\al}$ with $n\mu_{\al} = (n-1)\la_{\al} - \sum_{\ga \neq \al}\la_{\ga}$. There hold $\sum_{\al = 1}^{n}\mu_{\al} = 0$ and $\mu_{\al} - \mu_{\be} = \la_{\al} - \la_{\be}$. Hence if all the $\la_{\al}$ are equal the same is true for the $\mu_{\al}$, which implies all the $\mu_{\al}$ are $0$, and hence that $\om = 0$. Thus $\om = 0$ if and only if the $\la_{\al}$ are all equal. A computation shows that $n|\om|_{h}^{2} = n\sum_{\al = 1}^{n}\mu_{\al}^{2} = \sum_{\al \neq \be}(\la_{\al} - \la_{\be})^{2}$. From \eqref{qrsym2} it follows that a bound on the sectional curvatures of $h$ implies a bound on $\qR(\om)$, e.g. if the sectional curvatures are bounded from below by $-a^{2}$ (resp. above by $a^{2}$) then $2\qR(\om) \geq -a^{2}n|\om|^{2}$ (resp. $2\qR(\om) \leq a^{2}n|\om|^{2}$). Suppose are constants so that the sectional curvatures of $h$ satisfy $A \leq \inf_{x \in M}\inf_{L \subset Gr(2, T_{x}M)}\ka(L) \leq \sup_{x \in M}\sup_{L \subset Gr(2, T_{x}M)}\ka(L) \leq B$ where $\ka(L)$ denotes the sectional curvature of the two-dimensional subspace $L \subset Gr(2, T_{x}M)$. Then \eqref{qrsym2} shows that as quadratic forms on $S^{2}_{0}(\ctm)$ there holds $\q_{n(n-1)A\sH} \leq 2\qR \leq \q_{n(n-1)B\sH}$. In particular this implies Lemma \ref{qrsignlemma}, which is essentially Theorem $6.1$ of \cite{Berger-Ebin} (wherein only the positive case is stated). 
\begin{lemma}\label{qrsignlemma}
If the Riemannian metric $h$ has (strictly) negative, non-positive, (strictly) positive, or non-negative sectional curvature then $\qR$ has the same property on $S^{2}_{0}(\ctm)$. 
\end{lemma}
It follows, for example, that positive sectional curvature implies $\qR$ is positive on $S^{2}_{0}(\ctm)$, which implies positive Ricci curvature, but it is evident from \eqref{qrsym2} that it is unlikely that in general either of these implications is reversible. More precisely, the non-negativity of the coefficients $(\la_{\al} - \la_{\be})^{2}$ means that a sign condition on the sectional curvatures results in a sign condition for $\qR$ on $S^{2}_{0}(\ctm)$. However, a collection of $n(n-1)/2$ real numbers of the form $(\la_{\al} - \la_{\be})^{2}$ is not an arbitrary collection of non-negative numbers. For example, let $a, b, c \in \rea$ and let $x$, $y$, and $z$ be non-negative real numbers such that $x^{2} = (a-b)^{2}$, $y^{2} = (b-c)^{2}$, and $z^{2} = (c - a)^{2}$. Then $(x + y - z)(x - y + z)(-x + y + z) = 0$. A vector having a negative coordinate but positive inner product with $(x^{2}, y^{2}, z^{2})$ is given by $(x^{-1}, y^{-1}, \la z^{-1})$ for any $0 > \la > -1$. The point here is simply that, for example, the positivity of $\qR$ on $S^{2}_{0}(\ctm)$ does not in any obvious way imply the positivity of the sectional curvature for purely numerical reasons; if there were to be such an implication, it would have its origins in geometrical considerations.

This suggests that although one does not feel that one understands geometrically an algebraic condition such as the positivity of $\qR$ on $S^{2}_{0}(\ctm)$, such a condition is a reasonable condition in its own right, just as is a condition such as the positivity of isotropic curvature.

\subsubsection{}
When $n > 2$, $\qR$ can be rewritten in terms of the conformal Weyl and Schouten tensors $\sW_{ijk}\,^{l}$ and $\sW_{ij}$. Namely,
\begin{align}\begin{split}
\qR(\om) &= \qW(\om) + \tfrac{n+2(k-2)}{n-2}\sR_{p}\,^{q}\om_{qi_{1}\dots i_{k-1}}\om^{pi_{1}\dots i_{k-1}} +  \tfrac{1-k}{(n-1)(n-2)}\sR |\om|_{h}^{2}\\
& = \qW(\om) -(n+2(k-2))\sW_{p(i_{1}}\om_{i_{2}\dots i_{k})}\,^{p}\om^{i_{1}\dots i_{k}} +  \tfrac{1}{2(n-1)}\sR |\om|_{h}^{2}\\
& = \qW(\om) -(n+2(k-2))\mr{\sW}_{p(i_{1}}\om_{i_{2}\dots i_{k})}\,^{p}\om^{i_{1}\dots i_{k}} +  \tfrac{n-2+k}{n(n-1)}\sR |\om|_{h}^{2}.
\end{split}
\end{align}
If the Ricci curvature is bounded from below, $\sR_{ij} \geq \tfrac{c}{n}h_{ij}$, (resp. above, $ \sR_{ij} \leq \tfrac{c}{n}h_{ij}$) then $\qR(\om) \geq \qW(\om) + \tfrac{n+k-2}{(n-1)(n-2)}c|\om|^{2}$ (resp. $\qR(\om) \leq \qW(\om) + \tfrac{n+k-2}{(n-1)(n-2)}c|\om|^{2}$). It follows that on a conformally flat manifold $\qR$ is positive, non-negative, zero, negative, etc. according to whether the Ricci tensor has the corresponding property. For Einstein metrics understanding $\qR$ reduces to understanding $\qW$. Evidently stronger results can only be obtained when there is further control of the term $\qW(\om)$.

\subsection{Weitzenb\"ock type identities for trace-free symmetric tensors}\label{weitzenbocksection}
\subsubsection{}
For $\om \in \Ga(\symkt)$ straightforward computations using the Ricci identity and \eqref{cliedefined} show
\begin{align}
\label{liediv}
\begin{split}
\clie \div(\om)_{i_{1}\dots i_{k}} &= D_{(i_{1}}D^{p}\om_{i_{2}\dots i_{k})p} + \tfrac{1-k}{(n+2(k-2))}h_{(i_{1}i_{2}}D^{p}D^{q}\omega_{i_{3}\dots i_{k})pq},
\end{split}\\
\begin{split}
\label{divlie}
\tfrac{k+1}{k}\div \clie(\om) &  = \tfrac{n+2(k-2)}{n+2(k-1)}\clie\div(\om) + \tfrac{1}{k}\lap_{h}\om + \sR(\om) .
\end{split}
\end{align}
Contracting \eqref{domkl} with $D^{i}$ and using \eqref{liediv} gives
\begin{align}\label{lapom}
\lap_{h}\om& = \div \clie(\om) + \tfrac{k(n+2(k-2))}{(n-3+k)(n+2(k-1))}\clie \div(\om)+ \tfrac{2k}{k+1}\kliea\klie(\om),
\end{align}
in which has been used $D^{p}\klie(\om)_{p(i_{1}\dots i_{k})} = \kliea\klie(\om)_{i_{1}\dots i_{k}}$. Solving \eqref{divlie} for $\lap_{h}\om$ and substituting the result into \eqref{lapom} yields
\begin{align}\label{klieweitzenbock}
\begin{split}
&\sR(\om)  = \div \clie(\om) - \tfrac{(n-2+k)(n+2(k-2))}{(n-3+k)(n+2(k-1))}\clie\div(\om) -\tfrac{2}{k+1}\kliea\klie(\om).
\end{split}
\end{align}
Equation \eqref{divlie} and \eqref{klieweitzenbock} are the analogues of Equations $(2.8)$ and $(2.9)$ of \cite{Semmelmann}. Rewriting \eqref{lapom} in two different ways using \eqref{divlie} gives
\begin{align}
\label{lapom2} \lap_{h} \om + \tfrac{k}{n-2+k}\sR(\om)& =  \tfrac{n+2(k-1)}{n - 2 + k}\div \clie(\om)  + \tfrac{2k(n-3+k)}{(k+1)(n-2+k)}\kliea\klie(\om)\\
\label{lapom3} \lap_{h} \om - \sR(\om)& =  \tfrac{n+2(k-2)}{n - 3 + k}\clie \div(\om)   + 2\kliea\klie(\om).
\end{align}
Contracting each of \eqref{lapom2} and \eqref{lapom3} with $\om^{i_{1}\dots i_{k}}$ yields the Weitzenb\"ock identities
\begin{align}
\label{lapomdivlie}\tfrac{1}{2}\lap_{h}|\om|^{2} & = |D\om|^{2} + \tfrac{n+2(k-1)}{n - 2 + k}\lb \om, \div \clie(\om)\ra  + \tfrac{2k(n-3+k)}{(k+1)(n-2+k)}\lb \om, \kliea\klie(\om)\ra%
- \tfrac{k}{n-2+k}\qR(\om),\\
\label{lapomliediv}\tfrac{1}{2}\lap_{h}|\om|^{2} & = |D\om|^{2} + \tfrac{n+2(k-2)}{n - 3 + k}\lb\om, \clie \div(\om)\ra+ 2\lb\om, \kliea\klie(\om)\ra + \qR(\om) . 
\end{align}
On a compact manifold integrating either \eqref{lapomdivlie} or \eqref{lapomliediv} by parts against the Riemannian volume $\vol_{h}$ and using \eqref{normdom} to simplify the result gives
\begin{align}\label{bigbochner}
\begin{split}
\tfrac{2}{k+1}\int_{M}|\klie(\om)|^{2} &  + \tfrac{(n-2+k)(n+2(k-2))}{(n-3+k)(n+2(k-1))}\int_{M}|\div(\om)|^{2} - \int_{M}|\clie(\om)|^{2} =   \int_{M}\qR(\om).
\end{split}
\end{align}
The identity \eqref{bigbochner} generalizes the usual integrated Bochner identities for harmonic one-forms and conformal Killing vector fields. 

\begin{theorem}\label{bochnerliedivtheorem}
Let $h$ be a Riemannian metric on a compact manifold $M$ of dimension $n > 2$. If $\qR$ is non-negative on $S^{k}_{0}(\ctm)$ then any $\om \in \Ga(\symkt) \cap \ker \klie \cap \ker \div$ is parallel. If moreover $\qR$ is at some point of $M$ strictly positive on $S^{k}_{0}(\ctm)$ then $\Ga(\symkt) \cap \ker \klie \cap \ker \div = \{0\}$. If $\qR$ is non-positive on $S^{k}_{0}(\ctm)$ then any rank $k$ conformal Killing tensor is parallel, and if, moreover, $\qR$ is at some point of $M$ strictly negative on $S^{k}_{0}(\ctm)$, then any rank $k$ conformal Killing tensor is identically zero.
\end{theorem}

\begin{proof}
Throughout the proof $\om \in \Ga(\symkt)$. If $\om \in \ker \klie \cap \ker \div$ and $\qR \geq 0$ then \eqref{bigbochner} shows that $\om \in \ker \clie$ and from \eqref{normdom} it follows that $D\om = 0$. If moreover $\qR$ is somewhere positive then \eqref{lapomliediv} shows $\om \equiv 0$. If $\om \in \ker \clie$ and $\qR \leq 0$ then \eqref{bigbochner} shows that $\om \in \ker \klie \cap \ker \div$ and from \eqref{normdom} it follows that $D\om = 0$. If, moreover, $\qR$ is somewhere negative then \eqref{lapomdivlie} shows $\om \equiv 0$.
\end{proof}

\subsubsection{}
For $\al \in \rea$ define a formally self-adjoint second order elliptic differential operator $\culap_{\al}$ on $\Ga(S^{k}_{0}(\ctm))$ by $\culap_{\al}\om = \lap_{h}\om +\al \sR(\om)$. Let $M$ be a compact manifold. Define a functional $\cubic_{\al}$ with arguments a Riemannian metric $h_{ij}$ and a tensor $\om \in \Ga(S^{k}_{0}(\ctm))$ by 
\begin{align*}
\cubic_{\al}(h, \om) \defeq \int_{M}\left( |\clie(\om)|_{h}^{2} - \al \qR(\om) \right) \,d\vol_{h} = -\int_{M}\lb\om,\culap_{\al}\om\ra\,d\vol_{h}.
\end{align*}
For fixed $h$ the first variation of $\cubic_{\al}(h, \om)$ in $\om$ yields the equation $\culap_{\al}\om = 0$. Lemma \ref{culaplemma} essentially amounts to the observation that for $-1 \leq \al \leq \tfrac{k}{n-2+k}$ the functional $\cubic_{\al}(h, \om)$ is non-negative.

\begin{lemma}\label{culaplemma}
For a Riemannian metric $h$ on a compact manifold $M$ of dimension $n \geq 2$ there hold
\begin{enumerate}
\item If $\al = -1$ then $\ker \culap_{\al} \cap \Ga(S^{k}_{0}(\ctm)) = \ker \klie \cap \ker \div \cap \Ga(S^{k}_{0}(\ctm))$.
\item If $-1 < \al < \tfrac{k}{n-2+k}$ then $\ker \culap_{\al} \cap \Ga(S^{k}_{0}(\ctm)) = \ker D \cap \Ga(S^{k}_{0}(\ctm))$.
\item If $\al = \tfrac{k}{n-2+k}$ then $\ker \culap_{\al} \cap \Ga(S^{k}_{0}(\ctm)) = \ker \klie \cap \ker \clie \cap \Ga(S^{k}_{0}(\ctm))$.
\end{enumerate}
\end{lemma}
\begin{proof}
By \eqref{lapom3} there holds
\begin{align}\label{culap1}
\culap_{\al} \om = (1+\al)\div \clie(\om) + \tfrac{(n+2(k-2))(k - \al(n-2+k))}{(n - 3 + k)(n+2(k-1))}\clie \div(\om)  + \tfrac{2(k-\al)}{k+1}\kliea\klie(\om)
\end{align}
from which one of the inclusions in each of the claims is obvious (this does not use the compactness of $M$). If $M$ is compact and $\om \in \ker \culap \cap \Ga(S^{k}_{0}(\ctm))$ then integrating the righthand side of \eqref{culap1} gives 
\begin{align*}
0 = \int_{M}\left( (1+\al)|\clie(\om)|_{h}^{2} +  \tfrac{(n+2(k-2))(k - \al(n-2+k))}{(n - 3 + k)(n+2(k-1))}|\div(\om)|_{h}^{2} + \tfrac{2(k-\al)}{k+1}|\klie(\om)|_{h}^{2}\right) \, d\vol_{h},
\end{align*}
from which the opposite inclusions follow.
\end{proof}

The special case of $\culap_{\al}$ for $\al = -1$ acting on sections of $S^{2}_{0}(\ctm)$ was studied by J. Simons in \cite{Simons}, and this case of Lemma \ref{culaplemma} is given in section $6$.c of \cite{Berger-Ebin}. The operator $\culap_{-1}$ plays an important role in Section \ref{eigensection}.

The \textbf{Lichnerowicz Laplacian} $\lich$ defined on page $27$ of \cite{Lichnerowicz-propagateurs} is the formally self-adjoint operator which acts on an arbitary rank $k$ covariant tensor $\om_{i_{1}\dots i_{k}}$ by $-\lap_{h}\om + k\sR(\om)$. The linearization of the Ricci curvature of the metric $h$ at the symmetric two-tensor $a_{ij}$ is $\vr\ric_{h}(a)_{ij} = \tfrac{1}{2}\lich a_{ij} + D_{(i}D^{p}a_{j)p}$. On differential forms the Lichnerowicz Laplacian restricts to the usual Hodge Laplacian. Since $\sR$ and $\lap_{h}$ commute with taking traces, the Lichnerowicz operator restricts to $-\culap_{-k}$ on $\Ga(S^{k}_{0}(\ctm))$.

\subsubsection{Refined Kato inequalities}
\begin{lemma}\label{katolemma}
Let $M$ be an $n$-dimensional manifold with a Riemannian metric $h$ having Levi-Civita connection $D$, and let $\phi_{i_{1}\dots i_{k}} \in \Ga(\symkt)$. If $\phi \in \Ga(\symkt)\cap \ker \klie \cap \ker \div$ (that is $D_{[i}\phi_{j]i_{1}\dots i_{k-1}} = 0$) then where $|\phi|^{2} \neq 0$ there holds
\begin{align}\label{kato}
|d|\phi||^{2} \leq \tfrac{n-2+k}{n+2(k-1)}|D\phi|^{2}.
\end{align}
If $\phi \in \Ga(\symkt)\cap \ker \clie \cap \ker \div$ (that is $D_{(i_{1}}\phi_{i_{2}\dots i_{k+1})} = 0$) then where $|\phi|^{2} \neq 0$ there holds
\begin{align}\label{kato2}
|d|\phi||^{2} \leq \tfrac{k}{k+1}|D\phi|^{2}.
\end{align}
If $\phi \in \Ga(\symkt)\cap \ker \clie \cap \ker \klie$ (that is $D\phi$ is pure trace) then where $|\phi|^{2} \neq 0$ there holds
\begin{align}\label{kato3}
|d|\phi||^{2} \leq \tfrac{k}{n+2(k-1)}|D\phi|^{2}.
\end{align}
\end{lemma}
\begin{proof}
The inequality \eqref{kato} generalizes the estimate for the second fundamental form of a minimal hypersurface proved in \cite{Schoen-Simon-Yau}. With a $1$ in place of $\tfrac{n+k-2}{n+2(k-1)}$, it follows from the Cauchy-Schwarz inequality, and is known as a \textbf{Kato inequality}. The estimates \eqref{kato}-\eqref{kato3} are \textbf{refined Kato inequalities} in the sense of \cite{Calderbank-Gauduchon-Herzlich} and \cite{Branson-kato}, and can be deduced from the results in either of those papers. In particular the results of Section $6$ of \cite{Calderbank-Gauduchon-Herzlich} include the present lemma, and the discussion at the very end of Section $6$ of \cite{Calderbank-Gauduchon-Herzlich} gives the explicit constants of \eqref{kato}-\eqref{kato3} for the cases $k = 1, 2$. To keep the exposition self-contained, here a direct proof of \eqref{kato}-\eqref{kato3} is given following the general procedure described in the introduction of \cite{Branson-kato}. 

Recall that  $\sbl_{\clie}(Z)(\phi)$ is the symbol of $\clie$ applied to the vector $Z^{i}$ and $\phi \in \Ga(\symkt)$, and similarly for $\klie$ and $\div$. Write $(i(Z)\phi)_{i_{1}\dots i_{k-1}} = Z^{p}\phi_{pi_{1}\dots i_{k-1}}$ and define $|i(Z)\phi|^{2}$, $|\phi|^{2}$, $|\Phi|^{2}$, etc. by contracting indices using $h^{ij}$. Take $Z^{i}$ to be unit norm in what follows. Straightforward computations show
\begin{align}\label{cliesymbolnorm}
\begin{split}
|\sbl_{\clie}(Z)(\phi)|^{2}
& = \tfrac{1}{k+1}|\phi|^{2} + \tfrac{k(n+2(k-2))}{(k+1)(n+2(k-1))}|i(Z)\phi|^{2}.
\end{split}\\
\label{kliesymbolnorm}
\begin{split}
|\sbl_{\klie}(Z)(\phi)|^{2} 
& =  \tfrac{1}{2}\left(|\phi|^{2} -|i(Z)\phi|^{2}\right) + \tfrac{k-1}{2(3-n-k)}|i(Z)\phi|^{2} =  \tfrac{1}{2}|\phi|^{2} + \tfrac{n+2(k-2)}{2(3-n-k)}|i(Z)\phi|^{2}.
\end{split}
\end{align} 
When $k = 1$ and $n = 2$ the coefficient of the pure trace terms in \eqref{kliesymbolnorm} should be understood in a limiting sense. The non-negativity of $|\sbl_{\klie}(Z)(\phi)|^{2}$ in \eqref{kliesymbolnorm} yields
\begin{align}\label{klieinequality}
 \tfrac{n+2(k-2)}{n-3+k}|i(Z)\phi|^{2} \leq |\phi|^{2}.
\end{align}
Together \eqref{cliesymbolnorm} and \eqref{klieinequality} give 
\begin{align}\label{sblclie}
|\sbl_{\clie}(Z)(\phi)|^{2} \leq  \tfrac{1}{k+1}|\phi|^{2} + \tfrac{k(n-3+k)}{(k+1)(n+2(k-1))}|\phi|^{2} = \tfrac{n-2+k}{n+2(k-1)}|\phi|^{2}.
\end{align}
Contracting \eqref{domkl} with $\sbl_{D}(Z)(\om)$ shows
\begin{align}\label{katoineq1}
\begin{split}
&\tfrac{1}{2}Z^{i}d_{i}|\phi|^{2} = Z^{i}\phi^{i_{1}\dots i_{k}}D_{i}\phi_{i_{1}\dots i_{k}} = \lb \sbl_{D}(Z)(\phi), D\phi\ra \\
& = \lb\sbl_{\clie}(Z)(\phi), \clie(\phi)\ra + \tfrac{2k}{k+1} \lb \sbl_{\klie}(Z)(\phi), \klie(\phi)\ra + \tfrac{k(n+2(k-2))}{(n-3+k)(n+2(k-1))}\lb i(Z)\phi, \div(\phi)\ra,\\
& \leq |\sbl_{\clie}(Z)(\phi)||\clie(\phi)| + \tfrac{2k}{k+1}| \sbl_{\klie}(Z)(\phi)|| \klie(\phi)| + \tfrac{k(n+2(k-2))}{(n-3+k)(n+2(k-1))}|i(Z)\phi|| \div(\phi)|,
\end{split}
\end{align}
in which angled brackets denote the inner product given by $h_{ij}$, and the inequality is simply that of Cauchy-Schwarz.

Suppose $\phi \in \Ga(\symkt) \cap \ker \klie\cap \ker \div$. By \eqref{normdom} there holds $|D\phi|^{2} = |\clie(\phi)|^{2}$. Substituting this and \eqref{sblclie} into \eqref{katoineq1} gives
\begin{align*}
|\phi|^{2}|Z^{i}d_{i}|\phi||^{2} = \tfrac{1}{4}|Z^{i}d_{i}|\phi|^{2}|^{2}\leq |\sbl_{\clie}(Z)(\phi)|^{2}|\clie(\phi)|^{2}  = |\sbl_{\clie}(Z)(\phi)|^{2}|D\phi|^{2} \leq   \tfrac{n-2+k}{n+2(k-1)}|\phi|^{2}|D\phi|^{2}.
\end{align*}
This holds for all unit norm $Z^{i}$, so shows \eqref{kato}.

Suppose $\phi \in \Ga(\symkt) \cap \ker \clie\cap \ker \div$. By \eqref{normdom} there holds $|D\phi|^{2} = \tfrac{2k}{k+1}|\klie(\phi)|^{2}$, and by \eqref{kliesymbolnorm} there holds $2|\sbl_{\klie}(Z)(\phi)|^{2} \leq |\phi|^{2}$. In \eqref{katoineq1} these give
\begin{align*}
|\phi|^{2}|Z^{i}d_{i}|\phi||^{2} = \tfrac{1}{4}|Z^{i}d_{i}|\phi|^{2}|^{2}\leq (\tfrac{2k}{k+1})^{2}|\sbl_{\klie}(Z)(\phi)|^{2}|\klie(\phi)|^{2}  = \tfrac{2k}{k+1}|\sbl_{\klie}(Z)(\phi)|^{2}|D\phi|^{2} \leq   \tfrac{k}{n+1}|\phi|^{2}|D\phi|^{2}.
\end{align*}
This holds for all unit norm $Z^{i}$, so shows \eqref{kato2}.

Suppose $\phi \in \Ga(\symkt) \cap \ker \clie\cap \ker \klie$. There holds $|D\phi|^{2} =\tfrac{k(n+2(k-2))}{(n-3+k)(n+2(k-1))}|\div(\phi)|^{2}$ by \eqref{normdom}. With \eqref{klieinequality} in \eqref{katoineq1} this gives
\begin{align*}
|\phi|^{2}|Z^{i}d_{i}|\phi||^{2}& = \tfrac{1}{4}|Z^{i}d_{i}|\phi|^{2}|^{2}     \leq\left(\tfrac{k(n+2(k-2))}{(n-3+k)(n+2(k-1))}\right)^{2}|i(Z)\phi|^{2}| \div(\phi)|^{2} \\
&= \tfrac{k(n+2(k-2))}{(n-3+k)(n+2(k-1))}|i(Z)\phi|^{2}|D\phi|^{2} \leq  \tfrac{k}{n+2(k-1)}|\phi|^{2}|D\phi|^{2}.
\end{align*}
This holds for all unit norm $Z^{i}$, so shows \eqref{kato3}.
\end{proof}

\begin{remark}
It can be proved that when $n = 2$ the inequalities \eqref{kato}, \eqref{kato2}, and \eqref{kato3} are in fact equalities. Since in the two-dimensional case stronger results are available using Riemann-Roch, and attention here is focused on the case $n > 2$, the proof is omitted.
\end{remark}

\begin{remark}\label{katoremark}
The proof of Lemma \ref{katolemma} shows that if $\phi \in \Ga(\symkt)$ then the largest eigenvalue $\mu$ of the symmetric two-tensor $\phi_{ij} \defeq \phi_{ia_{1}\dots a_{k-1}}\phi_{j}\,^{a_{1}\dots a_{k-1}}$ satisfies $\mu \leq \tfrac{n-3+k}{n+2(k-2)}|\phi|^{2}$. By \eqref{klieinequality} for any vector field $X^{i}$ there holds $X^{i}X^{j}\phi_{ij} = |i(X)\phi|^{2} \leq \tfrac{n-3+k}{n+2(k-2)}|\phi|^{2}|X|^{2}$, which suffices to show the claim. This means $\phi_{ij} \leq \tfrac{n-3+k}{n+2(k-2)}|\phi|^{2}h_{ij}$. In particular, for any AH structure there holds $\bt_{ij} \leq \tfrac{n}{n+2}\bt H_{ij}$.
\end{remark}

\subsubsection{}
Let $\om \in \Ga(\symkt)$. Wherever $|\om| > 0$ there holds
\begin{align}\label{lapnormla}
\tfrac{1}{2\la}|\om|^{2(1-\la)}\lap_{h}|\om|^{2\la} = \tfrac{1}{2}\lap_{h}|\om|^{2} + 2(\la - 1)|d|\om||^{2}.
\end{align}
Combining \eqref{lapnormla} with \eqref{katolemma} and the equations \eqref{lapom}, \eqref{lapomdivlie}, and \eqref{lapomliediv} yields
\begin{lemma}\label{swlemma}
Let $h$ be a Riemannian metric on a manifold $M$ of dimension $n > 2$ and let $\om \in \Ga(\symkt)$. Then wherever $\om \neq 0$ there hold
\begin{align}
\label{sharplapom}& |\om|^{(n+2(k-1))/(n-2+k)}\lap_{h}|\om|^{(n-2)/(n-2+k)}\geq \tfrac{n-2}{n-2+k}\qR(\om), && \text{if}\,\, \om \in \ker \klie \cap \ker \div,\\
& |\om|^{(k+1)/k}\lap_{h}|\om|^{(k-1)/k}\geq \tfrac{2(k-1)}{k+1}\lb \om, \kliea\klie(\om)\ra, && \text{if}\,\, \om \in \ker \clie \cap \ker \div,\\
& |\om|^{(n+2(k-1))/k}\lap_{h}|\om|^{(2-n)/k}\leq \tfrac{n-2}{n-2+k}\qR(\om),&& \text{if}\,\, \om \in \ker \klie \cap \ker \clie.
\end{align}
\end{lemma}

\begin{remark}
A number of classical results are contained in the estimates leading to Lemma \ref{swlemma}. Here are mentioned a few typical ones. Suppose $h$ is flat and $f \in \cinf(M)$ is harmonic. Then it is easy to check that $\om_{i_{1}\dots i_{k}}\defeq D_{i_{1}}\dots D_{i_{k}}f$ is in $\Ga(\symkt) \cap \ker \klie \cap \ker \div$. By Lemma \ref{swlemma} the function $|\om|^{p}$ is subharmonic for all $p \geq (n-2)/(n-2+k)$. For the flat Euclidean connection on $\rea^{n}$ and $k = 1$ this is Theorem $A$ of \cite{Stein-Weiss-harmonic}, and for $k >1$ it is Theorem $1$ of \cite{Calderon-Zygmund}. In the opposite direction, Theorem $2(b)$ of \cite{Stein-Weiss} shows that on flat Euclidean space the best $p$ for which $|\om|^{p}$ is subharmonic is $(n-2)/(n-2+k)$, and Theorem $2(a)$ of \cite{Stein-Weiss} shows that on flat Euclidean space, given any section $\om$ of $\symkt$ there is around every point a neighborhood $U$ and a harmonic function $f \in \cinf(U)$ such that on $U$ there holds $\om_{i_{1}\dots i_{k}}= D_{i_{1}}\dots D_{i_{k}}f$.

If $f \in \cinf(M)$ is harmonic then $df \in \ker \klie \cap \ker \div$, and Lemma \ref{swlemma} shows that $(n-1) |df|^{n/(n-1)}\lap_{h}|df|^{(n-2)/(n-1)}\geq (n-2)\ric(df, df)$; if $h$ has non-negative Ricci curvature it follows that $|df|^{p}$ is subharmonic for any $p \geq (n-2)/(n-1)$. If there is $\ka \in \reap$ such that $\ric \geq - \ka(n-1)h$, then any harmonic $f$ satisfies $\lap_{h}|df|^{(n-2)/(n-1)}\geq -\ka (n-2)|df|^{(n-2)/(n-1)}$.

As is shown in \cite{Berger-Ebin}, an infinitesimal deformation of an Einstein metric $h$ on a compact manifold is identified with an $\om \in \Ga(S^{2}_{0}(\ctm))\cap \ker \div$ solving the equation $\lap \om = 2\sR(\om) - \tfrac{2\sR}{n}\om$. As is summarized in section $12$.H of \cite{Besse} (the notations there are somewhat different than those here), using this equation in conjunction with the positivity conditions given by integrating \eqref{divlie} and \eqref{lapom3} gives a proof of the criterion of N. Koiso, (Theorem $3.3$ of \cite{Koiso-nondeformability}), for the rigidity of an Einstein metric, in particular showing that an Einstein metric of negative sectional curvature is rigid provided $n \geq 3$.

Suppose $\om_{ij} \in \Ga(S^{2}_{0}(\ctm))\cap \ker \klie \cap \div$. If the sectional curvature is non-negative then $\qR$ is non-negative on $S^{2}_{0}(\ctm)$ so it follows from \eqref{sharplapom} that $|\om|^{(n-2)/n}$ is subharmonic; if $M$ is compact this means $|\om|$ is constant, and by \eqref{lapomliediv} this implies $\om$ is parallel. Moreover, if the sectional curvature is positive at some point, then $\qR$ is positive on $S^{2}_{0}(\ctm)$ at that point and $\om$ must be identically zero. This recovers the well-known
\begin{theorem}[M. Berger - D. Ebin, \cite{Berger-Ebin}]\label{codazzitensortheorem}
On a compact Riemannian manifold with non-negative sectional curvature a tensor $b_{ij} \in \Ga(S^{2}(\ctm))$ such that $D_{[i}b_{j]k} = 0$ and $D_{i}b_{p}\,^{p} = 0$ is parallel; if, moreover, the sectional curvature is somewhere positive, then $b_{ij}$ is a constant multiple of the metric.
\end{theorem}
\begin{corollary}
On a compact manifold of dimension $n > 2$ a Riemannian metric $h$ has harmonic curvature if and only if the trace-free part $\mr{\sR}_{ij}$ of its Ricci tensor is contained in $\ker \klie \cap \ker \div$. In this case if $h$ has non-negative sectional curvature which is somewhere positive then it is Einstein. 
\end{corollary}
\begin{proof}
There hold $\div(\mr{\sR})_{i} = \tfrac{n-2}{2n}D_{i}\sR$ and $2\klie(\mr{\sR})_{ijk} = D_{p}\sR_{ijk}\,^{p} + \tfrac{1}{n-1}h_{k[i}D_{j]}\sR$, from which it follows that if $n > 2$ then $\mr{\sR}_{ij} \in \ker \klie \cap \ker \div$ if and only if $D_{p}\sR_{ijk}\,^{p} = 0$. The claim follows by applying Theorem \ref{codazzitensortheorem} to $\mr{\sR}_{ij}$.
\end{proof}

\end{remark}

\subsection{Alternative formulations of the Riemannian Einstein AH equations}\label{eigensection}
The main result of this section is Theorem \ref{eigentheorem} which shows that the Einstein AH equations for a Riemannian AH structure with self-conjugate curvature can be formulated entirely in terms of equations for the Ricci curvature of a Gauduchon metric and partial differential equations on the cubic torsion. Conversely, given a metric and a tensor satisfying these equations, there results an Einstein AH structure. It is anticipated that the systematic construction of Einstein AH equations can be achieved by solving the equations of Theorem \ref{eigentheorem}, or by studying flows associated to these equations (not discussed here).

\subsubsection{}\label{estimatenotationsection}
Let $(\en, [h])$ be a Riemannian AH structure on a manifold of dimension at least $3$. Let $h \in [h]$ and let $\ga_{i}$ be the corresponding one-form and $D$ the Levi-Civita connection. In this section indices will be raised and lowered using $h^{ij}$ and $h_{ij}$. Define $L_{ijk} = \bt_{ij}\,^{p}h_{pk}$ and $L_{ij} = L_{ip}\,^{q}L_{jq}\,^{p} = \bt_{ij}$. The unweighted analogues of the tensors $\tbt_{ijkl}$, $\bt_{ijkl}$, $\bt_{ij}$, and $\bt$ are $C_{ijkl} \defeq \tbt_{ijk}\,^{p}h_{pl}$, $L_{ijkl} \defeq 2L_{k[i}\,^{p}L_{j]lp}$, $L_{ij} \defeq  L_{pij}\,^{p} = L_{ip}\,^{q}L_{jq}\,^{p} = \bt_{ij}$, and $L \defeq L_{p}\,^{p} = h^{ij}L_{ij} = |L|_{h}^{2} = L^{ijk}L_{ijk} = |\det h|^{-1/n}\bt$. Although an expression such as $|L|_{h}^{2}$ is ambiguous, here it will be used always to mean $L^{ijk}L_{ijk}$. Sometimes, for readability, the subscript $h$ will be omitted, in particular when $|L|_{h}$ is taken to some power other than $2$. Write $R^{\flat}_{ijkl} = R_{ijk}\,^{p}h_{pl}$ to indicate the unweighted analogue of $R_{ijkl}$ defined by lowering the last index using $h_{ij}$ rather than $H_{ij}$, and do likewise for other tensors, e.g. $A_{ijkl}^{\flat}$. An unweighted tensor such as $R_{ijk}\,^{l}$ or $R_{ij}$ is indicated in the same way as before.

\subsubsection{}
For any Riemannian signature AH structure $(\en, [h])$ and any $h \in [h]$ there result from \eqref{ddivbt} and \eqref{uplusv},
\begin{align}\label{cl1}
D_{[i}L_{j]kl} &= -E_{ijkl}^{\flat} + 2h_{l[i}E_{j]k} + 2h_{k[i}E_{j]l} + h_{l[i}L_{j]k}\,^{p}\ga_{p} + h_{k[i}L_{j]l}\,^{p}\ga_{p},\\
\label{divbt}\div(L)_{ij} & = D_{p}L_{ij}\,^{p} = 2nE_{ij} + n\ga_{p}L_{ij}\,^{p},\\
\label{kliebt}\klie(L)_{ijkl} &= -E_{ijkl}.
\end{align}
Expanding $L^{abc}E^{\flat}_{iabc}$ using \eqref{cl1}, observing $L^{abc}D_{i}L_{abc} - L^{abc}D_{a}L_{bci} = \tfrac{1}{2}D_{i}|L|^{2}_{h} - D^{a}L_{ia} + L_{bci}D_{a}L^{abc}$, using \eqref{ddivbt} to simplify the last expression, and using \eqref{ebtw} gives
\begin{align}\label{conservativediv}
2(2-n)A_{i}^{\flat} = \tfrac{1}{2}D^{p}\left(L_{ip} - \tfrac{1}{2}|L|^{2}h_{ip}\right) + (2-n)L_{i}\,^{pq}E_{pq} + \tfrac{2-n}{2}L_{i}\,^{p}\ga_{p}.
\end{align}
In particular, it follows from \eqref{conservativediv} that an exact naive Einstein AH structure is conservative (and hence Einstein) if and only if $L_{ij} - \tfrac{1}{2}|L|^{2}h_{ij}$ is $D$-divergence free. Also, by Theorem \ref{todtheorem} if $h$ is a Gauduchon metric for an Einstein AH structure on a compact manifold then $L_{ij} - \tfrac{1}{2}|L|^{2}h_{ij}$ is $D$-divergence free.
\begin{lemma}
Let $(\en, [h])$ be a Riemannian signature AH structure on an $n$-manifold $M$. There holds $E_{ijkl} = 0$ if and only if for any $h \in [h]$ there holds $L_{ijk} \in \Ga(S^{3}_{0}(\ctm))\cap \ker \klie$. If $M$ is compact then $E_{ijkl} = 0$ if and only if $\kliea\klie(L)_{ijk} = -D^{p}E_{p(ijk)} = 0$. If $(\en, [h])$ has self-conjugate curvature and either 
\begin{enumerate}
\item $(\en, [h])$ is Einstein, $M$ is compact, and $h \in [h]$ is a Gauduchon metric with associated Faraday form $\ga_{i}$; 
\item or $(\en, [h])$ is exact ($M$ is not necessarily compact and $(\en, [h])$ is not necessarily Einstein) with distinguished metric $h \in [h]$, 
\end{enumerate}
then $L_{ijk} \in  \Ga(S^{3}_{0}(\ctm))\cap \ker \klie \cap \ker \div$.
\end{lemma}
\begin{proof}
The first claim is immediate from \eqref{kliebt}. From \eqref{ebtw}, \eqref{cl1}, and \eqref{kliebt} there follows
\begin{align}\label{lbt1}
-L^{ijk}\kliea\klie(L)_{ijk} = L^{ijk}D^{p}E^{\flat}_{pijk} = D^{p}(L^{ijk}E^{\flat}_{pijk}) - E^{\flat}_{pijk}D^{p}L_{ijk} = 2(2-n)D^{p}A_{p}^{\flat} + |E|_{h}^{2}.
\end{align}
If $M$ is compact, integrating \eqref{lbt1} shows that $E_{ijkl} = 0$ if and only if $\kliea\klie(L)_{ijk} = -D^{p}E^{\flat}_{p(ijk)} = 0$. 
The final claim (treating self-conjugate curvature) follows from \eqref{divbt} and \eqref{kliebt}, in conjunction with Theorem \ref{todtheorem} in case $(1)$.
\end{proof}

\subsubsection{}
Let $h$ be a Riemannian metric. From \eqref{kliedefined} it follows that for any $\om \in \Ga(\symkt)$ there holds
\begin{align}\label{omklieom}
\begin{split}
\om^{a_{1}\dots a_{k}}\klie(\om)_{ia_{1}\dots a_{k}} &= \om^{a_{1}\dots a_{k}}D_{[i}\om_{a_{1}]a_{2}\dots a_{k}} + \tfrac{1-k}{2(n-3+k)}\om_{i}\,^{a_{1}\dots a_{k-1}}\div(\om)_{a_{1}\dots a_{k-1}}\\
& = \tfrac{1}{4}D_{i}|\om|^{2} - \tfrac{1}{2}\omega^{a_{1}\dots a_{k}}D_{a_{1}}\om_{a_{2}\dots a_{k}i} + \tfrac{1-k}{2(n-3+k)}\om_{i}\,^{a_{1}\dots a_{k-1}}\div(\om)_{a_{1}\dots a_{k-1}}.
\end{split}
\end{align}
Differentiating \eqref{divsi} shows that for the associated tensor $\si_{ij} \defeq \om_{ia_{1}\dots a_{k-1}}\om_{j}\,^{a_{1}\dots a_{k-1}}$ there hold
\begin{align}\label{divsi}
\begin{split}
\div(\si)_{i} &= \tfrac{n-2}{n-3+k}\div(\om)^{a_{1}\dots a_{k-1}}\om_{ia_{1}\dots a_{k-1}} - 2\om^{a_{1}\dots a_{k}}\klie(\om)_{ia_{1}\dots a_{k}} + \tfrac{1}{2}D_{i}|\om|^{2}.
\end{split}
\end{align}
In particular, if $\om \in \Ga(\symkt)\cap \ker \klie \cap \ker \div$ then $\si_{ij} - \tfrac{1}{2}|\om|_{h}^{2}$ is divergence free.

Let $\culap$ be the formally self-adjoint second order elliptic differential operator on $\Ga(S^{k}_{0}(\ctm))$ defined by $\culap\om = \lap_{h}\om - \sR(\om)$. 

\begin{lemma}\label{tracefreeconstantlemma}
Let $M$ be a manifold of dimension at least $3$ and let $h$ be a Riemannian metric. Let $\om \in\Ga(S^{k}_{0}(\ctm))$ and write $\si_{ij} = \om_{ia_{1}\dots a_{k-1}}\om_{j}\,^{a_{1}\dots a_{k-1}}$ and let $X^{i}$ be an $h$-Killing field and write $\ga_{i} = X^{p}h_{ip}$ for the dual one-form. If $\om \in \Ga(\symkt)  \cap \ker \div$ satisfies $\om^{a_{1}\dots a_{k}}\klie(\om)_{ia_{1}\dots a_{k}} = 0$, then the equations 
\begin{align}\label{stressenergy2}
\sR_{ij} - \tfrac{1}{2}\sR_{h}h_{ij} + \tfrac{n-2}{2n}\ka h_{ij} = \tfrac{1}{4}\left(\si_{ij} -  \tfrac{1}{2}|\om|_{h}^{2}h_{ij}\right) + (2-n)\left(\ga_{i}\ga_{j} + \tfrac{1}{2}|\ga|_{h}^{2}h_{ij}\right),
\end{align}
for $\om$, $X$, and $h$ are consistent and $\ka =  \sR_{h} - \tfrac{1}{4}|\om|_{h}^{2} - (n+2)|\ga|_{h}^{2}$ is a constant. In particular, these conclusions hold if $M$ is compact and $\om \in \Ga(\symkt)\cap \ker\culap$. In particular, if $M$ is compact and $\om \in \Ga(\symkt)\cap \ker \culap$, the equations \eqref{stressenergy2} hold for some constant $\ka$ if and only if there holds 
\begin{align}\label{tracefreericcondition}
\mr{R}_{ij} -\tfrac{1}{4} \mr{\si}_{ij} + (n-2)\mr{\ga \tensor \ga}_{ij} =0.
\end{align}
\end{lemma}

\begin{proof}
From \eqref{divsi} it follows that $\si_{ij} - \tfrac{1}{2}|\om|_{h}^{2}$ is divergence free. Since $X$ is Killing, also $\ga_{i}\ga_{j} + \tfrac{1}{2}|\ga|_{h}^{2}h_{ij}$ is divergence free. This forces $\ka$ to be constant, and the explicit form for $\ka$ follows by taking traces. If $M$ is compact, then Lemma \ref{culaplemma} shows that $\ker \culap \cap \Ga(S^{k}_{0}(\ctm)) = \ker \klie \cap \ker \div \cap \Ga(S^{k}_{0}(\ctm))$, and the remaining claims followfrom the preceeding.
\end{proof}

The equations \eqref{stressenergy2} resemble formally (the metric has the wrong signature) the gravitational field equations with a possibly non-zero cosmological constant and stress energy tensor equal to the righthand side of \eqref{stressenergy2}. Observe also that for the equations \eqref{stressenergy2} to make sense it need not be the case that $h$ be Riemannian; the operators $\klie$ and $\div$ are defined in any signature and the equation \eqref{divsi} is always valid, which is all that is needed to show the consistency of \eqref{stressenergy2} for $\om \in \Ga(\symkt)  \cap \ker \div$ satisfying $\om^{a_{1}\dots a_{k}}\klie(\om)_{ia_{1}\dots a_{k}} = 0$.

\begin{corollary}\label{stressenergycorollary}
Let $(\en, [h])$ be a Riemannian signature Einstein AH structure on a manifold $M$ of dimension $n \geq 3$. If either $M$ is compact and $h \in [h]$ is a Gauduchon metric with associated Faraday form $\ga_{i}$, or $(\en, [h])$ is exact (and $M$ is not necessarily compact) with distinguished metric $h \in [h]$ then $L_{ij} - \tfrac{1}{2}|L|_{h}^{2}$ is divergence free and the Ricci curvature $\sR_{ij}$ of $h$ satisfies
\begin{align}\label{stressenergy}
\sR_{ij} - \tfrac{1}{2}\sR_{h}h_{ij} + \tfrac{n-2}{2n}\ka h_{ij} = \tfrac{1}{4}\left(L_{ij} - \tfrac{1}{2}|L|_{h}^{2}h_{ij}\right) + (2-n)\left(\ga_{i}\ga_{j} + \tfrac{1}{2}|\ga|_{h}^{2}h_{ij}\right),
\end{align}
with $\ka = \sR_{h} - \tfrac{1}{4}|L|_{h}^{2} - (n+2)|\ga|_{h}^{2} = \uR_{h} + n(n-4)|\ga|_{h}^{2}$ a constant. 
\end{corollary}
\begin{proof}
By \eqref{ebtw} and \eqref{kliebt}, that $(\en, [h])$ be Einstein implies $L^{abc}\klie(L)_{iabc} =0$. By \eqref{divbt} and Theorem \ref{todtheorem}, $\div(L) = 0$. The claim follows from Lemma \ref{tracefreeconstantlemma}.
\end{proof}

\begin{theorem}\label{eigentheorem}
On a manifold of dimension $n \geq 3$, let $(\en, [h])$ be a Riemannian signature Einstein AH structure with self-conjugate curvature and cubic torsion $\bt_{ij}\,^{k}$. If either $M$ is compact and $h \in [h]$ is a Gauduchon metric with Faraday primitive $\ga_{i}$, or $(\en, [h])$ is exact with distinguished metric $h \in [h]$, and in either case $L_{ijk} \defeq \bt_{ij}\,^{p}h_{pk}$, then $h$, $\ga$, and $L$ solve the equations
\begin{align}
\label{lapeinstein1}&\div\clie(L) = \lap_{h}L = \sR(L),&\\
\label{lapeinstein3}
\begin{split}&\div(L) = 0,\\
&\klie(L) = 0,
\end{split}\\
\label{dicheq}
\begin{split}
&\ga_{p}L_{ij}\,^{p} = 0,\\
&D_{(i}\ga_{j)} = 0,
\end{split}\\
\label{stressenergy3}&\sR_{ij} - \tfrac{1}{2}\sR_{h}h_{ij} + \tfrac{n-2}{2n}\ka h_{ij} = \tfrac{1}{4}\left(L_{ij} - \tfrac{1}{2}|L|_{h}^{2}h_{ij}\right) + (2-n)\left(\ga_{i}\ga_{j} + \tfrac{1}{2}|\ga|_{h}^{2}h_{ij}\right),
\end{align}
in which $\ka = \sR_{h} - \tfrac{1}{4}|L|_{h}^{2} - (n+2)|\ga|_{h}^{2} = \uR_{h} + n(n-4)|\ga|_{h}^{2}$ is constant. If $(\en, [h])$ is exact then \eqref{stressenergy3} can be replaced by
\begin{align}
\label{lapeinstein2}&\sR_{ij} = \tfrac{1}{4}L_{ij} + \tfrac{\ka}{n} h_{ij},
\end{align}
while \eqref{dicheq} is vacuous. 

Given a Riemannian metric $h$ with Levi-Civita connection $D$, $L \in \Ga(\S^{3}_{0}(\ctm))$ solving \eqref{lapeinstein3}, and an $h$-Killing vector field $\ga^{i}$ such that $\ga_{i} = \ga^{p}h_{ip}$ solves \eqref{dicheq}, which all together solve \eqref{stressenergy3}, the connection $\nabla = D - \tfrac{1}{2}h^{kp}L_{ijp} - 2\ga_{(i}\delta_{j)}\,^{k} + h_{ij}\ga^{k}$ is the aligned representative of an Einstein AH structure with self-conjugate curvature for which $h$ is a Gauduchon metric with Faraday primitive $\ga_{i}$. If $M$ is compact then the same conclusion obtains provided $h$, $L$, and $\ga$ together solve \eqref{dicheq}, \eqref{tracefreericcondition}, and the second equation of \eqref{lapeinstein1}.
\end{theorem}
\begin{proof}
If $(\en, [h])$ is Riemannian Einstein with self-conjugate curvature then $L \in \ker \klie$ by \eqref{kliebt}. If $M$ is compact then Theorem \ref{todtheorem} implies \eqref{dicheq} and with \eqref{divbt} this shows $L \in \ker \div$, while if $(\en, [h])$ is exact, then $L \in \ker \div$ follows from \eqref{divbt} alone. In either case, by \eqref{lapom3} there holds the second equation of \eqref{lapeinstein1}; the first equation of \eqref{lapeinstein1} then follows from \eqref{lapom}. Equation \eqref{stressenergy3} results from \eqref{confric} and \eqref{dicheq}, while that $\ka$ is constant was concluded in Corollary \ref{stressenergycorollary}.

The converse is completely routine. If there are given $h$, $L$, and $\ga$ solving \eqref{lapeinstein3}, \eqref{dicheq}, and \eqref{stressenergy3}, then $\nabla = D - \tfrac{1}{2}h^{kp}L_{ijp} - 2\ga_{(i}\delta_{j)}\,^{k} + h_{ij}\ga^{k}$ generates with $[h]$ an AH structure for which $\ga_{i}$ is the Faraday primitive associated to $h$ and with cubic torsion $h^{kp}L_{ijp}$. Because \eqref{dicheq} implies that $\dad_{h}\ga = 0$, $h$ is a Gauduchon metric. By \eqref{dicheq} and \eqref{divbt} there holds $E_{ij} = 0$, while by \eqref{kliebt} and $\klie(L) = 0$ there holds $E_{ijkl} = 0$, so that the curvature of $(\en, [h])$ is self-conjugate. Together \eqref{dicheq}, \eqref{stressenergy3}, and \eqref{confric} show that $\mr{R}_{ij} = 0$, so that $(\en, [h])$ is Einstein. 

If $M$ is compact then by Lemma \ref{culaplemma} the second equation of \eqref{lapeinstein1} implies \eqref{lapeinstein3}; together \eqref{lapeinstein3} and \eqref{lapom} imply the first equation of \eqref{lapeinstein1}. By Corollary \ref{stressenergycorollary} and \eqref{tracefreericcondition} this means that there is a constant $\ka$ such that there holds \eqref{stressenergy3}, so that there hold all of \eqref{lapeinstein1}-\eqref{stressenergy3} and the preceeding paragraph applies.
\end{proof}

In the statement that \eqref{lapeinstein1} implies \eqref{lapeinstein3}, the compactness of $M$ could be relaxed to some condition, such as that $L$ have bounded $W^{1,2}$ norm, in the presence of which $(n+2)\clie \div(L)  + 2n\kliea\klie(L) = 0$ implies the vanishing of $\div(L)$ and $\klie(L)$.

In the remainder of this subsection there are made some remarks about the equations of Theorem \ref{eigentheorem} and some questions are raised.

\begin{corollary}
Let $M$ be a compact manifold of dimension $n \geq 3$ and let $h$ be a Riemannian metric such that $\qR$ is positive on $S^{3}_{0}(\ctm)$. If $h$ is the distinguished metric of an exact Einstein AH structure, then the AH structure is exact Weyl and $h$ is itself an Einstein metric.
\end{corollary}
\begin{proof}
Integrating \eqref{lapeinstein1} gives $0 \leq \int_{M}|\clie(L)|^{2}_{h} = -\int_{M}\qR(L)$, and so $L$ is identically zero. 
\end{proof}
On the other hand, on a compact manifold of dimension at least $3$ a necessary condition for the existence of a Riemannian signature Einstein AH structure with self-conjugate curvature and cubic torsion $h^{kp}L_{ijp}$ for a given $L_{ijk}$ and a distinguished metric $h_{ij}$ is that $\int_{M}\qR(L)\,d\vol_{h} \leq 0$. This suggests that it should be possible to solve \eqref{lapeinstein1}-\eqref{lapeinstein2} on a manifold for which $\qR$ is non-positive on $S^{3}_{0}$.

\subsubsection{}

Let $M$ be a compact $n$-manifold and $h$ a Riemannian metric of constant sectional curvature $\sR/n(n-1)$. From $\sR_{ijkl} = \tfrac{2\sR}{n(n-1)}h_{l[i}h_{j]k}$, equation \eqref{lapeinstein1} simplifies to $\lap_{h}L_{ijk} =\tfrac{\sR(n+1)}{n(n-1)}L_{ijk}$. Supposing $L \in \ker \klie \cap \ker \div$, integrating $\lap_{h}|L|^{2}_{h}$ and using Lemma \ref{katolemma} shows 
\begin{align}
\tfrac{n+4}{n+1}\int_{M}|d|L||^{2}_{h}\,d\vol_{h} \leq \int_{M}|DL|^{2}_{h}\,d\vol_{h} = - \tfrac{\sR(n+1)}{n(n-1)}\int_{M}|L|^{2}_{h}\,d\vol_{h}.
\end{align}
In particular, if $\sR > 0$ then $L$ must vanish identically, while if $\sR = 0$, then $L$ must be parallel. These observations suggest that Theorem \eqref{eigentheorem} is not useful in finding exact Einstein AH structures the distinguished metrics of which have in some sense positive curvature (and moreover, that there may not be many such examples). If $\sR < 0$ then computing using \eqref{katolemma} shows that where $|L| \neq 0$ there holds
\begin{align}\label{hyperineq}
\lap |L|^{(n-2)/(n+1)} \geq \tfrac{\sR(n-2)}{n(n-1)}|L|^{(n-2)/(n+1)}.
\end{align}
For a hyperbolic $n$-manifold with $n > 2$ it is not clear whether one should expect the pair of equations \eqref{lapeinstein1} and \eqref{lapeinstein3} to be solvable. Starting with a Riemannian metric with some non-positivity condition on its curvature, there needs to be solved the eigenvalue problem \eqref{lapeinstein1} subject to the algebraic constraint \eqref{lapeinstein2}; the first part is a typical problem of analyzing the spectrum of the Laplacian on a tensor bundle, but it is less clear how to analyze it in the presence of the second part. While the case of Riemann surfaces suggests yes, rigidity phenomena for higher dimensional hyperbolic manifolds suggest that \eqref{lapeinstein3} could present problems. 
It seems unlikely that the inequality \eqref{hyperineq} by itself gives any obstruction to solving \eqref{lapeinstein1} on a flat hyperbolic manifold. In section \ref{cubicformsection} it is shown that on flat Euclidean space there are solutions to \eqref{lapeinstein1} and \eqref{lapeinstein3}, and this suggests that there should be solutions on, for example, Cartan-Hadamard manifolds.

\subsubsection{}\label{variationalsection}
Let $M$ be a compact manifold. Define a functional $\cubic$ with arguments a Riemannian metric $h_{ij}$ and a tensor $A_{ijk} \in \Ga(S^{3}_{0}(\ctm))$ by 
\begin{align*}
\cubic(h, A) \defeq \int_{M}\left( |\clie(A)|_{h}^{2} + \qR(A) \right) \,d\vol_{h}.
\end{align*}
For fixed $h$ the first variation of $\cubic(h, A)$ in $A$ yields the equations \eqref{lapeinstein1}. Similarly, the first variation in $h$ of the functional
\begin{align*}
\mathscr{S}(h, A) \defeq (\vol_{h}(M))^{(2-n)/n}\int_{M}\left( \sR + \tfrac{1}{12}|A|^{2}_{h}\right)\,d\vol_{h},
\end{align*}
yields \eqref{lapeinstein2}. However, it is not clear how to give a variational formulation which directly leads to both \eqref{lapeinstein1} and \eqref{lapeinstein2}.

\subsubsection{}
Let $(\en, [h])$ be an Einstein AH structure on a compact $n$-manifold with $n > 2$ and let $h \in [h]$ be a Gauduchon metric. For a Gauduchon metric, because of Theorem \ref{todtheorem}, \eqref{confcurvijkl} simplifies to
\begin{align}
\begin{split}
\sR_{ijkl} &= T^{\flat}_{ijkl} + 2|\ga|_{h}^{2}h_{l[i}h_{j]k} + \tfrac{1}{4}L_{ijkl} - 4\ga_{[i}h_{j][k}\ga_{l]} ,
\end{split}
\end{align}
and if $(\en, [h])$ is Einstein this reduces further to
\begin{align}\label{srijklexact}
\begin{split}
\sR_{ijkl} &= A^{\flat}_{ijkl} +  \left(\tfrac{2}{n(n-1)}\uR_{h}  + 2|\ga|_{h}^{2}\right) h_{l[i}h_{j]k} + \tfrac{1}{4}L_{ijkl} - 4\ga_{[i}h_{j][k}\ga_{l]} \\
& = A^{\flat}_{ijkl} +  \left(\tfrac{2}{n(n-1)}\left(\sR_{h} - \tfrac{1}{4}|L|_{h}^{2}\right) -\tfrac{4}{n} |\ga|_{h}^{2}\right) h_{l[i}h_{j]k} + \tfrac{1}{4}L_{ijkl} - 4\ga_{[i}h_{j][k}\ga_{l]}.
\end{split}
\end{align}
In this case, using \eqref{confric} and \eqref{confscal} yields
\begin{align}\label{srl}
\begin{split}
\sR(L)_{ijk} & = A^{\flat}(L)_{ijk} - \tfrac{1}{2}L_{ia}\,^{b}L_{jb}\,^{c}L_{kc}\,^{a} + \tfrac{3}{4}L_{p(i}L_{jk)}\,^{p} + \left(\tfrac{n+1}{n(n-1)}\uR_{h} + n |\ga|_{h}^{2} \right)L_{ijk}\\
& =  A^{\flat}(L)_{ijk} - \tfrac{1}{2}L_{ia}\,^{b}L_{jb}\,^{c}L_{kc}\,^{a} + \tfrac{3}{4}L_{p(i}L_{jk)}\,^{p} \\ &\quad - \tfrac{n+1}{4n(n-1)}|L|_{h}^{2}L_{ijk} + \left(\tfrac{n+1}{n(n-1)}\sR_{h} + \tfrac{n+2}{n} |\ga|_{h}^{2} \right)L_{ijk}.
\end{split}
\end{align}
Substituting this into \eqref{lapeinstein1} is not useful in general because the righthand side of \eqref{srl} depends on $A_{ijkl}$, which is not determined by $h$, $\ga$, and $L$. For example, if $(\en, [h])$ is an exact Einstein AH structure, then by \eqref{srijklexact}, \eqref{srl}, and \eqref{lapom} there holds
\begin{align}\label{laphlexact}
\begin{split}
\lap_{h}L_{ijk} &- 2\kliea\klie(L)_{ijk} = \sR(L)_{ijk}\\
& = \tfrac{(n+1)}{n(n-1)}\uR_{h}L_{ijk} - 2A_{p(ij}\,^{q}L_{k)q}\,^{p} - \tfrac{1}{2}L_{p(ij}\,^{q}L_{k)q}\,^{p} + \tfrac{1}{4}L_{p(i}L_{jk)}\,^{p},\\
& =  \tfrac{(n+1)}{n(n-1)}\uR_{h}L_{ijk} - 2A_{p(ij}\,^{q}L_{k)q}\,^{p} -\tfrac{1}{2}L_{ia}\,^{b}L_{jb}\,^{c}L_{kc}\,^{a} + \tfrac{3}{4}L_{pa}\,^{b}L_{(jk}\,^{p}L_{i)b}\,^{a}.
\end{split}
\end{align}
However, by Theorem \ref{projflatahlemma} an Einstein AH structure is projectively and conjugate projectively flat if and only if it is exact with self-conjugate curvature and $A_{ijkl} = 0$. Hence, one can try to find projectively and conjugate projectively flat (exact) Einstein AH structures by trying to find $h$ and $L$ solving the result of setting equal to zero $A_{ijkl}$ and $E_{ijkl}$ in \eqref{laphlexact},
\begin{align}\label{modifiedlapeinstein}
\lap_{h}L_{ijk} = - \tfrac{1}{2}L_{ia}\,^{b}L_{jb}\,^{c}L_{kc}\,^{a} + \tfrac{3}{4}L_{p(i}L_{jk)}\,^{p} - \tfrac{n+1}{4n(n-1)}|L|_{h}^{2}L_{ijk} + \tfrac{n+1}{n(n-1)}\sR_{h}L_{ijk},
\end{align}
(which also results from substituting \eqref{srl} with $A_{ijkl} = 0$ into \eqref{lapeinstein1} and setting $\ga = 0$) in conjunction with \eqref{lapeinstein3} (which does not change). The point is that the equation \eqref{modifiedlapeinstein} involves only $h$ and $L$. When $(\en, [h])$ is projectively flat, so $A_{ijkl} = 0 = E_{ijkl}$, \eqref{modifiedlapeinstein} recovers Equation $2.10$ of Calabi's \cite{Calabi-completeaffine} for affine hyperspheres (Calabi's $n(n-1)H$ is $\uR_{h}$ and his $2A_{ijk}$ is $L_{ijk}$).

Theorem \ref{convextheorem} shows that the pair of equations \eqref{modifiedlapeinstein} and \eqref{lapeinstein3} has non-trivial solutions. However, except in the two-dimensional case, one does not have a good idea how to characterize the metrics $h$ which arise in this way. It seems an interesting and difficult problem to find $h$ for which the pair of equations \eqref{modifiedlapeinstein} and \eqref{lapeinstein3} has solutions.

\subsection{Growth of the cubic torsion}\label{einsteinestimatesection}
Let $(\en, [h])$ be a Riemannian AH structure on a manifold of dimension at least $3$ and let the notational conventions be as in section \ref{estimatenotationsection}.
\subsubsection{}
For convenience there is recorded
\begin{align}\label{mrcdefined}
\begin{split}
C_{ijkl} &\defeq L_{ijkl} - \tfrac{2}{n-2}\left(h_{l[i}L_{j]k} - h_{k[i}L_{j]l}\right) + \tfrac{2}{(n-1)(n-2)}|L|_{h}^{2} h_{l[i}h_{j]k},\\
& = L_{ijkl} - \tfrac{2}{n-2}\left(h_{l[i}\mr{L}_{j]k} - h_{k[i}\mr{L}_{j]l}\right) - \tfrac{2}{n(n-1)}|L|_{h}^{2} h_{l[i}h_{j]k}.
\end{split}
\end{align}
From the non-negativity of the norm of $\mr{L}_{ij}$ there results $nL^{ij}L_{ij} \geq |L|_{h}^{4}$. From \eqref{mrcdefined} there follows
\begin{align}
\begin{split}
L^{ijkl}L_{ijkl} &= C^{ijkl}C_{ijkl} + \tfrac{4}{n-2}L^{ij}L_{ij} - \tfrac{2}{(n-1)(n-2)}|L|_{h}^{4} \\ &=  C^{ijkl}C_{ijkl} + \tfrac{4}{n-2}\mr{L}^{ij}\mr{L}_{ij} + \tfrac{2}{n(n-1)}|L|^{4}_{h}.
\end{split}
\end{align}
From the non-negativity of $C^{ijkl}C_{ijkl}$ and $\mr{L}^{ij}\mr{L}_{ij}$ there results $L^{ijkl}L_{ijkl}\geq  \tfrac{2}{n(n-1)}|L|_{h}^{4}$, but a slightly sharper result will be obtained by maintaining these terms in subsequent formulas. What will be needed is
\begin{align}\label{btnormestimate}
\begin{split}
L^{ijkl}L_{ijkl} + L^{ij}L_{ij} - C^{ijkl}C_{ijkl}&= \tfrac{n+2}{n-2}\mr{L}^{ij}\mr{L}_{ij} +  \tfrac{n+1}{n(n-1)}|L|_{h}^{4}.
\end{split}
\end{align}
Rewritten in terms of $L_{ijk}$, \eqref{btnormestimate} yields
\begin{align}\label{clest1}
&L^{pq}L_{pq} - \mt \geq \tfrac{1}{n-1}L^{pq}L_{pq} \geq \tfrac{1}{n(n-1)}|L|^{4}_{h},&
&\tfrac{3}{4}L^{ij}L_{ij} -\tfrac{1}{2}\mt \geq \tfrac{n+1}{4n(n-1)}|L|_{h}^{4},
\end{align}
in which $\mt = L_{ic}\,^{a}L_{ja}\,^{b}L_{kb}\,^{c}L^{ijk}$.
The estimates \eqref{clest1} are essentially the form in which \eqref{btnormestimate} was used in section $2$ of \cite{Calabi-improper}; here their use will be avoided. 

\subsubsection{}
Because $A_{ijkl}$ is completely trace-free there holds $L^{ijkl}A_{ijkl}^{\flat} = C^{ijkl}A_{ijkl}^{\flat}$. By \eqref{brtaijkl} the conformal Weyl tensor of the Weyl structure underlying the Codazzi projective structure generated by $(\en, [h])$ is $A_{ijkl} + \tfrac{1}{4}\C_{ijkl}$. Because the Levi-Civita connection of $h$ and the aligned representative of the underlying Weyl structure are conformal projectively equivalent, Theorem \ref{aeinvariant} shows that the (usual) conformal Weyl tensor $\sW_{ijkl}$ of $h$ is equal to the conformal Weyl tensor of the underlying Weyl structure, that is to $A^{\flat}_{ijkl} + \tfrac{1}{4}C_{ijkl}$. Hence
\begin{align}\label{qwl}
\qW(L) = L^{ijkl}\sW_{ijkl} = C^{ijkl}\sW_{ijkl} = C^{ijkl}A^{\flat}_{ijkl} + \tfrac{1}{4}C^{ijkl}C_{ijkl} = \qA(L) + \tfrac{1}{4}C^{ijkl}C_{ijkl},
\end{align} 
Later the condition $\qW(L) \geq 0$ will be imposed. By \eqref{qwl} this is a slightly weaker assumption than is $\qA(L) = C^{ijkl}A^{\flat}_{ijkl} \geq 0$.

\subsubsection{}
In order to utilize \eqref{lapomliediv} to estimate $|L|_{h}^{2}$, it is necessary to express the term $\qR(L)$ in a more useful form. From the definitions of $\sR(L)$ and $L_{ijkl}$ there follows $\qR(L) = L^{ijkl}\sR_{ijkl} + L^{ij}\sR_{ij}$. Straightforward computation  using \eqref{confcurvijkl}, \eqref{confric}, and \eqref{gd2} shows
\begin{align}\label{qrbt}
\begin{split}
\qR(L) &= L^{ijkl}\sR_{ijkl} + L^{ij}\sR_{ij} \\
& = L^{ijkl}T_{ijkl}^{\flat} + \tfrac{1}{4}L^{ijkl}L_{ijkl} + L^{ij}T_{ij} + \tfrac{1}{4}L^{ij}L_{ij} \\
&\qquad + \tfrac{n+2}{2}L^{ij}L_{ij}\,^{p}\ga_{p} - (n+2)L^{ij}\ga_{(ij)}+ \dad_{h}\ga |L|^{2}_{h} + \tfrac{n-2}{2}|\ga|^{2}|L|_{h}^{2}\\
& = L^{ijkl}T_{ijkl}^{\flat} + \tfrac{1}{4}L^{ijkl}L_{ijkl} + L^{ij}T_{ij} + \tfrac{1}{4}L^{ij}L_{ij} \\
&\qquad - (n+2)L^{ij}D_{(i}\ga_{j)} - (n+2)L^{ij}\ga_{i}\ga_{j} +  |L|_{h}^{2}\dad_{h}\ga + n|\ga|^{2}|L|_{h}^{2}.
\end{split}
\end{align}
Using \eqref{qwl} to simplify \eqref{qrbt} yields
\begin{align}\label{btt}
\begin{split}
L^{ijkl}T_{ijkl}^{\flat} + L^{ij}T_{ij} & = C^{ijkl}A_{ijkl}^{\flat} +\tfrac{n+2}{n-2}L^{ij}\mr{T}_{ij} + \tfrac{n+1}{n(n-1)}\uR_{h}|L|_{h}^{2}\\
& =  C^{ijkl}\sW_{ijkl} - \tfrac{1}{4}C^{ijkl}C_{ijkl} +\tfrac{n+2}{n-2}L^{ij}\mr{T}_{ij} + \tfrac{n+1}{n(n-1)}\uR_{h}|L|_{h}^{2},
\end{split}
\end{align}
Substituting \eqref{btnormestimate} and \eqref{btt} into \eqref{qrbt}, and using \eqref{qwl} gives
\begin{align}\label{qrlestimate}
\begin{split}
\qR(L) &=  \qW(L) + \left(n|\ga|^{2} + \tfrac{n+1}{n(n-1)}\left(\uR_{h} + \tfrac{1}{4}|L|_{h}^{2}\right)\right)|L|_{h}^{2}  \\ & + \tfrac{n+2}{4(n-2)}\mr{L}^{ij}\mr{L}_{ij}  - (n+2)(L^{ij}D_{(i}\ga_{j)} + L^{ij}\ga_{i}\ga_{j}) + \dad_{h}\ga |L|_{h}^{2} .
\end{split}
\end{align}
Together \eqref{lapomliediv} and \eqref{lbt1} yield
\begin{align}\label{lapombt2}
\tfrac{1}{2}\lap_{h}|L|_{h}^{2} = |DL|^{2}_{h} + \tfrac{n+2}{n}\lb L, \clie \div(L)\ra  + \qR(L) + 4(n-2)D^{p}A_{p}^{\flat} -2 |E|_{h}^{2}.
\end{align}
Substituting \eqref{normdom} into \eqref{lapombt2} and using \eqref{kliebt} and \eqref{qrlestimate} gives
\begin{align}\label{lapombt3}
\begin{split}
\tfrac{1}{2}\lap_{h}|L|_{h}^{2} &= \qR(L) +4(n-2)D^{p}A_{p}^{\flat} - \tfrac{1}{2} |E|_{h}^{2} + |\clie(L)|^{2}\\&\qquad  + \tfrac{n+2}{n}\lb L, \clie \div(L)\ra  +   \tfrac{3(n+2)}{n(n+4)}|\div(L)|^{2}\\ 
& = \qW(L)  + \left(n|\ga|^{2} + \tfrac{n+1}{n(n-1)}\left(\uR_{h} + \tfrac{1}{4}|L|^{2}_{h}\right)\right)|L|_{h}^{2} +4(n-2)D^{p}A_{p}^{\flat} - \tfrac{1}{2} |E|_{h}^{2} \\ &+ |\clie(L)|^{2}  + \tfrac{n+2}{4(n-2)}\mr{L}^{ij}\mr{L}_{ij}  + \tfrac{n+2}{n}\lb L, \clie \div(L)\ra  +   \tfrac{3(n+2)}{n(n+4)}|\div(L)|^{2} \\ &- (n+2)(L^{ij}D_{(i}\ga_{j)} + L^{ij}\ga_{i}\ga_{j}) + |L|_{h}^{2}\dad_{h}\ga.
\end{split}
\end{align}
If $h$ is chosen to be a Gauduchon metric, then the term $\dad_{h}\ga$ in \eqref{lapombt3} vanishes; if moreover, $M$ is compact, so that Theorem \ref{todtheorem} and Theorem \ref{negexacttheorem} apply, then there vanish all terms in \eqref{lapombt3} containing a $\ga$ except for the term containing $|\ga|^{2}$, and the terms containing $\div(L)$, and there results
\begin{lemma}
If $(\en, [h])$ is a Riemannian signature Einstein AH structure on manifold $M$ of dimension at least $3$ and if either $M$ is compact and $h \in [h]$ is a Gauduchon metric with associated Faraday form $\ga_{i}$, or $(\en, [h])$ is exact with distinguished metric $h \in [h]$, then 
\begin{align}
\label{lapombt4}
\begin{split}
\tfrac{1}{2}\lap_{h}|L|_{h}^{2}  & =   \qW(L)  + \left(n|\ga|^{2} + \tfrac{n+1}{n(n-1)}\left(\uR_{h} + \tfrac{1}{4}|L|_{h}^{2}\right)\right)|L|_{h}^{2} \\&\qquad - \tfrac{1}{2} |E|_{h}^{2} + |\clie(L)|^{2}  + \tfrac{n+2}{4(n-2)}\mr{L}^{ij}\mr{L}_{ij}. 
\end{split}
\end{align}
\end{lemma}
Because of the non-positivity of the term involving $|E|_{h}^{2}$, it is not clear how to handle this term except by assuming $E_{ijkl} = 0$, in which case \eqref{sharplapom} yields the sharper estimate of Theorem \ref{calabiestimatetheorem}. Theorem \ref{calabiestimatetheorem} will be used to refine Theorem \ref{negexacttheorem}. 

\begin{theorem}\label{calabiestimatetheorem}
Let $(\en, [h])$ be a Riemannian signature Einstein AH structure with self-conjugate curvature on a manifold $M$ of dimension $n \geq 3$. If either $M$ is compact and $h \in [h]$ is a Gauduchon metric with associated Faraday form $\ga_{i}$, or $(\en, [h])$ is exact (and $M$ is not necessarily compact) with distinguished metric $h \in [h]$, then wherever $L_{ijk} \neq 0$ there holds
\begin{align}\label{calabi1}
\begin{split}
|L|^{(n+4)/(n+1)}\lap_{h}|L|^{(n-2)/(n+1)}  &\geq  \left(\tfrac{n-2}{n(n-1)}\left(\uR_{h} + \tfrac{1}{4}|L|^{2}\right)+ \tfrac{n(n-2)}{n+1}|\ga|^{2}\right)|L|^{2}_{h} +  \tfrac{n-2}{n+1}\qW(L)\\
& = \tfrac{n-2}{n(n-1)} \left(\sR_{h} + (n+2)|\ga|^{2}\right)|L|^{2} +  \tfrac{n-2}{n+1}\qW(L).
\end{split}
\end{align}
in which in the latter case $\ga \equiv 0$. If, moreover, $\qW(L)\geq 0$, then in either of these cases, wherever $L_{ijk} \neq 0$ there holds
\begin{align}\label{calabi2}
\begin{split}
\lap_{h}|L|^{(n-2)/(n+1)}  &\geq  \left(\tfrac{n-2}{n(n-1)}\left(\uR_{h} + \tfrac{1}{4}|L|^{2}\right)+ \tfrac{n(n-2)}{n+1}|\ga|^{2}\right)|L|^{(n-2)/(n+1)}\\
 & = \tfrac{n-2}{n(n-1)} \left(\sR_{h} + (n+2)|\ga|^{2}\right)|L|^{(n-2)(n+1)} .
\end{split}
\end{align}
\end{theorem}
\begin{proof}
If $E_{ijkl} = 0$ then $L_{ijk} \in \ker \klie$ by \eqref{kliebt}. If $(\en, [h])$ is Einstein, then $\div(L)_{ij} = n\ga_{p}L_{ij}\,^{p}$. If $h$ is a Gauduchon metric and $M$ is compact, then by Theorem \ref{classtheorem}, there holds $\ga_{p}L_{ij}\,^{p} =0$, so $\div(L) = 0$, while if $(\en, [h])$ is exact, then for a distinguished metric $h \in [g]$ there holds $\ga_{i} = 0$, so $\div(L) = 0$. Similarly, if $h$ is Gauduchon and $M$ is compact, by Theorem \ref{classtheorem} there vanish $L^{ij}D_{(i}\ga_{j)}$ and $L^{ij}\ga_{i}\ga_{j}$, while these of course vanish for the distinguished metric of an exact AH structure. Hence in either case, \eqref{qrlestimate} shows $\qR(L) \geq  C^{ijkl}\sW_{ijkl} + \left(n|\ga|^{2} + \tfrac{n+1}{n(n-1)}\left(\uR_{h} + \tfrac{1}{4}|L|^{2}_{h}\right)\right)L$, and \eqref{sharplapom} of Lemma \ref{swlemma} shows \eqref{calabi1}. The inequality \eqref{calabi2} is immediate from \eqref{calabi1} when $\qW(L) \geq 0$.
\end{proof}

\begin{remark}
The condition $\qW(L) \geq 0$ is satisfied if $[h]$ is locally conformally flat, or if $(\en, [h])$ is conjugate projectively flat (in the latter case $A_{ijkl} = 0$ so $\qW(L) = C^{ijkl}\sW_{ijkl}= \tfrac{1}{4}C^{ijkl}C_{ijkl} \geq 0$).
\end{remark}

\begin{remark}
The conclusion of Theorem \ref{calabiestimatetheorem} is strongest when $n = 3$ for in this case $\sW_{ijkl} \equiv 0$ so the condition $\qW(L) \geq 0$ is satisfied trivially.
\end{remark}

The inequality \eqref{calabi2} is slightly stronger than the similar inequality for improper affine hyperspheres in Proposition $1$ of \cite{Calabi-improper}; the improvement is due to the use of the refined Kato inequality \eqref{katolemma}.

\subsubsection{}
Theorem \ref{cyestimatetheorem} is an estimate due to S.Y. Cheng and S.T. Yau of the growth on a complete Riemannian manifold of a non-negative function the Laplacian of which satisfies a certain inequality. The content of Theorem \ref{cyestimatetheorem} is essentially that of Theorem $5$ and its Corollaries $1$ and $2$ in Section $4$ of \cite{Cheng-Yau-affinehyperspheresI}, and the argument is also the same one used by Cheng and Yau to prove Theorem $2$ of their \cite{Cheng-Yau-maximalspacelike}, which theorem estimates the growth of the norm of the second fundamental form of a complete maximal spacelike hypersurface in Minkowski space. %
For the convenience of the reader a full proof is given following Section $4$ of \cite{Cheng-Yau-affinehyperspheresI} (equivalently Section $3$ of \cite{Cheng-Yau-maximalspacelike}) and the proof of Yau's gradient estimate for harmonic functions given in I.$3$ of \cite{Schoen-Yau}.

\begin{theorem}[Cheng-Yau. \cite{Cheng-Yau-maximalspacelike}, \cite{Cheng-Yau-affinehyperspheresI}]\label{cyestimatetheorem}
Let $(M, h)$ be a complete Riemannian manifold with Ricci curvature bounded from below by $-\ka^{2}(n-1)h_{ij}$ for some real constant $\ka \geq 0$. Suppose $u \in \cinf(M)$ is non-negative and not identically $0$ and satisfies $\lap_{h}u \geq Bu^{1 + \si} - Au$ for constants $B > 0$, $\si > 0$, and $A \in \rea$. Then for any $x \in M$ at which $u(x) \neq 0$, and any $a \in \reat$, on the open ball $B(x, a)$ of radius $a$ centered at $x$ there holds 
\begin{align}\label{cyestimate}
u \leq (a^{2} - r^{2})^{-2/\si}\left|(\tfrac{A}{B})a^{4} + (\tfrac{4 \ka(n-1)}{B\si })a^{3} + (\tfrac{4((n+2)\si + 4)}{B\si^{2}})a^{2}\right|^{1/\si}
\end{align}
in which $r \defeq d(x, \dum)$ is the distance from $x$ determined by the Riemannian metric $h$. In particular, letting $a \to \infty$, there holds $\sup_{M}u \leq |A/B|^{1/\si}$.
\end{theorem}

\begin{proof}
Let $u \in \cinf(M)$ and suppose $u \geq 0$ and $u$ is not identically zero. Suppose $x$ chosen so that $u(x) \neq 0$. Let $a > 0$ and $\al > 0$ and define $f = (a^{2} - r^{2})^{\al}u$ which is by assumption not identically zero on $B(x, a)$. Recall that $r$ is smooth on the complement of the cut locus $\cut(x)$ of $x$, and so $f$ is smooth on the complement of $\cut(x)$ in the ball $B(x, a)$, and there hold
\begin{align}
\label{cy1}Df & = \left(\tfrac{du}{u} - \tfrac{2\al r dr}{a^{2} - r^{2}} \right)f,\\
\label{cy2}\lap f & = \left|\tfrac{du}{u} - \tfrac{2\al r dr}{a^{2} - r^{2}}\right|^{2}f + \left(\tfrac{\lap u}{u} - \tfrac{|du|^{2}}{u^{2}} - \tfrac{2\al(r\lap r + 1)}{a^{2} - r^{2}} - \tfrac{4\al r^{2}}{(a^{2} - r^{2})^{2}} \right)f,
\end{align}
in which $\lap$ is the Laplacian determined by $h$. The restriction to the boundary $\pr B(x, a)$ of $f$ is identically $0$, and since by construction $f$ is not identically $0$ on $B(x, a)$, it must be that its restriction to the closure of $B(x, a)$ attains its maximum at some point $x_{0} \in B(x, a)$. First suppose $x_{0} \notin \cut(x)$. The proof in the case $x_{0} \in \cut(x)$ is similar, and will be sketched at the end. Since $f(x_{0}) \neq 0$, also $u(x_{0}) \neq 0$. It follows from \eqref{cy1} that at $x_{0}$ there holds $\tfrac{du}{u} = \tfrac{2\al r dr}{a^{2} - r^{2}}$ and, as at $x_{0}$ there holds $0 \geq \lap f$, in \eqref{cy2} this implies that at $x_{0}$ there holds
\begin{align}\label{cy3}
\tfrac{\lap u}{u}  \leq \tfrac{2\al(r\lap r + 1)}{a^{2} - r^{2}} + \tfrac{4\al (\al + 1)r^{2}}{(a^{2} - r^{2})^{2}}.
\end{align}
By standard comparison arguments the assumed lower bound on the Ricci curvature implies that on $M \setminus \cut(x)$ there holds $r \lap r \leq (n-1)(1 + \ka r)$. In \eqref{cy3} this shows that at $x_{0}$ there holds
 \begin{align}\label{cy4}
Bu^{\si} - A \leq \tfrac{\lap u}{u}  \leq \tfrac{2\al(n + \ka(n-1)r)}{a^{2} - r^{2}} + \tfrac{4\al (\al + 1)r^{2}}{(a^{2} - r^{2})^{2}}.
\end{align}
Since $B$ is by assumption positive, \eqref{cy4} can be rewritten to show that at $x_{0}$ there holds
\begin{align}\label{cy5}
u^{\si} \leq \tfrac{A}{B} + \tfrac{2\al(n + \ka(n-1)r)}{B(a^{2} - r^{2})} + \tfrac{4\al (\al + 1)r^{2}}{B(a^{2} - r^{2})^{2}}
\end{align}
Let $\al = 2/\si$ and multiply \eqref{cy5} by $(a^{2} - r^{2})^{2}$ to obtain that at $x_{0}$ there holds
\begin{align}\label{cy6}
\begin{split}
f^{\si} &\leq \tfrac{A}{B}(a^{2} - r^{2})^{2} + \tfrac{4(n + \ka(n-1)r)(a^{2} - r^{2})}{B\si} + \tfrac{8\si (2 + \si)r^{2}}{B\si^{2}} \\
&\leq (\tfrac{A}{B})a^{4} + (\tfrac{4 \ka(n-1)}{B\si})a^{3} + (\tfrac{4((n + 2)\si ) + 4)}{B\si^{2}} )a^{2}.
\end{split}
\end{align}
which implies \eqref{cyestimate}. 
For the proof in the case $x_{0} \in \cut(x)$ the argument is modified as in the proof of Yau's gradient estimate on page $21$ in Section I.$3$ of \cite{Schoen-Yau}. For the convenience of the reader this is recalled here. By the assumption that $x_{0} \in \cut(x)$ there is a minimizing geodesic joining $x$ to $x_{0}$ the image $\si$ of which necessarily lies in $B(x, a)$ (because no point of $\si$ is farther from $x$ than is $x_{0}$). 
Let $\bar{x}$ be a point on $\si$ lying strictly between $x$ and $x_{0}$ at some distance $\ep> 0$ from $x$. Since $\si$ is minimizing, no point of the interior of $\si$ can be conjugate to $\bar{x}$. Were $x$ or $x_{0}$ conjugate to $\bar{x}$ then $\bar{x}$ would be in its cut locus, which it is not because $x_{0} \in \cut(x)$ (and so $x \in \cut(x_{0}))$. Thus no point of $\si$ is a conjugate point of $\bar{x}$ and 
hence there is some $\delta > 0$ for which there is an open $\delta$ neighborhood $N \subset B(x, a)$ of $\si$ containing no conjugate point of $\bar{x}$. Let $\bar{r} = d(\bar{x}, \dum)$. By the triangle inequality, $\bar{r} + \ep \geq r$. On the other hand $\bar{r}(x_{0}) + \ep = r(x_{0})$. Define $\bar{f} = (a^{2} - (\bar{r} + \ep)^{2})^{\al}u$. Then $\bar{f} \leq f$ on $N$ and $\bar{f}(x_{0}) = f(x_{0})$, so $\bar{f}$ attains its maximum value on $N$ at $x_{0}$. As $\bar{r}$ is smooth near $x_{0}$ the preceeding argument goes through with $\bar{f}$ in place of $f$, and letting $\ep \to 0$ at the end yields \eqref{cyestimate}.
\end{proof}

\subsubsection{Notions of completeness}
Since a metric is complete if and only if any positively homothetic metric is complete, it makes sense to say that an exact Riemannian AH structure $(\en, [h])$ is \textbf{metrically complete} if any distinguished metric is complete. More generally any condition distinguishing a positive homothety class within $[h]$ determines a notion of completeness; for example, a Riemannian signature AH structure on a compact manifold is \textbf{Gauduchon complete} if it is complete with respect to any Gauduchon metric. The reason for the qualifier \textit{metrically} is that for any CP pair $(\en, [h])$ it also makes sense to consider the completeness of the aligned representative $\nabla \in \en$. A CP pair is \textbf{affinely complete} if its aligned representative is complete. An interesting problem not addressed here is to understand the relatationships between different notions of completeness. 

By Lemma \ref{conjprojflatae} if $(\en, [h])$ is conjugate projectively flat then it is exact and $A_{ijk}\,^{l}$ and $E_{ijk}\,^{l}$ vanish identically. In particular it makes sense to a speak of a metrically complete conjugate projectively flat Riemannian AH structure.

\subsubsection{Growth of the cubic torsion}
Theorem \ref{cubictorsiontheorem} is the main structural result of this paper. %

\begin{theorem}\label{cubictorsiontheorem}
  Let $(\en, [h])$ be a Riemannian signature Einstein AH structure with self-conjugate curvature on a manifold $M$ of dimension $n \geq 3$. Suppose either $M$ is compact or $(\en, [h])$ is exact and metrically complete. Suppose that $C^{ijkl}\sW_{ijkl} \geq 0$ (this is automatic if $n = 3$). Then one of the following mutually exclusive possibilities holds:
\begin{enumerate}
\item $R > 0$ and $(\en, [h])$ is an Einstein Weyl structure which is either not closed or is exact. 
\item $R \equiv 0$ and $(\en, [h])$ is closed Einstein Weyl. If $M$ is compact and $(\en, [h])$ is not exact then the universal cover of $M$ equipped with the pullback of a Gauduchon metric is isometric to a product metric on $\rea \times N$ where $N$ is simply-connected with an Einstein metric of positive scalar curvature; if $3 \leq n \leq 4$, then $N$ is diffeomorphic to $S^{n-1}$. There is induced on $N$ an exact Riemannian Einstein Weyl structure which has positive scalar curvature, and for which the induced metric $g$ is a distinguished metric.
\item $R$ is negative and $\nabla$-parallel and $(\en, [h])$ is exact. A distinguished metric $h \in [h]$ has nonpositive scalar curvature $\sR_{h}$ and Ricci curvature $\sR_{ij}$ satisfying
\begin{align}\label{riccibound}
\sR_{ij} \leq  \tfrac{(n-2)(n+1)}{4n(n+2)}\bt H_{ij} = \tfrac{(n-2)(n+1)}{4n(n+2)}|L|_{h}^{2}h_{ij} \leq -\tfrac{(n-2)(n+1)}{n(n+2)}RH_{ij}.
\end{align}
Moreover, $\sR_{h} = 0$ if and only if $|L|^{2}_{h}$ is constant, in which case $L$ is $h$-parallel, $C^{ijkl}\sW_{ijkl} = 0$ and $4\sR_{ij} = \mr{L}_{ij}$. 
\item $n = 3$, $R$ is somewhere positive and somewhere non-positive, and $(\en, [h])$ is not closed. The scalar curvature $\sR$ of a Gauduchon metric $h \in [h]$ satisfies $\sR \leq 5|\ga|^{2}_{h}$. \end{enumerate}
If $[h]$ is locally conformally flat, then in cases $(1)$ and $(2)$, $(\en, [h])$ is exact and a distinguished metric has constant sectional curvature, positive or identically zero according to whether $R$ is positive or zero. If $(\en, [h])$ is conjugate projectively flat, then case $(4)$ does not occur, and in case $(1)$, $(\en, [h])$ is exact and a distinguished metric has constant positive sectional curvature.
\end{theorem}

\begin{proof}
If $M$ is compact let $h \in [h]$ be a Gauduchon metric, and if $(\en, [h])$ is exact let $h \in [h]$ be a distinguished metric (which is complete by assumption). Recall Theorem \ref{negexacttheorem}. In the first case, by Theorem \ref{negexacttheorem}, $R$ can have values both positive and non-positive if $n = 3$, while in the second case $R$ is parallel, so cannot change sign. In either case, by Theorem \ref{calabiestimatetheorem} there holds
\begin{align}
\lap_{h}|L|^{(n-2)/(n+1)}\geq \tfrac{n-2}{n(n-1)}\left(\uR_{h} + \tfrac{1}{4}|L|^{2}\right)|L|^{(n-2)/(n+1)},
\end{align}
wherever $L_{ijk}$ is not zero. If $(\en, [h])$ is exact, then $\uR_{h}$ is constant and \eqref{confric} implies $n\sR_{ij} \geq \uR_{h}h_{ij}$, so in either case the Ricci curvature of $h$ is bounded from below, and Theorem \ref{cyestimatetheorem} can be applied. If $R$ is positive or identically zero then Theorem \ref{cyestimatetheorem} with $A = 0$, $B = \tfrac{(n-2)}{4n(n-1)}$, and $\si = \tfrac{2(n+1)}{(n-2)}$ shows $L_{ijk} \equiv 0$. If $R$ is negative, then Theorem \ref{cyestimatetheorem} with $A = \tfrac{(n-2)}{n(n-1)}\uR_{h}$, $B = \tfrac{(n-2)}{4n(n-1)}$, and $\si = \tfrac{2(n+1)}{(n-2)}$ implies $|L|^{2}_{h} \leq -4\uR_{h}$, or, what is the same, $\sR_{h} = \uR_{h} + \tfrac{1}{4}|L|^{2}_{h} \leq 0$. By Remark \ref{katoremark} there holds $L_{ij} \leq \tfrac{n}{n+2}|L|^{2}h_{ij} \leq -\tfrac{4n}{n+2}\uR_{h}h_{ij}$, and with \eqref{confric} this shows that $\sR_{ij} \leq -\tfrac{(n-2)(n+1)}{n(n+2)}RH_{ij}$. However the bound \eqref{riccibound} claimed in $(3)$ is slightly stronger. To see it, suppose there is $p \in M$, a constant $a \neq 0$, and a vector $X \in T_{p}M$ such that at $p$ there holds $\sR_{ij}X^{i}X^{j} \geq a^{2}|X|^{2} > 0$. By \eqref{confric} and Remark \ref{katoremark} there holds $\sR_{ij}X^{i}X^{j} = \tfrac{1}{n}\uR_{h}|X|^{2} + \tfrac{1}{4}|i(X)L|^{2}_{h} \leq \left(\tfrac{1}{n}\uR_{h} + \tfrac{n}{4(n+2)}|L|^{2}\right)|X|^{2}$, so that, because $\sR_{h} \leq 0$, at $p$ there holds
\begin{align*}
\tfrac{(n-2)(n+1)}{4n(n+2)}|L|_{h}^{2} = \tfrac{n}{4(n+2)}|L|^{2} - \tfrac{1}{4n}|L|^{2}\geq a^{2} - \tfrac{1}{n}\uR_{h} - \tfrac{1}{4n}|L|^{2} = a^{2} - \tfrac{1}{n}\sR_{h} \geq a^{2}.
\end{align*}
This yields a contradiction if $a^{2}$ is bigger than the value at $p$ of $\tfrac{(n-2)(n+1)}{4n(n+2)}|L|_{h}^{2}$, and so there holds \eqref{riccibound}. Because $\uR_{h}$ is constant and $(\en, [h])$ is exact, there holds $\sR_{h} = 0$ if and only if $|L|_{h}^{2}$ is constant. With \eqref{lapombt2} and \eqref{calabi1} this implies $D_{i}L_{jkl} = 0$ and $C^{ijkl}\sW_{ijkl} = 0$, and with \eqref{confric}, that $4\sR_{ij} = \mr{L}_{ij}$.

The remaining conclusions of $(1)$ and $(2)$ in the compact case follow from Theorem \ref{negexacttheorem}. In particular, if $M$ is compact and $R \equiv 0$ and $(\en, [h])$ is not exact, then by Theorem \ref{negexacttheorem} the universal cover of $M$ equipped with the pullback of the Gauduchon metric is isometric to a product metric on $\rea \times N$ where $N$ is simply connected and compact and there is induced on $N$ an exact Riemannian Einstein AH structure of positive scalar curvature which by the preceeding this must in fact be an Einstein Weyl, and because the cubic torsion on $M$ is annihilated by the Faraday one-form, it follows that the original AH structure must be Weyl. 

If $n = 3$ and $R$ is somewhere positive and somewhere non-positive, then by Theorem \ref{negexacttheorem}, for a Gauduchon metric $\uR_{h} - 3|\ga|^{2}$ is a negative constant $\ka$ and by Theorem \ref{calabiestimatetheorem} there holds
\begin{align}
\lap_{h}|L|^{1/4}\geq \left(\tfrac{\ka}{6} + \tfrac{5}{4}|\ga|^{2} + \tfrac{1}{4}|L|^{2}\right)|L|^{1/4}\geq \left(\tfrac{\ka}{6} + \tfrac{1}{4}|L|^{2}\right)|L|^{1/4}.
\end{align}
By Theorem \ref{cyestimatetheorem} with $A = -\ka/6$, $B = 1/24$, and $\si = 8$ it follows that $\ka  + \tfrac{1}{4}|L|^{2} \leq 0$, so that $\sR = \uR_{h} + \tfrac{1}{4}|L|^{2} +2|\ga|^{2} = \ka  + \tfrac{1}{4}|L|^{2} + 5|\ga|^{2} \leq 5|\ga|^{2}$. 

For the conformally flat case the conclusions follow from Theorem \ref{eastwoodtodtheorem}. The conclusions in the conjugate projectively flat case follow because a conjugate projectively flat AH structure is necessarily closed.
\end{proof}
As is explained in section \ref{calabiriccisection} below, in the projectively flat case one has in case $(3)$, by a theorem of Calabi, the stronger result that the Ricci curvature is non-positive. 

\subsubsection{}\label{s3examplerevisited}
Recall the example of Section \ref{s3example} of a one-parameter family of strongly Einstein AH structures on $S^{3}$. Since the dimension is $3$ the metric $h$ is conformally flat, so there holds $\qW(L) =0$. Since when $t < 3/\sqrt{2}$ this AH structure is strongly Einstein with positive scalar curvature and non-zero cubic torsion, and satisfies $\qW(L) = 0$, this example shows that in Theorem \ref{cubictorsiontheorem} the condition $E_{ijk}\,^{l} = 0$ is necessary for conclusion $(1)$. Since $L_{ijk} = 2\Ga_{ijk}$, it follows from \eqref{confscal} that the scalar curvature of the distinguished metric $h_{ij}$ is $\sR_{h} = 3/4$, which is positive. Since for $t > 3/\sqrt{2}$ the scalar curvature of $(\en, [h])$ is negative, this shows that in Theorem \ref{cubictorsiontheorem} the condition $E_{ijk}\,^{l} = 0$ is necessary for conclusion $(3)$.

In this example $D$ is the invariant connection given by $A = \tfrac{1}{2}\ad$, so that $\sqrt{2}D_{i}X_{j} = -Y_{[i}Z_{j]}$, from which it follows that $D_{(i}\Ga_{jkl)} = 0$, so $\clie(\Ga) = 0$. Thus $\Ga_{ijk}$ is an example of a conformal Killing tensor.

\subsection{Infinitesimal automorphisms of convex flat projective structures}\label{automorphismsection}
In this section some results are deduced about infinitesimal projective automorphisms of convex flat real projective structures. The best result is obtained in two dimensions:

\begin{theorem}\label{twodconvexautotheorem}
A convex flat real projective structure on a compact, orientable surface of genus at least $2$ admits no non-trivial infinitesimal projective automorphism.
\end{theorem}
As mentioned in the introduction it would be unsurprising were Theorem \ref{twodconvexautotheorem}, but the more general Theorem \ref{convexoneparametertheorem} from which it follows is new, and the method of proof is as well. Theorem \ref{twodconvexautotheorem} is an illustration of how the AH formalism allows analytic tools to be applied to flat projective structures.

\subsubsection{}\label{calabiriccisection}
The proof of Theorem $5.1$ of Calabi's \cite{Calabi-completeaffine} implies 
\begin{theorem}[Theorem $5.1$ of \cite{Calabi-completeaffine}]\label{calabitheorem}
If for a projectively flat Riemannian signature Einstein AH structure with negative scalar curvature on a manifold of dimension at least $2$ a distinguished metric is complete, then the Ricci curvature of a distinguished metric is non-positive.
\end{theorem}

It seems likely that this need not be the case if $A_{ijkl}$ is non-zero. 

\begin{proof}
Theorem $5.1$ of \cite{Calabi-completeaffine} implies the statement for hyperbolic affine hyperspheres, and by Theorem \ref{affinehyperspheretheorem} it follows also if $M$ is compact and the underlying projective structure is known to be convex. However, it is easier to observe that Calabi's proof carries over to this more general setting unchanged, as will now be explained. Theorem $5.1$ of \cite{Calabi-completeaffine} is proved by applying a suitable modification of Theorem \ref{cyestimatetheorem} to a (weak) differential inequality for the Laplacian of the square-root of the largest eigenvalue of $L_{ij}$ obtained using \eqref{modifiedlapeinstein}. Precisely, Calabi considers a constant multiple of 
\begin{align*}
\phi(x) \defeq \sup_{X \in T_{x}M}\left(\tfrac{X^{i}X^{j}L_{ij}}{|X|_{h}^{2}}\right)^{1/2}
\end{align*}
and proves an estimate for $\lap_{h} \phi$. Of course $\phi$ is only continuous, and so some sense has to be made of $\lap_{h} \phi$; the resulting differential inequality is Lemma $5.2$ of \cite{Calabi-completeaffine}. %
When $A_{ijkl} = 0 = E_{ijkl}$, Calabi's estimate goes through unchanged in the \textit{a priori} more general setting of the present theorem to show that $nL_{ij} \leq - 4\uR_{h}h_{ij}$, which is exactly what is needed to show that $\sR_{ij}$ is non-positive. 
\end{proof}

\begin{remark}
A very similar estimate is made in section $4$ of \cite{Cheng-Yau-maximalspacelike}.
\end{remark}

When $A_{ijk}\,^{l}$ is not identically zero, Calabi's argument goes through with the only change being the appearance of some algebraically complicated terms involving $A_{ijk}\,^{l}$ and $L_{ijk}$ and corresponding to the term $\qW(L)$ in \eqref{calabi1}. Assuming the non-negativity of these terms, Calabi's original argument goes through to show that the Ricci curvature of a distinguished metric is non-positive, but the geometric meaning of such an assumption is so far completely opaque. When $A_{ijkl}$ is not identically zero but is globally bounded in norm by some constant, Calabi's estimate still goes through, and $L_{ij}$ is bounded from above by some constant multiple of $h_{ij}$; however this does not force $\sR_{ij}$ to be non-positive, and it seems likely that in general it need not be. Although no example showing this is in hand, the examples constructed in section \ref{cubicformsection} show that an exact Einstein AH structure with self-conjugate curvature and non-vanishing self-conjugate Weyl tensor can have a distinguished metric which is flat.

\subsubsection{}
By Theorem \ref{calabitheorem}, the Ricci curvature of a distinguished metric $h$ of the exact Riemannian Einstein AH structure determined by a convex flat real projective structure on a compact manifold is non-positive; Lemma \ref{ricflatflatlemma} shows that if it is zero, then $h$ is flat.
\begin{lemma}\label{ricflatflatlemma}
If a distinguished metric $h \in [h]$ of the exact Riemannian Einstein AH structure $(\en, [h])$ determined by a convex flat real projective structure on a compact, orientable $n$-manifold is Ricci flat, then it is flat, and the cubic torsion is parallel with respect to $h$.
\end{lemma}
\begin{proof}
When $n = 2$ there is nothing to prove, so assume $n > 2$. Because $\en$ is projectively flat, $A_{ijk}\,^{l} = 0 = E_{ijk}\,^{l}$. By \eqref{confric}, that $\sR_{ij} = 0$ implies $\mr{L}_{ij} = 0$, and $4\uR_{h} = -|L|^{2}_{h}$. By \eqref{confcurvijkl} and \eqref{mrcdefined}, $4\sR_{ijkl} = L_{ijkl} - \tfrac{2}{n(n-1)}|L|^{2}_{h}h_{l[i}h_{j]k} = C_{ijkl}$, and so $4\sW_{ijkl} = C_{ijkl}$ and $4\qW(L) = C^{ijkl}C_{ijkl} \geq 0$. Because $R$ is parallel, $|L|^{2}_{h} = -4 \uR_{h}$ is constant. If it is $0$ there is nothing to show. It it is positive, then by \eqref{calabi1} of Theorem \ref{calabiestimatetheorem} there holds $0 = |L|^{(n+4)/(n+1)}\lap_{h}|L|^{(n-2)/(n+1)} \geq  \tfrac{n-2}{n+1}\qW(L) = \tfrac{n-2}{4(n+1)}C^{ijkl}C_{ijkl}\geq 0$, so that $4\sR_{ijkl} = C_{ijkl} \equiv 0$, showing that $h$ is flat. From \eqref{lapombt2} it follows that the cubic torsion is parallel with respect to $h$.
\end{proof}

The Einstein AH structure induced by the density $u$ of \eqref{orthantexample} on the positive orthant $\{x \in \rea^{n}: x^{i} > 0\}$ is of the sort appearing in Lemma \ref{ricflatflatlemma}. For this Einstein AH structure, the aligned representative is $\nabla = \pr + 2\ga_{(i}\delta_{j)}\,^{k}$ where $\pr$ is the standard flat affine connection on $\rea^{n}$ and $\ga_{i} = -u^{-1}\pr_{i}u$. A distinguished metric is $h_{ij} = -u^{-1}\B(u)_{ij} = \pr_{i}\ga_{j} - \ga_{i}\ga_{j}$. The curvature of $\nabla$ is $R_{ijk}\,^{l} = 2h_{k[i}\delta_{j]}\,^{l}$, so that $\uR_{h} = n(1-n)$, while it can be checked that the metric $h$ is flat, so that from \eqref{confscal} it follows that $|L|_{h}^{2} = 4n(n-1)$. Thus this is an exact Einstein AH structure with negative scalar curvature for which a distinguished metric is flat and the cubic torsion is non-zero and parallel with respect to a distinguished metric. For this example, these facts were already observed by Calabi in the last paragraph of section $3$ of \cite{Calabi-completeaffine}, where they were noted as consequences of the homogeneity of the affine hypersphere given by the radial graph of $u$. More non-trivial examples of the sort appearing in Lemma \ref{ricflatflatlemma} are easily constructed by the methods described in Section \ref{cubicformsection}.

\subsubsection{}
The Lie derivative, $\lie_{X}\nabla$, along the vector field, $X$ of an affine connection, $\nabla$, is defined to be the derivative $\lie_{X}\nabla = \tfrac{d}{dt}_{|t = 0}\left((\phi^{t})^{\ast}\nabla  - \nabla\right)$ of the difference tensor with $\nabla$ of its pullback via the flow $\phi^{t}$ of $X$. By its definition, $\lie_{X}\nabla$ is a $\binom{1}{2}$-tensor. Formally $(\lie_{X}\nabla)(A, B)$ may be computed as if $\nabla$ were a tensor, with the result that for a torsion-free $\nabla$, $(\lie_{X}\nabla)_{ij}\,^{k} = \nabla_{i}\nabla_{j}X^{k} + X^{p}R_{pij}\,^{k}$. If $\nabla$ is an affine connection on $M$ write $\Aut(\nabla) = \{\phi \in \diff(M): \phi^{\ast}(\nabla) = \nabla\}$ and $\aut(\nabla) = \{X \in \Ga(TM): \lie_{X}\nabla = 0\}$ for the group of affine automorphisms and the Lie algebra of infinitesimal affine automorphisms of $\en$. 

\subsubsection{}\label{projdiffsection}
The difference tensor $[\bnabla] - \en$ of two projective structures is defined to be the trace-free part of the difference tensor $\bnabla - \nabla$ of any torsion-free representatives $\bnabla \in [\bnabla]$ and $\nabla \in \en$; this does not depend on the choices of representatives because the difference of any two torsion-free representatives is pure trace. The Lie derivative $\lie_{X}\en$ of $\en$ along $X \in \Ga(TM)$ is defined to be $\tfrac{d}{dt}_{|t = 0}\left((\phi^{t})^{\ast}\en - \en\right)$ where $\phi^{t}$ is the flow generated by $X$. By definition the difference tensor $(\phi^{t})^{\ast}\en - \en$ is the trace-free part of $(\phi^{t})^{\ast}\nabla - \nabla$ for any torsion-free representative $\nabla \in \en$, and it follows that $\lie_{X}\en$ is the completely trace-free part of $\lie_{X}\nabla$ for any torsion-free representative $\nabla \in \en$. Explicitly,
\begin{align}\label{lieproj}
(\lie_{X}\en)_{ij}\,^{k} & = \nabla_{i}\nabla_{j}X^{k} + X^{p}R_{pij}\,^{k} - \tfrac{2}{n+1}\delta_{(i}\,^{k}\nabla_{j)}\nabla_{p}X^{p} -2\delta_{i}\,^{k}X^{p}P_{[pj]} - 2\delta_{j}\,^{k}X^{p}P_{[pi]}\\
\notag & = \nabla_{i}\nabla_{j}X^{k} - \tfrac{2}{n+1}\delta_{(i}\,^{k}\nabla_{j)}\nabla_{p}X^{p} 
 + \left(B_{pij}\,^{k} - \delta_{p}\,^{k}P_{ij} + \delta_{i}\,^{k}P_{jp}\right)X^{p}.%
\end{align}
While by definition $(\lie_{X}\en)_{[ij]}\,^{k} = 0$, this can also be checked using \eqref{lieproj} and the Bianchi and Ricci identities. If $\en$ is a projective structure on $M$ write $\Aut(\en) = \{\phi \in \diff(M): \phi^{\ast}(\en) = \en\}$ and $\aut(\en) = \{X \in \Ga(TM): \lie_{X}\en = 0\}$ for the group of projective automorphisms and the Lie algebra of infinitesimal projective automorphisms of $\en$. Evidently $\aut(\nabla)$ is a subalgebra of $\aut(\en)$.

\subsubsection{}
If $h$ is a pseudo-Riemannian metric write $\aut(h) = \{X \in \Ga(TM): \lie_{X}h = 0\}$ for the Lie algebra of $h$-Killing fields. If $D$ is the Levi-Civita connection of $h$ then $\aut(h)$ is a subalgebra of $\aut(D)$. Given a torsion-free affine connection $\nabla$ and a pseudo-Riemannian metric $h$ define $\affhar(\nabla, h) = \{X \in \Ga(TM): h^{ij}(\lie_{X}\nabla)_{ij}\,^{k} = 0\}$ which is a subspace of $\vect(M)$. Likewise, given a projective structure $\en$ and a conformal metric $[h]$ define $\projhar(\en, [h]) = \{X \in \Ga(TM): H^{ij}(\lie_{X}\en)_{ij}\,^{k} = 0\}$, which is a subspace of $\vect(M)$. Evidently $\aut(\nabla)$ and $\aut(\en)$ are subspaces of $\affhar(\nabla, h)$ and $\projhar(\en, [h])$, respectively. Elements of $\affhar(\nabla, h)$ and $\projhar(\en, [h])$ are called respectively \textbf{affine harmonic} and \textbf{projective harmonic}. It would require too much of a digression to justify the terminology.

Tracing \eqref{lieproj} in $jk$ shows that $X \in \projhar([D], [h])$ if and only if 
\begin{align}\label{projkill2}
0 = h^{pq}(\lie_{X}[D])_{pq}\,^{k}h_{ki} = \lap_{h}X_{i} + X^{p}\sR_{pi} - \tfrac{2}{n+1}D_{i}D_{p}X^{p} %
\end{align}
Theorem $4$ of K. Yano's \cite{Yano} shows that for a Riemannian metric $h$ with Levi-Civita connection $D$ on a compact manifold there holds $\aut(D) = \aut(h)$. In \cite{Couty}, R. Couty proved that if on a compact manifold $h$ has non-positive Ricci curvature which is somewhere negative then $\aut([D]) = 0$. The proof is a Bochner type argument, and motivates the proof of Theorem \ref{projkilltheorem}.

\begin{theorem}\label{projkilltheorem}
Let $(\en, [h])$ be an exact, Riemannian signature AH structure on a compact, orientable manifold $M$ of dimension $n \geq 2$. 
\begin{enumerate}
\item If the tensor $T_{ij} + \tfrac{1}{4}\bt_{ij}+ \tfrac{n(n+1)}{(n+3)}E_{ij}$ is non-positive then a non-trivial $X \in \projhar(\en, [h])$ is parallel with respect to the Levi-Civita connection $D$ of a distinguished metric $h \in [h]$, and satisfies $X^{i}X^{j}E_{ij} = 0$ and $X^{i}X^{j}\sR_{ij} = 0$, in which $\sR_{ij} = T_{ij} + \tfrac{1}{4}\bt_{ij}$ is the Ricci curvature of $h$. As a consequence, the subspace $\projhar(\en, [h])$ is an abelian Lie subalgebra of $\vect(M)$ and spans an integrable totally $h$-geodesic flat subbundle of $TM$; hence $M$ is foliated by totally $h$-geodesic flat parallelizable submanifolds of dimension $\dim \projhar(\en)$. 
\item If $n = 2$ and $T_{ij} + \tfrac{1}{4}\bt_{ij}$ is non-positive and $M$ admits a non-trivial $X \in \projhar(\en, [h])$, then $M$ is a torus and a distinguished metric $h$ is flat.
\item If $T_{ij} + \tfrac{1}{4}\bt_{ij}+ \tfrac{n(n+1)}{(n+3)}E_{ij}$ is non-positive and somewhere strictly negative definite, then $\projhar(\en, [h]) = \{0\}$.
\end{enumerate}
\end{theorem}%
\begin{proof}

Fix a distinguished metric $h \in [h]$ and raise and lower indices using $h$. Write $L_{ijk} = \bt_{ij}\,^{p}h_{pk}$ and $\nabla = D - \tfrac{1}{2}\bt_{ij}\,^{k}$. By definition of the Lie derivative, $(\lie_{X}\nabla)_{ij}\,^{k} = (\lie_{X}D)_{ij}\,^{k} - \tfrac{1}{2}(\lie_{X}L)_{ij}\,^{k}$. Since $(\lie_{X}\bt)_{ip}\,^{p} = 0$, the definition of $\lie_{X}\en$ implies
\begin{align}\label{projlied}
\begin{split}
(\lie_{X}\en)_{ij}\,^{k} &= (\lie_{X}[D])_{ij}\,^{k} - \tfrac{1}{2}(\lie_{X}\bt)_{ij}\,^{k}\\
& = (\lie_{X}[D])_{ij}\,^{k} - \tfrac{1}{2}X^{p}D_{p}L_{ij}\,^{k} - L_{p(i}\,^{l}D_{j)}X^{p} + \tfrac{1}{2}L_{ij}\,^{p}D_{p}X^{k}.
\end{split}
\end{align}
If $X\in \projhar(\en, [h])$, then contracting \eqref{projlied} with $h^{ij}$ and using \eqref{projkill2} gives
\begin{align}\label{projkill2b}
0 = \lap_{h}X_{i} + X^{p}\sR_{pi} - \tfrac{2}{n+1}D_{i}D_{p}X^{p} - L_{i}\,^{pq}D_{p}X_{q}.
\end{align}
Integrating by parts and using \eqref{ddivbt} shows 
\begin{align*}
\int_{M}L^{ijk}X_{i}D_{j}X_{k} = - \int_{M} L^{ijk}X_{k}D_{j}X_{i} - \int_{M}X_{i}X_{k}D_{j}L^{ikj} = - \int_{M}L^{ijk}X_{i}D_{j}X_{k} - 2n\int_{M}X^{i}X^{j}E_{ij},
\end{align*}
so that $\int_{M}L^{ijk}X_{i}D_{j}X_{k} = - n\int_{M}X^{i}X^{j}E_{ij}$. Hence integrating $\lap_{h}|X|^{2}$ using \eqref{projkill2b} gives
\begin{align}\label{projkill4b}
0 =  \int_{M}\left( - \tfrac{2}{n+1}(\dad_{h}X^{\flat})^{2} + \tfrac{1}{4}|dX^{\flat}|^{2} + \tfrac{1}{4}|\lie_{X}h|^{2}   - X^{i}X^{j}\sR_{ij} - nX^{i}X^{j}E_{ij}\right)
\end{align} 
For any $h$ and any $X \in \Ga(TM)$ with dual one-form $X^{\flat} = i(X)h$ there holds
\begin{align}\label{projkill0}
0 = \int_{M}\left((\dad_{h}X^{\flat})^{2} + \tfrac{1}{4}|dX^{\flat}|^{2} - \tfrac{1}{4}|\lie_{X}h|^{2} - X^{i}X^{j}\sR_{ij}\right).
\end{align}
Taking  the appropriate linear combination of \eqref{projkill0} and \eqref{projkill4b} and using \eqref{confric} yields
\begin{align}\label{projkill5b}
\begin{split}
\tfrac{1}{4}\int_{M} \left( |dX^{\flat}|^{2} + \tfrac{n-1}{n+3}|\lie_{X}h|^{2}\right) &= \int_{M}X^{i}X^{j}\left(\sR_{ij} + \tfrac{n(n+1)}{n+3}E_{ij}\right)\\
&= \int_{M}X^{i}X^{j}\left(T_{ij} + \tfrac{1}{4}\bt_{ij} + \tfrac{n(n+1)}{n+3}E_{ij}\right).
\end{split}
\end{align}
If $\sR_{ij} + \tfrac{n(n+1)}{(n+3)}E_{ij} = T_{ij} + \tfrac{1}{4}\bt_{ij}+ \tfrac{n(n+1)}{(n+3)}E_{ij}$ is non-positive \eqref{projkill5b} implies $D_{i}X^{j} = 0$. In \eqref{projkill2b} there results $X^{i}\sR_{ij} = 0$, and so the assumed non-positivity of $\sR_{ij} + \tfrac{n(n+1)}{(n+3)}E_{ij}$ implies that $X^{i}X^{j}E_{ij} \leq 0$, which is consistent with \eqref{projkill5b} if and only if $X^{i}X^{j}E_{ij} \equiv 0$. If $X, Y \in \projhar(\en, [h])$ then $0 = D_{X}Y - D_{Y}X = [X, Y]$, so $\projhar(\en, [h])$ is an abelian Lie subalgebra of $\vect(M)$. If $X \in \projhar(\en, [h])$ is not identically zero, then it is nowhere zero, so the span of $\projhar(\en, [h])$ has constant rank, so is a totally $h$-geodesic flat integrable subbundle because $\projhar(\en, [h])$ is abelian and comprises parallel vector fields. Since each $X \in \projhar(\en, [h])$ is non-vanishing, an integrable submanifold of this subbundle is parallelizable. 

If $n = 2$ then $E_{ij} \equiv 0$ tautologically. If $T_{ij} + \tfrac{1}{4}\bt_{ij}$ is non-positive and $X \in \projhar(\en, [h])$ is non-trivial, then $M$ admits a non-vanishing vector field, so must have Euler characteristic zero. Since the curvature $\tfrac{\sR}{2}h_{ij} = \sR_{ij} =  T_{ij} + \tfrac{1}{4}\bt_{ij}$ is by assumption non-positive this is consistent with the Gau\ss-Bonnet Theorem if and only if $\sR \equiv 0$, that is $h$ is flat.

If $T_{ij} + \tfrac{1}{4}\bt_{ij} + \tfrac{n(n+1)}{n+3}E_{ij}$ is non-positive and somewhere strictly negative definite, then it is so on some open $U \subset M$, in which case equality can hold in \eqref{projkill5b} if and only if $X$ vanishes on $U$; because $X$ is $D$-parallel this implies $X \equiv 0$. 
\end{proof}

If $(\en, [h])$ is a Riemannian AH structure, then because $E_{ij}$ is symmetric and trace-free, it can be non-positive if and only if it vanishes identically. 

Note that Theorem \ref{projkilltheorem} shows that on a compact manifold, a Riemannian metric $h$ with Levi-Civita connection $D$, for which the Ricci tensor is non-positive and somewhere negative there holds $\projhar([D], [h]) = \{0\}$, which is a strengthening of Couty's theorem mentioned above.

\begin{theorem}\label{convexoneparametertheorem}
Let $(\en, [h])$ be a projectively flat Riemannian signature Einstein AH structure with negative scalar curvature on a compact, orientable manifold $M$ of dimension $n \geq 2$. Any distinguished metric $h \in [h]$ has non-positive Ricci curvature and if $\en$ admits a non-trivial infinitesimal projective automorphism $X \in \aut(\en)$ then the Ricci curvature of $h$ is degenerate in the direction of $X$. In particular, 
\begin{enumerate}
\item If $h$ has Ricci curvature which is negative definite at some point of $M$ then $X \equiv 0$.
\item If $n = 2$ then the existence of a non-trivial infinitesimal projective automorphism implies $M$ is a torus and a distinguished metric is flat. 
\item If $n > 2$ and there exists a non-trivial infinitesimal projective automorphism then $\aut(\en) = \aut(\nabla)$ is a Lie subalgebra of $\vect(M)$ of rank $1 \leq r \leq n$, and the universal cover $\tilde{M}$ of $M$ equipped with the pullback of a distinguished metric $h  \in [h]$ splits isometrically $\tilde{M} \simeq \rea^{r} \times N$ where $\rea^{r}$ is flat Euclidean space, and the $(n-r)$-dimensional Riemannian manifold $N$ is connected and simply-connected. Hence $M$ is foliated by totally $h$-geodesic flat $r$-dimensional submanifolds each of which is finitely covered by a flat Riemannian vector bundle over a torus.
\end{enumerate}
\end{theorem}
\begin{proof}
By Theorem \ref{convextheorem} there is a unique conformal structure $[h]$ such that $(\en, [h])$ is a Riemannian Einstein AH structure with negative scalar curvature. By Theorem \ref{calabitheorem} the Ricci curvature of a distinguished metric is non-positive. Theorem \ref{projkilltheorem} applies and yields immediately the first two claims. If $X \in \aut(\en)$ then $\nabla_{i}X^{j} = - \tfrac{1}{2}X^{p}L_{ip}\,^{j}$, so $\nabla_{p}X^{p} = 0$. With \eqref{lieproj} this implies $\lie_{X}\nabla = 0$, so $X \in \aut(\nabla)$. The containment $\aut(\nabla) \subset \aut(\en)$ is obvious, so $\aut(\en) = \aut(\nabla)$. By Theorem \ref{projkilltheorem}, $\aut(\en)$ is an abelian subalgebra of $\vect(M)$ comprising $D$-parallel vector fields, so generates a totally $h$-geodesic flat integrable subbundle; the claims about the universal cover follow from the argument showing the de Rham decomposition. The distribution on $M$ generated by $\aut(\en)$ is totally $h$-geodesic and so each leaf is complete in the induced metric, which is flat. By Bieberbach's theorem each leaf is finitely covered by a flat Riemannian vector bundle on a torus.
\end{proof}

In particular, Theorem \ref{convexoneparametertheorem} applies to the Einstein AH structure induced as in Theorem \ref{convextheorem} on a manifold with convex flat real projective structure. In this context, Theorem \ref{convexoneparametertheorem} can also be deduced from the usual Bochner theorem for conformal Killing fields on compact Riemannian manifolds with non-positive Ricci curvature coupled with Theorem \ref{calabitheorem} and the uniqueness claim in Theorem \ref{convextheorem}; this last implies that an infinitesimal projective automorphism of $\en$ preserves the induced $[h]$.

\begin{proof}[Proof of Theorem \ref{twodconvexautotheorem}]
By Theorem \ref{convexoneparametertheorem} were there a non-trivial infinitesimal projective automorphism then $M$ would admit a flat metric, which is impossible by the Gau\ss-Bonnet theorem.
\end{proof}

\section{Einstein AH structures from commutative nonassociative algebras}\label{cubicformsection}
This section presents a method which yields an abundance of complete exact Einstein Riemannian AH structures with negative scalar curvature and self-conjugate curvature for which the standard Euclidean metric is a distinguished metric, but which are neither projectively flat nor conjugate projectively flat (or, what is the same, for which $A_{ijk}\,^{l}$ is not identically zero). As a byproduct of the main construction there is obtained in section \ref{naiveexample} an example of a naive Einstein AH structure which is not Einstein. 

In section \ref{polynomialpdesection} it is shown how to construct the desired examples by finding a homogeneous cubic polynomial satisfying a certain system of partial differential equations, and some solutions are found. In section \ref{codazzialgebrasection} it is shown that these cubic polynomials can be interepreted as the structure tensors of a certain kind of algebra, and the resulting formalism is used to show that some of the resulting Einsten AH structures are not projectively flat. 

\setcounter{subsubsection}{0}

\subsection{Partial differential equations for homogeneous cubic polynomials}\label{polynomialpdesection}
The purpose of this section is the construction of solutions to \eqref{einsteinpolynomials}, which seems interesting in its own right.

\subsubsection{}\label{polynomialsection}
Let the notations be as in section \ref{symmetricharmonictensorsection}. Consider $\rea^{n}$ with the standard Euclidean metric $\delta_{ij}$, having Levi-Civita connection $D$, and let $\eul$ be the radial vector field generating dilations by $e^{t}$. Suppose there is given a homogeneous cubic polynomial $P(x) \in \pol^{3}(\rea^{n})$ such that 
\begin{align}\label{einsteinpolynomials}
&\lap_{\delta} P = 0,& &\text{and}& &|\hess P|^{2}_{\delta} = \ka E(x),
\end{align} 
for some constant $\ka > 0$ and $E(x) = \eul^{i}\eul^{j}\delta_{ij}$ the quadratic polynomial corresponding to the metric. By \eqref{lapsquarepsquare}, or simply $\lap_{\delta}^{2}P^{2} = 2\lap_{\delta}|DP|_{\delta}^{2} = 4|\hess P|^{2}_{\delta}$, the second equation of \eqref{einsteinpolynomials} is equivalent in the presence of the first, to either $\lap_{\delta}^{2} P^{2} = 4\ka E(x)$ or $\lap_{\delta}|DP|_{\delta}^{2} = 2\ka E(x)$, for the same constant  $\ka > 0$. 
Write $P_{i_{1}\dots i_{k}} = D_{i_{1}}\dots D_{i_{k}}P$ and raise and lower indices with $\delta_{ij}$ and the dual bivector $\delta^{ij}$. Let $L_{ij}\,^{k} = P_{ij}\,^{k}$. Differentiating the second equality of \eqref{einsteinpolynomials} and using that $P$ is a cubic polynomial gives $L_{ij} = L_{ipq}L_{j}\,^{pq} = \ka\delta_{ij}$. By the arguments proving Theorem \ref{eigentheorem} the connection $\nabla = D - \tfrac{1}{2}L_{ij}\,^{k}$ is the aligned representative of an exact proper strongly Einstein AH structure $(\en, [\delta])$ with distinguished metric $\delta_{ij}$ and satisfying $E_{ijk}\,^{l} = 0$. By \eqref{confcurvijkl} there holds $T_{ijk}\,^{l} = -\tfrac{1}{4}L_{ijk}\,^{l}$. The scalar curvature $\uR_{\delta}$ is $-\tfrac{1}{4}|L|_{\delta}^{2}$, which is negative if $L_{ijk}$ is not identically zero. Taking the $h$-trace free part of $L_{ijkl} = 2P_{k[i}\,^{p}P_{j]lp}$ gives $-4A_{ijk}\,^{p}h_{pl} = C_{ijkl}$. If $n = 3$ then $C_{ijkl} = 0$ necessarily, but if $n > 3$ then it can be that $C_{ijkl}$ is not identically zero, in which case the preceeding gives a construction of exact Riemannian proper strongly Einstein AH structures with self-conjugate curvature but which by Lemma \ref{projflatahlemma} are neither projectively flat nor conjugate projectively flat. 

Let $\Ga$ be a lattice in $\rea^{n}$. Since the tensor $L_{ij}\,^{k}$ is constant and the generators of $\Ga$ act by translations, the Einstein AH structures on $\rea^{n}$ obtained in this way descend to the quotient torus $\rea^{n}/\Ga$.

To find examples in this way two things must be done. First, the equations \eqref{einsteinpolynomials} must be solved. Second, there must be found ways of checking the non-vanishing of $C_{ijkl}$. In what follows there are first given examples of polynomials solving \eqref{einsteinpolynomials}; in particular examples are given for all $n \geq 2$. Later, it is shown how such a cubic polynomial determines a particular kind of algebra, and this formulation is used to check the non-vanishing of $C_{ijkl}$ in some cases. Examples for which $C_{ijkl}$ is not identically zero are constructed for all $n > 3$.

\subsubsection{}
If $P$ solves \eqref{einsteinpolynomials} then $e^{r}P$ solves \eqref{einsteinpolynomials} with $e^{2r}\ka$ in place of $\ka$. Also the equations \eqref{einsteinpolynomials} are orthogonally invariant in the sense that if $P \in \pol^{3}(\rea^{n})$ solves \eqref{einsteinpolynomials} then so too does $(g \cdot P)(x) \defeq P(gx)$ for any $g \in O(n)$. Hence it is desirable to describe their solutions modulo the action of the group $CO(n) = O(n) \times \reap$ of conformal linear transformations.

\subsubsection{}\label{2dharpolysection}
Writing $z = x + \j y$, the most general harmonic polynomial on $\rea^{2}$ is 
\begin{align}\label{2dharpoly}
r\cos\theta (x_{1}^{3} - 3x_{1}x_{2}^{2}) + r\sin \theta(x_{2}^{3}- 3x_{2}x_{1}^{2}) = r\re(e^{\j\theta/3}z)^{3}, 
\end{align}
for some $r \in [0, \infty)$ and $\theta \in [0, 2\pi)$. By Lemma \ref{twodcubicformlemma} or direct computation using \eqref{2dharpoly}, any harmonic polynomial $P \in \pol^{3}(\rea^{2})$ solves \eqref{einsteinpolynomials}, with $\ka = 72r^{2}$ for \eqref{2dharpoly}. Hence, after the orthogonal transformation sending $z$ to $e^{-\j\theta/3}z$, the polynomial $P$ has the form $r(x_{1}^{3} - 3x_{1}x_{2}^{2})$. Thus every non-trivial solution of \eqref{einsteinpolynomials} on $\rea^{2}$ is in the $CO(2)$-orbit of $x_{1}^{3} - 3x_{1}x_{2}^{2}$.

\subsubsection{}
If $P \in \pol^{3}(\rea^{p})$ and $Q \in \pol^{3}(\rea^{q})$ solve \eqref{einsteinpolynomials} for the same constant $\ka$, then $P \circ \pi_{p} + \Q \circ \pi_{q} \in \pol^{3}(\rea^{p+q})$, in which $\pi_{p}$ and $\pi_{q}$ are the projections from $\rea^{p+q}= \rea^{p}\oplus \rea^{q}$ onto $\rea^{p}$ and $\rea^{q}$, solves \eqref{einsteinpolynomials}. Since any $P \in \har^{3}(\rea^{2})$ satisfies \eqref{einsteinpolynomials}, there are solutions to \eqref{einsteinpolynomials} on $\rea^{2n}$ for any $n$. For example, for any $\theta_{i} \in [0, 2\pi)$,
\begin{align*}
P(x) = \sum_{i = 1}^{n}\left(\cos \theta_{i} (x_{2i-1}^{3} - 3x_{2i-1}x_{2i}^{2}) + \sin \theta_{i} (x_{2i}^{3} - 3x_{2i}x_{2i-1}^{2})\right).
\end{align*}
solves \eqref{einsteinpolynomials} on $\rea^{2n}$ with $\ka = 72$. 

As the goal is to produce examples for which $A_{ijk}\,^{l}$ does not vanish, and such examples can neither arise in $2$ and $3$ dimensions, nor by composing in the sense just described examples for which $A_{ijk}\,^{l}\equiv 0$, something more has to be done. However, as the simplest construction of such examples will utilize the $n = 2$ and $n = 3$ cases, the $n=3$ case is analyzed completely in the next section.

\subsubsection{}
Suppose $P \in \pol^{3}(\rea^{n+1})$ solves \eqref{einsteinpolynomials}. The restriction of $P$ to the sphere $E = 1$ has a maximum, at which point $D_{i}P$ is proportional to $D_{i}E$. By an orthogonal change of variables it may be supposed that the maximum occurs at the point $x_{1} = 0, \dots, x_{n} = 0$, $x_{n+1} = 1$, and that at this point $D_{i}P$ is proportional to $dx^{n+1}$. It follows that $P$ has the form $P = cx_{n+1}^{3} + x_{n+1}A(x_{1}, \dots x_{n}) + B(x_{1}, \dots, x_{n})$, in which $A$ and $B$ are homogeneous polynomials on $\rea^{n}$ of degrees $2$ and $3$, respectively, and $c$ is a non-negative constant. By an orthogonal change of variables it may be further supposed that $A$ has the form $A = \sum_{i = 1}^{n}a_{i}x_{1}^{2}$, so that, modulo $O(n+1)$, $P$ may be supposed to have the form
\begin{align}\label{pprenormal}
P = cx_{n+1}^{3} + x_{n+1}\sum_{i = 1}^{n}a_{i}x_{i}^{2} + B(x_{1}, \dots, x_{n}). 
\end{align}
Moreover, after a dilation, it may be supposed that $c$ is either $0$ or $1$. The equations \eqref{einsteinpolynomials} are equivalent to the following equations for $A$ and $B$
\begin{align}\label{epreduced}
\begin{split}
&\lap_{\delta}B = 0,\qquad \sum_{i = 1}^{n}a_{i} = -3c, \qquad 4(9c^{2} + \sum_{i = 1}^{n}a_{i}^{2}) = \ka,\\
&\sum_{i = 1}^{n}a_{i}\tfrac{\pr^{2}B}{\pr^{2} x^{i}} = 0,\qquad 8\sum_{i= 1}^{n}a_{i}^{2}x_{i}^{2} + |\hess B|_{\delta}^{2} = 4(9c^{2} + \sum_{i = 1}^{n}a_{i}^{2})E_{n}(x),
\end{split}
\end{align}
in which the subscript $n$ on $E_{n}$ indicates that it is the quadratic form defined by the standard Euclidean structure on $\rea^{n}$. It should be possible to describe all $O(n)$-equivalence classes of solutions to \eqref{einsteinpolynomials} by studying the equations \eqref{epreduced}. As is explained in section \ref{isoparametricsection} below, these equations generalize one case of the equations for isomparametric hypersurfaces solved by E. Cartan in \cite{Cartan-cubic}. It would be interesting to understand what the equations \eqref{einsteinpolynomials} say about the geometry of the level sets of $P$.

\begin{lemma}\label{3dpolylemma}
Any two non-trivial solutions of \eqref{einsteinpolynomials} (for the same $\ka > 0$) are in the same $O(3)$-orbit. In particular, any $P \in \pol^{3}(\rea^{3})$ solving \eqref{einsteinpolynomials} is on the $O(3)$ orbit of $\sqrt{\ka}/6\sqrt{3}$ multiplied by 
\begin{align}
\label{poly2}
\begin{split}
 x_{3}^{3}& - \tfrac{3}{2}x_{3}(x_{1}^{2} + x_{2}^{2}) + \tfrac{1}{\sqrt{2}}\left(x_{1}^{3} - 3x_{1}x_{2}^{2}\right) = x_{3}^{3} - \tfrac{3}{2}x_{3}|z|^{2} + \tfrac{1}{\sqrt{2}}\re z^{3}.
\end{split}
\end{align}
in which $z = x_{1} + \j x_{2}$ (\eqref{poly2} solves \eqref{einsteinpolynomials} with $\ka = 54$). Moreover, the polynomials \eqref{poly2} and $3\sqrt{3}x_{1}x_{2}x_{3}$ are on the same $O(3)$-orbit.
\end{lemma}

\begin{proof}
Suppose that $P$ is in the form \eqref{pprenormal}. First suppose $c = 1$. Since $B$ is harmonic it has the form \eqref{2dharpoly}. The second and last equations of \eqref{epreduced} yield $8a_{1}^{2}x_{1}^{2} + 8(a_{1} + 3)^{2}x_{2} = ( 36 + 4(2a_{1}^{2} + 6a_{1} + 9) -72r^{2})(x_{1}^{2} + x_{2}^{2})$, which forces $a_{1} = -3/2$ and $s = 1/\sqrt{2}$, so that $P$ has the form \begin{align*}
\begin{split}
P(x) &= x_{3}^{3} - \tfrac{3}{2}x_{3}(x_{1}^{2} + x_{2}^{2}) + \tfrac{1}{\sqrt{2}}\cos \theta \left(x_{1}^{3} - 3x_{1}x_{2}^{2}\right)+\tfrac{1}{\sqrt{2}}\sin \theta \left(x_{2}^{3} - 3x_{2}x_{1}^{2}\right).
\end{split}
\end{align*}
A rotation in the $(x_{1}, x_{2})$ plane preserves the first two terms, and, as in the discussion following \eqref{2dharpoly}, after such a rotation it may be supposed that $\theta = 0$, so that $P$ has the the form \eqref{poly2}.

If $c = 0$, then the second equation of \eqref{epreduced} yields $a_{2} = -a_{1} = -a$, and as in the preceeding $B$ may be assumed to have the form \eqref{2dharpoly}. The penultimate equation of \eqref{epreduced} yields $12ar(\cos \theta x_{1} - \sin \theta x_{2}) = 0$, so that either $a = 0$ or $r = 0$. If $a = 0$ then $P = B$, and, as in section \ref{2dharpolysection}, $B$ is orthogonally equivalent to a constant multiple of $x_{1}^{3} - 3x_{1}x_{2}^{2}$; however, in this case the third equation of \eqref{epreduced} implies $\ka = 0$, but $\ka$ is assumed non-zero in \eqref{einsteinpolynomials}, so this is not possible. If $r = 0$, then $P$ is a constant multiple of $\tfrac{3\sqrt{3}}{2}x_{3}(x_{1}^{2} - x_{2}^{2}) = \tfrac{3\sqrt{3}}{2}x_{3}(x_{1} - x_{2})(x_{1} + x_{2})$; here the constant factor is chosen so that the resulting $\ka$ is the same as for \eqref{poly2}, namely $54$. Via the orthogonal change of variables of the form $\sqrt{2}\bar{x}_{1} = x_{1} - x_{2}$ and $\sqrt{2}\bar{x}_{2} = x_{1} + x_{2}$, $P$ is equivalent to $3\sqrt{3}x_{1}x_{2}x_{3}$. That the polynomial \eqref{poly2} and $3\sqrt{3}x_{1}x_{2}x_{3}$ lie on the same $SO(3)$-orbit is evident from the factorization
\begin{align*}
 x_{3}^{3}& - \tfrac{3}{2}x_{3}(x_{1}^{2} + x_{2}^{2}) + \tfrac{1}{\sqrt{2}}\left(x_{1}^{3} - 3x_{1}x_{2}^{2}\right)\\
& = 3\sqrt{3}(\tfrac{\sqrt{2}}{\sqrt{3}}x_{1} + \tfrac{1}{\sqrt{3}}x_{3})(-\tfrac{1}{\sqrt{6}}x_{1} + \tfrac{1}{\sqrt{2}}x_{2} + \tfrac{1}{\sqrt{3}}x_{3})(-\tfrac{1}{\sqrt{6}}x_{1} - \tfrac{1}{\sqrt{2}}x_{2} + \tfrac{1}{\sqrt{3}}x_{3}),
\end{align*}
which exhibits \eqref{poly2} as the pullback of $3\sqrt{3}x_{1}x_{2}x_{3}$ via an element of $SO(3)$.
\end{proof}

\subsubsection{}
Next it is shown that there exist non-trivial solutions to \eqref{einsteinpolynomials} for all $n \geq 2$. let $P \in \pol^{3}(\rea^{n+1})$ be as in \eqref{pprenormal} and choose all the $a_{i}$ to be equal. By \eqref{epreduced} this forces $a_{i} = -3c/n$. If $B \in \har^{3}(\rea^{n})$, then the first and fourth equations of \eqref{epreduced} are satisfied. The last equation of \eqref{epreduced} will be satisfied if and only if $B$ solves \eqref{einsteinpolynomials} with constant $\ka_{n} = 36c^{2}(n+2)(n-1)/n^{2}$. The resulting $P$ solves \eqref{einsteinpolynomials} with constant $\ka_{n+1} = 36c^{2}(n+1)/n = \ka_{n}n(n+1)/((n+2)(n-1))$. If $c$ is chosen to be $n/6$ then $\ka_{n} = (n+2)n(n-1)$ and $\ka_{n+1} = n(n+1)$. This proves

\begin{lemma}
Suppose $Q_{n} \in \pol^{3}(\rea^{n})$ solves \eqref{einsteinpolynomials} with constant $\ka_{n} = n(n-1)$. Then
\begin{align}\label{qns}
Q_{n+1}(x_{1}, \dots, x_{n+1})\defeq \tfrac{n}{6}x_{n+1}^{3} - \tfrac{1}{2}x_{n+1}E_{n}(x_{1}, \dots, x_{n}) + \sqrt{\tfrac{n+2}{n}}Q_{n}(x_{1}, \dots, x_{n})
\end{align}
solves \eqref{einsteinpolynomials} with constant $\ka_{n+1} = n(n+1)$.
\end{lemma}
By section \ref{2dharpolysection}, $Q_{2}$ can be taken to be $6Q_{2} = x_{1}^{3} - 3x_{1}x_{2}^{2}$, in which case $Q_{3}$ is one third of \eqref{poly2}. The resulting $Q_{n}$ for $n > 3$ show that \eqref{einsteinpolynomials} admits non-trivial solutions for all $n$. These polynomials $Q_{n}$ are determined uniquely up to $O(n)$ equivalence, as the $Q_{n+1}$ defined by \eqref{qns} using $Q_{n}$ and using $g\cdot Q_{n}$ for $g \in O(n)$ are equivalent modulo $O(n+1)$. What distinguishes the $O(n)$-orbits of the polynomials $Q_{n}$ among the $O(n)$ orbits of all solutions of \eqref{einsteinpolynomials} is best described in the algebraic formalism introduced in section \ref{codazzialgebrasection}, and is implicit in Lemma \ref{confasslemma} (and the discussion following its proof) and Theorem \ref{confassclassificationtheorem}. It turns out that the Einstein AH structures arising from these $Q_{n}$ \textit{are} projectively flat, so more solutions to \eqref{einsteinpolynomials} are needed.

\subsubsection{}\label{explicitpolynomialsection}
Some more examples of cubic polynomials satisfying the second equation of \eqref{einsteinpolynomials} for $n > 3$ are (the polynomials \eqref{poly3} -\eqref{cartanpoly} are also harmonic):
\begin{align}
\label{poly3} 
\begin{split}
P(x) &= - \tfrac{1}{6}x_{3}^{3} + \tfrac{1}{2}x_{3}(x_{1}^{2} - x_{2}^{2} + x_{4}^{2}) - x_{1}x_{2}x_{4},
\end{split}
& &x \in \rea^{4},\\
\label{nahmpolyso3}
\begin{split}
P(x)  & = \tfrac{1}{2}\det \begin{pmatrix} x_{1} & x_{2} & x_{3} \\ x_{4} & x_{5} & x_{6}\\ x_{7} & x_{8} & x_{9}\end{pmatrix}\\
 &= \tfrac{1}{2}\left(x_{1}x_{5}x_{9} + x_{2}x_{6}x_{7} + x_{3}x_{4}x_{8} - x_{1}x_{6}x_{8} - x_{2}x_{4}x_{9} - x_{3}x_{5}x_{7}\right),
\end{split}&& x \in \rea^{9},\\
\label{cartanpoly}
\begin{split}
P(x) &= x_{5}^{3} + \tfrac{3}{2}x_{5}(x_{1}^{2} + x_{2}^{2} - 2x_{3}^{2} - 2x_{4}^{2})  + \tfrac{3\sqrt{3}}{2}x_{4}(x_{1}^{2} - x_{2}^{2}) + 3\sqrt{3}x_{1}x_{2}x_{3},
\end{split}&& x\in \rea^{5},\\
\label{prehomogpoly}
\begin{split}
P(x)  & = \det \begin{pmatrix} x_{11} & x_{12}/\sqrt{2} & x_{13}/\sqrt{2} \\ x_{12}/\sqrt{2} & x_{22} & x_{23}/\sqrt{2}\\ x_{13}/\sqrt{2} & x_{23}/\sqrt{2} & x_{33}\end{pmatrix}\\
 &= x_{11}x_{22}x_{33} - \tfrac{1}{2}\left(x_{11}x_{23}^{2} + x_{22}x_{13}^{2} + x_{33}x_{12}^{2}\right) + \tfrac{1}{\sqrt{2}}x_{12}x_{23}x_{13},
\end{split}&& x \in \rea^{6},\\
\label{prehomogpoly2}
\begin{split}
P(x)  & = \text{Pfaff}\, X \end{split}&& X \in \mathfrak{so}(6, \rea).
\end{align}
In these examples $\ka$ is, respectively, $4$, $1$, $126$, $3$, and $6$. In what follows there are described several approaches to solving \eqref{einsteinpolynomials}, which partly explain the origins of the polynomials \eqref{poly3}-\eqref{prehomogpoly2}. In each case the polynomial can be interpreted as structure constants of a certain kind of commutative nonassociative algebra, and this motivates the discussion of such algebras in section \ref{codazzialgebrasection}. The polynomial \eqref{poly3} is $1/6$ times the real part of the polynomial $3(x_{1} + \j x_{2})^{2}(x_{3} + \j x_{4}) - (x_{3} + \j x_{4})^{3}$ on $\com^{2}$. The polynomial \eqref{nahmpolyso3} corresponds to the multiplication on the Nahm algebra on $\mathfrak{so}(3)$, as is explained in section \ref{nahmsection}. As is explained in section \ref{isoparametricsection}, the polynomial \eqref{cartanpoly} arises from Cartan's classification of isoparametric hypersurfaces in spheres having three distinct principal curvatures. In section \ref{codazzialgebrasection} it will be proved that the Einstein AH structures resulting from these two examples have non-zero $A_{ijk}\,^{l}$. Finally, the polynomials \eqref{prehomogpoly} and \eqref{prehomogpoly2} (and also \eqref{nahmpolyso3}) arise as the relative invariants of prehomogeneous vector spaces (see the first three entries of table $I$ in section $7$ of \cite{Sato-Kimura}). In all the cases that have been checked, for a regular prehomogeneous vector space for which the irreducible relative invariant polynomial has degree $3$, this polynomial solves \eqref{einsteinpolynomials}, an observation for which the explanation is not yet available.

\subsubsection{}\label{naiveexample}
Here is made a digression to exhibit an example of a naive Einstein AH structure which is not Einstein. If at least one of the real numbers $\mu$ and $\la$ is non-zero then 
\begin{align}
\label{poly1} &P(x)  = \la x_{1}x_{2}x_{3} + \mu x_{2}(x_{1}^{2} - x_{3}^{2}),
\end{align}
solves \eqref{einsteinpolynomials} with $\ka$ equal to $8\mu^{2} + 2\la^{2}$. By Lemma \ref{3dpolylemma} all the polynomials \eqref{poly1} are on the same $O(3)$ orbit as is \eqref{poly2}. Let $P^{\mu, \la}(x)$ be the polynomial \eqref{poly1} and let $L(\mu, \la)_{ijk} = D_{i}D_{j}D_{k}P^{\mu, \la}$, which is a tensor depending on the parameters $\mu$ and $\la$. It makes sense to substitute into $L(\mu, \la)_{ijk}$ functions $f(x)$ and $g(x)$ in place of constants. Let $x^{i}$ be coordinates such that $dx^{i}$ is a parallel frame and let $\pr_{i}$ be the dual coframe. Let $\pr_{i_{1}\dots i_{k}}$ be the completely symmetric part of the tensor product $\pr_{i_{1}}\tensor \dots \tensor \pr_{i_{k}}$ The resulting tensor is explicitly
\begin{align}
L(f, g) = 6f\pr_{112} - 6f\pr_{233} + 6g \pr_{123}.
\end{align}
The divergence of $L(f, g)$ is 
\begin{align}
\div(L(f, g)) = (2g_{3} + 4f_{1})\pr_{12} + 2f_{2}(\pr_{11} +2g_{2}\pr_{12} - \pr_{33}) + (2g_{1}- 4f_{3})\pr_{23}, 
\end{align}
in which $f_{i} = \tfrac{\pr f}{\pr x^{i}}$ and similarly for $g$. Hence if $f$ and $g$ solve $f_{2} = 0 = g_{2}$, $g_{1} = 2f_{3}$, and $g_{3} = -2f_{1}$, then $L(f, g)$ is divergence free. In this case the exact AH structure $(\en, [\delta])$ with aligned representative $\nabla \defeq D - \tfrac{1}{2}L_{ij}\,^{k}$ and distinguished metric $\delta_{ij}$ has $\mr{R}_{ij} = 0$ and $E_{ij} = 0$, the former because $\mr{L}_{ij} = 0$ still holds pointwise, and the latter by \eqref{ddivbt}, so is naive Einstein. On the other hand, by \eqref{confric} its scalar curvature is $-4\uR_{\delta} = |L(f, g)|_{\delta}^{2} = 24f^{2} + 6g^{2}$. If $f$ and $g$ are properly chosen, then $\uR_{\delta}$ is not constant, and so $(\en, [\delta])$ is not Einstein. An explicit example is $f(x) = x_{1} + x_{3}$ and $g(x) = 2(x_{1} - x_{3})$, for which $\uR_{\delta} = -12(x_{1}^{2} + x_{3}^{2})$ is evidently non-constant.

\subsubsection{}
In this section a remark is made about an equivalent description of the equations \eqref{einsteinpolynomials} when $n = 3$, and a Bernstein type problem is proposed. The material of this section is not needed in what follows and the reader mainly interested in the construction of examples as in section \ref{polynomialsection} can skip ahead to section \ref{codazzialgebrasection}. 

The group $GL(3, \rea)$ acts on $\pol^{3}(\rea^{3})$ by $(g\cdot P)(x) \defeq P(gx)$. The equation 
\begin{align}\label{dethpkap}
 \det \hess P = \ka P,
\end{align}
transforms under this action by $\H(g \cdot P) = (\det g)^{2} g \cdot H(P)$, in which here, as in the rest of this section, it is convenient to write $\H(P) \defeq \det \hess P$. In particular, \eqref{dethpkap} is covariant with respect to the group $SL^{\pm}(3, \rea)$ of unimodular linear transformations. Theorem \ref{dethesstheorem} describes the solutions of \eqref{dethpkap} with non-zero $\ka$ up to $SL^{\pm}(3, \rea)$ equivalence. The sign of $\ka$ matters. 
Observe that the two-parameter family $P_{a, b}(x) = \tfrac{a}{6}(x_{1}^{3} + x_{2}^{3} + x_{3}^{3}) + bx_{1}x_{2}x_{3}$ of cubics, known in the nineteenth century as the \textbf{syzygetic pencil} of cubics, satisfies $\H(P_{a,b}) = P_{-ab^{2}, a^{3} + 2b^{3}}$. In the syzygetic pencil there are two one-parameter families of solutions to \eqref{dethpkap} for non-zero $\ka$, and distinguished by the sign of $\ka$. Namely $P_{6a, -3a}$ solves $\H(P) = -54 a^{2} P$, and $P_{0, b}(x) = bx_{1}x_{2}x_{3}$ solves $\H( P) = 2b^{2}P$. Unlike $P_{0, b}$, the polynomial $P_{6a, -3a}$ is not decomposable as a product of linearly independent linear forms, as is most easily seen from its explicit expression:
\begin{align*}
\begin{split}
P_{6a, -3a}(x) &= a(x_{1}^{3} + x_{2}^{3} + x_{3}^{3}) - 3ax_{1}x_{2}x_{3}  = \tfrac{3a}{2}(x_{1} + x_{2} + x_{3})E(x) - \tfrac{a}{2}(x_{1} + x_{2} + x_{3})^{3},\\
& = a(x_{1} + x_{2} + x_{3})(x_{1}^{2} + x_{2}^{2} + x_{3}^{2} - x_{1}x_{2} - x_{2}x_{3} - x_{3}x_{1})\\
& = \tfrac{a}{4}(x_{1} + x_{2}+x_{3})(3(x_{1} - x_{2})^{2} + (x_{1} + x_{2} - 2x_{3})^{2}).
\end{split}
\end{align*} 
The linear change of variables $2y_{1} = x_{1} - x_{2}$, $2\sqrt{3} y_{2} = x_{1} + x_{2} - 2x_{3}$, $3y_{3} = x_{1}+x_{2}+x_{3}$, sends $P_{6a, -3a}$ into $9a y_{3}(y_{1}^{2} + y_{2}^{2})$, which is evidently the product of a linear form and an irreducible degenerate quadratic form.

\begin{theorem}\label{dethesstheorem}
Let $P \in \pol^{3}(\rea^{3})$ be a homogeneous ternary cubic polynomial not identically zero.
\begin{enumerate}
\item $P$ solves \eqref{dethpkap} for some $\ka > 0$ if and only if $P$ decomposes as a product of linearly independent homogeneous linear forms, or, what is equivalent, $P$ is in the $SL^{\pm}(3, \rea)$ orbit of a positive multiple of $x_{1}x_{2}x_{3}$.
\item $P$ solves \eqref{dethpkap} for some $\ka < 0$ if and only if $P$ decomposes as a product of a non-zero linear form $\ell$ and an irreducible degenerate quadratic form $A$ such that the point in $\proj^{2}(\rea)$ determined by the kernel of $A$ is not contained in the line in $\proj^{2}(\rea)$ determined by the kernel of $\ell$, or, what is equivalent, $P$ is in the $SL^{\pm}(3, \rea)$ orbit of a positive multiple of $x_{1}(x_{2}^{2} + x_{3}^{2})$.
\item $P$ is harmonic and solves \eqref{dethpkap} for some $\ka > 0$ if and only if it is a product of orthogonal homogeneous linear forms, or, what is equivalent, $P$ is in the $O(3)$ orbit of a positive multiple of $x_{1}x_{2}x_{3}$. 
\end{enumerate}
Moreover, a harmonic $P \in \har^{3}(\rea^{3})$ solves \eqref{einsteinpolynomials} if and only if it solves \eqref{dethpkap}.
\end{theorem}

\begin{proof}
For $c \neq 0$, the polynomials $P = cx_{1}x_{2}x_{3}$ and $Q = cx_{3}(x_{1}^{2} + x_{2}^{2})$ solve \eqref{dethpkap} with $\ka$ equal to $2c^{2}$ and $-8c^{2}$, respectively, and hence so too do any polynomials in their $SL^{\pm}(3, \rea)$ orbits. By Lemma \ref{3dpolylemma}, it follows that any solution $P \in \pol^{3}(\rea^{3})$ of \eqref{einsteinpolynomials} solves \eqref{dethpkap} for the same $\ka$, which is necessarily positive.
The converse claims will now be proved.

A $P \in \pol^{3}(\rea^{3})$ is irreducible over $\rea$ if and only if it is irreducible over $\com$. The non-trivial direction follows from the claim that if $P$ is reducible over $\com$ then it is reducible over $\rea$. If $P$ factors over $\com$ into irreducible factors, then either $P = \ell Q$ with $\ell$ linear and $Q$ quadratic, or $P = \ell_{1}\ell_{2}\ell_{3}$ with the $\ell_{i}$ linear. In the former case $\bar{\ell}\bar{Q} = \bar{P} = P = \ell Q$, which implies that $\bar{\ell} = e^{2\j \theta}\ell$, in which case $e^{\j \theta}\ell$ is a real factor of $P$, and $e^{-\j\theta}Q$ must be real as well. In the latter case, $\bar{\ell}_{1}\bar{\ell}_{2}\bar{\ell}_{3} = \ell_{1}\ell_{2}\ell_{3}$, and one of the following occurs: the $\ell_{i}$ have multiplicity $1$ and there is $\si \in S_{3}$ such that $\bar{\ell}_{i}$ divides $\ell_{\si(i)}$; one of the factors has multiplicty two; or, there is a single factor of multiplicity three. %
In all the cases it is straightforward to check that there is at least one factor which, possibly after rescaling by a complex factor, is real. The following fact proved in chapter $1$ of \cite{Dolgachev-topics} will be used: a polynomial $A \in \pol^{k}(\com^{3})$ is a factor (over $\com$) of its Hessian determinant $\H(A)$ if and only if each multiplicity one factor of $A$ is linear. If $P \in \pol^{3}(\rea^{3})$ were irreducible over $\rea$ and solved \ref{dethpkap} with $\ka \neq 0$ then it would be irreducible over $\com$ and a multiplicity one factor of $\H(P)$, but not linear, a contradiction. Hence if $P$ solves \eqref{dethpkap} with $\ka \neq 0$ it is reducible over $\rea$ so, by the preceeding paragraph, has a real linear factor.

Now suppose $P = \ell_{1}\ell_{2}\ell_{3}$, in which $\ell_{i} = \lb v_{i}, \eul\ra$ for $v_{i} \in \rea^{3}$. If the $v_{i}$ are linearly dependent, then they lie in a plane, so, after a rotation, $P$ can be assumed to depend on only two variables, which yields $\H(P) = 0$. Hence, for $P$ to solve \eqref{dethpkap}, the $v_{i}$ must be linearly independent; it is evident that such $P$ is in the $SL^{\pm}(3, \rea)$ orbit of a positive multiple of $x_{1}x_{2}x_{3}$. Such a $P$ is moreover harmonic if and only if $\sum_{\si \in S_{3}}\lb v_{\si(1)}, v_{\si(2)}\ra v_{\si(3)} = 0$, and, since the $v_{i}$ are linearly independent, this holds if and only if they are pairwise orthogonal; in this case $P$ is in the $O(3)$-orbit of a multiple of $x_{1}x_{2}x_{3}$, which by Lemma \ref{3dpolylemma} implies also that $P$ solves \eqref{einsteinpolynomials}. 

Otherwise $P$ must be a product $\ell A$ where $\ell$ is linear and $A$ is a homogeneous quadratic polynomial irreducible over $\rea$. If $A$ is non-degenerate then after applying an element of $SL^{\pm}(3, \rea)$, it may be assumed to be a positive multiple of $E$ or $F \defeq x_{1}^{2} + x_{2}^{2} - x_{3}^{2}$. In the former case, after applying a rotation (which preserves $A$), it can be assumed $\ell$ is a multiple of $x_{1}$, so that $P = x_{1}E$. In the latter case, after applying an element of $O(2, 1)$, it can be assumed that $\ell_{1}$ is one of $x_{1}$, $x_{3}$, or $x_{1} + x_{2}$, according to whether the coefficient vector of $\ell$ is spacelike, timelike, or null with respect to $F$. In the null case, it simplifies computations to instead send $A$ to a multiple of $x_{1}x_{2} + x_{3}^{2}$ and $\ell$ to $x_{1}$. In all of the preceeding cases, direct computation shows that $\H(P)$ is not a multiple of $P$. 

There remains to consider the case $P = \ell A$ with $A$ an irreducible degenerate quadratic form. This can be only if $A$ has rank two and negative index of inertia equal to $0$, so after applying an element of $SL^{\pm}(3, \rea)$ 
 it can be supposed $A$ has the form $A = \pm(x_{1}^{2} + x_{2}^{2})$. If $\ell$ depends on only $x_{1}$ and $x_{2}$ then $P$ depends on only two variables, so $\H(P) = 0$; this condition can be stated invariantly as that $\ker \ell$ contains $\ker A$. If $\ker \ell$ does not contain $\ker A$, there is a triangular element of $SL(3, \rea)$ with diagonal entries equal to $1$ that fixes $x_{1}$ and $x_{2}$ and sends $\ell$ into a multiple of $x_{3}$. Hence $P$ can be assumed to have the form $P = cx_{3}(x_{1}^{2} + x_{2}^{2})$ with $c \in \reat$. In this case $\H(P) = -8c^{2} P$ and so $P$ solves \eqref{dethpkap} with negative $\ka$.
\end{proof}

Theorem \ref{dethesstheorem} gives credibility to the speculation that the following Bernstein like problem has an affirmative resolution. Namely, \textit{must a sufficiently smooth function $F$ on $\rea^{3}$ solving \eqref{dethpkap} for positive $\ka$ be equivalent modulo the action of the group of unimodular affine transformations to a multiple of $x_{1}x_{2}x_{3}$?} 

A homogeneous polynomial $P$ is \textbf{homaloidal} if the map of (complex) projective spaces defined by the differential $DP$ is birational.  Theorem \ref{dethesstheorem} is closely related to Theorem $4$ of I. Dolgachev's \cite{Dolgachev-polarcremona}, which characterizes (over $\com$) ternary homaloidal polynomials having no multiple factors. While it is not clear what form a characterization of cubic solutions of \eqref{einsteinpolynomials} might take for $n > 3$, the most natural generalization of \eqref{dethpkap} for $P \in \pol^{n}(\rea^{n})$ is $\H( P) = (-1)^{n-1}\ka P^{n-2}$. In Remark $3.5$ of \cite{Ciliberto-Russo-Simis} such polynomials are called \textbf{totally Hessian}, and it is asked whether they must be homaloidal. In this regard observe that a simple determinantal computation exploiting homogeneity proves:
\begin{lemma}\label{inductiveplemma}
If $Q \in \pol^{n}(\rea^{n})$ solves $\H( Q) = (-1)^{n-1}(n-1)Q^{n-2}$, then $P \in \pol^{n+1}(\rea^{n+1})$ defined by $P(x) = x_{n+1}Q(x_{1}, \dots, x_{n})$ solves $\H( P) = (-1)^{n}nP^{n-1}$.
\end{lemma}
In particular there is always the solution $P_{n}(x) = \prod_{i = 1}^{n}x_{i}$, for which $(-1)^{n-1}\ka = (-1)^{n-1}(n-1)$. It also seems interesting to ask what are all the smooth solutions to $\H( F) = \ka F^{n-2}$ up to unimodular affine equivalence. 

\subsection{Einstein commutative Codazzi algebras}\label{codazzialgebrasection}
The formalism of Einstein commutative Codazzi algebras is introduced with the immediate aim of organizing the computations necessary to analyze the vanishing or not of the tensor $A_{ijk}\,^{l}$ associated to the Einstein AH structures constructed as in section \ref{polynomialsection}. As it seems a classification of these algebras should be feasible, some basic structural features are elucidated in more detail than is immediately necessary.

\subsubsection{}
An \textbf{algebra} \textbf{$(\alg, \mprod)$} is a finite-dimensional real vector space $\alg$ and an element $\mprod$ of $\alg^{\star}\tensor \alg^{\star}\tensor \alg$, referred to as the \tbf{multiplication}. Here the base field has always characteristic zero; in all applications it will be $\rea$. An algebra is \textbf{unital} if it contains a left and right unit (in which case these are necessarily equal). Algebras need not be unital, and those of interest here will not be. That an algebra be \tbf{nonassociative} means that it need not be associative, although it could be. The \tbf{left (resp. right) regular representation} is the representation $L:\alg \to \eno(\alg)$ (resp. $R$) given by left (resp. right) multiplication, $L(x)(y) \defeq x \mprod y$. Here $\eno(\alg)$ means vector space endomorphisms of $\alg$ and is regarded as an algebra under composition. The multiplication $\mprod$ is identified with its \textbf{structure tensor} given in abstract index notation as $\mtens_{ij}\,^{k}$. As tensors $L(x)_{i}\,^{j} = x^{p}\mtens_{pi}\,^{j}$. The algebra is \tbfs{commutative}{algebra} if $\mtens_{[ij]}\,^{k} = 0$; equivalently $L(x) = R(x)$ for all $x \in \alg$. The \textbf{associator} of $\alg$ is the $\alg$-valued trilinear form $[x, y, z] \defeq (x\mprod y)\mprod z - x \mprod (y \mprod z)$ given by $[x, y, z]^{l} = x^{i}z^{j}y^{k}\mu_{ijk}\,^{l}$ (note the ordering convention) for the tensor $\mtens_{ijk}\,^{l} = \mtens_{pj}\,^{l}\mtens_{ik}\,^{p} - \mtens_{ip}\,^{l}\mtens_{kj}\,^{p}$. The algebra $\alg$ is \textbf{associative} if and only if $\mtens_{ijk}\,^{l} = 0$; equivalently the left regular representation is a homomorphism, $L(x\mprod y) = L(x)\circ L(y)$. If $\alg$ is commutative then the associator has the algebraic symmetries $\mtens_{[ij]k}\,^{l} = \mtens_{ijk}\,^{l}$ and $\mtens_{[ijk]}\,^{l} = 0$ of an affine curvature tensor. The former corresponds to $[x, y, z]  = -[z, y, x]$, or, what is the same (by polarization), $[x, y, x] = 0$. An algebra satisfying $[x, y, x] = 0$ for all $x, y \in \alg$ is called \tbf{flexible}, so the preceeding remarks say that any commutative algebra is flexible. For a commutative algebra there holds the stronger identity
\begin{align}\label{commassoc}
[x, y, z] - [x, z, y] = [y, x, z],
\end{align}
which corresponds to the Bianchi symmetry $\mtens_{[ijk]}\,^{l} = 0$. 

Many of the standard notions for associative algebras do not work as simply for nonassociative algebras, e.g. the obvious notions of semisimplicity, namely that the solvable radical be zero and that the algebra be expressible as a direct sum of simple ideals, need not coincide, and so when necessary such conditions will be stated explicitly. As is explained in \cite{Schafer} it makes sense to speak of a bimodule $\B$ over an arbitrary nonassociative algebra $\alg$. This means there are given linear maps $L$ and $R$ from $\alg$ to the algebra $\eno(\B)$ of linear endomorphisms of $\B$. In general there are no further conditions imposed. If $\alg$ is a member of some class of nonassociative algebras defined by multilinear identities then $L$ and $R$ are required to be compatible with these identities (see section $6$ of \cite{Schafer} for a precise statement). For example, if $\alg$ is commutative then there must hold $L_{a} = R_{a}$ for all $a \in \alg$, so that in this case $L = R$, and to specify a module $\B$ over a commutative algebra it suffices to give a single arbitrary linear map $\rho:\alg \to \eno(\B)$. For associative algebras, the compatibility conditions are $R_{b}R_{a} = R_{ab}$, $L_{a}L_{b} = L_{ab}$ and $R_{a}L_{b} = L_{b}R_{a}$ for all $a, b \in \alg$. In this case it makes sense to speak of a left or right module, which is just the structure corresponding to $L$ or $R$ with the compatibility condition involving only $L$ or $R$. For any finite-dimensional algebra $\alg$ the vector space dual $\alg^{\ast} \defeq \hom_{\rea}(\alg, \rea)$ is a (bi)module over $\alg$ with left action $L$ and right action $R$ defined respectively by $L_{a}\mu \defeq \mu \circ R_{a}$ and $R_{a}\mu \defeq \mu \circ L_{a}$ for $a \in \alg$ and $\mu \in \alg^{\ast}$. It is sometimes more readable to write the pairing of dual vector spaces using angled brackets so that the result of applying $L_{a}\mu \in \alg^{\ast}$ to $b \in \alg$ is $\lb L_{a}\mu, b\ra = \lb \mu, ba\ra$. Thus, although for a general nonassociative algebra $\alg$ it does not make sense to distinguish a left action from a right action, it does always make sense to speak of the left and right actions of $\alg$ on its dual $\alg^{\ast}$. The interchange of $L$ and $R$ in the definitions of the actions on the dual is made so that in the case that $\alg$ is associative, the map $L$ (resp. $R$) makes $\alg^{\ast}$ a left (resp. right) module over $\alg$.

In what follows there will be considered a finite-dimensional commutative nonassociative (nonunital) algebra equipped with a non-degenerate symmetric invariant bilinear form, and some terminology and context is needed for discussing such an object. Although there are many papers written about this or that nonassociative algebra, there are substantial theories only for special cases arising in geometry and representation theory of Lie groups, e.g. Jordan algebras and composition algebras. Some references are \cite{Schafer-book} and \cite{Knus-Merkurjev-Rost-Tignol}. One context in which commutative nonassociative algebras appear naturally is that of vertex algebras; in that context such algebras come equipped with a lot of other structure. They appear also in the study of autonomous systems of ordinary differential equations quadratic in the dependent variables initiated by L. Markus in \cite{Markus}; see the survey \cite{Kinyon-Sagle-quadratic} for background. Still, it seems no one has developed a theory of general commutative nonassociative algebras, and it may be that such structure is simply too weak to admit a general theory. However, as is familiar from the study of Lie algebras, the extra data of a non-degenerate symmetric bilinear form having some invariance property allows for a good structure theory. Part of the reason for the algebraic digression which follows is that it may suggest some interesting classification problems.

\subsubsection{}
A bilinear form $h$ on $\alg$ is \textbf{left invariant} if $h(x\mprod y, z) = h(x, y \mprod z)$ for all $x, y, z \in \alg$, \textbf{right invariant} if $h(x, y \mprod z) = h(z \mprod x, y)$ for all $x, y, z \in \alg$, and \textbf{braided} if it is both left and right invariant, or, what is the same it satisfies the braid relations
\begin{align}\label{braid}
h(z, x \mprod y) = h(y \mprod z, x) = h(y, z \mprod x) = h(x\mprod y, z) = h(x, y \mprod z) = h(z \mprod x, y)
\end{align}
for all $x, y, z \in \alg$. A braided bilinear form necessarily satisfies $h(x \mprod y, z) = h(z, x \mprod y)$ for all $x, y, z \in \alg$. For example, the Killing form of a Lie algebra is braided. Evidently a left or right invariant bilinear form on a commutative algebra is automatically braided.

A bilinear form $h$ is right (resp. left) invariant if and only if the \textbf{opposite} bilinear form $\op{h}$ defined by $\op{h}(x, y) = h(y, x)$ is left (resp. right) invariant. For this reason it is customary to speak only of left invariant forms, and these are called simply \textbf{invariant} (or \textbf{associative}). A bilinear form is \textbf{symmetric} if it equals its opposite. Evidently, for a symmetric bilinear form the three notions, left invariant, right invariant, and braided are mutually equivalent, and so such a form will be called simply \textbf{invariant}.

\begin{lemma}\label{frobeniuslemma}
For a finite-dimensional nonassociative algebra $\alg$ any two of the following sets are in canonical bijection.
\begin{enumerate}
\item Isomorphisms from $\alg$ to $\alg^{\ast}$ intertwining the left actions of $\alg$.
\item Non-degenerate right invariant bilinear forms on $\alg$.
\item Isomorphisms from $\alg$ to $\alg^{\ast}$ intertwining the right actions of $\alg$.
\item Non-degenerate left invariant bilinear forms on $\alg$.
\end{enumerate}
Also, the following sets are in canonical bijection.
\begin{enumerate}
\setcounter{enumi}{4}
\item $\alg$-bimodule isomorphisms from $\alg$ to $\alg^{\ast}$.
\item Non-degenerate braided bilinear forms on $\alg$.
\end{enumerate}
\end{lemma}
\begin{proof}
Suppose given a linear isomorphism $\Psi:\alg \to \alg^{\ast}$ such that $L_{a}(\Psi(b)) = \Psi \circ L_{a}(b)$ for all $a, b \in \alg$. That is, $L_{a}(\Psi(b)) = \Psi(a\mprod b)$. Define $h$ by $h(a, b) = \lb \Psi(a), b\ra$. Obviously $h$ is bilinear and non-degenerate. By the assumption on $\Psi$, 
\begin{align}\label{frob1}
\begin{split}
h(a \mprod b, c) & = \lb \Psi(a\mprod b), c\ra = \lb \Psi \circ L_{a}(b) , c\ra = \lb L_{a}\circ\Psi(b), c\ra = \lb \Psi(b), R_{a}(c)\ra  = h(b, c \mprod a),
\end{split}
\end{align}
so $h$ is right-invariant. On the other hand, if given a non-degenerate right-invariant bilinear form $h$, then define $\Psi$ by $\lb \Psi(a), b\ra = h(a, b)$. Because $h$ is non-degenerate, $\Psi$ is injective, so a linear isomorphism because $\alg$ is finite-dimensional. For all $a, b, c \in \alg$ there holds 
\begin{align}\label{frob2}
\begin{split}
\lb L_{a}(\Psi(b)), c\ra &= \lb \Psi(b), c \mprod a \ra = h(b, c \mprod a) = h(a \mprod b, c) = \lb \Psi(a \mprod b), c\ra = \lb \Psi \circ L_{a}(b), c\ra,\\
\end{split}
\end{align}
showing that $L_{a}\circ\Psi = \Psi \circ L_{a}$. Evidently these associations are inverses so establish a bijection between $(1)$ and $(2)$. Given $\Psi$ as in $(3)$, the form $h$ defined as above will be left invariant by a computation like \eqref{frob1}, while given $h$ as in $(4)$ the map $\Psi$ defined as above will intertwine the right actions by a computation like \eqref{frob2}; that these associations are inverse is evident. The bijection between $(2)$ and $(4)$ is given by passing to the opposite bilinear form. Given $(5)$ define $h$ as before; it will be both left and right invariant, so braided. Given $(6)$ define $\Psi$ as before; it will intertwine both the left and right actions, so will be a bimodule map.
\end{proof}

Usually a finite-dimensional, unital, associative algebra $\alg$ over a field $\fie$ is called a \textbf{Frobenius algebra} if it is equipped with one of the following structures, each of which determines the others in a canonical manner:
\begin{enumerate}
\item A left $\alg$-module isomorphism from $\alg$ to $\alg^{\ast}$.
\item A right $\alg$-module isomorphism from $\alg$ to $\alg^{\ast}$.
\item A non-degenerate (left) invariant bilinear form $h$ on $\alg$.
\item A linear form $f:\alg \to \fie$ such that $\ker f$ contains no non-trivial ideal of $\alg$.
\end{enumerate}
The first three structures determine one another even if there is no unit. Thus it makes sense to call \textbf{Frobenius} a nonassociative algebra equipped with a non-degenerate left-invariant bilinear form (or one of the equivalent structures of $(2)$ of Lemma \ref{frobeniuslemma}), and to call a Frobenius algebra \textbf{braided} if this form is braided (or, equivalently, the corresponding linear map satisfies $(5)-(6)$ of Lemma \ref{frobeniuslemma}). Really these should be called \textit{left} Frobenius algebras. As by Lemma \ref{frobeniuslemma}, by replacing $h$ with $h^{\op}$, a left Frobenius algebra determines a right Frobenius algebra , and conversely, it suffices to work with the former. With this terminology, what is usually called a \textit{Frobenius algebra} is a unital, associative, Frobenius algebra. A Frobenius algebra is \textbf{symmetric} if the invariant form is symmetric. A symmetric Frobenius algebra is braided. With this terminology, a semisimple real Lie algebra is a symmetric Frobenius algebra. 
\begin{lemma}\label{braidedsymmetriclemma}
A unital braided Frobenius algebra $(\alg, h)$ is symmetric. 
\end{lemma}
\begin{proof}
If $e \in \alg$ is the unit then $h(x, y) = h(x, e\mprod y) = h(y\mprod x, e) = h(y, x \mprod e) = h(y, x)$.
\end{proof}

A (not necessarily unital) \textbf{composition algebra} is an algebra $\alg$ equipped with a non-degenerate quadratic form $E$ which is multiplicative in the sense that $E(x \mprod y) = E(x)E(y)$. A \textbf{symmetric composition algebra} is a composition algebra for which the symmetric bilinear form $h(x, y) = E(x + y) - E(x) - E(y)$ determined by $E$ is left invariant. In the present terminology these are symmetric Frobenius algebras for which the quadratic form $E$ associated to the metric $h$ is multiplicative. Such algebras have been classfied by A. Elduque and H.~C. Myung in \cite{Elduque-Myung}; not quite all examples are given by the para-Hurwitz and Okubo algebras. Basic results about such algebras are given in chapter $34$ of \cite{Knus-Merkurjev-Rost-Tignol}. In particular they are non-unital if $\dim \alg > 1$ and they are alternative. However, there are no interesting examples of symmetric composition algebras which are commutative, so these will not serve present aims.

The multiplication defining an algebra $\alg$ is \textbf{trivial} if the product of any two-elements is $0$. An algebra is \textbf{simple} if its multiplication is non-trivial and it has no non-trivial proper (two-sided) ideal; note that the definition in the nonassociative setting is the same as the usual definition. The commutative case of Lemma \ref{simplelemma} is Lemma $7.4$ of \cite{Dong-Griess-Lam}, and the easy proof is essentially the same. 
\begin{lemma}\label{simplelemma}
The space of non-degenerate braided bilinear forms on a simple finite-dimensional nonassociative algebra $\alg$ over a field $\fie$ of characteristic $0$ is at most one-dimensional. 
\end{lemma}
\begin{proof}
Let $\bar{h}$ and $h$ be two such forms. From the non-degeneracy of $h$ and $\bar{h}$ it follows that the linear map $\Psi:\alg \to \alg$ defined by $\bar{h}(x, y) = h(\Psi(x), y)$ is a linear isomorphism. Then 
\begin{align*}
h(\Psi(x\mprod y), z) = \bar{h}(x \mprod y, z) = \bar{h}(y,  z\mprod x) = h(\Psi(y), z \mprod x) = h(x \mprod \Psi(y), z),\\
h(\Psi(y\mprod x), z) = \bar{h}(y \mprod x, z) = \bar{h}(y,  x\mprod z) = h(\Psi(y), x \mprod z) = h( \Psi(y)\mprod x, z),
\end{align*}
in which the second and fourth equalities of the first line (resp. second line) use respectively the right (resp. left) invariance of $\bar{h}$ and of $h$. By the non-degeneracy of $h$ the first chain of equalities implies $\Psi(x \mprod y) = x \mprod \Psi(y)$, while the second implies $\Psi(y \mprod x) = \Psi(y) \mprod x$. Since $\Psi$ is invertible there is a non-zero $\la$ in the algebraic closure $\bar{\fie}$ of $\fie$ such that the $\la$-eigenspace in $\bar{\alg} = \alg \tensor_{\fie}\bar{\fie}$ of $\Psi$ is non-zero. By an obvious extension argument, $\bar{\alg}$ equipped with the extended multiplication is simple if and only if $\alg$ is. Since $\Psi(x)\mprod y = \Psi(x \mprod y) = x \mprod \Psi(y)$, the $\la$-eigenspace of $\Psi$ is a two-sided ideal in $\bar{\alg}$, so must equal $\bar{\alg}$. It follows that $\Psi$ is multiplication by $\la$, and moreover that $\la \in \fie$. Hence $\bar{h} = \la h$.
\end{proof}

The multiplication defining an algebra $\alg$ is \textbf{left (resp. right) non-degenerate} if for any $ x\in \alg$ there is $y \in \alg$ such that $x\mprod y \neq 0$ (resp. such that $y\mprod x \neq 0$). The multiplication is non-degenerate if it is both left and right non-degenerate.

\begin{lemma}
The multiplication of a simple finite-dimensional Frobenius algebra is non-degenerate.
\end{lemma}
\begin{proof}
For any algebra the square $\alg^{2}$, defined to be the subspace spanned by products of elements of $\alg$, is a two-sided ideal. If $\alg$ is simple then the multiplication is by assumption non-trivial and so there must hold $\alg^{2} = \alg$. Let $h$ be the non-degenerate left invariant bilinear form defining the Frobenius algebra structure. Suppose there is $x \in \alg$ such that $x\mprod y = 0$ for all $y \in \alg$. Then $0 = h(x\mprod y, z) = h(x, y\mprod z)$ for all $y, z \in \alg$. Since $\alg^{2} = \alg$, every element of $\alg$ can be written as a finite sum of terms of the form $y \mprod z$, and so the non-degeneracy of $h$ implies that $x = 0$. Thus the multiplication on $\alg$ is left non-degenerate. The analogous argument with $h^{\op}$ in place of $h$ and left and right multiplication interchanged shows the right non-degeneracy of the multiplication. 
\end{proof}

\subsubsection{}
This section shows that an arbitrary Frobenius algebra determines a unital Frobenius algebra of one dimension more. This construction is something like passing from projective space to the affine space over it. It is needed in the proof of Theorem \ref{confassclassificationtheorem} below.

Let $(\alg, h)$ be a finite-dimensional $\fie$-algebra equipped with a bilinear form $h$. Define a unital algebra $(\hat{\alg}_{h}, \hat{h})$ with bilinear form $\hat{h}$ as follows. As a vector space $\hat{\alg}_{h}$ comprises pairs $(x, \la)$ with $x \in \alg$ and $\la \in \fie$. The multiplication on $\hat{\alg}_{h}$ is defined by $(x, \la)\mprod(y, \mu) = (x\mprod y + \la y + \mu x, h(x, y) + \la \mu)$ and has unit $(0, 1)$. Evidently $\hat{\alg}_{h}$ is commutative if and only if $\alg$ is commutative and $h$ is symmetric. However, associativity is not preserved by this construction for
\begin{align}\label{alghassoc}
\begin{split}
[(x, \la), (y, \mu), (z, \ga)] = ([x, y, z] + h(x, y)z - h(y, z)x, h(x \mprod y, z) - h(x, y \mprod z)).
\end{split}
\end{align}
On the other hand, it follows from \eqref{alghassoc} that $\hat{\alg}_{h}$ is associative if and only if $h$ is left invariant and $[x, y, z] = h(y, z)x - h(x, y)z$ for all $x, y, z \in \alg$. The bilinear form $\hat{h}$ on $\hat{\alg}_{h}$ is defined by $\hat{h}((x, \la), (y, \mu)) = h(x, y) + \la \mu$. The $\hat{h}$ norm of the unit in $\hat{\alg}_{h}$ is $1$. Evidently $\hat{h}$ is non-degenerate or symmetric if and only if $h$ has the same property. Routine computations show
\begin{align*}
\begin{split}
&\hat{h}((x, \la)\mprod (y, \mu), (z, \ga)) - h((x, \la), (y, \mu)\mprod (z, \ga)) = h(x\mprod y, z) - h(x, y \mprod z),
\end{split}
\end{align*}
from which it is evident that $\hat{h}$ is left invariant if and only if $h$ is left invariant. 
On the other hand,
\begin{align}\label{alghright}
\begin{split}
&\hat{h}((z, \ga)\mprod(x, \la), (y, \mu)) - \hat{h}((x, \la), (y, \mu)\mprod (z, \ga)) \\
&= h(z \mprod x, y) - h(x, y \mprod z) + \mu(h(z, y) - h(y, z)) + \la (h(z, x) - h(x, z)),
\end{split}
\end{align}
from which it follows that $\hat{h}$ is right invariant if and only if $h$ is both right invariant and symmetric. The preceeding proves the first part of
\begin{lemma}
Let $\alg$ be a finite-dimensional algebra with a bilinear form $h$. Then $(\hat{\alg}_{h}, \hat{h})$ is a unital Frobenius algebra if and only if $(\alg, h)$ is a Frobenius algebra. In this case the following are equivalent: $\hat{h}$ is braided; $\hat{h}$ is symmetric; $h$ is symmetric and braided.
\end{lemma}
\begin{proof}
There remains to prove only the last claim. If left invariant $h$ or $\hat{h}$ is symmetric, then it is braided. If $\hat{h}$ is braided, then by \eqref{alghright}, $h$ is symmetric, and so by definition $\hat{h}$ is symmetric as well. Alternatively, if $\hat{h}$ is braided then by Lemma \ref{braidedsymmetriclemma}, $\hat{h}$ is symmetric.
\end{proof}

In any unital Frobenius algebra $(\alg, h)$ there holds $h(e, y) = h(y, e)$ for all $y \in \alg$, for $h(e, y) = h(e \mprod e, y) = h(e, e\mprod y) = h(e, y\mprod e) = h(e \mprod y, e) = h(y, e)$. Hence if $h(e, e) \neq 0$, the projection $P_{e}$ onto the subspace $\{y \in \alg: h(y, e) = 0 = h(e, y)\}$ can be written as $P_{e}(y) =  y - h(y, e)h(e, e)^{-1}e$. 

\begin{lemma}\label{reconstructionlemma}
Let $(\balg, g)$ be a finite-dimensional unital Frobenius algebra such that the unit $e \in \balg$ satisfies $g(e, e) \neq 0$. Then the subspace $\alg \defeq \{x \in \balg: g(x, e) = 0 = g(e, x)\}$ equipped with the multiplication $x \star y  \defeq x \mprod y - g(x, y)g(e, e)^{-1}e$ and the restriction of $g$ is a Frobenius algebra. 
\end{lemma}
\begin{proof}
If $x, y \in \alg$ then $g(x \star y, e) = g(x \mprod y, e) - g(x, y) = 0$ and $g(e, x\star y) = g(e, x\mprod y) - g(x, y) = 0$, both by left invariance of $g$, so $x \star y \in \alg$. If $x \in \alg$, by non-degeneracy of $g$ there exist $y, z \in \balg$ such that $g(x, y) \neq 0$ and $g(z, x) \neq 0$, and since $x \in \alg$ there hold $g(x, P_{e}(y)) = g(x, y) \neq 0$ and $g(P_{e}(z), x) = g(z, x) \neq 0$, so $y$ and $z$ could have been chosen from $\alg$. This shows the non-degeneracy of the restriction of $g$ to $\alg$. Finally, for $x, y, z \in \alg$ there holds $g(x \star y, z) - g(x, y \star z) = g(x \mprod y, z) - g(x, y, \mprod z) = 0$, so that $g$ is left invariant with respect to the multiplication $\star$ on $\alg$.
\end{proof}
If $(\alg, h)$ is Frobenius, the construction of Lemma \ref{reconstructionlemma} applied to $(\hat{\alg}_{h}, \hat{h})$ recovers $(\alg, h)$. It makes sense to refer to $(\hat{\alg}_{h}, \hat{h})$ as the \textbf{unitalization} of the Frobenius algebra $(\alg, h)$.

\begin{lemma}\label{unitalizationuniquelemma}
Two finite dimensional Frobenius algebras are isomorphic if and only if their unitalizations are isomorphic.
\end{lemma}
\begin{proof}
An isomorphism $\Psi$ of Frobenius algebras extends to an isomorphism $\hat{\Psi}$ of their unitalizations defined by $\hat{\Psi}(x, \la) = (\Psi(x), \la)$, so the unitalizations of isomorphic Frobenius algebras are isomorphic. Since an isomorphism of Frobenius algebras is an isometry, to prove the converse it suffices to prove that an automorphism $\Phi$ of the unitalization $(\hat{\alg}_{h}, \hat{h})$ is necessarily of the form $\hat{\Psi}$. The most general linear endomorphism of $\hat{\alg}_{h}$ has the form $\Phi(x, \la) = (Ax + \la b, h(c, x) + \la d)$ for some $A \in \eno(\alg)$, $b, c \in \alg$, and $d \in \rea$. Since 
\begin{align}\label{unt1}
\begin{split}
0 &= \hat{h}((0, 1), (x, 0)) = \hat{h}(\Phi(0, 1), \Phi(x, 0)) = \hat{h}((b, d), (Ax, h(c, x))) = h(b, Ax) + dh(c, x),
\end{split}
\end{align}
there holds
\begin{align}\label{unt2}
\begin{split}
0 & = \Phi(0, 1)\mlt \Phi(x, 0) - \Phi((0,1)\mlt (x,0)) = (b, d)\mlt(Ax, h(c, x)) - \Phi(x, 0) \\
& = (b \mlt Ax + (d-1)Ax + h(c, x)b, h(b, Ax) + (d-1)h(c, x))\\
& = (b \mlt Ax + (d-1)Ax + h(c, x)b, -h(c, x)).
\end{split}
\end{align}
Hence $h(c, x) = 0$ for all $x \in \alg$, and since $h$ is non-degenerate this implies $c = 0$. Since $c = 0$, $A$ must be invertible. In \eqref{unt1} this gives $h(b, AX) = 0$, and by the invertibility of $A$ and the non-degeneracy of $h$ this implies $b = 0$. In \eqref{unt2} this gives $(d-1)Ax$, which implies $d = 1$, so that $\Phi = \hat{\Psi}$ where $\Psi(x) = Ax$ is an automorphism of $(\alg, h)$.
\end{proof}

\subsubsection{}
The \textbf{(left) trace form} $\tau$ on a nonassociative algebra $\alg$ is the symmetric bilinear form defined by $\tau(x, y) = \tr L(x)L(y) = L(x)_{p}\,^{q}L(y)_{q}\,^{p} = x^{a}y^{b}\mu_{ab}$, where, consistent with the notations used throughout this paper, $\mu_{ab} \defeq \mtens_{ap}\,^{q}\mtens_{bq}\,^{p}$.

\subsubsection{}
There is a useful analogy between algebras and connections which explains the point of view behind much of what follows. The structure tensor $\mtens_{ij}\,^{k}$ is thought of as analogous to a connection, and its skew part $\mtens_{[ij]}\,^{k}$ is the analogue of the torsion. The associator is the analogue of the curvature tensor. The trace form is something like the Ricci tensor. A non-degenerate invariant form should be regarded as a metric (for which there is not assumed any \textit{a priori} relation with the connection/multiplication). This analogy can be made a genuine correspondence by using the formalism of algebroids, but this will not be needed here. Some ideas along these lines are explored by L. Ionescu in \cite{Ionescu}.

\subsubsection{}
The only systematic study of  non-degenerate left invariant bilinear forms on nonassociative algebras of which I am aware is that made by M. Bordemann in \cite{Bordemann}. In \cite{Bordemann} a nonassociative algebra admitting a non-degenerate left invariant symmetric bilinear form $h$ is called \textbf{metrizable} and the pair $(\alg, h)$ is called a \textbf{metric algebra}. In Bordemann's terminology, a unital associative symmetric Frobenius algebra is a \textit{unital associative metric algebra}. Following the analogy between algebras and connections described above, the terminology \textit{metric} is not quite apt, for by this analogy such terminology should apply to an algebra equipped with a non-degenerate invariant symmetric bilinear form $h$ compatible with the multiplication in the sense that $h(x\mprod y, z) + h(y, x \mprod z) = 0$; this last condition is the algebraic condition that the multiplication $x \mprod y = \nabla_{x}y$ determined by a flat connection $\nabla$ with respect to a parallel frame be compatible with a metric $h$. What is of interest here is the algebraic condition corresponding to the Codazzi compatibility between a connection and a metric, which is
\begin{align}\label{codazzialgebra}
h(x\mprod y, z) + h(y, x\mprod z) = h(y\mprod x, z) + h(x, y \mprod z),
\end{align}
A nonassociative algebra $\alg$ with a non-degenerate symmetric bilinear form $h$ satisfying \eqref{codazzialgebra} is a \textbf{Codazzi algebra}. A Codazzi algebra is \textbf{Riemannian} if $h$ is positive definite.

Plainly if a non-degenerate symmetric bilinear form $h$ on $\alg$ is invariant then $(\alg, h)$ is Codazzi. From \eqref{codazzialgebra} it is evident that $(\alg, h)$ be commutative Codazzi is equivalent to the complete symmetry of $h(x\mprod y, z)$ in $x$, $y$, and $z$. That is, $\mtens_{ijk} = \mtens_{ij}\,^{p}h_{pk}$ satisfies $\mtens_{ijk} = \mtens_{(ijk)}$. In particular, if $(\alg, h)$ is Codazzi and $\alg$ is commutative, then $h$ is invariant, so a commutative Codazzi algebra is the same thing as a commutative symmetric Frobenius algebra. Since commutative Frobenius algebras are exactly the objects obtained by naively relaxing the associativity and unitality conditions in the usual definition of a symmetric Frobenius algebra, commutative Codazzi algebras seem to be the right generalization to the nonunital, nonassociative setting of unital associative symmetric Frobenius algebras.

\begin{remark}
An algebra is a \textbf{left symmetric algebra} (\textbf{LSA}) (also called \textit{pre-Lie algebras} or \textit{Vinberg algebras}) if its associator satisfies $[x, y, z] = [y, x, z]$ for all $x, y, z \in \alg$. From \eqref{commassoc} it is evident that a commutative algebra is left symmetric if and only if it is associative. A left symmetric Codazzi algebra is called a \textbf{Hessian algebra}. Hessian algebras arise naturally in the description of homogeneous convex cones, and this is not unrelated to the appearance here of Codazzi algebras, although this point will not be pursued; see H. Shima's book \cite{Shima} for background and references.
\end{remark}

\subsubsection{}
To an algebra $\alg$ with a metric $h$ there is associated a cubic polynomial $P = P^{\mtens}  \in S^{3}(\alg^{\ast})$ defined by $6P^{\mtens}(x) = h(x \mprod x, x) = x^{i}x^{j}x^{k}\mtens_{(ijk)}$. For example, the cubic polynomial $\hat{P}(x, r)$ associated to the unitalization of a Frobenius algebra $(\alg, h)$ is given in terms of the cubic polynomial $P(x)$ associated to $(\alg, h)$ by $\hat{P}(x, r) = P(x) + \tfrac{1}{2}rE(x) + \tfrac{1}{6}r^{3}$. Let $D$ be the Levi-Civita connection of $h_{ij}$, which is just the flat affine connection on $\alg$, and write $P_{i_{1}\dots i_{k}} = D_{i_{1}}\dots D_{i_{k}}P$. Raise and lower indices with $h_{ij}$ and the dual bivector $h^{ij}$. Observe $P_{i_{1}\dots i_{k}} \equiv 0$ if $k > 3$. If $(\alg, h)$ is a commutative Codazzi algebra, then $\mtens_{ijk} \defeq \mtens_{ij}\,^{p}h_{pk}$ is completely symmetric, so that by construction $P_{ij}\,^{k} = \mtens_{ij}\,^{k}$. That is, a commutative Codazzi algebra is completely determined by its associated cubic polynomial.

\subsubsection{}
An algebra $\alg$ is \textbf{(left) special} if $\tr L(x) = 0$ for all $x \in \alg$. A special algebra cannot have a (left) unit. The Lie algebra of a unimodular Lie group is special in this sense. That $\alg$ be special is equivalent to $\mtens_{ip}\,^{p} = 0$, and it follows in particular that a commutative Codazzi algebra is special if and only if its associated cubic polynomial is harmonic. A special commutative algebra satisfies $\mtens_{pij}\,^{p} = \mu_{ij}$; that is the Ricci trace of its associativity tensor equals its trace form.

\begin{lemma}\label{flatalgebralemma}
A finite-dimensional real Riemannian commutative associative Codazzi algebra $(\alg, h)$ is isomorphic to Euclidean space $(\rea^{n}, \delta)$ with a multiplication given in the standard basis $e_{i}$ by $e_{i}\mprod e_{i} = \la_{i}e_{i}$ (no summation) and $e_{i}\mprod e_{j} = 0$ if $i \neq j$ for some real constants $\la_{i} \in \rea$. If $(\alg, h)$ is moreover special, then it is isomorphic to $(\rea^{n}, \delta)$ with the trivial multiplication.
\end{lemma}

\begin{proof}
It can be assumed without loss of generality that the invariant bilinear form is $\delta_{ij}$. Let $e_{i}$ be a standard basis of $\rea^{n}$. Associativity means that the symmetric endomorphisms $L(e_{i})$ given by left multiplication by $e_{i}$ form a commuting family, so are simultaneously diagonalizable by an orthogonal linear transformation, so it may be assumed from the beginning that the $L(e_{i})$ are diagonal. Then by the complete symmetry of $\delta(L(e_{i})( e_{j}), e_{k})$ in $ijk$ and the fact that $L(e_{i})$ is diagonal, $ \delta(L(e_{i})( e_{j}), e_{k})$ vanishes unless $i = j = k$. This means that $L(e_{i})$ has a non-zero entry $\la_{i}$ only in its $ii$ entry. If $\alg$ is special, then each $L(e_{i})$ must be trace-free, and this evidently forces $\la_{i} = 0$.
\end{proof}
With the obvious modifications Lemma \ref{flatalgebralemma} holds for $h$ of other signatures. For a unital, associative Frobenius algebra, Lemma \ref{flatalgebralemma} is well known; see e.g. Proposition $2.2$ of \cite{Hitchin-frobenius}. From the point of view taken here, Lemma \ref{flatalgebralemma} should be interpreted as saying that a special associative commutative Riemannian Codazzi algebra is flat.   

\begin{remark}
Given a Frobenius algebra $(\alg, h)$ associate to each pair of functions $f, g \in \cinf(\alg)$ the vector field $Q(f, g)^{i}\defeq df_{a}dg_{b}\mu^{abi}$, in which indices are raised using $h^{ij}$. In particular, taking $f = g = E(x) \defeq \tfrac{1}{2}h(\eul, \eul)$, where $\eul$ is the vector field generating radial dilations by $e^{t}$, the dynamics of the integral curves of the vector field $Q^{i} \defeq dE_{a}dE_{b}\mu^{abi}$ reflect the algebraic properties of $\alg$. It was the idea of L. Markus, \cite{Markus}, to associate to an autonomous system of quadratic differential equations the nonassociative algebra determined by its coefficients. It seems the idea might be useful also in the opposite direction; that is using the associated system of differential equations to study the algebra.
\end{remark}

\subsubsection{}\label{einsteincodazzisection}
A commutative Codazzi algebra $(\alg, \mprod, h)$ is \textbf{Einstein} if its trace form $\mu_{ij}$ is a constant multiple of $h_{ij}$. An Einstein commutative Codazzi algebra is \textbf{proper} or \textbf{improper} according to whether this constant is non-zero or zero. In either case, the constant is $\tfrac{1}{n}|\mtens|_{h}^{2} = \tfrac{1}{n}\mtens_{ip}\,^{q}\mtens_{jq}\,^{p}h^{ij}$, in which $n = \dim \alg$. The trace form of a proper Einstein commutative Codazzi algebra is invariant. A commutative algebra with non-degenerate invariant trace-form is a proper Einstein commutative Codazzi algebra with $h_{ij} = c\mu_{ij}$ for any non-zero $c$ in the base field. By Lemma \ref{simplelemma} these are the only Einstein structures on a simple commutative algebra with invariant non-degenerate trace-form. 

\begin{remark}
Something like infinite-dimensional Einstein commutative Codazzi algebras have been studied by N. Sasakura in a series of papers of which \cite{Sasakura} and \cite{Sasakura-emergent} are representative.
\end{remark}

\subsubsection{}\label{hessianmultiplicationsection}
Let $\eul$ be the vector field generating dilations by $e^{t}$ on $\alg$. That $P$ be a cubic polynomial implies $P_{ij} = \eul^{k}P_{ijk}$, so that a commutative Codazzi algebra $(\alg, h)$ satisfies the Einstein condition if and only if $|\hess P|_{h}^{2} = P^{ij}P_{ij} = \eul^{a}\eul^{b}\mtens_{ija}\mtens_{b}\,^{ij} = \tfrac{1}{n}|\mtens|_{h}^{2}E(x)$ where $E(x) \defeq \eul^{i}\eul^{j}h_{ij}$.  

On the other hand, suppose that as in section \ref{polynomialsection} there is given on $\rea^{n}$ with the Euclidean metric $\delta_{ij}$ a homogeneous cubic polynomial $P(x)$ solving \eqref{einsteinpolynomials} for some constant $\ka > 0$. The multiplication $\mprod$ defined by $\mtens_{ij}\,^{k} = P_{ij}\,^{k}$ makes $\rea^{n}$ into a special proper Einstein commutative Codazzi algebra. Because $(x\mprod y)^{j} = \eul_{y}^{i}\hess P(x)_{i}\,^{j}$, the multiplication table can be found by computing the Hessian of $P$. Concretely this means that for the standard basis $e_{i}$ of $\rea^{n}$ the product $e_{i}\circ e_{j}$ is the result of applying to $e_{j}$ the matrix corresponding to $\hess P$ evaluated at $e_{i}$. For example, the algebra on $\rea^{4}$ determined by the polynomial $P$ of \eqref{poly3} has the multiplication table:
\begin{align}\label{poly3table}
&e_{1}\mprod e_{1} = e_{3},& &e_{1}\mprod e_{2} = -e_{4},& &e_{1}\mprod  e_{3} = e_{1},&& e_{1}\mprod e_{4} = - e_{2},& &e_{2}\mprod e_{2} = -e_{3},&\\
\notag &e_{2}\mprod e_{3} = - e_{2},& &e_{2}\mprod e_{4} = -e_{1},& &e_{3}\mprod e_{3} = -e_{3}, & &e_{3}\mprod e_{4} = e_{4},& &e_{4}\mprod e_{4} = e_{3}.
\end{align}
in which $\{e_{i}\}$ is the standard basis of $\rea^{3}$. From \eqref{poly3table} it can be seen that this algebra is special. The polynomial can be reconstructed from the multiplication table as follows. The polynomial is $1/6$ time the sum of the monomials constructed according to the rule: to each occurrence of $ce_{k}$ in the expansion of the product $e_{i} \mprod e_{j}$ associate the monomial $cx^{i}x^{j}x^{k}$. The sum is take over all ordered triples $\{i, j, k\}$. %

\subsubsection{}\label{confasssection}
Suppose $\alg \simeq \rea^{n}$ carries a Riemannian special proper Einstein commutative Codazzi structure $(\alg, \mprod, h_{ij})$ (so $h_{ij}$ is a constant tensor). Then, as explained in section \ref{polynomialsection}, the connection $\nabla = D - \tfrac{1}{2}\mtens_{ij}\,^{k}$ is the aligned representative of an exact proper Einstein AH structure $(\en, [h])$ with distinguished metric $h_{ij}$ and satisfying $E_{ijk}\,^{l} = 0$. By \eqref{confcurvijkl} there holds $T_{ijk}\,^{l} = -\tfrac{1}{4}\mu_{ijk}\,^{l}$. Hence $-4A_{ijk}\,^{p}h_{pl} = C_{ijkl}$, which is the $h$-trace free part of $\mu_{ijkl}$. Say that a commutative Codazzi algebra is \tbf{conformally associative} if the $h$-trace free part of $\mu_{ijkl}$ vanishes identically. By Lemma \ref{weylcriterion}, a three-dimensional commutative Codazzi algebra is necessarily conformally associative. The goal of this section is the construction of Riemannian special proper Einstein commutative Codazzi algebras which are \textit{not} conformally associative, because by the preceeding discussion these yield exact Riemannian Einstein AH structures satisfying $E_{ijk}\,^{l} = 0$ but for which $A_{ijk}\,^{l}$ is not identically zero, and which are therefore, by Lemma \ref{projflatahlemma}, neither projectively nor conjugate projectively flat. The polynomial $P = P^{\mtens}$ associated to a Riemannian special proper Einstein commutative Codazzi structure $(\alg, \mprod, h_{ij})$ solves \eqref{einsteinpolynomials} for some $\ka > 0$. The trace-free part of $\mu_{ijkl}$ is $\mu_{ijkl} - \tfrac{2\ka}{(n-1)}h_{l[i}h_{j]k}$, and so there needs to be shown that there are $x, y, z, v \in \alg$ such that $x^{i}y^{j}z^{k}v^{l}(\mu_{ijkl} - \tfrac{2\ka}{(n-1)}h_{l[i}h_{j]k}) \neq 0$, or, what is the same
\begin{align}\label{confasscriterion}
 h(x \mprod z, y \mprod v) - h(y \mprod z, x \mprod v) = h([x, z, y], v)  \neq \tfrac{\ka}{(n-1)}(h(x, v)h(y, z) - h(y, v)h(x, z)).
\end{align}
For the algebra associated to \eqref{poly3}, taking $x = e_{1}$, $y = e_{2}$, $z = e_{3}$, and $v = e_{4}$ in \eqref{confasscriterion}, the righthand side is $0$, while it follows from \eqref{poly3table} that the left hand side is $-2$; hence this algebra is not conformally associative and yields an example of the desired sort.

If $(\alg, \mlt, h)$ is a Riemannian special Einstein commutative Codazzi algebra with positive Einstein constant $\ka$, then $(\alg, \mlt, \tfrac{\ka}{n-1}h)$ is an algebra of the same sort, with Einstein constant $(n-1)$.
\begin{lemma}\label{unitalassociativelemma}
The unitalization $(\hat{\alg}_{h}, \hmlt, \hat{h})$ of an $n$-dimensional special proper Einstein commutative Codazzi algebra $(\alg, \mlt, h)$ with Einstein constant $(n-1)$ is an Einstein commutative Codazzi algebra with Einstein constant $(n+1)$. Moreover, $(\hat{\alg}_{h}, \hmlt, \hat{h})$ is associative if and only if $(\alg, \mlt, h)$ is conformally associative.
\end{lemma}
\begin{proof}
For $(x, \la)\in \hat{\alg}_{h}$, the left multiplication operator $L(x, \la)$ is related to $L(x)$ by
\begin{align*}
L(x, \la) = \begin{pmatrix} L(x) + \la I_{n} & x\\ x^{\flat}& \la \end{pmatrix},
\end{align*}
and so the trace form $\hat{\tau}$ of $\hat{\alg}_{h}$ is $\hat{\tau}((x, \la), (y, \mu)) = \tr L(x, \la)\circ L(y,\mu) = \tau(x, y) + 2h(x, y) + (n+1)\la\mu = (n+1)\hat{h}((x, \la), (y, \mu))$, showing the first claim. The second part follows by comparing \eqref{alghassoc} and \eqref{confasscriterion} and using the non-degeneracy of $\hat{h}$.
\end{proof}

\subsubsection{Example: Nahm Algebras}\label{nahmsection}
In \cite{Kinyon-Sagle}, M. Kinyon and A. Sagle define the \tbf{Nahm algebra $\nahm(\g)$} of a finite-dimensional (real) Lie algebra $\g$ to be the vector space $\nahm(\g) = \g\oplus\g\oplus\g$ equipped with the evidently commutative multiplication
\begin{align}
&\begin{pmatrix}x_{1}\\x_{2}\\ x_{3} \end{pmatrix}\mprod\begin{pmatrix}y_{1}\\y_{2}\\ y_{3} \end{pmatrix} = \frac{1}{2}\begin{pmatrix}[x_{2}, y_{3}] + [y_{2}, x_{3}]\\
[x_{3}, y_{1}] + [y_{3}, x_{1}] \\ [x_{1}, y_{2}] + [y_{1}, x_{2}]\end{pmatrix},& &x_{i}, y_{j} \in \g.
\end{align}
The left multiplication $L(x)$ has the block matrix form
\begin{align}
L(x) = \tfrac{1}{2}\begin{pmatrix} 0 & -\ad(x_{3}) & \ad(x_{2}) \\ \ad(x_{3}) & 0 & -\ad(x_{1})\\ -\ad(x_{2})) & \ad(x_{1}) & 0\end{pmatrix},
\end{align}
from which it is evident that $\tr L(x) = 0$ and so $\nahm(\g)$ is special. A straightforward calculation (see Theorem $7.2$ of \cite{Kinyon-Sagle}) shows that the trace form $\tau$ is related to the Killing form $B_{\g}$ by 
\begin{align}\label{tracebg}
-2\tau(x, y) = \sum_{i = 1}^{3}B_{g}(x_{i}, y_{i}). 
\end{align}
From this and the invariance of the Killing form it follows that $\tau$ is invariant (Theorem $7.4$ of \cite{Kinyon-Sagle}). Theorem $5.1$ of \cite{Kinyon-Sagle} shows that $\nahm(\g)$ is simple if and only if $\g$ is simple, and Theorem $5.3$ of \cite{Kinyon-Sagle} shows that $\nahm(\g)$ is a direct sum of simple ideals if and only if $\g$ is semisimple. From these and the preceeding remarks there follows Corollary $7.7$ of \cite{Kinyon-Sagle}.
\begin{theorem}[Kinyon-Sagle]
For a finite-dimensional real Lie algebra $\g$ the following are equivalent: $(1).$ $\g$ is semisimple; $(2).$ $B_{\g}$ is non-degenerate; $(3).$ $\nahm(\g)$ is a direct sum of simple ideals; $(4).$ $\tau$ is non-degenerate. 
\end{theorem}
It follows that if $\g$ is a semisimple real Lie algebra, then $(\nahm(\g), \mprod, \tau)$ is a special proper Einstein commutative Codazzi algebra. The simplest example is $\g = \so(3)$ regarded as $\rea^{3}$ with the cross product. Precisely, identifying $x \in \rea^{3}$ with $\hat{x} \in \so(3)$ defined by $\hat{x}y = x \times y$ gives an isomorphism $(\rea^{3}, \times) \to (\so(3), [\dum, \dum])$. Then $\nahm(\so(3))$ is $\rea^{9}$ with the multiplication
\begin{align}
x \mprod y = 
\frac{1}{2}\begin{pmatrix} 
x_{5}y_{9} - x_{6}y_{8} + y_{5}x_{9} - y_{6}x_{8}\\ 
x_{6}y_{7} - x_{4}y_{9} + y_{6}x_{7} - y_{4}x_{9} \\ 
x_{4}y_{8} - x_{5}y_{7} + y_{4}x_{8} - y_{5}x_{7}\\ 
x_{8}y_{3} - x_{9}y_{2} +y_{8}x_{3} - y_{9}x_{2} \\ 
x_{9}y_{1} - x_{7}y_{3} +y_{9}x_{1} - y_{7}x_{3}\\ 
x_{7}y_{2} - x_{8}y_{1} +y_{7}x_{2} - y_{8}x_{1}\\ 
x_{2}y_{6} - x_{3}y_{5} +y_{2}x_{6} - y_{3}x_{5} \\ 
x_{3}y_{4} - x_{1}y_{6} +y_{3}x_{4} - y_{1}x_{6}\\ 
x_{1}y_{5} - x_{2}y_{4}+y_{1}x_{5} - y_{2}x_{4}
\end{pmatrix}.
\end{align} 
The Killing form is $B_{\so(3)}(x, y) = \tr \widehat{x\times y} = \tr \hat{x}\hat{y} = -2\delta_{ij}x^{i}y^{j}$, where $\delta_{ij}$ is the usual Euclidean inner product on $\rea^{3}$. By \eqref{tracebg} the trace-form on $\nahm(\so(3))$ is $\delta_{ij}$ where $\delta_{ij}$ is the usual Euclidean inner product on $\rea^{9}$. The harmonic polynomial $P(x) = \tfrac{1}{6}h(x\mprod x, x)$ corresponding to $\mprod$ is \eqref{nahmpolyso3}, and it is easily checked that $|\hess P|^{2} = E(x) = \tfrac{1}{9}P^{ijk}P_{ijk}E(x)$. 

Finally, $\nahm(\so(3))$ is not conformally associative. In \eqref{confasscriterion} take $x = e_{1}$, $y = e_{2}$, $z = e_{4}$, and $v = e_{5}$, and $h$ to be the trace-form on $\nahm(\g)$. From \eqref{tracebg} it is evident that these vectors are pairwise $h$-orthogonal so the righthand side of \eqref{confasscriterion} vanishes. On the other hand, the lefthand side is $h(e_{1}\mprod e_{4}, e_{2} \mprod e_{5}) - h(e_{2}\mprod e_{4}, e_{1}\mprod e_{5})$ and computing the Hessian of \eqref{nahmpolyso3} shows that $2e_{2}\mprod e_{4} = -e_{9} = -2 e_{1}\mprod e_{5}$, so that the lefthand side of \eqref{confasscriterion} equals $1/4$, which suffices to show that $\nahm(\so(3))$ is not conformally associative. That this same argument applies to any compact simple real $\g$ is the content of Theorem \ref{compactnahmtheorem}.

If $\h \subset \g$ is a subalgebra then the subset $\{x\in \nahm(\g): x_{i} \in \h \,\, \text{for}\,\, i = 1, 2, 3\}$ is a subalgebra of $\nahm(\g)$ isomorphic to $\nahm(\h)$.
\begin{theorem}\label{compactnahmtheorem}
Let $\g$ be a compact simple real Lie algebra. Then $\nahm(\g)$ is not conformally associative. Hence in this case $(\nahm(\g), \mprod, \tau)$ is a special proper Einstein commutative Codazzi algebra which is not conformally associative.
\end{theorem}
\begin{proof}
It follows from the general theory of $\mathfrak{sl}(2)$ triples that there can be found elements $e, f \in \g$ such that $\spn\{e, f, [e,f]\}$ is a subalgebra isomorphic to $\so(3)$ in such a way that $e$, $f$, and $[e, f]$ map to the standard basis elements of $\rea^{3}$, and the Lie bracket is identified with the cross-product. Let $\nu_{i}: \g \to \nahm(\g)$ be the inclusion in the $i$th component. In \eqref{confasscriterion} take $x = \nu_{1}(e)$, $y = \nu_{1}(f)$, $z = \nu_{2}(e)$, and $v = \nu_{2}(f)$, and $h$ to be the trace-form on $\nahm(\g)$. From \eqref{tracebg} it is evident that these vectors are pairwise $h$-orthogonal, so that the righthand side of \eqref{confasscriterion} vanishes. By definition of the Nahm product, $\nu_{i}(a)\mprod \nu_{i}(a) = 0$ for any $a \in \g$, so the first term on the lefthand side of \eqref{confasscriterion} vanishes. On the other hand $2 \nu_{1}(e) \mprod \nu_{2}(f) = \nu_{3}([e, f])$ and so $8h(\nu_{1}(f) \mprod \nu_{2}(e), \nu_{1}(e) \mprod \nu_{2}(y)) = -2 h(\nu_{3}([e, f]), \nu_{3}([e, f])) = B_{\g}([e, f], [e, f])$, which is not $0$ because $[e, f]$ is not zero and $\g$ is compact. Hence the second term on the lefthand side of \eqref{confasscriterion} is non-zero, and so $\nahm(\g)$ is not conformally associative.
\end{proof}

\subsubsection{Example: isoparametric hypersurfaces}\label{isoparametricsection}
A hypersurface in the round sphere $S^{n-1}$ is \textbf{isoparametric} if its principal curvatures (the eigenvalues of its Riemannian shape operator) are constant. See \cite{Cecil} for background and references. In \cite{Cartan-cubic}, Cartan classified the isoparametric hypersurfaces with at most three distinct principal curvatures. In the case of three distinct principal curvatures, the principal curvatures must all have the same multiplicity $m$, which must be one of $1$, $2$, $4$, or $8$, and the hypersurface must be a tube of constant radius over the image in $S^{3m+1}$ of the Veronese embedding of the projective plane over one of the real definite signature composition algebras. This yields four one-parameter families (the parameter is the radius), each of which can be realized as the level sets $P(x) = \cos 3t$ of a homogeneous cubic polynomial $P \in \pol^{3}(\rea^{n})$ solving 
\begin{align}\label{cartaniso}
&\lap_{\delta}P = 0, && |DP|_{\delta}^{2} = 9E^{2}.
\end{align}
\begin{lemma}
If $P \in \pol^{3}(\rea^{n})$ solves \eqref{cartaniso} then it solves \eqref{einsteinpolynomials} with $\ka = 18(n+2)$.
\end{lemma}
\begin{proof}
By \eqref{cartaniso} there holds $\lap_{\delta}|DP|_{\delta}^{2} = 9\lap_{\delta}E^{2} = 18(E\lap_{\delta}E + |DE|_{\delta}^{2}) = 36(n+2)E$. By the remark immediately following \eqref{einsteinpolynomials}, this means $P$ solves \eqref{einsteinpolynomials} with $\ka = 18(n+2)$.
\end{proof}
Let $\fie$ be one of $\rea$, $\com$, $\quat$ (quaternions), or $\cayley$ (octonions), and let $m = \dim_{\rea}\fie$. Let $z_{1}, z_{2}, z_{3} \in \fie$ and let $\bar{z}_{i}$ denote the canonical conjugation on $\fie$ fixing the real subfield. Let $x, y \in \rea$. By \cite{Cartan-cubic} any solution of \eqref{cartaniso} is equivalent modulo a rotation to one of 
\begin{align}\label{cartanformula}
\begin{split}
P(x, y, z_{1}, z_{2}, z_{3}) &= x^{3} - 3xy^{2} + \tfrac{3}{2}x\left(z_{1}\bar{z}_{1} + z_{2}\bar{z}_{2} - 2z_{3}\bar{z}_{3}\right) \\
&\qquad + \tfrac{3\sqrt{3}}{2}y\left(z_{1}\bar{z}_{1} - z_{2}\bar{z}_{2}\right) + \tfrac{3\sqrt{3}}{2}\left((z_{1}z_{2})z_{3} + \bar{z}_{3}(\bar{z}_{2}\bar{z}_{1}) \right)
\in \pol^{3}(\rea^{3m+2}),
\end{split}
\end{align}
which is, up to changes of notation, equation $(17)$ of \cite{Cartan-cubic}. The parentheses in the last term of \eqref{cartanformula} are necessary when $m = 8$, because $\cayley$ is not associative. The case $m = 1$ of \eqref{cartanformula} is \eqref{cartanpoly}.

\begin{lemma}
For $m \in \{1, 2, 4, 8\}$ the multiplication on $\rea^{3m+2}$ determined by the homogeneous cubic polynomial $P$ of \eqref{cartanformula} associated to the unique family of isoparametric hypersurfaces in $S^{3m+1}$ having three distinct principal curvatures forms with the Euclidean metric a special proper commutative Codazzi Einstein algebra which is not conformally associative.
\end{lemma}

\begin{proof}
It suffices to verify that the resulting multiplication is not conformally associative. This is shown using \eqref{confasscriterion}, as in the proof of Theorem \ref{compactnahmtheorem}. For $z_{i} \in \fie$, write $z_{i} = x_{i} + y_{i}$, where $x_{i}$ is the real part, and $y_{i}$ is the imaginary part. Let $e_{x}$ and $e_{y}$ be the basis vectors in the $x$ and $y$ diretions, and for $i = 1, 2, 3$, let $e_{i}$ be the basis vector in the $x_{i}$ direction (that is spanning the real part of the $i$th copy of $\fie$). There will be used \eqref{confasscriterion} applied to the vectors $e_{1}$, $e_{2}$, $e_{3}$ and $e_{x}$. The polynomial \eqref{cartanformula} becomes 
\begin{align*}%
\begin{split}
P(x, y,& z_{1}, z_{2}, z_{3}) = x^{3} - 3xy^{2} + \tfrac{3}{2}x\left(x_{1}^{2} + x_{2}^{2} - 2x_{3}^{2} + y_{1}^{2} + y_{2}^{2} - 2y_{3}^{2}\right)  + \tfrac{3\sqrt{3}}{2}y\left(x_{1}^{2} - x_{2}^{2} + y_{1}^{2} - y_{2}^{2}\right) \\
&\qquad + \tfrac{3\sqrt{3}}{2}\left(2x_{1}x_{2}x_{3} + x_{3}(y_{1}y_{2} + y_{2}y_{1}) + x_{1}(y_{2}y_{3} + y_{3}y_{2}) + x_{2}(y_{3}y_{1} + y_{1}y_{3}) + [y_{1}, y_{2}, y_{3}]\right).
\end{split}
\end{align*}
Now it is easy to compute the components of the Hessian involving $x$, $y$, and the $x_{i}$ alone (in all cases these are just as in the $\fie = \rea$ case). In computing a product such as $e_{1}\mlt e_{2}$ there matter only terms in $P$ in which after differentiating by $x_{1}$ and then some other variable there remains an $x_{2}$; such a term must contain $x_{1}x_{2}$, and there is only one such term. The relevant products in the associated algebra are
\begin{align*}
\begin{split}
&e_{1}\mlt e_{1} = 3e_{x} + 3\sqrt{3}e_{5},\quad e_{1}\mlt e_{2} = 3\sqrt{3}e_{3},\quad e_{2}\mlt e_{3} = 3\sqrt{3}e_{1}, \quad e_{3}\mlt e_{3} = -6e_{x},
\end{split}
\end{align*}
so that
\begin{align*}
[e_{1}, e_{2}, e_{3}] = (e_{1}\mlt e_{2})\mlt e_{3} - e_{1}\mlt(e_{2}\mlt e_{3}) = 3\sqrt{3} e_{3}\mlt e_{3} - 3\sqrt{3}e_{1} \mlt e_{1} =  -27\sqrt{3}e_{x} - 27 e_{y}.
\end{align*}
Hence $h([e_{1}, e_{2}, e_{3}], e_{x}) \neq 0$, showing that the lefthand side of \eqref{confasscriterion} is non-zero, while the righthand side of \eqref{confasscriterion} is zero, because the considered vectors are pairwise orthogonal. 
\end{proof}

\subsubsection{}
Next it is shown that an $n$-dimensional special proper Einstein commutative Codazzi structure gives rise to one of $(n+1)$-dimensions. 
 
Let $\ka_{n}$ and $\ka_{n+1}$ be positive constants such that $\ka_{n+1} > n$. Let $(\alg, \mlt,  h)$ be an $n$-dimensional Einstein commutative Codazzi structure with trace form $\mu_{ij}$ satisfying $\mu_{ij} = \ka_{n} h_{ij}$. Let $\halg = \alg \oplus \spn_{\rea}\{e\}$. For $x, y \in \alg$ and $r, s \in \rea$ define $\hat{h}(x + re, y + rs) = h(x, y) + rs$ and define a commutative multiplication $\hmlt$ on $\halg$ by
\begin{align}\label{extensionmlt}
\begin{split}
(x + re)\hmlt(y + se) &= \sqrt{\tfrac{\ka_{n+1}}{(n+1)n}}\left( \sqrt{\tfrac{(n+2)(n-1)}{\ka_{n}}}x \mlt y - sx - ry + (nrs - h(x, y))e\right),
\end{split}
\end{align}
This is the transcription in the language of Codazzi algebras of the identity \eqref{qns} exhibiting a solution on $\rea^{n+1}$ of \eqref{einsteinpolynomials} in terms of a solution of the same equations on $\rea^{n}$. The multiplications determined by $Q_{n}$ and $Q_{n+1}$ can be computed from their Hessians, as in section \ref{hessianmultiplicationsection}, and the results are related as in \eqref{extensionmlt}. That $\hat{h}$ is invariant follows directly from \eqref{extensionmlt} and the definition of $\hat{h}$, and so $(\halg, \hmlt, \hat{h})$ is a commutative Codazzi algebra by \eqref{codazzialgebra}. Let $L_{x}$ and $\hat{L}_{x + re}$ denote the left multiplications on $\alg$ and $\halg$, respectively. Schematically,
\begin{align*}
\hat{L}_{x + re} 
& = \sqrt{\tfrac{\ka_{n+1}}{n(n+1)}} \begin{pmatrix} \sqrt{\tfrac{(n+2)(n-1)}{\ka_{n}}} L_{x} - rI_{n} & - x \\ - x^{\flat} & nr\end{pmatrix}.
\end{align*}
Hence $\tr \hat{L}_{x + re} = \sqrt{\tfrac{\ka_{n+1}(n+2)(n-1)}{\ka_{n}(n+1)n}} \tr L_{x}$, so that $(\halg, \hmlt)$ is special if and only if $(\alg, \mlt)$ is special. In this case 
\begin{align*}
\begin{split}
\tr (\hat{L}_{x + re}\circ \hat{L}_{y + se}) & = \tfrac{\ka_{n+1}}{n(n+1)}\left( \tfrac{(n+2)(n-1)}{\ka_{n}}\tr(L_{x}\circ L_{y}) + 2h(x, y) + n(n+1)rs\right)\\
&= \ka_{n+1}\hat{h}(x + re, y + se),
\end{split}
\end{align*}
showing that $(\halg, \hmlt, \hat{h})$ is Einstein with constant $\ka_{n+1}$. The multiplication \eqref{extensionmlt} takes a reasonably nice form for the choices $\ka_{n} = n(n-1)$ and $\ka_{n+1} = (n+1)n$.

\begin{lemma}\label{confasslemma}
Let $(\alg_{n}, \mlt_{n},  h_{n})$ be an $n$-dimensional Einstein special commutative Codazzi structure with constant $\ka_{n} = n(n-1)$ and let $(\alg_{n+1}, \mlt_{n+1}, h_{n+1})$ be the associated $(n+1)$-dimensional Einstein special commutative Codazzi structure with constant $\ka_{n+1} = n(n+1)$. Then $(\alg_{n+1}, \mlt_{n+1}, h_{n+1})$ is conformally associative if and only if $(\alg_{n+1}, \mlt_{n+1}, h_{n+1})$ is conformally associative.
\end{lemma}
\begin{proof}
Let $\hat{x} = x + re$, $\hat{y} = y + se$, $\hat{z} = z + te$, and $\hat{u} = u + qe$ for $x, y, z, u \in \alg_{n}$ and $r, s, t, q \in \rea$. 
The associators of $(\alg_{n+1}, \mlt_{n+1}, h_{n+1})$ are related to those of $(\alg_{n}, \mlt_{n},  h_{n})$ by 
\begin{align*}
\begin{split}
[\hat{x}, \hat{y}, \hat{z}]_{n+1} &= \tfrac{n+2}{n}[x, y, z]_{n} \\
&\quad + h_{n}(x, y)z - h_{n}(y, z)x + (n+1)\left(stx - srz\right) + (n+1)\left(rh_{n}(y, z) - th_{n}(x, y) \right)e.
\end{split}
\end{align*}
Hence
\begin{align*}
\begin{split}
h_{n+1}([\hat{x}, \hat{y}, \hat{z}]_{n+1}, \hat{u}) &= \tfrac{n+2}{n}h_{n}([x, y, z]_{n}, u) + h_{n}(x, y)h_{n}(z, u) - h_{n}(y, z)h_{n}(x, u) \\
&\quad + (n+1)\left(sth_{n}(x, u) - srh_{n}(z, u) + qrh_{n}(y, z) - qth_{n}(x, y) \right),
\end{split}
\end{align*}
and so
\begin{align}\label{confassn}
\begin{split}
&h_{n+1}([\hat{x}, \hat{y}, \hat{z}]_{n+1}, \hat{u})  - (n+1)\left(h_{n+1}(\hat{x}, \hat{u})h_{n+1}(\hat{y}, \hat{z}) - h_{n+1}(\hat{x}, \hat{y})h_{n+1}(\hat{z}, \hat{u})\right)\\
& = \tfrac{n+2}{n}\left[h_{n}([x, y, z]_{n}, u) - n\left(h_{n}(x, u)h_{n}(y, z) - h_{n}(x, y)h_{n}(z, u)\right)\right],
\end{split}
\end{align}
from which the claim is evident.
\end{proof}

While Lemma \ref{confasslemma} means that the algebras $(\alg_{n}, \mlt_{n},  h_{n})$ constructed from the polynomials $Q_{n}$ as in \eqref{qns} \textit{are} conformally associative, it also means that once there has been found a non conformally associative Einstein special commutative Codazzi algebra in dimension $n$, there result such algebras in all dimensions $m > n$. In particular, applying the preceeding construction to the algebra of section \ref{confasssection} constructed from the polynomial \eqref{poly3}, shows that such algebras exist in all dimensions $n \geq 4$.

\subsubsection{}
The classification of the finite-dimensional Riemannian Einstein commutative Codazzi algebras over $\rea$ seems to be a reasonable problem. Here nothing much is said about this, except to treat the conformally associative case. Lemma \ref{3dpolylemma} implies that up to isomorphism there is a unique three-dimensional Riemannian signature Einstein special commutative Codazzi algebra with a given Einstein constant. Because in three dimensions conformal associativity is automatic, this is the $n = 3$ special case of the following.
\begin{theorem}\label{confassclassificationtheorem}
Up to isomorphism there is a unique $n$-dimensional conformally associative Riemannian signature Einstein special commutative Codazzi algebra with a given Einstein constant.
\end{theorem}
\begin{proof}
Because the signature is Riemannian, the Einstein constant is zero if and only if the multiplication is trivial. If the Einstein constant of $(\alg, \mlt, h)$ is $\ka > 0$ then by Lemma \ref{unitalassociativelemma} the unitalization of $(\alg, \mlt, \tfrac{\ka}{n-1}h)$ is an associative Riemannian signature Einstein commutative Codazzi algebra with Einstein constant $(n+1)$. By Lemma \ref{flatalgebralemma} such an algebra is isomorphic to $(\rea^{n+1}, \delta)$ with the multiplication $e_{i}\mlt e_{j} = \sqrt{n+1} \delta_{ij} e_{i}$. Thus the unitalization of $(\alg, \mlt,  \tfrac{\ka}{n-1}h)$ is unique up to isomorphism, and by Lemma \ref{unitalizationuniquelemma}, $(\alg, \mlt,  \tfrac{\ka}{n-1}h)$ is unique up to isomorphism.
\end{proof}

In the $n = 3$ case, the multiplication of such an algebra with Einstein constant $2c^{2}$ is given by
\begin{align}\label{n3alg}
\begin{split}
&e_{1}\mlt e_{1} = e_{2}\mlt e_{2} = e_{3}\mlt e_{3} = 0, \qquad e_{1}\mlt e_{2} = ce_{3}, \qquad e_{2}\mlt e_{3} = ce_{1}, \qquad e_{3}\mlt e_{1} = ce_{2},
\end{split}
\end{align}
corresponding to the cubic polynomial $P(x) = cx_{1}x_{2}x_{3}$. Note that this algebra is not power associative; e.g. for $x = e_{2} + e_{3}$, $(xx)(xx) = 0$, while $((xx)x)x = 4c^{3}e_{1}$. The unitalization of the $c = 1$ case of this algebra is the four-dimensional unital associative commutative Codazzi algebra with multiplication determined by the polynomial $P = \tfrac{1}{6}x_{4}^{3} + \tfrac{1}{2}x_{4}(x_{1}^{2} + x_{2}^{2} + x_{3}^{2}) + x_{1}x_{2}x_{3}$. The $f_{i}$ defined by $4f_{1} = e_{1} - e_{2} - e_{3} + e_{4}$, $4f_{2} = -e_{1} + e_{2} - e_{3} + e_{4}$, $4f_{3} = -e_{1}-e_{2} + e_{3} + e_{4}$, $4f_{4} = e_{1}+e_{2}+e_{3}+e_{4}$ are orthogonal idempotents of square-norm $1/4$. The coordinates $y_{i}$ defined by $y = 2\sum_{i = 1}^{4}y_{i}f_{i}$, are related to the coordinates $x_{i}$ by $2y_{1} = x_{1} - x_{2} - x_{3} + x_{4}$, etc., and in these coordinates the polynomial $P$ has the form $\tfrac{1}{3}(y_{1}^{3} + y_{2}^{3} + y_{3}^{3} + y_{4}^{3})$, as predicted by Lemma \ref{flatalgebralemma}. That the non power associative algebra \eqref{n3alg} is really a slice of an extremely simple algebra, was not evident without the unitalization construction coupled with Lemma \ref{flatalgebralemma}.

\section{The Einstein AH structure on a mean curvature zero Lagrangian submanifold of a para-K\"ahler space form}\label{lipkcurvaturesection}

\setcounter{subsubsection}{0}

In this section it is shown that there is induced on a mean curvature zero spacelike Lagrangian submanifold of a para-K\"ahler manifold of constant para-holmorphic sectional curvature an exact Einstein AH structure. 

\subsubsection{}
If $V$ and $H$ are transverse subbundles of $TN$ the unique idempotent endomorphism $P_{HV}\in \Ga(\eno(TN))$ having kernel $H$ and image $V$ is called \textbf{projection onto $V$ along $H$}. There holds $\Id = P_{HV} + P_{VH}$. A splitting $TN = V \oplus H$ is \textbf{ordered} if there is distinguished an ordering of the pair $(V, H)$. An ordered splitting is identified with the involutive endomorphism $A \defeq P_{HV}- P_{VH}$. %

Given an ordered splitting $TN = T^{+} \oplus T^{-1}$, a completely symmetric or completely anti-symmetric contravariant $(p+q)$ tensor has \textbf{type} $(p, q)$ if it is contained in the span of the monomials formed by the images of tensor products of $p$ sections of $T^{+}$ and $q$ sections of $T^{-}$.

\subsubsection{}
A section $A$ of the bundle $\eno(TM)$ of fiberwise endomorphisms of $TM$ is an \textbf{almost para-complex} structure if $A_{I}\,^{P}A_{P}\,^{J} = \delta_{I}\,^{J}$ and its fiberwise $\pm 1$ eigensubspaces $T^{\pm}$ have the same constant rank, and a \textbf{para-complex} structure if moreover it is integrable in the sense that there vanishes its Nijenhuis tensor. The terminology is justified by the equivalence of $(2)$ and $(6)$ in Theorem \ref{lagrangiansplittingtheorem} below. A \textbf{para-Hermitian structure} is a triple $(G, A, \Omega)$ comprising a pseudo-Riemannian metric $G_{IJ}$, a para-complex structure, and an almost symplectic form $\Omega_{IJ}$, which are \textbf{compatible} in the sense that there hold
\begin{align}\label{aphcompatibility}
&A_{I}\,^{P}A_{P}\,^{J} = \delta_{I}\,^{J}, &\Omega_{IJ} = A_{I}\,^{P}G_{PJ},& &A_{I}\,^{P}\Omega_{PJ} =  G_{IJ},& &\Omega_{IP}G^{JP} = A_{I}\,^{J}. 
\end{align}
 A para-Hermitian structure is \textbf{para-K\"ahler} if $d\Omega = 0$. Any two of $G$, $A$, and $\Omega$ determine the third by \eqref{aphcompatibility}. The metric $G$ must have split signature and equals its $(1,1)$ part with respect to the splitting $T = T^{+}\oplus T^{-}$ determined by $A$. The subbundles $T^{\pm}$ are Lagrangian and $G$-null, and this Lagrangian splitting completely determines the para-Hermitian structure. For background on para-Hermitian structures see the papers of S. Kaneyuki and collaborators, e.g. \cite{Kaneyuki-classification}.

\begin{theorem}\label{lagrangiansplittingtheorem}
There is associated to a para-Hermitian structure $(G, A, \Omega)$ a unique affine connection $\hnabla$ having torsion $\htau$ such that $\hnabla \Omega = 0$, $\hnabla A= 0$, $\hnabla  G = 0$, and the torsion is pure in the sense that there vanishes its $(1,1)$ part $\htau^{(1,1)} = 0$. Let $D$ be the Levi-Civita connection of $G$. The following are equivalent: $(1).$ $DA = 0$;
$(2).$ $A$ is integrable; $(3).$ $D = \hnabla$; $(4).$ $D \Omega = 0$; $(5).$ $\hnabla$ is torsion free; $(6).$ $T^{+}$ and $T^{-}$ are integrable. Hence $(G, \Omega, A)$ is para-K\"ahler if and only if there holds one and hence all of these conditions.
\end{theorem}
The connection $\hnabla$ of Theorem \ref{lagrangiansplittingtheorem} is the \textbf{canonical connection}. It is the para-Hermitian analogue of the canonical or Chern connection associated to a Hermitian structure. A version of Theorem \ref{lagrangiansplittingtheorem} in terms of Lagrangian splittings is well known and was proved by H. Hess in \cite{Hess}, although some version of Theorem \ref{lagrangiansplittingtheorem} can be found in the much older paper of P. Libermann, \cite{Libermann}.

For $X \in \Ga(TM)$ let $X^{\sflat} = i(X)\Omega$, and for $\al \in \Ga(\ctm)$ define $\al^{\ssharp}$ by $i(\al^{\ssharp})\Omega = -\al$. The canonical connection is defined for $X, Y \in \Ga(TN)$ by
\begin{align}
\hnabla_{X}Y = [X_{+}, Y_{-}]_{-} + [X_{-}, Y_{+}]_{+} - (\lie_{X_{+}}Y_{+}^{\sflat})^{\ssharp}_{+}- (\lie_{X_{-}}Y_{-}^{\sflat})^{\ssharp}_{-},
\end{align}
in which $X_{\pm}$ denote the projections of $X$ onto $T^{+}$ along $T{-}$ and vice-versa. The torsion can be expressed explicitly in terms of the Nijenhuis tensor of $A$, but the explicit expression will not be needed here.

\subsubsection{}
The curvature of the canonical connection will be written $\hat{R}_{IJK}\,^{L}$, and the tensors derived from it will similarly be decorated with hats. The \textbf{Ricci form $\hrho$} of a (para)-K\"ahler structure $(G, A, \Omega)$ is defined by $\hrho_{IJ} = A_{I}\,^{P}\hat{R}_{PJ}$. The proof of Lemma \ref{ricciformlemma} in the para-K\"ahler case is formally analogous to the proof in the K\"ahler case.
\begin{lemma}\label{ricciformlemma}
The Ricci form of a para-K\"ahler structure is a closed $(1,1)$-form, and the curvature of the connection induced on the bundle of $(n,0)$-forms by its canonical connection $\hnabla$ is $-\hrho$.
\end{lemma}
A $2n$-dimensional para-K\"ahler structure is \textbf{Einstein} if $G$ is Einstein, or, what is the same, $\hrho = c\Omega$. Since both $\hrho$ and $\Omega$ are closed, this can be the case only if $dc \wedge \Omega = 0$, and because wedge product with $\Omega$ is injective on one-forms, this means $c$ must be locally constant. 

The duality pairing gives rise to a tautological flat para-K\"ahler structure on the direct sum of a vector space with its dual; see \cite{Hitchin-speciallagrangian} for the definition.

\subsubsection{}
While compact examples of para-K\"ahler manifolds are not abundant, there are many such structures arising on cotangent bundles. A \tbf{para-Hermitian symmetric space} is an affine symmetric space $G/H$ with an almost para-Hermitian structure such that the symmetries act as automorphisms of the almost para-Hermitian structure. The almost para-Hermitian structure of a para-Hermitian symmetric space is necessarily para-K\"ahler, and $G$ acts by para-K\"ahler automorphisms. This and other basic facts about these spaces are due to S. Kaneyuki and collaborators in a series of papers, from which there results 

\begin{theorem}[\cite{Hou-Deng-Kaneyuki-Nishiyama}, \cite{Kaneyuki-compactification}, \cite{Kaneyuki-Kozai}]
Let $G$ be a connected, semisimple Lie group and $H \subset G$ a closed subgroup. The following are equivalent
\begin{enumerate}
\item $G/H$ is a homogeneous para-K\"ahler manifold.
\item $H$ is an open subgroup of a Levi subgroup of a parabolic subgroup $P$ of $G$ having abelian nilradical.
\item $G/H$ is a $G$-equivariant covering space of the adjoint orbit of a hyperbolic semisimple element of $\g$.
\end{enumerate}
Up to covering para-Hermitian symmetric spaces of semisimple Lie groups are in bijection with semsimimple graded Lie algebras $\g = \g_{-1} \oplus \g_{0}\oplus \g_{1}$ in such a way that $\g = \mathfrak{lie}(G)$ and $\g_{0} = \mathfrak{lie}(H)$. 
\end{theorem}
$G/H$ is diffeomorphic to the cotangent bundle of $G/P$. The symplectic form on $G/H$ is the pullback of the Kostant-Kirillov symplectic form pulled back from the coadjoint orbit. The simplest example is $G = SL(2, \rea)$ with $H = \reat$ as the diagonal. The symmetric space $SL(2, \rea)/\reat$ is the one-sheeted hyperboloid and the para-K\"ahler structure is given by the two rulings. 

The para-Hermitian symmetric spaces are Einstein. The proof is similar to the proof that Hermitian symmetric spaces are Einstein.

\subsubsection{}
A (para)-K\"ahler manifold $(N, G, A)$ has \textbf{constant (para)-holomorphic sectional curvature $4c$} if its curvature has the form
\begin{align}\label{chsc}
\hat{R}_{IJK}\,^{L} 
= 2c\left( \delta_{[I}\,^{L}G_{J]K}- A_{[I}\,^{L}\Omega_{J]K} + \Omega_{IJ}A_{K}\,^{L}\right).
\end{align}
 The para-Hermitian symmetric space structure on the coadjoint orbit of the element 
\begin{align}
\begin{pmatrix}\tfrac{n}{n+1} & 0 \\ 0 & -\tfrac{1}{n+1}\delta_{i}\,^{j} \end{pmatrix}
\end{align}
of $\sll(n+1, \rea)$ has constant non-zero para-holomorphic sectional curvature. This orbit is identified with either of the connected components $\{([u], [\mu]) \in \projp(\ste) \times \projp(\sted): \pm \mu(u) < 0\}$ of the complement of the tautological incidence correspondence, and the para-K\"ahler structures on it having constant para-holomorphic sectional curvature can be obtained from the flat para-K\"ahler structure on $\ste \oplus \sted$ via an analogue of the construction of the Fubini-Study metric on $\proj^{n}(\com)$ by reduction of the flat K\"ahler structure on complex Euclidean space via the Hopf fibration.

\subsubsection{}
The map sending $X \to X^{\sflat}$ identifies the normal bundle of the Lagrangian immersion $i:M \to N$ with the cotangent bundle $\ctm$, and via this identification the second fundamental form $\Pi^{\nabla}(i)_{ij}\,^{Q}$ is identified with the covariant $3$-tensor $\Pi_{ijk} \defeq \Pi^{\nabla}(i)_{ij}\,^{Q}\Omega_{Qk}$ on $M$ (that is, $\Pi(X, Y, Z) = \Omega(\nabla_{X}Ti(Y), Ti(Z))$). Here lowercase (resp. uppercase) Latin abstract indices indicate tensors on $M$ (resp. $N$). Because $\nabla_{I}\Omega_{JK} = 0$ there holds $\Pi_{i[jk]} = 0$, while because $i$ is Lagrangian there holds $2\Pi_{[ij]k} = i^{\ast}(\tau)_{ijk}$, so that $\Pi_{ijk}$ is completely symmetric if and only if $\nabla$ is torsion-free. In a slight abuse of terminology, the tensor $\Pi_{ijk}$ will be called the \tbfs{second fundamental form}{lagrangian immersion} of the Lagrangian immersion $i$.

\subsubsection{}\label{lagrangianimmersionsection}
An immersion $i:M \to N$ into a para-Hermitian manifold $(N, \Omega, A, G)$ is \textbf{non-degenerate} if the induced tensor $h \defeq i^{\ast}(G)$ is non-degenerate. A non-degenerate immersion for which the induced metric $h$ is Riemannian is called \textbf{positive definite} (\textbf{spacelike} is an alternative terminology). 
\begin{lemma}
A Lagrangian immersion into a para-Hermitian manifold is non-degenerate if and only if it is transverse to each of the subbundles of the associated Lagrangian splitting. 
\end{lemma}

\subsubsection{}
Let $(N, G, A, \Omega)$ be a $2n$-dimensional para-K\"ahler manifold with canonical (Levi-Civita) connection $\hnabla$ and let $i:M \to N$ be a non-degenerate Lagrangian immersion with second fundamental form $\Pi_{ijk} = \Pi_{(ijk)}$. Let $h_{ij} = i^{\ast}(G)_{ij}$ be the metric induced on $M$ and let $D$ be its Levi-Civita connection. In the rest of this section indices of tensors defined on $M$ are raised and lowered using $h_{ij}$ and $h^{ij}$. Define the \textbf{mean curvature one-form} $H_{i}$ by $H_{i} = \Pi_{ip}\,^{p} = h^{jk}\Pi_{ijk}$. The usual mean curvature vector is the image under $A$ of $H^{i}$. For para-K\"ahler structures the condition $H_{i} = 0$ will be called \textbf{mean curvature zero}; in fact when critical for the induced volume such immersions are locally volume \textit{maximizing}. 

The K\"ahler analogue of Lemma \ref{dazordlemma} is well known and is due to P. Dazord, \cite{Dazord}. The proof in the para-K\"ahler case is formally the same.
\begin{lemma}\label{dazordlemma}
Let $i:M \to N$ be a non-degenerate Lagrangian immersion into the para-K\"ahler manifold $(N, \Omega, A, G)$ having Ricci form $\hat{\rho}$. Then $(dH)_{ij} = 2D_{[i}H_{j]} = i^{\ast}(\hat{\rho})_{ij}$. In particular, if $(N, \Omega, A, G)$ is para-K\"ahler Einstein, then the mean curvature one-form $H_{i}$ is closed. 
\end{lemma}

\subsubsection{}\label{lagpathsection}
Let $\lag$ be the Grassmannian of Lagrangian subspaces of the symplectic vector space $(\ste, \Omega)$. For transverse $L, K \in \lag$, define $P_{KL} \in \eno(\ste)$ to be the projection onto $L$ along $K$ (the unique idempotent linear operator with kernel $K$ and image $L$). Suppose $L, K_{1}, K_{2} \in \lag$ are pairwise transverse. From the identity 
\begin{align*}
P_{K_{1}L}P_{K_{2}L} + P_{K_{2}L}P_{K_{1}L} = (\Id - P_{LK_{1}})P_{K_{2}L} + (\Id - P_{LK_{2}})P_{K_{1}L} = P_{K_{2}L} + P_{K_{1}L},
\end{align*}
it follows that $P(t) \defeq tP_{K_{1}L} + (1-t)P_{K_{2}L} \in \eno(\ste)$ satisfies $P(t)^{2} = P(t)$, so defines a one-parameter family of projections. Because $\im P(t)$ is contained in $L$, it is an isotropic subspace. The complementary projection $\Id - P(t)$ has the form $ tP_{LK_{1}} + (1-t)P_{LK_{2}}$. If $K$ and $L$ are transverse and Lagrangian the operators $P_{KL}$ and $P_{LK}$ are symplectic adjoints in the sense that there holds $\Omega(P_{KL}x, y) = \Omega(x, P_{LK}y)$ for all $x, y \in \ste)$. Using this and the observation $P_{K_{1}L}P_{LK_{2}} + P_{K_{2}L}P_{LK_{1}} = 0$, it follows that for all $x, y \in \ste$,
\begin{align*}
\Omega((\Id - P(t))x, (\Id - P(t))y)  &= t(1-t)\left(\Omega(P_{LK_{1}}x, P_{LK_{2}}y) + \Omega(P_{LK_{2}}x, P_{LK_{1}}y)\right)\\
& = t(1-t)\Omega(x, (P_{K_{1}L}P_{LK_{2}} + P_{K_{2}L}P_{LK_{1}})y) = 0,
\end{align*}
which shows that $\ker P(t) = \im (\Id - P(t))$ is isotropic.
Since $\ste = \im P(t) \oplus \im (\Id - P(t)) = \im P(t) \oplus \ker P(t)$, and an isotropic subspace has dimension at most half the dimension of $\ste$, it must be that $\im P(t)$ and $\ker P(t)$ are Lagrangian.  Because $P(t)(L) = L$, the one-parameter family $K^{t}_{L} \defeq \ker P(t) \in \lag$ is everywhere transverse to $L$ and has $K^{0}_{L} = K_{2}$, $K^{1}_{L} = K_{1}$.

\subsubsection{}
Let $\hnabla$ be an affine connection on $N$ with torsion $\htau$, and let $V \subset TN$ be a subbundle. If $i:M \to N$ is an immersion transverse to $V$ (that is $i^{\ast}(TN) = i^{\ast}(V) \oplus TM$) there is induced on $M$ a connection $\tnabla$ as in section \ref{inducedconnectionsection}. Explicitly $Ti(\tnabla_{X}Y) = P_{V L} \hnabla_{X}Ti(Y)$ in which $\hnabla$ is used also for the connection induced on $i^{\ast}(TN)$, $L$ is $Ti(TM)$ viewed as a subbundle of $i^{\ast}(TN)$, and $P_{VL}$ is viewed as a section of $i^{\ast}(\eno(TN))$. 

Let $(G, A, \Omega)$ be a para-Hermitian structure on $N$ with canonical connection $\hnabla$, and let $i:M \to N$ be a non-degenerate Lagrangian immersion. Write $L = Ti(TM)$, and define $P(t)$ as in section \ref{lagpathsection}, taking for $K_{1}$ and $K_{2}$ the Lagrangian subbundles $i^{\ast}(T^{-})$ and $i^{\ast}(T^{+})$ of the pullback $i^{\ast}(TN)$. Then $P(t) \defeq tP_{T^{-}L} + (1-t)P_{T^{+}L}$ is a projection operator and $V^{t}\defeq \ker P(t)$ is a one-parameter family of Lagrangian subbundles of $i^{\ast}(TN)$ transverse to $L$ such that $V^{0} = T^{+}$ and $V^{1} = T^{-}$. Let $\pktnabla$ be the affine connection induced on $M$ by $\hnabla$ and $V^{t}$, and let $\brt{\tau}$ be its torsion, which satisfies $Ti(\brt{\tau}(X, Y)) = P(t)\htau(Ti(X), Ti(Y))$. Hence if $(G, A, \Omega)$ is para-K\"ahler then $\pktnabla$ is torsion-free for all $t$. From $P(t)A = tP_{T^{-}L} - (1-t)P_{T^{+}L}$ it follows that if $l \in L$ then $P(1/2)A(l) = (2t-1)l$. Hence $L \subset \ker P(1/2)A$, and since $\dim  \ker P(1/2)A = \dim \ker P(1/2) = \dim L$, there must hold $L = \ker P(1/2)A$, and so also $V^{1/2} = \ker P(1/2) = A(L)$. Since $A(L)$ is the $G$ orthocomplement of $L$, this means that $\tnabla^{1/2}$ is the Levi-Civita connection of the pseudo-Riemannian metric $i^{\ast}(G)$ on $M$, as is also evident from Theorem \ref{inducedcodazzitheorem}.

\begin{theorem}\label{inducedcodazzitheorem}
Let $(G, A, \Omega)$ be a para-Hermitian structure on $N$ and let $i:M \to N$ be a non-degenerate Lagrangian immersion. Let $h = i^{\ast}(G)$ be the induced metric. Let $\Pi_{ijk} = \Pi_{(ijk)}$ be the second fundamental form viewed as a section of $S^{3}(\ctm)$. The connections $\pktnabla$ satisfy 
\begin{align}\label{dhijk}
&\pktnabla_{i}h_{jk} = 2(1-2t)\Pi_{ijk}, &&\pktnabla_{[i}h_{j]k} = 2(1-2t)\Pi_{[ij]k} = (1-2t)i^{\ast}(\htau^{\sflat})_{ijk}.
\end{align}
For all $t$ the connections $\pktnabla$ and $\brtopt{\tnabla}$ are $h$-conjugate. If $(G, A, \Omega)$ is para-K\"ahler then the pencil $(\pktnabla, h)$ interpolates between the conjugate Codazzi structures $(\brtzero{\tnabla}, h)$ and $(\brtone{\tnabla}, h)$. %
\end{theorem}

\begin{proof}
For the sake of readability notation will be abused throughout this proof by there being written $X$ in place of $Ti(X)$ for $X \in \Ga(TM)$. Observe that $(\Id - P(t))A = A  - tP_{T^{-}L} + (1-T)P_{T^{+}L}$, so for $X \in \Ga(TM)$ there holds $(\Id - P(t))A X = A X + (1-2t)X$. Hence
\begin{align*}
&(\pktnabla_{X}h)(Y, Z) = \lie_{X}(G(Y, Z)) + \Omega(P(t)\hnabla_{X}Y, A Z) + \Omega(P(t)\hnabla_{X}Z, A Y)\\
& = \lie_{X}(G(Y, Z)) + \Omega(\hnabla_{X}Y, (\Id - P(t))A Z) + \Omega(\hnabla_{X}Z, (\Id - P(t))A Y)\\
& = (\hnabla_{X}G)(Y, Z)+ (1-2t)\Omega(\hnabla_{X}Y, Z) + (1-2t)\Omega(\hnabla_{X}Z, Y)\\
& = 2(1-2t)\Pi(X, Y, Z).
\end{align*}
Using again $(\Id - P(t))A X = A X + (1-2t)X$ for $X \in \Ga(TM)$,
\begin{align*}
&h(\pktnabla_{X}Y, Z) + h(Y, \brtopt{\tnabla}_{X}Z) = \Omega(A Z, P(t)\hnabla_{X}Y)+ \Omega(A Y, P^{1-t}\hnabla_{X}Z)\\
& = \Omega((\Id - P(t))A Z, \hnabla_{X}Z) + \Omega((\Id - P^{1-t})A Y, \hnabla_{X}Z)\\
& = G(\hnabla_{X}Y, Z) + G(Y, \hnabla_{X}Z) + (2t-1)\left(\Omega(\hnabla_{X}Y, Z) + \Omega(Y, \hnabla_{X}Z)\right)\\
& = X(h(Y, Z)) + X(\Omega(Y, Z)) =  X(h(Y, Z)),
\end{align*}
which shows that $\pktnabla$ and $\brtopt{\tnabla}$ are $h$-conjugate. From \eqref{dhijk} it follows that if $T^{\pm}$ are integrable then $\pktnabla_{[i}h_{j]k} = 0$, so $(\pktnabla, h)$ is a Codazzi structure. 
\end{proof}

\subsubsection{}\label{kisection}
Let $(N, G, A, \Omega)$ be a $2n$-dimensional para-K\"ahler manifold with canonical connection $\hnabla$ and let $i:M \to N$ be a non-degenerate Lagrangian immersion with second fundamental form $\Pi_{ijk} = \Pi_{(ijk)}$ and induced metric $h_{ij} = i^{\ast}(G)_{ij}$. Note that it need not be the case that $\pktnabla$ is the aligned representative of the AH structure $(\pkten, [h])$ because it need not be the case that $\det h$ is $\pktnabla$-parallel. In fact, the aligned representative of $(\pkten, [h])$ is $\pnabla = \pktnabla + \tfrac{4(1-2t)}{n+2}H_{(i}\delta_{j)}\,^{k}$. The Faraday primitive of $(\pkten, [h]) = (\pen, [h])$ corresponding to $h$ is $\tfrac{(2t -1)}{n+2}H_{i}$, and
\begin{align}\label{pkinduced}
\pnabla = D + (2t - 1)L_{ij}\,^{k} + \tfrac{(2t - 1)}{n+2}(h_{ij}H^{k} - 2H_{(i}\delta_{j)}\,^{k}),
\end{align}
in which $D$ is the Levi-Civita connection of $h$ and $L_{ijk} = L_{(ijk)} = \Pi_{ijk} - \tfrac{3}{n+2}H_{(i}h_{jk)}$ is the completely $h$-trace-free part of the second fundamental form. Define the \textbf{induced pencil of AH structures $(\pen, [h])$} to be the family of AH structures $(\pen, [h])$ having cubic torsions $\brt{\bt}_{ij}\,^{k} = (2t -1)L_{ij}\,^{k}$ the aligned representatives $\pnabla$ of which have the forms \eqref{pkinduced}. Write $\nabla$ for the specialization $t = 1$, and define the \textbf{AH structure induced by $i:M \to N$} to be the AH structure $(\en, [h])$ given by specializing $t = 1$ in $\pen$. The Faraday form satisfies $(n+2)F_{ij} = 2 (1-2t) D_{[i}H_{j]}$, and so by Lemma \ref{dazordlemma} it is a multiple of the pullback via $i$ of the Ricci form. Hence if the given para-K\"ahler structure is Einstein then $(\pen, [h])$ is closed.

Define $|L|_{h}^{2} \defeq L_{ijk}L_{abc}h^{ia}h^{jb}h^{kc}$ and define similarly $|\Pi|_{h}^{2}$. Then $|\Pi|_{h}^{2} = |L|^{2}_{h} + \tfrac{3}{n+2}|H|_{h}^{2}$. Define also $L_{ij} \defeq L_{ip}\,^{q}L_{jq}\,^{p}$. There hold
\begin{align}
\label{kic5} 
\begin{split}
2\Pi_{p[i}\,^{l}\Pi_{j]k}\,^{p}& = 2L_{p[i}\,^{l}L_{j]k}\,^{p} + \tfrac{2}{n+2}H^{p}\left(L_{p[i}\,^{l}h_{j]k} + \delta_{[i}\,^{l}L_{j]kp} \right) \\ \notag &+ \tfrac{2}{(n+2)^{2}}\left(H_{[i}h_{j]k}H^{l} +\delta_{[i}\,^{l}h_{j]k}|H|_{h}^{2} + \delta_{[i}\,^{l}H_{j]}H_{k}\right),\\
 D_{i}\Pi_{jk}\,^{l} & = D_{i}L_{jk}\,^{l} + \tfrac{1}{n+2}\left(h_{jk}D_{i}H^{l} + \delta_{k}\,^{l}D_{i}H_{j} + \delta_{j}\,^{l}D_{i}H_{k} \right),\\
 2D_{[p}\Pi_{i]j}\,^{p} &  = D_{p}L_{ij}\,^{p} - \tfrac{n}{n+2}\left(D_{(i}H_{j)} - \tfrac{1}{n}h_{ij}D_{p}H^{p} \right) - D_{[i}H_{j]}.
\end{split}
\end{align}
Note that the last line simplifies when $H^{i}$ is a $[h]$-conformal Killing vector field.

Let $P \in \Ga(\eno(i^{\ast}(TN)))$ be projection onto $Ti(TM)$ along $ATi(TM)$. For $X, Y \in \Ga(TM)$ there hold $G(PX, Y) = 0$ and
\begin{align}
&D_{X}Y = P(\hnabla_{X}Y),& &\hnabla_{X}Y = D_{X}Y + A\Pis(X, Y),& &\hnabla_{X}AY = AD_{X}Y + \Pis(X, Y),
\end{align}
in which $\Pis(X, Y)$ is defined by $h(\Pis(X, Y), Z) = \Pi(X, Y, Z)$, and notation has been abused by writing $X$ and $Y$ where what is meant are extensions of $Ti(X)$ and $Ti(Y)$ defined near $i(M)$. 

Let $\hat{R}_{IJK}\,^{L}$ be the curvature of $\hnabla$ and $\sR_{ijk}\,^{l}$ be the curvature of $D$. By definition of $P$ it makes sense to define tensors $\sN_{ijk}\,^{l}$ and $\sT_{ijk}\,^{l}$ on $M$ by $\sT_{ijk}\,^{l} \defeq (P\hat{R})_{ijk}\,^{l}$ and $\sN_{ijk}\,^{l} \defeq (PA\hat{R})_{ijk}\,^{l}$, and there hold
\begin{align}\label{kic}
&\sR_{ijk}\,^{l}  = \sT_{ijk}\,^{l} - 2 \Pi_{p[i}\,^{l}\Pi_{j]k}\,^{p},& %
&\sN_{ijk}\,^{l} = 2D_{[i}\Pi_{j]k}\,^{l}.
\end{align}
Let $\sT_{ij} \defeq \sT_{pij}\,^{p}$, $\sN_{ij} \defeq \sN_{pij}\,^{p}$, and $\sT \defeq \sT_{p}\,^{p}$, and note $\sN_{p}\,^{p} = 0$. Plugging \eqref{kic5} into \eqref{kic} and simplifying yields
\begin{align}
\label{gic2}&\sR_{ijkl} = \sT_{ijkl} + L_{ijkl} - \tfrac{2}{n+2}H^{p}\left(h_{l[i}L_{j]kp} + h_{k[j}L_{i]lp} \right)  \\ \notag &\qquad \qquad - \tfrac{2}{(n+2)^{2}}\left( H_{[i}h_{j]k}H_{l} + h_{l[i}h_{j]k}|H|_{h}^{2} + h_{l[i}H_{j]}H_{k}\right),& \\
\label{gic2b}&\sR_{ij} = \sT_{ij} +  L_{ij} - \tfrac{n }{(n+2)^{2}}|H|_{h}^{2}h_{ij}  + \tfrac{(2-n)}{(n+2)^{2}}H_{i}H_{j} + \tfrac{(2-n)}{n+2}H^{p}L_{ijp} ,\\
\label{gic3}&\sR = \sT +|L|^{2}_{h} + \tfrac{(1-n)}{n+2}|H|_{h}^{2},\\
\label{qr}& \qR(L) = \qT(L) + (L_{ijkl}L^{ijkl} + L^{ij}L_{ij} - H^{p}L^{ij}L_{ijp} - \tfrac{1}{n+2} H^{i}H^{j}L_{ij} - \tfrac{1}{n+2}|H|^{2}|L|^{2}),\\
\label{nic1}&\div(L)_{ij} = D_{p}L_{ij}\,^{p} =  \sN_{ij} + \tfrac{n}{n+2}\left(D_{(i}H_{j)} - \tfrac{1}{n}h_{ij}D_{p}H^{p}\right) +  D_{[i}H_{j]},\\
\label{nic0}&2D_{[i}L_{j]kl}  = \sN_{ijkl} + \tfrac{2}{n+2}\left(h_{k[i}D_{j]}H_{l} + h_{l[i}D_{j]}H_{k} - h_{kl}D_{[i}H_{j]} \right),\\
\label{nic2}&\klie(L)_{ijkl} = \tfrac{1}{2}\left(\sN_{ijkl} - \tfrac{1}{2n}h_{k[i}\sN_{j]l} - \tfrac{1}{2n}h_{l[i}\sN_{j]k}\right)\\
\notag & \qquad \qquad - \tfrac{1}{2n(n+2)}\left(2h_{k[i}dH_{j]l} +2 h_{l[i}dH_{j]k} + nh_{kl}dH_{ij}\right),
\end{align}
in which $dH_{ij} = 2D_{[i}H_{j]}$ and $L_{ijkl} \defeq 2L_{k[i}\,^{p}L_{j]lp}$.

Recall that $L_{ij} \defeq L_{ip}\,^{q}L_{jq}\,^{p}$, so $h^{ij}L_{ij} = |L|_{h}^{2}$. The relations between the curvature $\sR_{ijk}\,^{l}$ of $D$ and the curvature $\brt{R}_{ijk}\,^{l}$ of $\pnabla$ follow from substituting $\tfrac{1}{2}\bt_{ij}\,^{k} = (1-2t)L_{ij}\,^{k}$ and $\ga_{i} = \tfrac{2t-1}{n+2}H_{i}$ into \eqref{confcurvijkl}-\eqref{confricfree}, and using \eqref{gic2}-\eqref{nic2}. Writing $\la = (2t - 1)$ and $\mu = |\det h|^{1/n}$, 
\begin{align}
\brt{T}_{ij} & = \sR_{ij}   - \la^{2}L_{ij} + \tfrac{(n-2)\la}{n+2}D_{(i}H_{j)} + \tfrac{\la}{n+2}D^{p}H_{p} h_{ij}  + \tfrac{(n-2)\la^{2}}{(n+2)^{2}}\left(H_{i}H_{j} - |H|_{h}^{2}h_{ij}\right),\\
\notag & =  \sT_{ij} + (1-\la^{2}) \left(L_{ij}  + \tfrac{2-n}{(n+2)^{2}}H_{i}H_{j}\right) - \tfrac{(n +(n-2)\la^{2})}{(n+2)^{2}}|H|_{h}^{2}h_{ij} + \tfrac{(2-n)}{n+2}H^{p}L_{ijp} \\
\notag & + \tfrac{(n-2)\la}{n+2} D_{(i}H_{j)}  + \tfrac{\la}{n+2}D^{p}H_{p} h_{ij},\\
\label{mrtscal} \mu \,\brt{R} & \defeq h^{ij}\brt{R}_{ij}  = \sT + (1 - \la^{2})|L|^{2}_{h} + \tfrac{1-n}{n+2}(1 +\tfrac{n-2}{n+2} \la^{2})|H|_{h}^{2}   + \tfrac{2(n-1)\la}{n+2}D^{p}H_{p},\\
\label{mrmrtr}  \mr{\brt{T}}_{ij} & = \mr{\sT}_{ij} + (1-\la^{2})\mr{L}_{ij}+ (n-2)\left( \tfrac{(\la^{2} - 1)}{(n+2)^{2}}\mr{H\tensor H}_{ij} - \tfrac{1}{n+2}H^{p}L_{ijp} + \tfrac{\la}{n+2}\mr{DH}_{ij} \right) ,\\
\label{mrtu}\brt{\uf}_{ij} & = \la \left(D_{p}L_{ij}\,^{p} - \tfrac{n\la}{n+2}H_{p}L_{ij}\,^{p}\right),\\
\notag & = \la \left(\tfrac{1}{2}\sN_{ij} + \tfrac{n}{2(n+2)}\left(D_{(i}H_{j)} - \tfrac{1}{n}D^{p}H_{p}h_{ij}\right) + \tfrac{1}{2}D_{[i}H_{j]}  - \tfrac{n\la}{n+2}H_{p}L_{ij}\,^{p} \right).
\end{align}
Using $(n+2)|\det h|^{-1}\pnabla_{i}|\det h| = 2n\la H_{i}$ gives 
\begin{align}\label{cke}
\pnabla_{i}\brt{R} &=  D_{i}\brt{R} + \tfrac{2\la}{n+2}H_{i}R.
\end{align}

\begin{theorem}\label{mczerotheorem}
If a positive definite Lagrangian immersion $i:M \to N$ of $M$ into an $2n$-dimensional para-K\"ahler manifold $(N, G, A)$ of constant holomorphic sectional curvature $4c$ has zero mean curvature then the induced AH structure $(\en, [h])$ (corresponding to $t = 1$) on $M$ is exact Einstein with the induced metric $h$ as a distinguished metric, and there hold $R = 2c|\det h|^{-1/n}$, $T_{ijk}\,^{l} = 2c \delta_{[i}\,^{l}h_{j]k}$, and $E_{ijkl} \equiv 0$. In particular, if $n > 2$ then $(\en, [h])$ is projectively flat and conjugate projectively flat. 
\end{theorem}
\begin{proof}
Since $(N, G, A)$ is Einstein, $D_{[i}H_{j]} = 0$. From \eqref{chsc} there follow 
\begin{align}
&\sN_{ijk}\,^{l} = 0, &&\sT_{ijk}\,^{l} = 2c\delta_{[i}\,^{l}h_{j]k},& &\sT_{ij} = c(n-1)h_{ij}, && \sT = cn(n-1). 
\end{align}From \eqref{nic1} and \eqref{nic2} there follow
\begin{align}\label{knic}
&(n+2)\div(L)_{ij} =n\mr{\lie_{H^{\sharp}}h}_{ij} ,&&\klie(L)_{ijkl} = 0.
\end{align}
The vanishing of $\mr{T}_{ij}$ and $E_{ij}$ is immediate from \eqref{mrmrtr} and \eqref{mrtu}. In this case $(\en, [h])$ is exact and \eqref{mrtscal} shows $\mu R = 2c$, so that $R$ is parallel and $(\en, [h])$ is Einstein. The claim about the curvature follows from \eqref{confcurvijkl} and \eqref{gic2}. Since $\klie(L) = 0$ there holds $E_{ijkl} = 0$ by \eqref{kliebt}.
\end{proof}

Theorem \ref{mczerotheorem} suggests a relation between mean curvature zero submanifolds of a para-K\"ahler manifold of constant para-holomorphic sectional curvature and affine hyperspheres. In fact, as will be explained elsewhere, there is a kind of local equivalence between the two.

\section{Analogy with conformal structures}\label{codprojsection}
It has been claimed that Codazzi projective structures should be viewed as analogues of conformal structures; in the analogy subordinate AH structures play the role of representative metrics. This final section is intended to give a bit more evidence for this claim, and to propose that many of the standard constructions and problems of conformal geometry, e.g. the ambient metric construction or the Yamabe problem, should have analogues/generalizations in this context.

Throughout this section it is convenient to work with c-weights (see section \ref{cweightsection}).

\subsection{M\"obius structures, the Codazzi Laplacian, and an analogue of the Yamabe problem}\label{mobiussection}
Recall from \cite{Calderbank-mobius} that a \textbf{M\"obius structure} on a conformal manifold $(M, [h])$ is a smooth second order linear differential operator from $\cmf^{1}$ to covariant $\cmf^{1}$-valued symmetric $[h]$-trace free tensors which is equal modulo zeroth order linear operators to the symmetric trace-free part of the Hessian determined by some covariant derivative on $\cmf^{1}$. In dimensions greater than $2$ a conformal structure determines a canonical M\"obius structure; namely, the operator obtained from a torsion-free conformal connection by modifying the trace-free Hessian by subtracting the trace-free Schouten tensor does not depend on the choice of torsion-free conformal connection. This is a special case of the more general occurrence that in dimension $n > 2$ a Codazzi projective structure determines a canonical M\"obius structure. This is described next.

\subsubsection{}
Let $\nabla$ be the aligned representative of an AH structure generating a given Codazzi projective structure and for $u \in \Ga(\cmf^{\la})$ define associated operators $\lap$, $\clap$, and $\mob$ by $\lap(u)\defeq  \nabla^{p}\nabla_{p}u$ and
\begin{align*}
  & \clap(u) \defeq \lap(u) + \tfrac{2-n}{4(n-1)}Ru,&
&\mob(u)_{ij} \defeq \mr{\nabla\nabla u}_{ij} - \mr{W}_{ij}u.%
\end{align*}

\begin{lemma}
If $\nabla$ and $\tnabla$ are the aligned representatives of AH structures generating the same Codazzi projective structure and 
\begin{align*}
\tilde{\mob}, \mob: \Ga(\cmf^{1}) \to \Ga(\cmf^{1}\tensor S^{2}(\ctm)) \,\, \text{and}\,\, \tilde{\clap}, \clap: \Ga(\cmf^{(2-n)/2}) \to \Ga(\cmf^{-(n+2)/2})
\end{align*}
are the associated operators, then $\tilde{\mob}(u)_{ij} = \mob(u)_{ij}$ and $\tilde{\clap}(u) = \clap(u)$. 
\end{lemma}
\begin{proof}
Let $u \in \Ga(\cmf^{\la})$ and recall that $\tnabla - \nabla$ has the form $2\al_{(i}\delta_{j)}\,^{k} - H_{ij}\al^{k}$%
for some one-form $\al_{i}$. Then, as $\tnabla_{i}u - \nabla_{i} u = \la \al_{i}u$,
\begin{align}\label{lambdahessian}
&\tnabla_{i}\tnabla_{j}u - \nabla_{i}\nabla_{j}u = 2(\la - 1)\al_{(i}\nabla_{j)}u + H_{ij}\al^{p}\nabla_{p}u + \la\left(\nabla_{i}\al_{j} + (\la -2)\al_{i}\al_{j} + H_{ij}\al^{p}\al_{p} \right)u,
\end{align}
so that when $\la = 1$ there holds $\mr{\tnabla\tnabla u}_{ij} = \mr{\nabla \nabla u}_{ij} + \mr{\al}_{ij}$,
in which $\mr{\al}_{ij} = \al_{ij} - \tfrac{1}{n}H_{ij}\al_{p}\,^{p}$ is defined as in \eqref{alijshort}. The claimed independence of $\mob(u)$ from the choice of subordinate AH structure follows from \eqref{wijdiff}. Tracing \eqref{lambdahessian} shows that the operator $\lap:\Ga(\cmf^{\la}) \to \Ga(\cmf^{\la - 2})$ transforms as
\begin{align*}
\tilde{\lap}(u) - \lap(u) = (2\la + n-2)\al^{p}\nabla_{p}u + \la (\nabla_{p}\al^{p} + (n-2 + \la)\al^{p}\al_{p})u,
\end{align*}
so that if $\la = \tfrac{2-n}{2}$ then $\tilde{\lap}(u) - \lap(u) = \tfrac{2-n}{2}\al_{p}\,^{p}$, which with \eqref{cpdiffscalar} shows that the operator $\clap$ from $\cmf^{(2-n)/2} \to \cmf^{-(n+2)/2}$ depends only on the Codazzi projective structure, and not on the choice of representative AH structure.
\end{proof}

\subsubsection{}
Evidently the \textbf{Codazzi Laplacian} $\clap$ and $\mob$ respectively generalize the usual conformally invariant Laplacian and M\"obius operator. The following makes this more precise. Given a conformal structure $[h]$ the usual conformal Laplacian is the operator $\clap_{[h]}(u)$ defined for $u \in \Ga(\cmf^{(2-n)/2})$ and any $h \in [h]$ by $\clap_{[h]}(u) \defeq |\det h|^{1/n}\left(\lap_{h}u + \tfrac{2-n}{4(n-1)}\sR_{h}u\right)$, and the usual conformal M\"obius operator is the operator $\mob_{[h]}(u)_{ij}$ defined for $u \in \Ga(\cmf^{1})$ by $\mob_{[h]}(u)_{ij} \defeq \mr{DDu}_{ij} - \mr{\sW}_{ij}u$, in which $\sW_{ij}$ is the usual conformal Schouten tensor. Straightforward computations using $D_{i}u = \nabla_{i}u + \la \ga_{i} u$ for $u \in \Ga(\cmf^{\la})$ and \eqref{confscal} show that for $u \in \Ga(\cmf^{(2-n)/2})$ there holds
\begin{align}\label{claps}
\clap(u) + \tfrac{2-n}{16(n-1)}\bt u = \clap_{[h]}(u).
\end{align}
Given an AH structure, by \eqref{aijdiff} the first order operator defined on $\cmf^{\la}$ by $-\tfrac{1}{2}\bt_{ij}\,^{p}\nabla_{p}u +\la E_{ij}u$ depends only on the Codazzi projective structure generated by the AH structure. Straightforward computations using \eqref{confric} and \eqref{ddivbt} show that for $u \in \Ga(\cmf^{1})$ there holds 
\begin{align}
\begin{split}
\mob_{[h]}(u)_{ij} - \mob(u)_{ij} & = - \tfrac{1}{2}\bt_{ij}\,^{p}\nabla_{p}u + E_{ij}u = -\tfrac{1}{2}\bt_{ij}\,^{p}D_{p}u + \tfrac{1}{2n}(D_{p}\bt_{ij}\,^{p})u .
\end{split}
\end{align}
Where $u \neq 0$ there holds $-n\bt_{ij}\,^{p}D_{p}u + (D_{p}\bt_{ij}\,^{p})u = u^{n+1}D_{p}(u^{-n}\bt_{ij}\,^{p})$ so that
\begin{align}
2n(\mob_{[h]}(u)_{ij} - \mob(u)_{ij})= u^{n+1}D_{p}(u^{-n}\bt_{ij}\,^{p}).
\end{align}

\subsubsection{Yamabe problem}\label{yamabesection}
Now it will be shown that the usual Yamabe problem admits a generalization in terms of AH structures and Codazzi projective structures. Two standard references for the Yamabe problem are \cite{Lee-Parker} and \cite{Aubin}, although here nothing is used from either.

Let $M$ be a compact $n$-dimensional manifold. Define a functional $\gfunc^{k}(\en, [h])$ on the space of Riemannian signature AH structures by
\begin{align}
\gfunc^{k}(\en, [h]) \defeq \vol_{h}(M)^{(2k-n)/n}\int_{M}R^{k}|\det h|^{(n-2k)/2n} = \vol_{h}(M)^{(2k-n)/n}\int_{M}\uR_{h}^{k}|\det h|^{1/2},
\end{align}
in which $h \in [h]$ is any Gauduchon metric. The volume term is included so that the functional does not depend on the choice of Gauduchon metric. 

By part $(1)$ of Lemma \ref{conformalprojectivediff} it makes sense to say that two AH structures are \textbf{restricted conformal projectively equivalent} if they are conformal projectively equivalent and they induce the same equivalence class of Faraday primitives; that is the one form determining the difference of their aligned representatives is exact. Define a \textbf{restricted Codazzi projective structure} $(\enbr, [h])$ to be an equivalence class of AH structures modulo this notion of equivalence. A generalization of the Yamabe problem is to minimize $\gfunc \defeq \gfunc^{1}$ over a restricted Codazzi projective structure. If $(\enbr, [h])$ contains an exact Weyl structure $(\en, [h])$, then the aligned representative of $\en$ is the Levi-Civita connection of a distinguished metric $h \in [h]$, and $\gfunc(\en, [h])$ is just the volume normalized total scalar curvature of $h$; the restricted conformal projectively equivalent AH structure $(\ben, [h])$ determined by $\si_{i} = df$ is exact Weyl with distinguished metric $e^{2f}h$, and so minimizing $\gfunc$ over such a restricted Codazzi projective structure is the same as minimizing the volume normalized total scalar curvature functional over a conformal class, which is the usual Yamabe problem. 

Let $M$ be compact of dimension $n > 2$ and let $(\ten, [h])$ and $(\en, [h])$ be AH structures generating the same Codazzi projective structure. By Lemma \ref{conformalprojectivediff} there is a one-form $\al_{i}$ such that $\tnabla - \nabla = 2\al_{(i}\delta_{j)}\,^{k} - h_{ij}\al^{k}$. Let $\tilde{h}, h \in [h]$ be representative metrics and write $\tilde{h} = fh$. In what follows raise and lower indices using $h$. It is easily verified that the Faraday primitives are related by $\tilde{\ga}_{i} = \ga_{i} + \si_{i} - \al_{i}$ in which $2\si = d\log{f}$, and the Levi-Civita connections are related by $\tD - D = 2\si_{(i}\delta_{j)}\,^{k} - h_{ij}\si^{k}$. Straightforward computation using \eqref{cpdiffscalar} shows that the unweighted scalar curvatures of $(\ten, [h])$ and $(\en, [h])$, taken with respect to the chosen metrics are related by
\begin{align}\label{twogaud2}
f\tilde{\uR}_{\tilde{h}} = \uR_{h} + 2(n-1)\dad_{h}\al - (n-1)(n-2)|\al|_{h}^{2} + 2(n-1)(n-2)\al^{p}\ga_{p}.
\end{align}
If now it is assumed that $(\ten, [h])$ and $(\en, [h])$ generate the same restricted Codazzi projective structure, so that $\al$ is exact, and there is written $(n-2)\al = 2d\log{u}$ for $0 < u \in \cinf(M)$, then \eqref{twogaud2} becomes
\begin{align}\label{twogaud2b}
\begin{split}
f\tilde{\uR}_{\tilde{h}} &= \uR_{h} + \tfrac{4(1-n)}{n-2}u^{-1}\lap_{h}u  + 4(n-1)(n-2)u^{-1}h^{pq}du_{p}\ga_{q}\\
& = \tfrac{4(1-n)}{n-2}|\det h|^{-(n+2)/4n}u^{-1}\clap(u|\det h|^{(n-2)/4n})\\&\quad - 2(n-1)\dad_{h}\ga - (n-1)(n-2)|\ga|_{h}^{2} + 4(n-1)(n-2)u^{-1}h^{pq}du_{p}\ga_{q},
\end{split}
\end{align}
the second equality following from \eqref{claps}. Imposing some condition on the representative metrics has the effect of relating $f$ and $u$.

If now there is imposed the requirement that $\tilde{h}$ and $h$ be Gauduchon and there is written $f = \phi^{2/(n-2)}$ then $\al$ and $\phi$ are related by 
\begin{align}\label{twogaud}
\dad_{h}\left(d\phi + (n-2)\phi(\ga - \al) \right) = 0
\end{align}
By Theorem \ref{gauduchontheorem}, there is for given $\al$ a positive $\phi \in \cinf(M)$ solving \eqref{twogaud}, and such $\phi$ is determined uniquely up to multiplication by a positive constant. Now suppose $(\en, [h])$ is exact so that, in particular, $\ga = 0$. That $(\ten, [h])$ be in the same restricted Codazzi projective equivalence class as $(\en, [h])$, means there is $q \in \cinf(M)$ such that $\al = dq$, in which case one sees by inspection that $\phi = e^{(n-2)q}$ solves \eqref{twogaud}. It follows from the uniqueness (modulo positive constants) of the solution to \eqref{twogaud} that $(n-2)\al = d\log{\phi} = (n-2)\si$. This shows that $(\ten, [h])$ is also exact. This proves that on a compact manifold of dimension $n > 2$, a restricted Codazzi projective structure contains a subordinate AH structure which is exact if and only if any subordinate AH structure is exact (the same is true for $n = 2$, with a similar proof). Thus it makes sense to speak of an exact restricted Codazzi projective structure. 

In this case the difference tensor of the aligned representatives is generated by the logarithmic differential of the conformal factor of the Gauduchon metric. In particular, since $\tilde{\ga}_{i} = 0 = \ga_{i}$ there holds $\al_{i} = \si_{i}$, and so, writing $f$ and $u$ as before, $f$ is a multiple of $u^{4/(n-2)}$ by a positive constant which may be taken to be $1$. In \eqref{twogaud2b} and with \eqref{claps} this yields
\begin{align}\label{yamabe}
\begin{split}
u^{(n+2)/(n-2)}\tilde{\uR}_{\tilde{h}} & = \tfrac{4(1-n)}{n-2}|\det h|^{-(n+2)/4n}\clap(u|\det h|^{(n-2)/4n}),\\
& = \tfrac{4(1-n)}{(n-2)}\lap_{h}u + (\sR_{h} - \tfrac{1}{4}|L|_{h}^{2})u,
\end{split}
\end{align}
which generalizes the usual Yamabe equation. If $h$ is the Gauduchon metric of an exact AH structure $(\en, [h])$ and $u$ is a positive function solving 
\begin{align}\label{yamabe1}
\begin{split}
u^{(n+2)/(n-2)}\ka & = \tfrac{4(1-n)}{n-2}|\det h|^{-(n+2)/4n}\clap(u|\det h|^{(n-2)/4n}),
\end{split}
\end{align}
for a constant $\ka$, then $\tilde{h} = u^{4/(n-2)}h$ is the Gauduchon metric of an exact AH structure $(\ten, [h])$  generating the same Codazzi projective structure as $(\en, [h])$ and such that the unweighted scalar curvature $\tilde{\uR}_{\tilde{h}}$ in the Gauduchon metric is equal to the constant $\ka$.

Let $||\psi||_{p,h}$ denote the $L^{p}$ norm of $\psi$ with respect to the $h$ volume measure $ d\vol_{h}| = |\det h|^{1/2}$. The restriction to $(\enbr, [h])$ of $\gfunc$ can be expressed in terms of the functional $\gfunc_{(\en, [h])}$ on $\{\psi\in\cinf(M): \psi > 0\}$ defined in terms of the fixed representative $(\en, [h]) \subset (\enbr, [h])$ and obtained from $\gfunc$ by substitution of \eqref{twogaud2} (with the $f$ in \eqref{twogaud2} written as $\phi^{2/(n-2)}$): 
\begin{align}\label{gfunc2}
\begin{split}
\gfunc_{(\en, [h])}(\psi) &= \gfunc((\ten, [h])) \\
&= ||\phi||_{n/(n-2), h}^{-1}\int_{M}\phi\left(\uR_{h} + 2(n-1)\dad_{h}\al - (n-1)(n-2)|\al|_{h}^{2}\right)\, d\vol_{h}, %
\end{split}
\end{align}
in which $(n-2)\al \defeq d\log{\psi}$ is the one-form determining the difference $\tnabla - \nabla$, and $\phi$ is determined by $\psi$ as the solution to \eqref{twogaud}. Although \eqref{twogaud} determines $\phi$ only up to multiplication by a positive constant the functional is well-defined.

Let $\phi(t)$ be a one-parameter family of positive functions such that $\phi(0) = 1$ and $\tfrac{d}{dt}_{|t = 0}\phi(t) = \dot{\phi}$. Given an exact AH structure consider the one-parameter family $(\ven, [h])$ of AH structures within its restricted conformal projective equivalence class and having aligned representative $\vnabla = \nabla + 2\al(t)_{(i}\delta_{j)}\,^{k} - h_{ij}\al(t)^{k}$, in which $(n-2)\al(t) = d\log{\phi(t)}$. Let $R(t)$ be the scalar curvature of $(\ven, [h])$, and let $h(t) = \phi(t)^{2/(n-2)}h$, which is a Gauduchon metric. Using \eqref{gfunc2} there results
\begin{align}\label{yamvar}
\begin{split}
\tfrac{d}{dt}_{|t = 0}\gfunc(\ven, [h]) & = \tfrac{d}{dt}_{|t = 0}\gfunc_{(\en, [h])}(\phi(t)) = \vol_{h}(M)^{(2-n)/n}\left(\int_{M}\uR_{h}\ctr(\dot{\phi}) \,d\vol_{h} \right),
\end{split}
\end{align}
in which $\ctr(\dot{\phi})$ denotes $\dot{\phi}$ minus its mean over $M$ (so $\int_{M}\ctr(\dot{\phi}) = 0$). Though $\phi(t)$ can be replaced by $e^{r(t)}\phi(t)$, the condition $\phi(0) = 1$ forces $r(0) = 0$, and determines $\dot{\phi}$ up to translation (by $\dot{r}(0)$), so $\ctr(\dot{\phi})$ is independent of the choice of $r(t)$ provided $r(0) = 1$. Since on a compact manifold all functions of mean zero can be realized as $\ctr(\psi)$ for some positive function $\psi$, \eqref{yamvar} shows that an exact AH structure $(\en, [h])$ is critical for $\gfunc$ with respect to variations within $(\enbr, [h])$ if and only if $\uR_{h}$ is constant. Such a critical AH structure $(\ten, [h])$ can be found by solving \eqref{yamabe} with $\tilde{\uR}_{\tilde{h}}$ constant.

In the case that $M$ is compact and $(\en, [h])$ is an exact AH structure such that $\inf_{\phi > 0}\gfunc_{(\en, [h])}(\phi)$ is non-positive it seems likely that modifications of the arguments showing the solvability of the usual Yamabe problem in the case of non-positive Yamabe constant will show that it is possible to solve \eqref{yamabe} for $\tilde{\uR}_{\tilde{h}}$ a negative constant. The non-vanishing cubic torsion only makes $\uR_{h}$ less than $\sR_{h}$, and this appears to help rather than to cause problems. 

In the case the infimum $\inf_{\psi > 0}\gfunc_{(\en, [h])}(\psi)$ is positive, then it will be necessary to consider non-exact AH structures. In this case, in addition to the difficulties one expects from the usual Yamabe problem, the dependence of $\phi$ on $\psi$ is less explicit, and this complicates the computations.

\subsection{Generalizing the Bach tensor}\label{bachsection}
For a Weyl structure $(\en, [h])$ the Bach tensor $\bach_{ij}$ is the trace-free symmetric covariant $2$-tensor usually defined to be $\bach_{ij} = \nabla^{p}W_{pij} + W^{pq}W_{pijq} = \nabla^{p}A_{pij} + W^{pq}A_{pijq} = \tfrac{1}{3-n}\nabla^{p}\nabla^{q}W_{pijq} + W^{pq}W_{pijq}$. That $\bach_{ij}$ is symmetric can be checked with multiple uses of the Ricci and Bianchi identities. When $n = 4$ then $\bach_{ij}$ is divergence free and depends only on the Codazzi projective structure underlying $(\en, [h])$. On $4$-manifolds closed Einstein Weyl structures and (anti)-self-dual conformal structures are Bach flat. The usual derivation of the Bach tensor is that the equation $\bach_{ij} = 0$ describes the critical variations of the functional $\int_{M}W_{ijkl}W^{ijkl}$ on Weyl structures on compact $4$-manifolds. Proofs of the basic properties just described can be found in many places, for example \cite{Calderbank-faraday} or \cite{Pedersen-Swann}.

From the work of C. Fefferman and C. R. Graham on their ambient metric construction for conformal metrics (see \cite{Fefferman-Graham}, \cite{Fefferman-Graham-ambient}) there emerges a better understanding of the variational origin of the Bach tensor. The Bach tensor of a metric is the $4$-dimensional specialization of the \textbf{ambient obstruction tensor} $\bach_{ij}$, which is a trace and divergence free symmetric conformally invariant two-tensor associated to a conformal structure on an even-dimensional manifold. In \cite{Graham-Hirachi-obstruction}, Graham and K. Hirachi have shown that $\bach_{ij}$ arises as the gradient of the integral of T. Branson's $Q$-curvature. The $Q$-curvature is an expression built from the derivatives of the curvature tensor of a $2n$-dimensional conformal metric. It generalizes the scalar curvature of a two-dimensional metric in the sense that when the metric is changed conformally the $Q$-curvature transforms by addition of the image of the conformal factor under a conformally invariant differential operator the top order piece of which is the $n$th power of the Laplacian (this is a GJMS operator). Because of this transformation rule, the integral of the $Q$-curvature is a conformal invariant, and in $4$-dimensions it follows from the generalized Gau\ss-Bonnet formula that its integral differs from that of $W_{ijkl}W^{ijkl}$ by a constant multiple of the Euler characteristic.

The result of Graham-Hirachi shows that the Bach tensor should be regarded as the first variation of the integrated $Q$-curvature, and that its description in terms of the variation of $W_{ijkl}W^{ijkl}$ is in some sense an artifact of the generalized Gau\ss\,-Bonnet theorem. This suggests that in the study of Codazzi projective structures one aim should be to generalize the ambient obstruction tensor, rather than the Bach tensor \textit{per se}, and that to do so will require generalizing to AH structures constructions such as the following: the GJMS conformally invariant powers of the Laplacian; the $Q$-curvature; and the Fefferman-Graham ambient construction or Poincar\'e metric construction. Evidently this is an entire research project in its own right, beyond the scope of the present article. Here, instead, there are reported attempts to generalize the Bach tensor to AH structures following an interpretation due to D. Calderbank and T. Diemer. These attempts have not succeeded simply because the computations become rather complicated, but they suggest that the research project just described should be realizable, at least for some subclass of AH structures, e.g. the conservative ones, or those with self-conjugate curvature.

The point of view of Section $9$ of \cite{Calderbank-Diemer} is that the Bach tensor results from applying to the conformal Weyl tensor the (second-order) conformally invariant differential operator formally adjoint to the operator on trace-free symmetric covariant two-tensors given by linearization of the Weyl curvature, which is in \cite{Calderbank-Diemer} called the \textit{Bach operator}. As the Weyl curvature of an AH structure depends on both the connection and the conformal structures, what is meant by its linearization requires some discussion of the nature of their simultaneous variation (because the connection and the metric are linked by the vanishing of the conformal torsion). Nonetheless, it turns out to be straightforward to defined directly such an operator for Codazzi projective structures, although in this more general context there is really a family of such operators, differing by an invariant first order operator; see Lemma \ref{bachtransformlemma} below.

Precisely, there is sought to associate to an AH structure on a $4$-manifold a trace-free, symmetric tensor $\bach_{ij}$ having c-weight $-2$ and the following properties:
\begin{enumerate}
\item $\bach_{ij}$ depends only on the underlying Codazzi projective structure.
\item $\bach_{ij}$ is self-conjugate.
\item $\nabla^{p}\bach_{ip} = 0$.
\end{enumerate}
Such a tensor will be called a \textbf{generalized Bach tensor}. Additionally, it would be desirable that
\begin{enumerate}
\setcounter{enumi}{3}
\item The equations $\bach_{ij} = 0$ are the Euler-Lagrange equations of the integral of an appropriate generalization of the Q-curvature.
\end{enumerate}If $\om_{ij} = \om_{(ij)}$ is trace-free of c-weight $\la$ then $\tnabla^{p}\om_{ip} = \nabla^{p}\om_{ip} + (n-2+\la)\al^{p}\om_{ip}$. It follows that when $n = 4$ the modified divergences $\nabla^{p}\bach_{ip} + \mu \bt_{i}\,^{pq}\bach_{pq}$ are invariant for any $\mu \in \rea$. This shows that condition $(3)$ is sensible. A tensor satisfying only $(1)$ is certainly not unique. For example to such a tensor can be added invariant expressions such as $A_{abc(i}E^{abc}\,_{j)}$ and $A_{abc(i}\tbt^{abc}\,_{j)}$; the latter expression is self-conjugate, so a tensor satisfying even both $(1)$ and $(2)$ need not be unique. Taking the self-conjugate part of a tensor satisfying $(1)$ yields a tensor satisfying $(1)$ and $(2)$. The unresolved difficulty is to produce such a tensor which is divergence free. This would almost certainly be a routine consequence of a variational description as in $(4)$.

Fix an AH structure $(\en, [h])$ with conjugate $(\ben, [h])$. In this section (only) let $\A_{\la}$ denote the vector space of completely trace-free tensors having c-weight $\la$ and the algebraic symmetries of the same type as has the self-conjugate conformal Weyl tensor. That is $\om_{ijkl} \in \A$ satisfies $\om_{ijkl} = \om_{[ij][kl]} = \om_{klij}$, $\om_{[ijk]l} = 0$, and $\om_{pij}\,^{p} =0$.

\begin{lemma}\label{bachtransformlemma}
Let $(\en, [h])$ be an AH structure on a manifold of dimension $n$. For $\om \in \A_{\la}$ define operators $\bachl$ and $\lbachl$ taking values in $\Ga(S^{2}_{0}(\ctm)\tensor \A_{\la - 4})$ by 
\begin{align}\label{bachopdefined}
&\bachl(\om)_{ij}  \defeq \nabla^{p}\nabla^{q}\om_{pijq} - (n - 5 + \la)\bar{W}^{pq}\om_{pijq},\\
&\lbachl(\om)_{ij} \defeq \bt^{abc}\nabla_{a}\om_{bijc} + 2nE^{pq}\om_{pijq} - 2\bt_{(i}\,^{ab}\nabla^{p}\om_{j)abp},
\end{align}
The skew part of $\bachl$ is $\bachl(\om)_{[ij]} = -\tfrac{1}{2}W^{ab}\,_{[i}\,^{p}\om_{j]pab}$, which is invariant for all $\la$. 
If $(\en, [h])$ and $(\ten, [h])$ are AH structures generating the same Codazzi projective structure and having aligned representatives related by $\tnabla - \nabla = 2\al_{(i}\delta_{j)}\,^{k} - H_{ij}\al^{k}$, then 
\begin{align}
\label{bachtransform}
&\tilde{\bachl}(\om)_{ij} = \bachl(\om)_{ij}  + (n-6+\la)\left(2\al^{p}\nabla^{q}\om_{p(ij)q} + (n-5+\la)\al^{p}\al^{q}\om_{p(ij)q} \right),\\
\label{lbachtransform}
&\tilde{\lbachl}(\om)_{ij} = \lbachl(\om)_{ij} + (n-6+\la)\left(\bt^{abc}\al_{a}\om_{bijc} - 2\bt_{(i}\,^{ab}\om_{j)abp}\al^{p}\right).
\end{align}
In particular on $\A_{6-n}$ any linear combination of the operators $\bachi$ and $\lbachi$ depends only on the underlying Codazzi projective structure.
\end{lemma}
\begin{proof}
From the Ricci identity, $2Q_{[ij]}  = (4-n)F_{ij}$, and \eqref{rintermsofw} there follows
\begin{align*}
2\nabla^{p}\nabla^{q}\om_{ijpq} = (n-4+\la)F^{pq}\om_{ijpq} + 2R^{ab}\,_{[i}\,^{p}\om_{j]pab} = (n-5+\la)F^{pq}\om_{ijpq} + 2W^{ab}\,_{[i}\,^{p}\om_{j]pab},
\end{align*}
from which follows $\bachl(\om)_{[ij]} = -\tfrac{1}{2}W^{ab}\,_{[i}\,^{p}\om_{j]pab}$. From this it is evident that the skew-part $\bachl(\om)_{[ij]}$ is invariant.
Keeping in mind that $\nabla^{p}\om_{ijkp}$ has c-weight $(\la - 2)$, straightforward computations show
\begin{align}
\label{bl0}\begin{split}
\tnabla_{p}\om_{ijkl} &= \nabla_{p}\om_{ijkl} + (\la - 4)\al_{p}\om_{ijkl} + 2\al_{[i}\om_{j]pkl} + 2\al_{[k}\om_{l]pij}\\ 
& + 2\ga^{q}\left(H_{p[j}\om_{i]qkl} + H_{p[l}\om_{k]qij} \right),
\end{split}\\
\label{bl1}\tnabla^{p}\om_{ijkp} &= \nabla^{p}\om_{ijkp} + (n-5+\la)\al^{p}\om_{ijkp},\\
\label{bl2}\begin{split}
\tnabla^{p}\tnabla^{q}\om_{pijq}&  = \nabla^{p}\nabla^{q}\om_{pijq} + (n-5+\la)(\nabla^{p}\al^{q} + (n-7+\la)\al^{p}\al^{q})\om_{pijq}\\ & + 2(n-6+\la)\al^{p}\nabla^{q}\om_{p(ij)q}, 
\end{split}
\\
\label{bl3}\tilde{W}^{pq}\om_{pijq} & = W^{pq}\om_{pijq} + (\nabla^{p}\al^{q} - \al^{p}\al^{q} + \al^{a}\bt_{a}\,^{pq})\om_{pijq},\\
\label{bl4} \tilde{E}^{pq}\om_{pijq} & = E^{pq}\om_{pijq} + \tfrac{1}{2}\al^{a}\bt^{abc}\al_{a}\om_{pijq},\\
\label{bl5}\bt^{abc}\tnabla_{a}\om_{bijc}& = \bt^{abc}\nabla_{a}\om_{bijc}  + (\la - 6)\bt^{abc}\al_{a}\om_{bijc} + 2\bt_{(i}\,^{ab}\om_{j)abp}\al^{p},\\
\label{bl6}\bt_{(i}\,^{ab}\tnabla^{p}\om_{j)abp} &= \bt_{(i}\,^{ab}\nabla^{p}\om_{j)abp} + (n-5+\la)\bt_{(i}\,^{ab}\om_{j)abp}\al^{p}.
\end{align}
From $\bar{W}_{ij} = W_{ij} - 2E_{ij}$ and \eqref{bl2}, \eqref{bl3}, and \eqref{bl4} there follows \eqref{bachtransform}. From \eqref{bl4}, \eqref{bl5}, and \eqref{bl6} there follows \eqref{lbachtransform}. Hence when $\la = 2$ the operators $\bachi$ and $\lbachi$ depend only on the underlying Codazzi projective structure. 
\end{proof}

The operator $\bach \defeq \bachtwo$ is what is usually called the \textit{Bach operator}. When $n = 4$ it makes sense to apply the Bach operator to the self-conjugate Weyl tensor $A_{ijkl}$. By Lemma \ref{bachtransformlemma} there holds
\begin{align}\label{bachskew}
\bach(A)_{[ij]} = -\tfrac{1}{2}W^{ab}\,_{[i}\,^{p}A_{j]pab} =-\tfrac{1}{2}E^{ab}\,_{[i}\,^{p}A_{j]pab},
\end{align}
so that $\bach(A)_{ij} + \tfrac{1}{2}E^{ab}\,_{[i}\,^{p}A_{j]pab}$ is symmetric. 

\begin{lemma}
On a $4$-manifold, there holds $\bach(A) = 0$ for a closed Einstein AH structure with self-conjugate curvature.
\end{lemma}
\begin{proof}
This follows from the definition of $\bach$, \eqref{adiffbianchi}, and Lemma \ref{faradaycottonlemma}.
\end{proof}

\begin{lemma}\label{bachweyllemma}
If $(\en, [h])$ is a Weyl structure with self-conjugate Weyl tensor $A_{ijkl}$ on an $n$-manifold then the self-conjugate tensor $\bach_{ij}$ defined by $2(3-n)\bach_{ij} \defeq \bach(A)_{(ij)} + \bar{\bach}(A)_{(ij)}$ satisfies $\bach_{ij} = \nabla^{p}A_{pij} + A^{pq}A_{pijq} - \tfrac{1}{4}F^{pq}A_{ijpq}$ and equals the usual Bach tensor.
\end{lemma}
\begin{proof}
For a Weyl structure, \eqref{adiffbianchi} shows $\nabla^{p}\nabla^{q}A_{pijq} = (3-n)\nabla^{p}A_{pij}$, so that $\tfrac{1}{3-n}\bach(A)_{ij} = \nabla^{p}A_{pij} + A^{pq}A_{pijq} - \tfrac{1}{4}F^{pq}A_{ijpq}$. By \eqref{bachskew} there holds $\bach(A)_{[ij]} = 0$, and the claim follows.
\end{proof}

Given an AH structure $(\en, [h])$ on a $4$-manifold, let $\brtb{\bach}_{ij}$ be the Bach tensor of the associated Weyl structure $([\lcp], [h])$. By Lemma \ref{bachtransformlemma} and Remark \ref{conjugateremark} if $\brtb{\bach}_{ij}$ is viewed as a tensor associated to $(\en, [h])$, then $\brtb{\bach}_{ij}$ does not change when $(\en, [h])$ is varied within the Codazzi projective structure which it generates. Since $\lcp^{p}\brtb{\bach}_{ip} = 0$, $2\nabla^{p}\brtb{\bach}_{ip} = 2\lcp^{p}\brtb{\bach}_{ip}+ \bt_{i}\,^{pq}\brtb{\bach}_{pq} = \bt_{i}\,^{pq}\brtb{\bach}_{pq} $, which will not generally be zero. As a generalized Bach tensor of $(\en, [h])$ one would like a tensor which is divergence free in dimension $4$, so $\brtb{\bach}_{ij}$ is not the desired generalization; it needs to be modified by adding a term invariant in dimension $4$ and having divergence equal to $-\bt_{i}\,^{pq}\brtb{\bach}_{pq} $. Evidently finding such an expression by brute force computation would be both difficult and unilluminating.

\def\cprime{$'$} \def\cprime{$'$} \def\cprime{$'$} \def\cprime{$'$}
  \def\cprime{$'$} \def\cprime{$'$} \def\cprime{$'$} \def\cprime{$'$}
  \def\cprime{$'$} \def\cprime{$'$} \def\cprime{$'$} \def\cprime{$'$}
  \def\dbar{\leavevmode\hbox to 0pt{\hskip.2ex \accent"16\hss}d}
  \def\cprime{$'$} \def\cprime{$'$} \def\cprime{$'$} \def\cprime{$'$}
  \def\cprime{$'$} \def\cprime{$'$} \def\cprime{$'$} \def\cprime{$'$}
  \def\cprime{$'$} \def\cprime{$'$}
\providecommand{\bysame}{\leavevmode\hbox to3em{\hrulefill}\thinspace}
\providecommand{\MR}{\relax\ifhmode\unskip\space\fi MR }
\providecommand{\MRhref}[2]{%
  \href{http://www.ams.org/mathscinet-getitem?mr=#1}{#2}
}
\providecommand{\href}[2]{#2}


\begin{thebibliography}{100}

\bibitem{Alexandrov-Ivanov}
B.~Alexandrov and S.~Ivanov, \emph{Weyl structures with positive {R}icci
  tensor}, Differential Geom. Appl. \textbf{18} (2003), no.~3, 343--350.

\bibitem{Armstrong-einstein}
S.~Armstrong, \emph{Generalised {E}instein condition and cone construction for
  parabolic geometries},
  \href{http://www.arXiv.org/abs/0705.2390}{arXiv:0705.2390}.

\bibitem{Aubin}
T.~Aubin, \emph{Some nonlinear problems in {R}iemannian geometry}, Springer
  Monographs in Mathematics, Springer-Verlag, Berlin, 1998.

\bibitem{Baekler-Boulanger-Hehl}
P.~Baekler, N.~Boulanger, and F.~W. Hehl, \emph{Linear connections with a
  propagating spin-3 field in gravity}, Phys. Rev. D (3) \textbf{74} (2006),
  no.~12, 125009, 20.

\bibitem{Bailey-Eastwood-Gover}
T.~N. Bailey, M.~G. Eastwood, and A.~R. Gover, \emph{Thomas's structure bundle
  for conformal, projective and related structures}, Rocky Mountain J. Math.
  \textbf{24} (1994), no.~4, 1191--1217.

\bibitem{Benoist-convexesdivisblesI}
Y.~Benoist, \emph{Convexes divisibles. {I}}, Algebraic groups and arithmetic,
  Tata Inst. Fund. Res., Mumbai, 2004, pp.~339--374.

\bibitem{Benoist-survey}
\bysame, \emph{A survey on divisible convex sets}, Geometry, analysis and
  topology of discrete groups, Adv. Lect. Math. (ALM), vol.~6, Int. Press,
  Somerville, MA, 2008, pp.~1--18.

\bibitem{Berger-Ebin}
M.~Berger and D.~Ebin, \emph{Some decompositions of the space of symmetric
  tensors on a {R}iemannian manifold}, J. Differential Geometry \textbf{3}
  (1969), 379--392.

\bibitem{Bershadsky-Cecotti-Ooguri-Vafa}
M.~Bershadsky, S.~Cecotti, H.~Ooguri, and C.~Vafa, \emph{Kodaira-{S}pencer
  theory of gravity and exact results for quantum string amplitudes}, Comm.
  Math. Phys. \textbf{165} (1994), no.~2, 311--427.

\bibitem{Besse}
A.~L. Besse, \emph{Einstein manifolds}, Ergebnisse der Mathematik und ihrer
  Grenzgebiete (3), vol.~10, Springer-Verlag, Berlin, 1987.

\bibitem{Bordemann}
M.~Bordemann, \emph{Nondegenerate invariant bilinear forms on nonassociative
  algebras}, Acta Math. Univ. Comenian. (N.S.) \textbf{66} (1997), no.~2,
  151--201.

\bibitem{Boulanger-Kirsch}
N.~Boulanger and I.~Kirsch, \emph{Higgs mechanism for gravity. {II}. higher
  spin connections}, Phys. Rev. D (3) \textbf{73} (2006), no.~12, 124023.

\bibitem{Boulware-Deser-Kay}
D.~G. Boulware, S.~Deser, and J.~H. Kay, \emph{Supergravity from
  self-interaction}, Phys. A \textbf{96} (1979), no.~1-2, 141--162.

\bibitem{Branson-steinweiss}
T.~Branson, \emph{Stein-{W}eiss operators and ellipticity}, J. Funct. Anal.
  \textbf{151} (1997), no.~2, 334--383.

\bibitem{Branson-kato}
\bysame, \emph{Kato constants in {R}iemannian geometry}, Math. Res. Lett.
  \textbf{7} (2000), no.~2-3, 245--261.

\bibitem{Calabi-improper}
E.~Calabi, \emph{Improper affine hyperspheres of convex type and a
  generalization of a theorem by {K}. {J}\"orgens}, Michigan Math. J.
  \textbf{5} (1958), 105--126.

\bibitem{Calabi-constantcurvature}
\bysame, \emph{On compact, {R}iemannian manifolds with constant curvature.
  {I}}, Proc. {S}ympos. {P}ure {M}ath., {V}ol. {III}, American Mathematical
  Society, Providence, R.I., 1961, pp.~155--180.

\bibitem{Calabi-bernsteinproblems}
\bysame, \emph{Examples of {B}ernstein problems for some nonlinear equations},
  Global {A}nalysis ({P}roc. {S}ympos. {P}ure {M}ath., {V}ol. {XV}, {B}erkeley,
  {C}alif., 1968), Amer. Math. Soc., Providence, R.I., 1970, pp.~223--230.

\bibitem{Calabi-completeaffine}
\bysame, \emph{Complete affine hyperspheres. {I}}, Symposia Mathematica, Vol. X
  (Convegno di Geometria Differenziale, INDAM, Rome, 1971), Academic Press,
  London, 1972, pp.~19--38.

\bibitem{Calabi-affinelyinvariant}
\bysame, \emph{Hypersurfaces with maximal affinely invariant area}, Amer. J.
  Math. \textbf{104} (1982), no.~1, 91--126.

\bibitem{Calabi-affinemaximal}
\bysame, \emph{Convex affine maximal surfaces}, Results Math. \textbf{13}
  (1988), no.~3-4, 199--223, Reprinted in {\it Affine Differentialgeometrie}
  [(Oberwolfach, 1986), 199--223, Tech. Univ. Berlin, Berlin, 1988].

\bibitem{Calderbank-mobius}
D.~M.~J. Calderbank, \emph{M\"obius structures and two-dimensional
  {E}instein-{W}eyl geometry}, J. Reine Angew. Math. \textbf{504} (1998),
  37--53.

\bibitem{Calderbank-faraday}
\bysame, \emph{The {F}araday 2-form in {E}instein-{W}eyl geometry}, Math.
  Scand. \textbf{89} (2001), no.~1, 97--116.

\bibitem{Calderbank-Diemer}
D.~M.~J. Calderbank and T.~Diemer, \emph{Differential invariants and curved
  {B}ernstein-{G}elfand-{G}elfand sequences}, J. Reine Angew. Math.
  \textbf{537} (2001), 67--103.

\bibitem{Calderbank-Gauduchon-Herzlich}
D.~M.~J. Calderbank, P.~Gauduchon, and M.~Herzlich, \emph{Refined {K}ato
  inequalities and conformal weights in {R}iemannian geometry}, J. Funct. Anal.
  \textbf{173} (2000), no.~1, 214--255.

\bibitem{Calderbank-Pedersen}
D.~M.~J. Calderbank and H.~Pedersen, \emph{Einstein-{W}eyl geometry}, Surveys
  in differential geometry: essays on Einstein manifolds, Surv. Differ. Geom.,
  VI, Int. Press, Boston, MA, 1999, pp.~387--423.

\bibitem{Calderon-Zygmund}
A.~P. Calder{\'o}n and A.~Zygmund, \emph{On higher gradients of harmonic
  functions}, Studia Math. \textbf{24} (1964), 211--226.

\bibitem{Cap-Gover}
A.~{\v{C}}ap and A.~R. Gover, \emph{Tractor calculi for parabolic geometries},
  Trans. Amer. Math. Soc. \textbf{354} (2002), no.~4, 1511--1548.

\bibitem{Cap-Slovak-Soucek}
A.~{\v{C}}ap, J.~Slov{\'a}k, and V.~Sou{\v{c}}ek,
  \emph{Bernstein-{G}elfand-{G}elfand sequences}, Ann. of Math. (2)
  \textbf{154} (2001), no.~1, 97--113.

\bibitem{Cartan-cubic}
E.~Cartan, \emph{Sur des familles remarquables d'hypersurfaces
  isoparam\'etriques dans les espaces sph\'eriques}, Math. Z. \textbf{45}
  (1939), 335--367.

\bibitem{Cecil}
T.~E. Cecil, \emph{Isoparametric and {D}upin hypersurfaces}, SIGMA Symmetry
  Integrability Geom. Methods Appl. \textbf{4} (2008), Paper 062, 28.

\bibitem{Chen-Kontsevich-Schwarz}
Y.~Chen, M.~Kontsevich, and A.~Schwarz, \emph{Symmetries of {WDVV} equations},
  Nuclear Phys. B \textbf{730} (2005), no.~3, 352--363.

\bibitem{Cheng-Yau-maximalspacelike}
S.~Y. Cheng and S.~T. Yau, \emph{Maximal space-like hypersurfaces in the
  {L}orentz-{M}inkowski spaces}, Ann. of Math. (2) \textbf{104} (1976), no.~3,
  407--419.

\bibitem{Cheng-Yau-mongeampere}
\bysame, \emph{On the regularity of the {M}onge-{A}mp\`ere equation {${\rm
  det}(\partial^{2} u/\partial x\sb{i}\partial x\sb{j})=F(x,u)$}}, Comm. Pure
  Appl. Math. \textbf{30} (1977), no.~1, 41--68.

\bibitem{Cheng-Yau-realmongeampere}
\bysame, \emph{The real {M}onge-{A}mp\`ere equation and affine flat
  structures}, Proceedings of the 1980 Beijing Symposium on Differential
  Geometry and Differential Equations, Vol. 1, 2, 3 (Beijing, 1980) (Beijing),
  Science Press, 1982, pp.~339--370.

\bibitem{Cheng-Yau-affinehyperspheresI}
\bysame, \emph{Complete affine hypersurfaces. {I}. {T}he completeness of affine
  metrics}, Comm. Pure Appl. Math. \textbf{39} (1986), no.~6, 839--866.

\bibitem{Ciliberto-Russo-Simis}
C.~Ciliberto, F.~Russo, and A.~Simis, \emph{Homaloidal hypersurfaces and
  hypersurfaces with vanishing {H}essian}, Adv. Math. \textbf{218} (2008),
  no.~6, 1759--1805.

\bibitem{Couty}
R.~Couty, \emph{Transformations infinit\'esimales projectives}, C. R. Acad.
  Sci. Paris \textbf{247} (1958), 804--806.

\bibitem{Curtright-massless}
T.~Curtright, \emph{{Massless field supermultiplets with arbitrary spin}},
  Physics Letters B \textbf{85} (1979), 219--224.

\bibitem{Dazord}
P.~Dazord, \emph{Sur la g\'eom\'etrie des sous-fibr\'es et des feuilletages
  lagrangiens}, Ann. Sci. \'Ecole Norm. Sup. (4) \textbf{14} (1981), no.~4,
  465--480 (1982).

\bibitem{deWit-Freedman}
B.~de~Wit and D.~Z. Freedman, \emph{Systematics of higher-spin gauge fields},
  Phys. Rev. D (3) \textbf{21} (1980), no.~2, 358--367.

\bibitem{Delanoe}
P.~Delano{\"e}, \emph{Remarques sur les vari\'et\'es localement hessiennes},
  Osaka J. Math. \textbf{26} (1989), no.~1, 65--69.

\bibitem{Deser}
S.~Deser, \emph{Inequivalence of first- and second-order formulations in
  {$D=2$} gravity models}, Found. Phys. \textbf{26} (1996), no.~5, 617--621.

\bibitem{Deser-Waldron-nullpropagation}
S.~Deser and A.~Waldron, \emph{{Null propagation of partially massless higher
  spins in (A)dS and cosmological constant speculations}}, Phys. Lett.
  \textbf{B513} (2001), 137--141.

\bibitem{Deser-Waldron-conformalinvariance}
\bysame, \emph{Conformal invariance of partially massless higher spins}, Phys.
  Lett. B \textbf{603} (2004), no.~1-2, 30--34.

\bibitem{Deser-Waldron-partiallymassless}
\bysame, \emph{Partially massless spin-2 electrodynamics}, Phys. Rev. D (3)
  \textbf{74} (2006), no.~8, 084036, 7.

\bibitem{Deser-Zumino}
S.~Deser and B.~Zumino, \emph{Consistent supergravity}, Phys. Lett. B
  \textbf{62} (1976), no.~3, 335--337.

\bibitem{Dillen-Nomizu-Vranken}
F.~Dillen, K.~Nomizu, and L.~Vranken, \emph{Conjugate connections and {R}adon's
  theorem in affine differential geometry}, Monatsh. Math. \textbf{109} (1990),
  no.~3, 221--235.

\bibitem{Dolan-Nappi-Witten}
L.~Dolan, C.~R. Nappi, and E.~Witten, \emph{Conformal operators for partially
  massless states}, J. High Energy Phys. (2001), no.~10, Paper 16, 13.

\bibitem{Dolgachev-polarcremona}
I.~V. Dolgachev, \emph{Polar {C}remona transformations}, Michigan Math. J.
  \textbf{48} (2000), 191--202.

\bibitem{Dolgachev-topics}
\bysame, \emph{Topics in classical algebraic geometry},
  \href{http://www.math.lsa.umich.edu/~idolga/lecturenotes.html}{book draft},
  2010.

\bibitem{Dong-Griess-Lam}
C.~Dong, R.~L. Griess, Jr., and C.~H. Lam, \emph{Uniqueness results for the
  moonshine vertex operator algebra}, Amer. J. Math. \textbf{129} (2007),
  no.~2, 583--609.

\bibitem{Eastwood-laplacian}
M.~Eastwood, \emph{Higher symmetries of the {L}aplacian}, Ann. of Math. (2)
  \textbf{161} (2005), no.~3, 1645--1665.

\bibitem{Eastwood-Tod}
M.~G. Eastwood and K.~P. Tod, \emph{Local constraints on {E}instein-{W}eyl
  geometries}, J. Reine Angew. Math. \textbf{491} (1997), 183--198.

\bibitem{Elduque-Myung}
A.~Elduque and H.~C. Myung, \emph{Composition algebras satisfying the flexible
  law}, Comm. Algebra \textbf{32} (2004), no.~5, 1997--2013.

\bibitem{Fefferman-Graham-ambient}
C.~Fefferman and C.~R. Graham, \emph{The ambient metric},
  \href{http://arxiv.org/abs/0710.0919}{arXiv:0710.0919}.

\bibitem{Fefferman-Graham}
C.~Fefferman and C.~R. Graham, \emph{Conformal invariants}, Ast\'erisque
  (1985), no.~Numero Hors Serie, 95--116, The mathematical heritage of \'Elie
  Cartan (Lyon, 1984).

\bibitem{Fierz-Pauli}
M.~Fierz and W.~Pauli, \emph{On relativistic wave equations for particles of
  arbitrary spin in an electromagnetic field}, Proc. Roy. Soc. London, Ser. A.
  \textbf{173} (1939), 211--232.

\bibitem{Fox-2dahs}
D.~J.~F. Fox, \emph{Einstein-like geometric structures on surfaces}, preprint,
  2010.

\bibitem{Freed}
D.~S. Freed, \emph{Special {K}\"ahler manifolds}, Comm. Math. Phys.
  \textbf{203} (1999), no.~1, 31--52.

\bibitem{Fronsdal}
C.~Fronsdal, \emph{Massless fields with integer spin}, Phys. Rev. D (3)
  \textbf{18} (1978), no.~10, 3624--3629.

\bibitem{Gauduchon}
P.~Gauduchon, \emph{La {$1$}-forme de torsion d'une vari\'et\'e hermitienne
  compacte}, Math. Ann. \textbf{267} (1984), no.~4, 495--518.

\bibitem{Gigena}
S.~Gigena, \emph{On a conjecture by {E}. {C}alabi}, Geom. Dedicata \textbf{11}
  (1981), no.~4, 387--396.

\bibitem{Goldman-convex}
W.~M. Goldman, \emph{Convex real projective structures on compact surfaces}, J.
  Differential Geom. \textbf{31} (1990), no.~3, 791--845.

\bibitem{Graham-Hirachi-obstruction}
C.~R. Graham and K.~Hirachi, \emph{The ambient obstruction tensor and
  {$Q$}-curvature}, Ad{S}/{CFT} correspondence: {E}instein metrics and their
  conformal boundaries, IRMA Lect. Math. Theor. Phys., vol.~8, Eur. Math. Soc.,
  Z\"urich, 2005, pp.~59--71.

\bibitem{Gromov-Thurston}
M.~Gromov and W.~Thurston, \emph{Pinching constants for hyperbolic manifolds},
  Invent. Math. \textbf{89} (1987), no.~1, 1--12.

\bibitem{Hamilton}
R.~S. Hamilton, \emph{Three-manifolds with positive {R}icci curvature}, J.
  Differential Geom. \textbf{17} (1982), no.~2, 255--306.

\bibitem{Helgason}
S.~Helgason, \emph{Differential geometry and symmetric spaces}, Pure and
  Applied Mathematics, Vol. XII, Academic Press, New York, 1962.

\bibitem{Hess}
H.~Hess, \emph{Connections on symplectic manifolds and geometric quantization},
  Differential geometrical methods in mathematical physics (Proc. Conf.,
  Aix-en-Provence/Salamanca, 1979), Lecture Notes in Math., vol. 836, Springer,
  Berlin, 1980, pp.~153--166.

\bibitem{Hitchin-frobenius}
N.~Hitchin, \emph{Frobenius manifolds}, Gauge theory and symplectic geometry
  ({M}ontreal, {PQ}, 1995), NATO Adv. Sci. Inst. Ser. C Math. Phys. Sci., vol.
  488, Kluwer Acad. Publ., Dordrecht, 1997, pp.~69--112.

\bibitem{Hitchin}
N.~J. Hitchin, \emph{Complex manifolds and {E}instein's equations}, Twistor
  geometry and nonlinear systems (Primorsko, 1980), Springer, Berlin, 1982,
  pp.~73--99.

\bibitem{Hitchin-speciallagrangian}
N.~J. Hitchin, \emph{The moduli space of special {L}agrangian submanifolds},
  Ann. Scuola Norm. Sup. Pisa Cl. Sci. (4) \textbf{25} (1997), no.~3-4,
  503--515.

\bibitem{Hou-Deng-Kaneyuki-Nishiyama}
Z.~Hou, S.~Deng, S.~Kaneyuki, and K.~Nishiyama, \emph{Dipolarizations in
  semisimple {L}ie algebras and homogeneous para-{K}\"ahler manifolds}, J. Lie
  Theory \textbf{9} (1999), no.~1, 215--232.

\bibitem{Ionescu}
L.~M. Ionescu, \emph{Nonassociative algebras: a framework for differential
  geometry}, Int. J. Math. Math. Sci. (2003), no.~60, 3777--3795.

\bibitem{Jacobson}
N.~Jacobson, \emph{Lie algebras}, Interscience Tracts in Pure and Applied
  Mathematics, No. 10, Interscience Publishers, New York-London, 1962.

\bibitem{Jorgens}
K.~J{\"o}rgens, \emph{\"{U}ber die {L}\"osungen der {D}ifferentialgleichung
  {$rt-s^{2}=1$}}, Math. Ann. \textbf{127} (1954), 130--134.

\bibitem{Verbitsky-Kaledin}
D.~Kaledin and M.~Verbitsky, \emph{Non-{H}ermitian {Y}ang-{M}ills connections},
  Selecta Math. (N.S.) \textbf{4} (1998), no.~2, 279--320.

\bibitem{Kaneyuki-classification}
S.~Kaneyuki, \emph{On classification of para-{H}ermitian symmetric spaces},
  Tokyo J. Math. \textbf{8} (1985), no.~2, 473--482.

\bibitem{Kaneyuki-compactification}
\bysame, \emph{Compactification of parahermitian symmetric spaces and its
  applications. {II}. {S}tratifications and automorphism groups}, J. Lie Theory
  \textbf{13} (2003), no.~2, 535--563.

\bibitem{Kaneyuki-Kozai}
S.~Kaneyuki and M.~Kozai, \emph{Paracomplex structures and affine symmetric
  spaces}, Tokyo J. Math. \textbf{8} (1985), no.~1, 81--98.

\bibitem{Kapovich}
M.~Kapovich, \emph{Convex projective structures on {G}romov-{T}hurston
  manifolds}, Geom. Topol. \textbf{11} (2007), 1777--1830.

\bibitem{Kinyon-Sagle-quadratic}
M.~K. Kinyon and A.~A. Sagle, \emph{Quadratic dynamical systems and algebras},
  J. Differential Equations \textbf{117} (1995), no.~1, 67--126.

\bibitem{Kinyon-Sagle}
\bysame, \emph{Nahm algebras}, J. Algebra \textbf{247} (2002), no.~2, 269--294.

\bibitem{Kirsch}
I.~Kirsch, \emph{Higgs mechanism for gravity}, Phys. Rev. D (3) \textbf{72}
  (2005), no.~2, 024001.

\bibitem{Knus-Merkurjev-Rost-Tignol}
M.-A. Knus, A.~Merkurjev, M.~Rost, and J.-P. Tignol, \emph{The book of
  involutions}, American Mathematical Society Colloquium Publications, vol.~44,
  American Mathematical Society, Providence, RI, 1998.

\bibitem{Kobayashi-fixedpoints}
S.~Kobayashi, \emph{Fixed points of isometries}, Nagoya Math. J. \textbf{13}
  (1958), 63--68.

\bibitem{Kobayashi-holomorphicsymmetric}
\bysame, \emph{The first {C}hern class and holomorphic symmetric tensor
  fields}, J. Math. Soc. Japan \textbf{32} (1980), no.~2, 325--329.

\bibitem{Kobayashi-holomorphictensor}
\bysame, \emph{First {C}hern class and holomorphic tensor fields}, Nagoya Math.
  J. \textbf{77} (1980), 5--11.

\bibitem{Koiso-nondeformability}
N.~Koiso, \emph{Nondeformability of {E}instein metrics}, Osaka J. Math.
  \textbf{15} (1978), no.~2, 419--433.

\bibitem{Kostant-triples}
B.~Kostant, \emph{The principal three-dimensional subgroup and the {B}etti
  numbers of a complex simple {L}ie group}, Amer. J. Math. \textbf{81} (1959),
  973--1032.

\bibitem{Labourie-flatprojective}
F.~Labourie, \emph{Flat projective structures on surfaces and cubic holomorphic
  differentials}, Pure Appl. Math. Q. \textbf{3} (2007), no.~4, part 1,
  1057--1099.

\bibitem{Lebrun-Mason-einsteinweyl}
C.~LeBrun and L.~J. Mason, \emph{The {E}instein-{W}eyl equations, scattering
  maps, and holomorphic disks}, Math. Res. Lett. \textbf{16} (2009), no.~2,
  291--301.

\bibitem{Lee-Parker}
J.~M. Lee and T.~H. Parker, \emph{The {Y}amabe problem}, Bull. Amer. Math. Soc.
  (N.S.) \textbf{17} (1987), no.~1, 37--91.

\bibitem{Li-1}
A.~M. Li, \emph{Calabi conjecture on hyperbolic affine hyperspheres}, Math. Z.
  \textbf{203} (1990), no.~3, 483--491.

\bibitem{Li-2}
\bysame, \emph{Calabi conjecture on hyperbolic affine hyperspheres. {II}},
  Math. Ann. \textbf{293} (1992), no.~3, 485--493.

\bibitem{Libermann}
P.~Libermann, \emph{Sur le probl\`eme d'\'equivalence de certaines structures
  infinit\'esimales}, Ann. Mat. Pura Appl. (4) \textbf{36} (1954), 27--120.

\bibitem{Lichnerowicz-propagateurs}
A.~Lichnerowicz, \emph{Propagateurs et commutateurs en relativit\'e
  g\'en\'erale}, Inst. Hautes \'Etudes Sci. Publ. Math. (1961), no.~10, 56.

\bibitem{Loewner-Nirenberg}
C.~Loewner and L.~Nirenberg, \emph{Partial differential equations invariant
  under conformal or projective transformations}, Contributions to analysis (a
  collection of papers dedicated to {L}ipman {B}ers), Academic Press, New York,
  1974, pp.~245--272.

\bibitem{Loftin-survey}
J.~C. Loftin, \emph{Survey on affine spheres},
  \href{http://arxiv.org/abs/0809.1186}{arXiv:0809.1186}.

\bibitem{Loftin-affinespheres}
\bysame, \emph{Affine spheres and convex {$\Bbb{RP}^{n}$}-manifolds}, Amer. J.
  Math. \textbf{123} (2001), no.~2, 255--274.

\bibitem{Loftin-affinekahler}
\bysame, \emph{Affine spheres and {K}\"ahler-{E}instein metrics}, Math. Res.
  Lett. \textbf{9} (2002), no.~4, 425--432.

\bibitem{Loftin-riemannian}
\bysame, \emph{Riemannian metrics on locally projectively flat manifolds},
  Amer. J. Math. \textbf{124} (2002), no.~3, 595--609.

\bibitem{Loftin-cubic}
\bysame, \emph{Flat metrics, cubic differentials and limits of projective
  holonomies}, Geom. Dedicata \textbf{128} (2007), 97--106.

\bibitem{Loftin-Yau-Zaslow}
J.~C. Loftin, S.~T. Yau, and E.~Zaslow, \emph{Affine manifolds, {SYZ} geometry
  and the ``{Y}'' vertex}, J. Differential Geom. \textbf{71} (2005), no.~1,
  129--158.

\bibitem{Loftin-Yau-Zaslow-erratum}
\bysame, \emph{Erratum to affine manifolds, {SYZ} geometry and the ``{Y}''
  vertex}, \href{http://andromeda.rutgers.edu/~loftin/}{Loftin's web page},
  2008.

\bibitem{Manin-frobenius}
Y.~I. Manin, \emph{Frobenius manifolds, quantum cohomology, and moduli spaces},
  American Mathematical Society Colloquium Publications, vol.~47, American
  Mathematical Society, Providence, RI, 1999.

\bibitem{Markus}
L.~Markus, \emph{Quadratic differential equations and non-associative
  algebras}, Contributions to the theory of nonlinear oscillations, {V}ol. {V},
  Princeton Univ. Press, Princeton, N.J., 1960, pp.~185--213.

\bibitem{Nomizu-Sasaki}
K.~Nomizu and T.~Sasaki, \emph{Affine differential geometry}, Cambridge Tracts
  in Mathematics, vol. 111, Cambridge University Press, Cambridge, 1994.

\bibitem{Oliynyk-Suneeta-Woolgar}
T.~Oliynyk, V.~Suneeta, and E.~Woolgar, \emph{A gradient flow for worldsheet
  nonlinear sigma models}, Nuclear Phys. B \textbf{739} (2006), no.~3,
  441--458.

\bibitem{Pedersen-Swann}
H.~Pedersen and A.~Swann, \emph{Einstein-{W}eyl geometry, the {B}ach tensor and
  conformal scalar curvature}, J. Reine Angew. Math. \textbf{441} (1993),
  99--113.

\bibitem{Pedersen-Tod}
H.~Pedersen and K.~P. Tod, \emph{Three-dimensional {E}instein-{W}eyl geometry},
  Adv. Math. \textbf{97} (1993), no.~1, 74--109.

\bibitem{Pogorelov}
A.~V. Pogorelov, \emph{On the improper convex affine hyperspheres}, Geometriae
  Dedicata \textbf{1} (1972), no.~1, 33--46.

\bibitem{Rani-Edgar-Barnes}
R.~Rani, S.~B. Edgar, and A.~Barnes, \emph{Killing tensors and conformal
  {K}illing tensors from conformal {K}illing vectors}, Classical Quantum
  Gravity \textbf{20} (2003), no.~11, 1929--1942.

\bibitem{Sasaki}
T.~Sasaki, \emph{Hyperbolic affine hyperspheres}, Nagoya Math. J. \textbf{77}
  (1980), 107--123.

\bibitem{Sasakura}
N.~Sasakura, \emph{Tensor model and dynamical generation of commutative
  non-associative fuzzy spaces}, Classical Quantum Gravity \textbf{23} (2006),
  no.~17, 5397--5416.

\bibitem{Sasakura-emergent}
\bysame, \emph{{Emergent General Relativity in Fuzzy Spaces from Tensor
  Models}}, Progress of Theoretical Physics \textbf{119} (2008), 1029--1040.

\bibitem{Sato-Kimura}
M.~Sato and T.~Kimura, \emph{A classification of irreducible prehomogeneous
  vector spaces and their relative invariants}, Nagoya Math. J. \textbf{65}
  (1977), 1--155.

\bibitem{Schafer}
R.~D. Schafer, \emph{Structure and representation of nonassociative algebras},
  Bull. Amer. Math. Soc. \textbf{61} (1955), 469--484.

\bibitem{Schafer-book}
\bysame, \emph{An introduction to nonassociative algebras}, Pure and Applied
  Mathematics, Vol. 22, Academic Press, New York, 1966.

\bibitem{Schoen-Simon-Yau}
R.~Schoen, L.~Simon, and S.~T. Yau, \emph{Curvature estimates for minimal
  hypersurfaces}, Acta Math. \textbf{134} (1975), no.~3-4, 275--288.

\bibitem{Schoen-Yau}
R.~Schoen and S.~T. Yau, \emph{Lectures on differential geometry}, Conference
  Proceedings and Lecture Notes in Geometry and Topology, I, International
  Press, Cambridge, MA, 1994.

\bibitem{Schoen-variationaltheory}
R.~M. Schoen, \emph{Variational theory for the total scalar curvature
  functional for {R}iemannian metrics and related topics}, Topics in calculus
  of variations ({M}ontecatini {T}erme, 1987), Lecture Notes in Math., vol.
  1365, Springer, Berlin, 1989, pp.~120--154.

\bibitem{Semmelmann}
U.~Semmelmann, \emph{Conformal {K}illing forms on {R}iemannian manifolds},
  Math. Z. \textbf{245} (2003), no.~3, 503--527.

\bibitem{Shima}
H.~Shima, \emph{The geometry of {H}essian structures}, World Scientific
  Publishing Co. Pte. Ltd., Hackensack, NJ, 2007.

\bibitem{Simons}
J.~Simons, \emph{Minimal varieties in riemannian manifolds}, Ann. of Math. (2)
  \textbf{88} (1968), 62--105.

\bibitem{Stein-Weiss-harmonic}
E.~M. Stein and G.~Weiss, \emph{On the theory of harmonic functions of several
  variables. {I}. {T}he theory of {$H^{p}$}-spaces}, Acta Math. \textbf{103}
  (1960), 25--62.

\bibitem{Stein-Weiss}
\bysame, \emph{Generalization of the {C}auchy-{R}iemann equations and
  representations of the rotation group}, Amer. J. Math. \textbf{90} (1968),
  163--196.

\bibitem{Streets}
J.~Streets, \emph{Regularity and expanding entropy for connection {R}icci
  flow}, J. Geom. Phys. \textbf{58} (2008), no.~7, 900--912.

\bibitem{Strominger}
A.~Strominger, \emph{Special geometry}, Comm. Math. Phys. \textbf{133} (1990),
  no.~1, 163--180.

\bibitem{Tod-compact}
K.~P. Tod, \emph{Compact {$3$}-dimensional {E}instein-{W}eyl structures}, J.
  London Math. Soc. (2) \textbf{45} (1992), no.~2, 341--351.

\bibitem{Vasiliev-survey}
M.~A. Vasiliev, \emph{Progress in higher spin gauge theories}, Quantization,
  gauge theory, and strings, {V}ol. {I} ({M}oscow, 2000), Sci. World, Moscow,
  2001, pp.~452--471.

\bibitem{Wang-icm}
X.-J. Wang, \emph{Affine maximal hypersurfaces}, Proceedings of the
  International Congress of Mathematicians, Vol. III (Beijing, 2002) (Beijing),
  Higher Ed. Press, 2002, pp.~221--231.

\bibitem{Weinberg-Witten}
S.~Weinberg and E.~Witten, \emph{Limits of massless particles}, Phys. Lett. B
  \textbf{96} (1980), no.~1-2, 59--62.

\bibitem{Weyl}
H.~Weyl, \emph{The classical groups. {T}heir invariants and representations},
  Princeton University Press, Princeton, N.J., 1939.

\bibitem{Woodhouse}
N.~M.~J. Woodhouse, \emph{Killing tensors and the separation of the
  {H}amilton-{J}acobi equation}, Comm. Math. Phys. \textbf{44} (1975), no.~1,
  9--38.

\bibitem{Yano}
K.~Yano, \emph{On harmonic and {K}illing vector fields}, Ann. of Math. (2)
  \textbf{55} (1952), 38--45.

\bibitem{Yau-perspectives}
S.-T. Yau, \emph{Perspectives in geometric analysis}, Surveys in differential
  geometry: Essays in geometry in memory of S. S. Chern., Surv. Differ. Geom.,
  X, International Press, 2006, pp.~275--379.

\end{thebibliography}
\end{document}